\documentclass[10pt]{amsart}
\usepackage{amsthm,amssymb,amscd,amsmath,amsfonts,amssymb,amscd,hyperref,mathrsfs,bbm} 
\usepackage{stmaryrd}
\usepackage{yhmath}
\usepackage[shortlabels]{enumitem}
\usepackage[all]{xy}
\usepackage{hyperref} 					
\usepackage{tikz}

\theoremstyle{plain}
\makeatletter
\def\cal@symb#1|#2{\expandafter\def\csname #2#1\endcsname{\mathcal{#1}}}
\def\calsymbols#1#2{\@for\@tmpz:=#2\do{\expandafter\cal@symb\@tmpz|{#1}}}
\def\bb@symb#1|#2{\expandafter\def\csname #2#1\endcsname{\mathbb{#1}}}
\def\bbsymbols#1#2{\@for\@tmpz:=#2\do{\expandafter\bb@symb\@tmpz|{#1}}}
\def\bold@symb#1|#2{\expandafter\def\csname #2#1\endcsname{\mathbf{#1}}}
\def\boldsymbols#1#2{\@for\@tmpz:=#2\do{\expandafter\bold@symb\@tmpz|{#1}}}
\def\scr@symb#1|#2{\expandafter\def\csname #2#1\endcsname{\mathscr{#1}}}
\def\scrsymbols#1#2{\@for\@tmpz:=#2\do{\expandafter\scr@symb\@tmpz|{#1}}}
\def\frak@symb#1|#2{\expandafter\def\csname #2#1\endcsname{\mathfrak{#1}}}
\def\fraksymbols#1#2{\@for\@tmpz:=#2\do{\expandafter\frak@symb\@tmpz|{#1}}}

\def\dmth@p#1|{\expandafter\let\csname#1\endcsname\relax
  \expandafter\DeclareMathOperator\csname#1\endcsname{#1}}
\def\operators#1{\@for\@tmpz:=#1\do{\expandafter\dmth@p\@tmpz|}}
\makeatother
\calsymbols{c}{A,B,C,D,E,F,G,H,I,J,K,L,M,N,O,P,Q,R,S,T,U,V,W,X,Y,Z}
\bbsymbols{b}{A,B,C,D,E,F,G,H,I,J,K,L,M,N,O,P,Q,R,S,T,U,V,W,X,Y,Z}
\boldsymbols{bf}{r,q,A,B,C,D,E,F,G,H,I,J,K,L,M,N,O,P,Q,R,S,T,U,V,W,X,Y,Z}
\scrsymbols{s}{A,B,C,D,E,F,G,H,I,J,K,L,M,N,O,P,Q,R,S,T,U,V,W,X,Y,Z}
\fraksymbols{fr}{a,b,c,d,e,f,g,h,i,j,k,l,m,n,o,p,q,r,s,t,u,v,w,x,y,z,F,M,O,R,X,Z}

\theoremstyle{definition}
\newtheorem{defn}{Definition}[subsection]
\newtheorem{thm}[defn]{Theorem}
\newtheorem{cor}[defn]{Corollary}
\newtheorem{prop}[defn]{Proposition}

\newtheorem{lem}[defn]{Lemma}
\newtheorem{ex}[defn]{Example}

\newtheorem{rmk}[defn]{Remark}

\newtheorem{hypn}[defn]{Hypothesis}
\newtheorem{notn}[defn]{Notation}

\newtheorem{MainThm}{Theorem}

\newtheorem{MainCor}{Corollary}

\calsymbols{c}{A,B,C,D,E,F,G,H,I,J,K,L,M,N,O,P,Q,R,S,T,U,V,W,X,Y,Z}
\bbsymbols{b}{A,B,C,D,E,F,G,H,I,J,K,L,M,N,O,P,Q,R,S,T,U,V,W,X,Y,Z}
\scrsymbols{s}{A,B,C,D,E,L,M,N,P,R,S,T,U,V,X,Y}
\fraksymbols{f}{g,m,h,t,n,b,p}
\bbsymbols{}{A,R,C,Z,N,Q,F,G,H,X,D}

\DeclareSymbolFont{largesymbols}{OMX}{yhex}{m}{n}
\DeclareMathAccent{\wideparen}{\mathord}{largesymbols}{"F3}
\DeclareMathAlphabet{\mathpzc}{OT1}{pzc}{m}{it}

\operators{Tor,Der,gr,Spec,Proj,MaxSpec,ad,Loc,Stab,Aut,Hom,End,im,coker,Lie,Spf,Sp,Ad,aug,ac,tors,Sym,mod,Homeo,op,qc,Res,alg,GL,dlog,ann,Ch,rk,colim,pre,an,sp,Gr,Supp,rig,Cycl,Gal,Irr,PicCon,Pic,Con,coh,Ab,Aff,PreSh,ind,Ban,pr}

\DeclareMathOperator{\Char}{Char}
\DeclareMathOperator{\Div}{div}

\DeclareMathAlphabet{\mathpzc}{OT1}{pzc}{m}{it}

\newcommand{\h}[1]{\widehat{#1}}

\newcommand{\hK}[1]{\h{#1_K}}

\newcommand{\w}[1]{\wideparen{#1}}

\newcommand{\fr}[1]{\mathfrak{{#1}}}

\newcommand{\ts}[1]{\texorpdfstring{$#1$}{}}
\newcommand{\st}{\mid}
\newcommand\binomb[2]{\genfrac{\{}{\}}{0pt}{}{#1}{#2}}
\newcommand\binoma[2]{\genfrac{\langle}{\rangle}{0pt}{}{#1}{#2}}
\newcommand{\be}{\begin{enumerate}[{(}a{)}]}
\newcommand{\ee}{\end{enumerate}}
\let\le=\leqslant  \let\leq=\leqslant
  \let\geq=\geqslant
\newcommand{\qmb}[1]{\quad\mbox{#1}\quad}

\newcommand{\hsp}{\hspace{0.1cm}}

\newcommand{\Bin}[2]{\langle #1 \vert #2 \rangle}
\newcommand{\Abe}[2]{\widehat{\sE}_{\mathscr{#1},\mathbb{Q}}^{(#2)}}
\newcommand{\Berth}[2]{\widehat{\sD}_{\mathscr{#1},\mathbb{Q}}^{(#2)}}
\newcommand{\Cot}[1]{\overset{\circ}{T}{}^\ast {\mathscr{#1}_0}}
\newcommand{\myceil}[1]{\left \lceil #1 \right \rceil }
\newcommand{\myfloor}[1]{\left \lfloor #1 \right \rfloor }
\newcommand{\sq}[1]{\sqrt{|#1^\times|}}
\newcommand{\uset}[1]{\underset{#1}{\otimes}{}}
\newcommand{\uwset}[1]{\underset{#1}{\w\otimes}{}}

\def\arrowlim#1#2{\mathop{\underset{\scriptstyle #1}{\underset
    {\raisebox{0ex}[0.25ex][-0.5ex]{$#2$}}{\operatorname{lim}}}}}
\newcommand{\invlim}[1][]{\arrowlim{#1}{\longleftarrow}}        
\let\le=\leqslant  \let\leq=\leqslant
  \let\geq=\geqslant

\def\Qp{\Q_p}
\def\Zp{\Z_p}

\begin{document}
\title[Global sections of equivariant line bundles]{Global sections of equivariant line bundles on the $p$-adic upper half plane}
\author{Konstantin Ardakov}
\address{Mathematical Institute\\University of Oxford\\Oxford OX2 6GG}
\email{ardakov@maths.ox.ac.uk}

\author{Simon Wadsley}
\address{Homerton College, Cambridge, CB2 8PH}
\email{S.J.Wadsley@dpmms.cam.ac.uk}
\vspace{-2cm}

\begin{abstract} Let $F$ be a finite extension of $\bQ_p$, let $\Omega_F$ be Drinfeld's upper half-plane over $F$ and let $G^0$ the subgroup of $GL_2(F)$ consisting of elements whose determinant has norm $1$. Let $\sL$ be a torsion $G^0$-equivariant line bundle with connection on $\Omega_F$. We show that the strong dual of $\sL(\Omega_F)$ is an admissible locally $F$-analytic representation of $G^0$ of topological length at most $2$. It is topologically irreducible if and only if the underlying connection on $\sL$ is non-trivial. We give an explicit formula for the length of the strong dual of the space of globally-defined rigid analytic functions on a $G^0$-equivariant finite \'etale rigid analytic covering of $\Omega_F$ with abelian Galois group as an admissible locally $F$-analytic representation of $G^0$.
\end{abstract}

\maketitle
\vspace{-1cm}
\tableofcontents

\section{Introduction}

\subsection{Main results}
Let $p$ be a prime number, let $F$ be a finite field extension of $\bQ_p$ with residue field $k_F$ of size $q$ and let $K$ be a complete discretely valued field extension of $F$. For any rigid $F$-analytic variety $X$, we denote by $X_K := K \times_F X$ its base-change to a rigid $K$-analytic variety. Let $\Omega_F$ be the rigid $F$-analytic $p$-adic upper half plane and let $G^0$ denote the group of $2 \times 2$ matrices with entries in $F$ whose determinant has absolute value $1$. In \cite{ArdWad2023} the authors classified torsion $G^0$-equivariant line bundles with connection on $\Omega:=\Omega_F\times_F K$. As explained there, this was motivated by wanting to understand the first covering in Drinfeld's tower. In this paper we seek to understand the global sections of these line bundles as modules over the locally $F$-analytic distribution algebra $D(G^0,K)$ of $G^0$, or equivalently by \cite[p.176, Definition]{ST} and \cite[Corollary 3.3]{ST2}, the strong duals of these global sections as locally $F$-analytic representations of $G^0$.
\begin{MainThm}\label{ThmA}
Let $\sL$ be a torsion $G^0$-equivariant line bundle with connection on $\Omega$. Then $\sL(\Omega)$ is a coadmissible $D(G^0,K)$-module of length at most $2$. 
Moreover $\sL(\Omega)$ is topologically irreducible as a $D(G^0,K)$-module if and only if the underlying connection on $\sL$ is non-trivial.
\end{MainThm}

Let $\bfC$ denote a complete and algebraically closed field extension of $K$. Theorem \ref{ThmA} has the following consequence. 
\setcounter{MainCor}{1}
\begin{MainCor}\label{ThmB}Let $f : X \to \Omega$ be a $G^0$-equivariant finite \'etale rigid $K$-analytic covering of $\Omega$, with abelian Galois group $\Gamma$ of exponent $e$, where we assume that $K$ contains a primitive $e$-th root of unity. Let $H$ be a closed subgroup of $G^0$ that contains an open subgroup of $\mathbb{SL}_2(F)$. Then $\cO(X)$ is a coadmissible $D(H,K)$-module, whose length is given by
\[ \ell_{D(H,K)} \left(\cO(X)\right) = |\Gamma| + |H \backslash \bP^1(F)| \cdot |\pi_0(X_{\bfC})|.\]
\end{MainCor}
Following \cite[Introduction]{ArdWad2023} we write $\Sigma_1:=\cM_1\times_{\cM_0} \Omega$, defined with respect to some inclusion $\Omega \hookrightarrow \cM_0$ into the level-zero Rapoport-Zink space $\cM_0$, so that for a suitable choice of the ground field $K$, $\Sigma_1$ is a $G^0$-equivariant $\mathbb{F}_{q^2}^\times$-Galois covering of $\Omega$ arising in Drinfeld's tower. It follows from Corollary B that $\cO(\Sigma_1)$ is a coadmissible $D(G^0,K)$-module of length $q^2+q-2$. We note in passing that the coadmissibility part of this statement was already addressed by Patel, Schmidt and Strauch in \cite{PSSJussieu}.

\subsection{Summary of this paper}
\subsubsection{Chapter 2} We consider the rigid analytic affine line $\bA$ equipped with a fixed local coordinate $x$. We view the $K$-algebra $\w\cD(\bA)$ of infinite-order differential operators on $\bA$ that was introduced in \cite{DCapOne} as a `rigid analytic quantisation' of the cotangent bundle $T^\ast \bA$. The choice of coordinate $x$ on $\bA$ gives us a coordinate $(x,y)$ on $T^\ast \bA \cong \bA^2$, and we study various completions of $\w\cD(\bA)$. For every affinoid subdomain $X$ of $\bA$, we consider the `affinoid box' $X \times \{|y| \leq r\} \subseteq T^\ast \bA$ and we show that provided the positive real number $r$ is big enough (depending on $X$) we can define a completion $\cD_r(X)$ of $\w\cD(\bA)$ which we regard as a `rigid analytic quantisation` of this box: $\cD_r(X)$ is isomorphic to $\cO(X \times \{|y| \leq r\})$ as a Banach $\cO(X)$-module. We show that this construction, as well as its overconvergent version $\cD_r^\dag$, satisfies an analogue of Tate's Acyclicity Theorem - see Theorem \ref{NCDagTate}. With an eye on the applications in later chapters, we also consider certain `twisting-automorphisms' $\theta_{u,d}$ of the algebra of finite-order differential operators $\cD(X)$ that are morally given by the conjugation by a $d$-th root of a unit $u \in \cO(X)$, and we show in Theorem \ref{ThetaU} that these automorphisms extend to $\cD_r(X)$ whenever $d$ is coprime to $p$.

\subsubsection{Chapter 3}  In this more technical section, we establish Noetherianity of the rings $\cD_r(X)$ at Corollary \ref{DrXNoeth}. We show in Theorem \ref{changeofbase} that $\cD_r(Y)$ is (abstractly) flat as a $\cD_r(X)$-module on both sides whenever $Y$ is an affinoid subdomain $X$ such that both $\cD_r(X)$ and $\cD_r(Y)$ are well-defined. We also show in Theorem \ref{flatcotangent} that $\cD_r(X)$ is (abstractly) flat as a $\cD_s(X)$-module whenever $s > r$. These results further support the `rigid-analytic quantisation' intuition: more general precise conjectures in the direction of this philosophy can be found at \cite[Chapter 5]{Ciappara}. The proofs in Chapter 3 borrow heavily from the earlier work of Berthelot \cite{Berth}.

\subsubsection{Chapter 4} Here we begin to study in earnest the torsion line bundles with connection $\sL$ on the upper half plane $\Omega$. Write $\bD$ and $w \bD$ for the two closed subdiscs of $\bP^1$ of radius $1$, centred at $0$ and $\infty$ respectively. With an eye on the rigid-analytic Beilinson-Bernstein theorem \cite{EqDCap}, we view the two-element affinoid covering $\{\bD,w\bD\}$ of the projective line $\bP^1$ as fundamental. Our approach is therefore to first restrict these line bundles to what we call the `local Drinfeld space' $\Upsilon := \bD \cap \Omega$, and then to study their further restrictions to the affinoid subdomains $\Upsilon_n$ of $\bD$ obtained by removing from $\bD$ all open discs of radius $|\pi_F^n|$ centred at the $F$-rational points, simultaneously for all $n \geq 0$. Since $\{\Upsilon_n : n \geq 0\}$ forms an admissible affinoid open covering of $\Upsilon$, no information about $\sL$ is lost in this way. Actually, we study the sections of $\sL$ on $\Upsilon_n$ that \emph{overconverge} into those holes of $\Upsilon_n$ that are wholly contained within $\bD$. These sections form a sheaf $\sL_n$ on $\Upsilon$ that could reasonably be viewed as a `truncation' of $\sL$. 

Logically, our first major result of Chapter 4 is Theorem \ref{AlgGensExist} which ensures that the torsion overconvergent line bundle with connection $\sL_n$ is \emph{algebraic}: assuming that the order $d$ of $\sL$ in $\PicCon(\Upsilon)$ is coprime to $p$, we show that for any fixed (finite) set $\cS_n$ of coset representatives for $\pi_F^{n+1}\cO_F$ in $\cO_F$, there is a \emph{rational function} $u_n$ with zeroes and poles contained in $\cS_n$ such that $\sL_n$ can be extended as a $\cD$-module to a $d$-torsion line bundle with connection $\cM(\cS_n, u_n, d)$ on $\bD$ that is generated as an $\cO$-module by a $d$-th root of $u_n$ and has singularities in $\cS_n$ only.

In $\S \ref{KummerSect}$, we perform a more general study of the line bundles with connection $\cM(\cS,u,d)$ on the affine line that are generated as an $\cO$-module by $u^{\frac{1}{d}}$ and have singularities constrained to a finite set $\cS$. Again, assuming that $d$ is coprime to $p$, we show in Theorem \ref{SheafyPres} that the corresponding subsheaf $\cM(\cS,u,d)_{\overline{U_t}}$ of sections of $\cM(\cS,u,d)$ overconverging into the open discs of radius $t$ centred at the finite set of (possible) singularities $\cS$ of $\cM(\cS,u,d)$ is in fact \emph{finitely presented} as a module over the appropriate sheaf of infinite order differential operators $\cD^\dag_{\varpi/t}$ that was introduced earlier on in Chapter 2. More precisely, we recall a certain well-known order-one differential operator $R(u,d)$ at Definition \ref{TheRelator}(b), and show that the well-known classical presentation $\cD / \cD \cdot R(u,d)$ for the algebraic-$\cD$-module version of $\cM(\cS,u,d)$ in fact induces `the same' presentation of $\cM(\cS,u,d)_{\overline{U_t}}$ as a $\cD^\dag_{\varpi/t}$-module: $\cM(\cS,u,d)_{\overline{U_t}} \cong \cD^\dag_{\varpi/t} / \cD^\dag_{\varpi/t} \cdot R(u,d)$.

We conclude Chapter 4 by applying methods of Berthelot in $\S \ref{LocIrred}$ --- notably, his theory of \emph{Frobenius descent} --- to show that the simplest possible examples of the cyclic one-relator $\cD^\dag_{\varpi}$-modules of this form, namely those where the relator $R(u,d)$ equals $x \partial_x - \lambda$ for some scalar $\lambda$, are in fact \emph{irreducible} for generic values of $\lambda$.


\subsubsection{Chapter 5} Returning to the truncation $\sL_n$ to $\Upsilon_n$ of our line bundle with connection $\sL$ on $\Upsilon$, Theorem \ref{AlgGensExist} and Theorem \ref{SheafyPres} give us a presentation 
\begin{equation} \label{PresLeveln} \sL_n \cong \sD_n / \sD_n \cdot R(u_n,d)\end{equation}
for a certain rational function $u_n$: here $\sD_n$ is a convenient abbreviation for the sheaf of infinite-order differential operators $\cD^\dag_{\varpi/|\pi_F|^n}$ that is of relevance when studying $\Upsilon_n$ and $\sL_n$.  Chapter 5 forms the core of our paper: we investigate what happens with these presentations as we begin to vary the parameter $n$. Since for any affinoid subdomain $X$ of $\bD$, the inverse limit of the rings $\sD_n(X)$ over all sufficiently large $n$ can be shown to be isomorphic to the algebra $\w\cD(X)$ from \cite{DCapOne} --- see Corollary \ref{wDJ} --- one may wonder, writing $j : \Upsilon \hookrightarrow \bD$ to denote the open inclusion, whether $j_\ast \sL$ is in fact a coadmissible $\w\cD$-module on $\bD$ in the sense of \cite{DCapOne}. Although we cannot rule this out, we believe this not to be the case, the main reason being that the divisors $\Div(u_n)$ of the rational functions $u_n$ grow without bound in complexity as $n$ increases: consequently, the characteristic cycles of the algebraic approximations $\cD / \cD R(u_n,d)$ to $\sL$ grow without bound with $n$ as well. In fact, one can use Berthelot--Abe's \emph{characteristic cycles} from \cite{Berth2} and \cite{AbeMicro} to show that the natural comparison map
\begin{equation}\label{NaiveCompMapIntro} \sD_n \underset{\sD_{n+1}}{\otimes}{} \sL_{n+1} \to \sL_n\end{equation}
is in fact \emph{not} an isomorphism. 

In this situation, the underlying \emph{equivariance} of the line bundle $\sL$ comes to the rescue. Let $I$ denote the \emph{Iwahori subgroup}: it is of interest to us as it maps onto the stabiliser of $\bD$ in $\GL_2(F)$ under the natural M\"{o}bius action of $\GL_2(F)$ on $\bP^1$. Under the assumption that $\sL$ is $I$-equivariant, the action of $\sD_n$ on the truncated line bundles $\sL_n$ naturally extends to the (rather large) skew-group ring $\sD_n \rtimes I$. We show in Theorem \ref{DnJnJmodule} that that sufficiently small open subgroups $I_{n+1}$ of $I$ act on $\sL_n$ through certain infinite-order differential operators contained in the unit group of $\sD_n$: this is done as a consequence of the contents of $\S \ref{IwahoriAction}$, where we show that the natural action of elements $g  \in \GL_2(K)$ on rigid analytic functions by M\"{o}bius transformations can in fact be expressed using certain infinite-order differential operators $\beta(g) \in \sD_n^\times$, provided $g$ is sufficiently close to $1$: see Proposition \ref{SigmaRho} for a more precise statement. Consequently, the action of $\sD_n \rtimes I$ on $\sL_n$ factors through a particular crossed product  $\sD_n \rtimes_{I_{n+1}} I$ of $\sD_n$ with $I / I_{n+1}$. These crossed products are of relevance to us for two reasons: firstly, they are more manageable than the skew-group ring $\sD_n \rtimes I$ as the group $I / I_{n+1}$ is finite, and secondly, for any affinoid subdomain $X$ of $\bD$, the inverse limit of the crossed products $\sD_n(X) \rtimes_{I_{n+1}} I$ over all sufficiently large $n$ is naturally isomorphic to the completed skew-group ring $\w\cD(X,I)$ that underpins the definition of coadmissible equivariant $\cD$-modules given in \cite{EqDCap} --- see Corollary \ref{wDJ}. 

With these crossed products acting on $\sL_n$, we have the natural comparison map
\begin{equation}\label{CompMapIntro} \sD_n \underset{I_{n+1}}{\rtimes}I \quad \underset{ \sD_{n+1} \underset{I_{n+2}}{\rtimes}I }{\otimes}{} \quad \sL_{n+1} \quad \longrightarrow \quad \sL_n.\end{equation}
In $\S \ref{CPandSA}$ we explain how to work with the complicated-looking tensor product appearing on the left hand side: we show in Theorem \ref{IndModCoinv} that it is isomorphic to a particular quotient of the $\sD_n$-module $\sD_n \underset{\sD_{n+1}}{\otimes}{} \sL_{n+1}$ appearing in (\ref{NaiveCompMapIntro}) above. In fact, we explain how the operators $\beta(g)$ from Definition \ref{Beta} give rise to a certain left $\sD_n$-linear action of the finite group $I_{n+1}/I_{n+2}$ on $\sD_n \underset{\sD_{n+1}}{\otimes}{} \sL_{n+1}$ that we call the `secret action': the left hand side of $(\ref{CompMapIntro})$ can then be identified with the $\sD_n$-module of coinvariants of $\sD_n \underset{\sD_{n+1}}{\otimes}{} \sL_{n+1}$ under this secret action. 

The main result of Chapter 5 is then Theorem \ref{XiIso}, which (essentially) says that the comparison map (\ref{CompMapIntro}) is always an isomorphism. 

\subsubsection{Chapter 6} Here we give a proof of Theorem \ref{XiIso}. Working locally on $\bD$, at the end of $\S \ref{CompSect}$ we reduce the problem to showing that the comparison map (\ref{CompMapIntro}) is an isomorphism upon restriction to a particular affinoid subdomain $W_{a,n}$ contained in a closed disc of radius $|\pi_F|^n$ centred at some $a \in \cO_F$ --- see Theorem \ref{CompOnW}. Some further work reduces us further to the case where $n = 0$ and where the algebraic approximations to $\sL_1$ and $\sL_0$ given by (\ref{PresLeveln}) take a particularly simple explicit form --- see Theorem \ref{MandMdash}. Sections $\S \ref{MicRings}$ and $\S \ref{CharCycleSect}$ are devoted to the calculation of the characteristic cycles of $\sD_0 \otimes_{\sD_1} \sL_1$ and $\sL_0$, building on the work of Abe \cite{AbeMicro}, where it becomes apparent that the cycle of $\sD_0 \otimes_{\sD_1} \sL_1$ is larger than that of $\sL_0$. An abstract criterion given by Theorem \ref{Criterion} then provides three conditions that we must verify in order to deduce Theorem \ref{CompOnW}. The last two of these are relatively straightforward, however the first one, namely the non-triviality of the secret action, turns out to be surprisingly difficult to check. Arguing by contradiction and working with some microlocalisation of $\sD_n$, we show in $\S \ref{TheIntegral}$ that the triviality of this secret action would imply that a certain $p$-adic differential equation would have a rigid analytic solution $\zeta_{w,d}$ everywhere on $W_{a,n}$. We find an explicit local formal solution of this differential equation at Proposition \ref{UniqueSol}. Looking at the growth rate of the $p$-adic valuations of the Taylor coefficients of this solution then provides the required contradiction --- see Theorem \ref{ZetaUnbounded}. This relies on detailed estimates of $p$-adic valuations of certain binomial coefficients, which we perform in $\S \ref{BinEsts}$.

\subsubsection{Chapter 7} Here we draw everything together and invoke the results of \cite{EqDCap} by the first author to establish Theorem A and Corollary B. In $\S \ref{QcohTower}$ and $\S \ref{CompareSandT}$ we carry out the necessary preliminary work to establish Corollary \ref{wDJ} that was mentioned above. In $\S \ref{LocalCoad}$ we crucially use Theorem \ref{XiIso} together with the elementary fact that $j_\ast \sL \cong \invlim \sL_n$ as an $I$-equivariant locally Fr\'echet $\cD$-module on $\bD$ to establish that $j_\ast \sL \in \cC_{\bD/I}$ in Theorem \ref{LolCoad}. It is interesting to point out that for this local coadmissibility statement, we do not need to use the action of the entire Iwahori subgroup $I$ on $\sL$: we are able to show that $j_\ast \sL \in \cC_{\bD / J}$ for \emph{any} closed subgroup $J$ of $I$ that contains an open subgroup of the `translation subgroup' $\begin{pmatrix} 1 & F \\0 & 1 \end{pmatrix}$ of $\GL_2(F)$. In $\S \ref{CoadIrredSect}$ we globalise the material in $\S \ref{LocalCoad}$ from $\bD$ to $\bP^1$ and apply the equivariant Beilinson-Bernstein Localisation Theorem --- Theorem \ref{RigidEqBB} --- to prove the coadmissibility assertion on $\sL(\Omega)$ in Theorem A. Along the way we apply the local irreducibility results from $\S$\ref{LocIrred} to establish Theorem \ref{PushForwardIsSimple} which deals with those $\sL$ in Theorem A that are non-trivial as line bundles with connection. Finally, in $\S \ref{JstarO}$, we treat the case of the trivial line bundle with connection $\cO_\Omega$. We use the Induction Equivalence and the equivariant Kashiwara equivalence from \cite{ArdEqKash} to show that the pushforward of $\cO_\Omega$ to $\bP^1$ is a coadmissible $G^0$-equivariant $\cD$-module on $\bP^1$ of length $2$, and use this to deduce Theorem A together with Corollary B.

\subsection{Conventions and Notation}  \label{ConvNotn}
$F$ will denote a finite extension of $\mathbb{Q}_p$ with ring of integers $\cO_F$, uniformiser $\pi_F$ and residue field $k_F$ of order $q$. $K$ will denote a field containing $F$ that is complete with respect to a non-archimedean norm $|\cdot|$ such that $|p|=1/p$. We will write: 
\begin{itemize}
\item $\sqrt{|K^\times|}$ to denote the divisible subgroup of $\R^\times$ generated by $|K^\times|$, 
\item $K^\circ := \{a \in K : |a| \leq 1\}$ for the valuation ring of $K$,
\item $K^{\circ\circ} := \{a \in K : |a| < 1\}$ for the maximal ideal of $K^\circ$,
\item $\overline{K}$ for a fixed algebraic closure of $K$, and
\item $\bfC$ for the completion of $\overline{K}$.
\end{itemize}
Note that $\sqrt{|K^\times|} = \left|\overline{K}^\times\right|$.

A \emph{$K$-Banach space} will be a $K$-vector space $V$ equipped with a norm $|\cdot|$ compatible with the norm on $K$. A morphism of Banach spaces $T\colon V\to W$ will be a bounded linear map from $V$ to $W$ and $||T||$ will denote the operator norm of $T$: $||T||=\sup_{v\in V\backslash 0} \frac{|Tv|}{|v|}$. Thus a \emph{$K$-Banach algebra} $A$ will be a $K$-Banach space equipped with an associative unital multiplication map $A\times A\to A$ that is bounded: there is a constant $C$ such that $|ab|\leq C|a||b|$ for all $a,b\in A$.

Given a $K$-Banach space $V$ we will write $\cB(V)$ to denote the Banach algebra of endomorphisms of $V$ equipped with the operator norm. Following \cite[Definition 6.1.3]{Kedlaya} the \emph{spectral radius} of $T\in \cB(V)$ is defined to be
\[ |T|_{\sp,V} := \lim\limits_{k\to\infty} ||T^k||^{1/k}.\]

When $V$ and $W$ are $K$-Banach spaces we recall \cite[\S 17]{SchNFA} that the \emph{tensor product norm} on $V \otimes_KW$ is given by 
\[|u| = \inf\left\{ \max\limits_{1 \leq i\leq r} |v_i| |w_i| : u = \sum_{i=1}^r v_i \otimes w_i, v_i \in V, w_i \in W\right\}\]
and that $V \h{\otimes}_K V$ is the completion of $V \otimes_K W$ with respect to this norm. When we write $V\otimes W$ or $V\h\otimes W$ unadorned we will always mean tensor product or completed tensor product over $K$.

Let $X$ be a rigid $K$-analytic variety. When $Y$ is a subset $X$, we will say that $Y$ is an \emph{affinoid subdomain} of $X$ to mean that $Y$ is an admissible open subspace of $X$, itself isomorphic to an affinoid $K$-variety. When $X$ itself happens to be a $K$-affinoid variety, this agrees by \cite[Corollary 8.2.1/4]{BGR} with the standard definition found at \cite[Definition 7.2.2/2]{BGR}. We will write
\begin{itemize}
\item $|\cdot|_X$ to denote the (power-multiplicative) supremum seminorm on $X$,
\item $\cO(X)^\circ := \{f \in \cO(X) : |f|_X \leq 1\}$,
\item $\cO(X)^{\circ\circ} := \{f \in \cO(X) : |f|_X < 1\}$, and
\item $\cO(X)^{\times\times} := 1 + \cO(X)^{\circ\circ}$, the subgroup of \emph{small units} in $\cO(X)^\times$.
\end{itemize}

When $X$ is a $K$-affinoid variety, $g_0,\ldots,g_n\in \cO(X)$ generate the unit ideal, and $r\in \sqrt{|K|^\times}$ we will write $X(rg_1/g_0,\ldots,rg_n/g_0)$ to denote the rational subdomain of $X$ given by \[ X(rg_1/g_0,\ldots,rg_n/g_0):=\{ x\in X\mid r|g_i(x)|\leq |g_0(x)|\mbox{ for }i=1,\ldots,n \}\] c.f. \cite[Example 2.1.9]{ConradSurvey}. We note that when $a\in K^\times$, \[X(|a|g_1/g_0,\ldots, |a|g_n/g_0)=X((ag_1)/g_0,\ldots, (ag_n)/g_0).\]

When $X$ is a smooth rigid analytic variety over $K$ we may form the tangent sheaf $\cT_X$ which together with its identity map forms a Lie algebroid on $X$. The sheaf of enveloping algebras $\cD_X$ of (finite order) differential operators is given on affinoid subdomains $Y$ of $X$ by the enveloping algebra $U(\cT_X(Y))$. See \cite{DCapOne} for more details. As in \cite[Definition 3.1]{ArdWad2023}, $\PicCon(X)$ will denote the group of the isomorphism classes of line bundles with flat connection on $X$  with the group operation given by tensor product and $\Con(X)$ will denote the subgroup of $\PicCon(X)$ consisting of isomorphism classes of line bundles with flat connection on $X$ that are trivial after forgetting the connection.

We will write $\A := \A_K := \A_K^{1,\an}$ to denote the rigid $K$-analytic affine line, equipped with a fixed choice of local coordinate $x \in \cO(\A)$ and $\bD$ to denote the unit disc $\Sp K\langle x\rangle$ in $\bA$ with respect to this coordinate. We write $\bP^1$ to denote the rigid $K$-analytic projective line.  

We recall \cite[\S4.1]{ArdWad2023} that a \emph{$K$-cheese} is an affinoid subdomain of $\bA$ of the form \[ \Sp K\left\langle \frac{x-\alpha_0}{s_0},\frac{s_1}{x-\alpha_1},\cdots, \frac{s_g}{x-\alpha_g}\right\rangle\] for some $\alpha_0,\ldots,\alpha_g\in K$ and $s_0,\ldots,s_g\in K^\times$ satisfying certain conditions. An affinoid subdomain of $\bA$ \emph{splits} over a finite field extension $K'/K$ if $X_{K'}$ is a finite union of pairwise disjoint $K'$-cheeses. Every affinoid subdomain of $\bA$ splits for some finite extension $K'/K$ by \cite[Theorem 4.1.8]{ArdWad2023}. 

Let $A$ be an abelian group and let $d$ be a non-zero integer. We will write
\begin{itemize} 
\item $A[d]=\{a\in A\st da=0\}$ to denote the $d$-torsion subgroup of $A$, 
\item $A[p']:= \bigcup_{(d,p)=1} A[d]$ to denote the prime-to-$p$ torsion subgroup of $A$,
\item $A[p^\infty]:=\bigcup_{n=1}^\infty A[p^n]$ to denote the $p$-power torsion subgroup of $A$. 
\end{itemize} 

Let a group $G$ act on a set $X$. We will write
\begin{itemize}
\item $G_x:=\{g\in G:gx=x\}$ for the stabilizer of a point $x\in X$, and
\item $X^G:=\{x\in X: gx=x $ for all $g\in G\}$ for the set of elements fixed by $G$.
\end{itemize}


It  will be convenient to define $\varpi=p^{-\frac{1}{p-1}}\in \bR_{>0}$ and $\varepsilon:=\left\{\begin{array}{ll} 1 & \mbox{if } p > 2, \\ 2 & \mbox{if } p = 2. \end{array} \right.$
\section{Sheaves of completed differential operators}\label{sheaves}

\subsection{Skew-Ore and skew-Laurent extensions}\label{skewLaurent}
We recall \cite[p27]{GoodearlWarfield} that if $A$ is a ring and $\delta\colon A\to A$ is a derivation, then one can form the \emph{skew-Ore extension} $A[T;\delta]$ which is an over-ring of $A$ that as a left $A$-module is free on the symbols $\{T^i\st i\geq 0\}$ and satisfies $T a-aT=\delta(a)$ for all $a\in A$ and $T^iT^j=T^{i+j}$ for $i,j\geq 0$. 

We recall that $A[T;\delta]$ satisfies a natural looking universal property.

\begin{prop}\label{Univskewore} (\cite[Exercise 2F]{GoodearlWarfield})  For every ring homomorphism $\phi\colon A\to B$ and element $b\in B$ such that $b\phi(a)-\phi(a)b=\phi(\delta(a))$ for all $a\in A$, there is a unique ring homomorphism $\psi\colon A[T;\delta]\to B$  such that $\psi|_A=\phi$ and $\psi(T)=b$.
\end{prop}

We wish to form a kind of skew-Laurent-polynomial ring $A[T,T^{-1};\delta]$ by formally inverting $T$ in $A[T;\delta]$ and also understand its subring generated by $A$ and $T^{-1}$. This is not possible in general but we will see that it \emph{is} possible if $\delta$ is locally nilpotent; i.e. for each $a\in A$ there is $n\geq 0$ such that $\delta^n(a)=0$. The first step is to prove that the set $\{T^n\st n\geq 0\}$ satisfies the Ore condition (both left and the right versions) under this locally nilpotent hypothesis. 

We recall a criterion for a multiplicatively closed subset of a ring to satisfy Ore's condition.

\begin{prop}\label{GabberOre}(\cite[p457]{Gabber}) If $S$ is a multiplicatively closed subset of a ring $A$ such that $\ad(s)$ is locally nilpotent for all $s\in S$ then $S$ satisfies both the left and right Ore conditions. \end{prop} 

\begin{lem}\label{Oregens} If $s,t\in A$ such that $\ad(s)$ and $\ad(t)$ are both locally nilpotent and $st=ts$ then $\ad(st)$ is locally nilpotent.
\end{lem}
\begin{proof} If we write $l_s\colon A\to A$ to denote left multiplication by $a$ and $r_t\colon A\to A$ to denote right multiplication by $t$ then for $a\in A$ \[ \ad(st)(a)= sta-ast= (sta - sat) + (sat - ast) = l_s\circ\ad(t)(a) + r_t\circ \ad(s)(a).\] Thus since $\ad(s), l_s, r_t$ and $\ad(t)$ all commute \[ (\ad(st))^n = \sum_{i=0}^n \binom{n}{i} l_s^i r_t^{n-i} \ad(t)^i \ad(s)^{n-i}  \] and it is follows easily from the local nilpotence of $\ad(s)$ and $\ad(t)$ that $\ad(st)$ is locally nilpotent.
\end{proof}

\begin{cor}\label{locnilOre}If $\delta\colon A\to A$ is a locally nilpotent derivation then $\{T^n\st n\geq 0\}$ is a left and right Ore set in $A[T;\delta]$. 
\end{cor}
\begin{proof} By Proposition \ref{GabberOre} it suffices to show that $\ad(T^n)$ acts locally nilpotently on $A[T;\delta]$ and by Lemma \ref{Oregens} this can be reduced to the case $n=1$. 
	
	But \[ \ad(T)\left(\sum_{i=0}^m a_iT^i\right) = \sum_{i=0}^m \delta(a_i)T^i \] so local nilpotence of $\ad(T)$ on $A[T;\delta]$ is a straightforward consequence of the local nilpotence of $\delta$ on $A$.   
\end{proof}

\begin{notn} For $\delta $ a locally nilpotent derivation of $A$ we write $A[T,T^{-1};\delta]$ to denote the localisation of $A[T;\delta]$ at the Ore set $\{T^n\st n\geq 0\}$. 
\end{notn}
We can perform the following calculation in $A[T,T^{-1};\delta]$:
\begin{lem}\label{adTinverse} For each $a\in A$, $T^{-1}a = \sum\limits_{n=0}^\infty (-1)^n \delta^n(a)T^{-n-1}$. \end{lem}
\begin{proof} First note that the right-hand side always makes sense because $\delta$ is locally nilpotent. Since $T$ is a unit in $A[T,T^{-1};\delta]$ it thus suffices to show that \[ T\left(\sum_{n\geq 0} (-1)^n\delta^n(a)T^{-n-1}\right) =a.\] But $T\delta^n(a) = \delta^n(a)T + \delta^{n+1}(a)$ so \begin{eqnarray*} T\left(\sum_{n\geq 0} (-1)^n\delta^n(a)T^{-n-1}\right) & = & \sum_{n\geq 0} (-1)^n\left(\delta^n(a) T^{-n} + \delta^{n+1}(a) T^{-n-1}\right)   \\
		& = & a + \sum_{n\geq 0} \left( (-1)^n + (-1)^{n+1} \right) \delta^{n+1}(a)T^{-n-1} \\ & = & a. \qedhere
	\end{eqnarray*} 
\end{proof}

It follows immediately that the subring of $A[T,T^{-1};\delta]$ generated by $A$ and $T^{-1}$ is the (free) left $A$-submodule of $A[T,T^{-1};\delta]$ on the set $\{T^{-n}\st n\geq 0 \}$.

\begin{defn}\label{NegLaurPoly} For $\delta$ a locally nilpotent derivation of $A$ we write $A[T^{-1}; \delta]$ to denote the subring of $A[T,T^{-1};\delta]$ generated by $A$ and $T^{-1}$. 
\end{defn}
This algebra satisfies the following universal property.

\begin{prop}\label{univTinverse} Let $f\colon A\to B$ be a ring homomorphism and let $\delta$ be a locally nilpotent derivation of $A$. Suppose that $b\in B$ is such that 
	\[bf(a)=\sum\limits_{n\geq 0} (-1)^n f(\delta^n(a))b^{n+1} \qmb{for all} a \in A.\] 
	Then there is a unique ring homomorphism 
	\[g\colon A[T^{-1};\delta]\to B\]
	extending $f$ such that $g(T^{-1})=b$.
\end{prop}
\begin{proof} Given a pair $(f,b)$ as in the statement, we define $g : A[T^{-1};\delta] \to B$ by
	\[g\left(\sum\limits_{n\geq 0} a_nT^{-n}\right):=\sum f(a_n)b^n \qmb{for all} a_n\in A.\] 
	It suffices to prove that this is a ring homomorphism. Since for $a\in A$ and $m\geq 0$, 
	\begin{eqnarray*} g(T^{-1}aT^{-m}) & = &g\left(\sum_{n\geq 0}(-1)^n\delta^n(a)T^{-n-1-m}\right)\\ & = & \sum_{n\geq 0} (-1)^n f(\delta^n(a))b^{n+1+m} \\ & = & bf(a)b^m, \end{eqnarray*} we see that $g(T^{-1}aT^{-m})=g(T^{-1})g(aT^{-m})$. An easy induction argument together with left $A$-linearity of $g$ shows that $g$ respects multiplication.  
\end{proof}

\subsection{Skew-Laurent and skew-Tate algebras}
We recall from \S\ref{skewLaurent} that for a ring $A$ and a derivation $\delta\colon A\to A$ one can form the \emph{skew-Ore extension} $A[\partial;\delta]$. As noted previously, the set $\{\partial^n\st n\geq 0\}$ is not an Ore set in $A[\partial;\delta]$ in general so we cannot form a Laurent polynomial-style ring extension $A[\partial,\partial^{-1};\delta]$ of $A[\partial;\delta]$; though by Corollary \ref{locnilOre} this is possible when $\delta$ is locally nilpotent. In this section we will show that when $A$ is a $K$-Banach algebra and $\delta$ is topologically nilpotent in a suitable sense we can form certain completed versions of $A[\partial;\delta]$ and $A[\partial,\partial^{-1};\delta]$ that are $K$-Banach algebras. These will satisfy suitable analogues of Proposition \ref{Univskewore}. In the case where $A = \cO(X)$ is the $K$-affinoid algebra of rigid analytic functions on a $K$-affinoid variety $X$, we might think of $\cO(X) \langle \partial/r, s/\partial\rangle$ as being the algebra of `rigid analytic functions on a non-commutative annulus over $X$', and we refer the reader to Schneider's Appendix in \cite{ZabSch} for some related constructions.
	
Let $r\geq s$ denote arbitrary positive real numbers; frequently it will be convenient to restrict our attention to only those $r$ and $s$ lying in $\sq{K}$.

\begin{defn}\label{AdelSR} Let $A$ be a $K$-Banach space and let $\partial$ be a formal variable. We define \[ A\langle \partial/r,s/\partial\rangle := \left\{\sum_{j\in \mathbb{Z}}a_j\partial^j\in \prod_{j\in\mathbb{Z}} A\partial^j : \lim_{|j|\to \infty}|a_j|t^j= 0 \mbox{ if } s\le t \le r \right\} \]
	and give it the norm  
	\[  \left| \sum_{j\in \mathbb{Z}} a_j \partial^j \right| := \sup\limits_{s\le t\le r}\sup\limits_{j\in \mathbb{Z}} |a_j| t^j .\] 
\end{defn}
Of course this norm can be written in the somewhat less symmetric form
\begin{equation}\label{AsymNorm} \left| \sum_{j\in \mathbb{Z}} a_j \partial^j \right| = \max \left\{ \sup\limits_{j \geq 0} |a_j| r^j, \sup\limits_{j < 0} |a_j| s^j\right\}.\end{equation}
Note that $A\langle \partial/r,s/\partial\rangle$ becomes $K$-Banach space when equipped with this norm. We will now explain how to give it the structure of an associative non-commutative $K$-Banach algebra.
\begin{lem}\label{StarProdEstimate} Let $A$ be a $K$-Banach algebra and let $\delta \in \cB(A)$ be such that 
\[|\delta|_{\sp,A}<s.\]
Let $u, v \in A \langle \partial /r ,s/\partial\rangle$.
	\be
	\item For each $k \in \bZ$, the following limit exists in $A$:
	\begin{equation} \label{MicroMult} u \ast_k v := \lim\limits_{I \to \infty} \sum\limits_{i=-I}^I u_i \lim\limits_{M \to \infty} \sum\limits_{m=0}^M \binom{i}{m} \delta^m(v_{k-i+m}).\end{equation}
	\item There is a constant $C>0$ depending only on $A$ and $\delta$ such that 
	\[\sup\limits_{k \in \bZ} |u \ast_k v| t^k \leq C \cdot |u| \cdot |v| \qmb{for all} t \in [s,r].\]
	\item For all $t \in [s,r]$, $|u \ast_k v| t^k \to 0$ as $|k| \to \infty$.
	\ee\end{lem}
\begin{proof} (a) and (b). We can find $C_1$ such that $|ab|\leq C_1|a||b|$ for all $a,b\in A$; let $C_2=\sup\limits_{\ell \geq 0}||\delta^\ell||/s^\ell<\infty$. Fix $k,i \in \Z$ and $t \in [s,r]$. Then for all $m \geq 0$,
	\begin{equation} \label{StarProdEst} \begin{array}{lll} \left|u_i \binom{i}{m} \delta^m(v_{k-i+m})\right|  \cdot t^k &\leq & C_1 \cdot |u_i| \cdot ||\delta^m|| \cdot |v_{k-i+m}| \cdot t^k  \\
			&\leq & C_1 \cdot |u_i| t^i \cdot C_2 \cdot |v_{k-i+m}| t^{k-i+m}\end{array}\end{equation}
	which tends to $0$ as $m \to \infty$ because $v \in A \langle \partial/r, s/\partial\rangle$. Therefore the inner series in the definition of $u \ast_k v$ converges to an element $(u \ast_k v)_i \in A$, such that
	\[ \left\vert (u \ast_k v)_i \right\vert \cdot t^k \leq C_1 \cdot |u_i| t^i \cdot C_2 \cdot |v|.\]
	Since $|u_i|t^i \to 0$ as $|i| \to \infty$ because $u \in A \langle \partial/r, s /\partial\rangle$, we see that the partial sums $\sum\limits_{i = -I}^I (u \ast_k v)_i$ converge to an element $u \ast_k v \in A$ as $|i| \to \infty$, which satisfies
	\[ |u \ast_k v| t^k \leq C_1 \cdot |u| \cdot |v| \cdot C_2.\]
	We may therefore take $C := C_1C_2$.

	(c)  Let $\epsilon > 0$ be given, and choose positive integers $I, M$ and $J$ such that
	\begin{itemize}
		\item $|u_i| t^i < \frac{\epsilon}{C_1C_2 |v|}$ whenever $|i| > I$,
		\item $||\delta^m|| / s^m < \frac{\epsilon}{C_1 |u| \hsp |v|}$ whenever $m > M$, and 
		\item $|v_j| r^j < \frac{\epsilon}{C_1 C_2 |u|}$ whenever $|j| > J$.
	\end{itemize}
	Now suppose that $|k| > I + M+ J$. 
	
	Firstly, if $|i| > I$, then using (\ref{StarProdEst}) we see that
	\[  \left\vert u_i \delta^m(v_{k-i+m})\right\vert \cdot t^k  < C_1 \cdot\frac{\epsilon}{C_1 C_2 |v|} \cdot C_2 \cdot |v| = \epsilon.\]
	Next, if $m > M$ then again using (\ref{StarProdEst}) we have
	\[ \left\vert u_i \delta^m(v_{k-i+m})\right\vert \cdot t^k  < C_1 \cdot |u| \cdot \frac{\epsilon}{C_1 |u| |v|} \cdot |v| = \epsilon.\]
	Finally if $|i| \leq I$ and $m \leq M$, then since $|i|+m+|k-i+m|\geq |k| > I + M + J$, we must have $|k-i+m|> J$ and  so  ($\ref{StarProdEst})$ gives
	\[ \left\vert u_i \delta^m(v_{k-i+m})\right\vert \cdot t^k < C_1 \cdot |u| \cdot C_2 \cdot \frac{\epsilon}{C_1 C_2 |u|} = \epsilon.\]
	We conclude that in all cases, we have the inequality
	\[ |u\ast_kv| t^k < \epsilon  \qmb{whenever} |k| > I + J + M \]
	and part (c) follows.
\end{proof}

\begin{defn}\label{DefnStarProd} Let $A$ be a $K$-Banach algebra and let $\delta : A \to A$ be a bounded $K$-linear derivation such that $|\delta|_{\sp, A} <s$. We define the \emph{star product} of two elements $u, v \in A \langle \partial / r, s / \partial \rangle$ to be
\[ u \ast v := \sum\limits_{k \in \bZ} (u \ast_k v) \partial^k.\]
\end{defn}
It is not hard to see that this star product gives a $K$-bilinear map $\cE \times \cE \to \cE$ where $\cE := A \langle \partial /r ,s/\partial\rangle$. Our next task will be to show that this star product on $\cE$ is associative. Because $\delta : A \to A$ is a derivation we have at our disposal the skew-Ore extension $A[\partial; \delta]$. Its associated graded ring with respect to the filtration $F_\bullet A[\partial;\delta]$ by degree in $\partial$ is the polynomial ring $A[y]$ where $y = \gr \partial$ is the principal symbol of $\partial$; we now consider the classical algebraic microlocalisation $Q := Q_S A[\partial,\delta]$, where $S := \cup_{n = 0}^\infty (\partial^n + F_{n-1}A[\partial,\delta])$; see, for example, \cite[\S 3]{AVVO}. By construction, this is an associative ring $(Q, \cdot)$ equipped with a complete and exhaustive $\Z$-filtration, whose associated graded ring is $A[y, y^{-1}]$. In this way, we can identify $Q$ with the following set of formal Laurent power series:
\[ Q = A[[\partial^{-1}]][\partial] := \left\{ \sum\limits_{k \in \Z} a_k \partial^k : a_k = 0 \qmb{for all} k >> 0\right\}.\]
The ring $Q$, as well as the $A$-module $\cE$, both contain the $A$-submodule $A[\partial,\partial^{-1}]$ whose elements are \emph{finite} $A$-linear combinations of integer powers of the formal symbol $\partial$. Also, note that $Q$ and $\cE$ are both contained in the large $A$-module of all bi-directional power series $P := \prod\limits_{k \in \bZ} A \partial^k$.

\begin{lem}\label{multbypartial} For all $n\in \mathbb{Z}$ and $a\in A$, $\partial\cdot(\partial^n \ast a)=\partial^{n+1}\ast a$.
	\end{lem}

\begin{proof}
	A direct computation using Definition \ref{DefnStarProd} gives that 
	\begin{equation}\label{dnstara} \partial^n\ast a=\sum_{k=-\infty}^n \binom{n}{n-k}\delta^{n-k}(a)\partial^k. \end{equation}
	Thus, since $(Q,\cdot)$ is associative and $\partial\cdot a=\delta(a)+a\partial$ is the defining property of the skew-Ore extension $A[\partial;\delta]$, we can further compute 	 \begin{eqnarray*} \partial\cdot(\partial^n\ast a) & = & \partial\cdot \left(\sum_{k=-\infty}^{n}\binom{n}{n-k}\delta^{n-k}(a)\partial^k\right) \\ &=&
		\sum_{k=-\infty}^{n}\binom{n}{n-k}\delta^{n-k+1}(a)\partial^k + \sum_{k=-\infty}^{n}\binom{n}{n-k}\delta^{n-k}(a)\partial^{k+1} \\ & = &
		\sum_{k=-\infty}^{n+1}\left(\binom{n}{n-k}+\binom{n}{n-k+1}\right) \delta^{n-k+1}(a)\partial^k \\ & = &
		\sum_{k=-\infty}^{n+1} \binom{n+1}{n+1-k}\delta^{n+1-k}(a)\partial^k \\ & = & \delta^{n+1}\ast a
	\end{eqnarray*} as required.
\end{proof}
\begin{lem}\label{HexCheck} We have $u \cdot v = u \ast v$ in $P$ for all $u,v \in A[\partial,\partial^{-1}]$. 
\end{lem}
\begin{proof} 
	First we note that for $a\in A$ and $n\in \mathbb{Z}$ an easy computation gives that 
	\[(au)\ast (v\partial^n)=a\cdot (u\ast v)\cdot \partial^n.\] 
	Next, using the $\bZ$-bilinearity of both products we may reduce to the case $u=a\partial^n$ and $v=b\partial^{n'}$ with $a,b\in A$ and $n,n'\in \mathbb{Z}$. Then $u\ast v=(a\partial^n)\ast (b\partial^{n'})=a\cdot(\partial^n\ast b)\cdot\partial^{n'}$ so as $\cdot$ is associative we may further reduce to the case $a=1$ and $n'=0$. So, it remains to prove that 
	\[\partial^n\cdot b=\partial^n \ast b \qmb{for all} n\in \bZ \qmb{and} b\in A.\]
When $n\geq 0$ we can verify this directly by induction on $n$ using Lemma \ref{multbypartial} and associativity of $(Q,\cdot)$: 
\[ \partial^{n+1}\ast b=\partial\cdot (\partial^n\ast b)=\partial\cdot \partial^n\cdot b=\partial^{n+1}\cdot b.\]
We also prove the case $n<0$ by induction, this time on $-n$. Suppose inductively that $\partial^n \cdot b = \partial^n \ast b$. Then using Lemma \ref{multbypartial} we have
   \[\partial \cdot( \partial^{n-1} \ast b ) = \partial^n \ast b =  \partial \cdot( \partial^{-1} \cdot (\partial^n \ast b)) = \partial \cdot (\partial^{n-1} \cdot b).\]
Now $\partial^{n-1} \cdot b \in Q$ and $\partial^{n-1} \ast b \in Q$ by equation (\ref{dnstara}). Since $\partial$ is a unit in the associative ring $(Q,\cdot)$, we conclude that $\partial^{n-1}\cdot b=\partial^{n-1}\ast b$.\end{proof}

\begin{thm} \label{skewannulus} Suppose that $A$ is a $K$-Banach algebra and that $\delta\in \cB(A)$ is a derivation such that $|\delta|_{\sp,A}<s$. Then $A\langle\partial/r, s/\partial\rangle$ becomes an associative $K$-Banach algebra when equipped with the star product $\ast$. \end{thm}
\begin{proof} We have to show that the star product $\ast$ on $\cE := A\langle \partial/r, s/\partial\rangle$ is associative. Using Lemma \ref{HexCheck}, we see that the two possible $P = \prod\limits_{k \in \bZ} A \partial^k$-valued multiplications on $A[\partial, \partial^{-1}] \subseteq Q \cap \cE$ coincide in the sense that the diagram
\[\xymatrix{ & Q^2 \ar[r]^(0.6){m_Q} & Q \ar[rd] & \\ A[\partial,\partial^{-1}]^2 \ar[ru]\ar[rd] &  & & P \\ & \cE^2 \ar[r]_(0.6){m_{\cE}} & \cE \ar[ur] }\]
is commutative. Because $(Q, \cdot)$ is an associative ring, the associator $\Phi : \cE^3 \to \cE$ must vanish on $A[\partial, \partial^{-1}]^3$. Since $A[\partial,\partial^{-1}]$ is dense in $\cE$ and $\Phi$ is continuous, we conclude that $\Phi$ must be identically zero.\end{proof}

Because of Lemma \ref{HexCheck} and Theorem \ref{skewannulus}, we will now drop the symbol $\ast$ and simply write $ab$ to denote the product in the associative skew-Laurent algebra $A\langle \partial/r, s/\partial\rangle$.  Next, we establish a universal property for this algebra.

\begin{prop}\label{SkewLaurentUP} Suppose that $A$, $\delta$ and $0 < s \leq r$ are as in Theorem \ref{skewannulus}. Let $B$ be a $K$-Banach algebra, let $\theta\colon A\to B$ is a $K$-Banach algebra homomorphism and $b\in B$. Then the following are equivalent:
\be \item the map $\theta : A \to B$ extends to a $K$-Banach algebra homomorphism $\phi\colon A\langle \partial/r,s/\partial\rangle\to B$ which sends $\partial$ to $b$;
\item 
	\begin{enumerate}[(i)]
		\item $b\in B^\times$,
		\item $b\theta(a)-\theta(a)b=\theta(\delta(a))$ for all $a\in A$,
		\item $\sup\limits_{n \geq 0} |b^n|/r^n < \infty$ and $\sup\limits_{n \leq 0} |b^n|/s^n < \infty$.
	\end{enumerate}
\ee\end{prop}

\begin{proof} (a) $\Rightarrow$ (b). Let $C := A \langle \partial/r, s/\partial \rangle$ and suppose that $\phi : C \to B$ is a bounded $K$-algebra homomorphism that extends $\theta$ and that sends $\partial$ to $b$. The first two conditions hold because $\partial \in C^\times$ and $\partial a - a \partial = \delta(a)$ holds for all $a \in A$. 

For the last condition: since $\theta : C \to B$ is bounded, there is a constant $L$ such that $|\theta(c)| \leq  L|c|$ for all $c \in C$. Hence $|b^n| = |\theta(\partial^n)| \leq L |\partial^n|$ for all $n \in \bZ$. By Definition \ref{AdelSR}, we get $|b^n| \leq L r^n$ for all $n \geq 0$ and $|b^n| \leq L s^n$ for all $n \leq 0$.

(b) $\Rightarrow$ (a). Note first that $b^n$ makes sense for all $n\in \mathbb{Z}$ by condition (i). Choose constants $M,D>0$ such that $|\theta(a)|\leq M|a|$ for all $a\in A$ and $|b_1b_2|\leq D|b_1||b_2|$ for all $b_1,b_2\in B$; then for all $a \in A$ and all $n \in \bZ$ we have $|\theta(a)b^n| \leq MD |a| |b^n|$. Using (iii), choose $L>0$ such that $|b^n| \leq L r^n$ for all $n \geq 0$ and $|b^n| \leq L s^n$ for all $n \leq 0$. Then for any $\sum\limits_{n\in \Z} a_n\partial^n\in C$ we have
\begin{equation}\label{ThetaEst} |\theta(a_n) b^n|  \leq \left\{ \begin{array}{lll} LMD |a_n|r^n &\qmb{for all}& n \geq 0, \\ LMD |a_n|s^n &\qmb{for all}& n \leq 0. \end{array} \right.\end{equation}
Since $|a_n| r^n \to 0$ as $n \to +\infty$ and $|a_n|s^n \to 0$ as $n \to -\infty$ by Definition \ref{AdelSR}, we obtain $|\theta(a_n)b^n|\to 0$ as $|n|\to \infty$. So, we can define a map $\phi : C\to B$ by setting
\[ \phi\left(\sum\limits_{n\in \Z}a_n\partial^n\right) := \sum\limits_{n\in \Z} \theta(a_n)b^n \qmb{for any}\sum\limits_{n\in \Z} a_n\partial^n\in C.\] 
It is easy to check that $\phi$ is $K$-linear. Using (\ref{ThetaEst}) and (\ref{AsymNorm}), we see that for all $\sum\limits_{n\in \Z} a_n\partial^n\in C$ we have
\[\left|\phi\left( \sum\limits_{n\in \Z}a_n\partial^n\right)\right| \leq LMD \max\{ \sup\limits_{n \geq 0} |a_n| r^n, \sup\limits_{n \leq 0} |a_n| s^n\} = LMD\left|\sum\limits_{n\in \Z} a_n\partial^n\right|. \]  Thus $\phi$ is a bounded $K$-linear map which satisfies $||\phi||\leq LMD$.
	
To see that $\phi$ is a ring homomorphism, by its continuity and $K$-linearity it suffices to verify that $\phi(\partial^na)=\phi(\partial^n)\phi(a)$ for all $n$ in $\Z$. By induction on $n$ we can reduce to the cases $n=\pm 1$. The case $n=1$ is immediate from condition (iii). For the case $n=-1$, we know from Lemma \ref{HexCheck} together with equation (\ref{dnstara}) that
\[ \partial^{-1}a = \sum_{m\geq 0}\binom{-1}{m}\delta^m(a)\partial^{-m-1}. \]
Since $\binom{-1}{m}=(-1)^m$ we must prove that $b^{-1}\theta(a)=\sum\limits_{m\geq 0}(-1)^m\theta(\delta^m(a))b^{-m-1}$. Since $b$ is a unit in $B$ it suffices to prove that $\theta(a)=\sum\limits_{m\geq 0} (-1)^m b\theta(\delta^m(a))b^{-m-1}$. But, by condition (iii), the right-hand side of this formula is \[ \sum\limits_{m\geq 0} (-1)^m\left( \theta(\delta^m(a))b^{-m} + \theta(\delta^{m+1}(a))b^{-m-1}\right).\]  
We can now see that this sum telescopes down to $\theta(a)$. \end{proof}
We will now discuss the \emph{skew-Tate} algebra $A\langle \partial/r \rangle$. 
\begin{defn}\label{AdelR} Let $A$ be a $K$-Banach space and let $\partial$ be a formal variable. We define the \emph{Skew-Tate algebra}
	\[ A \langle \partial / r \rangle := \left\{ \sum_{j=0}^\infty a_j \partial^j \in A[[\partial]] : \lim\limits_{j\to\infty} |a_j| r^j = 0\right\}\] and give it the norm  \[ \left| \sum_{j=0}^\infty a_j \partial^j \right| := \sup\limits_{j \geq 0} |a_j| r^j .\]
\end{defn}
Again, $A \langle \partial / r \rangle$ is a $K$-Banach space when equipped with this norm. 
\begin{lem}\label{BanachStarProd} Suppose that $A$ and $\delta$ are as in Theorem \ref{skewannulus}. 
\be \item The natural map  $A\langle \partial/r\rangle \to A\langle \partial/r,s/\partial\rangle$ is an isometric embedding of $K$-Banach spaces with closed image.
\item $A \langle \partial/r \rangle$ is closed under the star product in $A\langle \partial/r,s/\partial\rangle$.
\ee\end{lem}
\begin{proof} (a) This is evident.

(b) We have to check that $u \ast_k v = 0$ for any $u, v \in A \langle \partial/r \rangle$ and any $k < 0$. Fix $m \geq 0$ and consider the term $\binom{i}{m} \delta^m(v_{k-i+m})$ appearing in the definition of $u \ast_k v$ in Lemma \ref{StarProdEstimate}(a). If $m > i$ then the binomial coefficient is zero, so suppose that $m \leq i$. But then $k - i + m \leq k < 0$ so $v_{k-i+m} = 0$ as $v \in A \langle \partial / r\rangle$. 
\end{proof}

When $||\delta|| \leq r$ and the norm on $A$ satisfies $|ab| = |a| \cdot |b|$ for all $a, b \in A$, it also follows from the work of Pangalos --- see \cite[Proposition 2.1.2]{Pangalos} --- that $A \langle \partial/r \rangle$ is an associative $K$-Banach algebra.

\begin{lem} \label{BanachUnivProp} Let $B$ be another $K$-Banach algebra. Given a bounded $K$-Banach algebra homomorphism $f\colon A\to B$ and an element $b\in B$, the following are equivalent.
\be \item
There is a bounded $K$-Banach algebra homomorphism $g\colon A\langle \partial/r\rangle \to B$ extending $f$ such that $g(\partial)=b$;
\item 
\begin{enumerate}[(i)]
		\item $bf(a)-f(a)b=f(\delta(a))$ for all $a\in A$ and \item $\sup\limits_{\ell \geq 0}|b^\ell|/r^\ell<\infty$.\end{enumerate} 
\ee 
Moreover when such a $g$ exists it is unique. \end{lem}

\begin{proof} (a) $\Rightarrow$ (b). Since for all $a\in A$, $\partial a-a\partial=\delta(a)$, we have \[ bf(a)-f(a)b=g(\partial)g(a)-g(a)g(\partial)=g(\delta(a))=f(\delta(a)),\] so (i) holds. Moreover, because $|\partial^\ell|=r^\ell$ in $A \langle \partial/r \rangle$ for all $\ell\geq 0$, \[ ||g||\geq |g(\partial^\ell)| / r^\ell = |b^\ell| / r^\ell \qmb{for all} \ell \geq 0, \]
so (ii) holds. 
	
(b) $\Rightarrow$ (a) Let $L:=\sup \limits_{\ell \geq 0}|b^\ell|/r^\ell$ and choose constants $M,D>0$ such that $|f(a)|\leq M|a|$ for all $a\in A$ and $|b_1b_2|\leq D|b_1||b_2|$ for all $b_1,b_2\in B$. 
	
	By condition (i), together with the universal property of the skew-Ore extension $A[T;\delta]$ --- see Proposition \ref{Univskewore} --- $f$ extends uniquely to a $K$-algebra map $h\colon A[T;\delta]\to B$ such that $h(T)=b$. We compute that for $a_0,\ldots,a_n\in A$ \[ \left|\sum\limits_{i=0}^n h(a_iT^i)\right|=\left|\sum\limits_{i=0}^n f(a_i)b^i\right|\leq D\sup\limits_{i\geq 0}(|f(a_i)||b^i|)\leq DML\sup\limits_{i\geq 0} |a_i| r^i .\]
	
	Thus, identifying $A[T;\delta]$ with its image in $A\langle \partial/r \rangle$ under the map $\varphi$ in the proof of Lemma \ref{BanachStarProd} we see that $h$ is bounded with respect to the subspace norm on $A[T;\delta]$ and so $h$ extends uniquely to a bounded $K$-linear map $g\colon A\langle \partial/r\rangle \to B$. Finally, consider the continuous map $\Psi\colon A\langle \partial/r\rangle \times A\langle \partial/r\rangle \to B$ given by \[ \Psi(u,v)=g(uv)-g(u)g(v).\] Since $\Psi$ vanishes on the dense subset $A[T,\delta]\times A[T,\delta]$ of its domain, it is identically zero and so $g$ is a $K$-algebra homomorphism. 
\end{proof}
It turns out that the skew-Tate algebra $A \langle \partial/r \rangle$ admits a natural involution.
\begin{lem} \label{Transpose} Let $A$ be a $K$-Banach algebra and let $\delta\in \cB(A)$ be a derivation such that $|\delta|_{\sp,A}<r$.  Then there is a bounded $K$-Banach algebra isomorphism
\[ (-)^T : A \langle \partial/r \rangle \to A \langle \partial/r \rangle^{\op} \]
such that $a^T = a$ for all $a \in A$, $\partial^T = -\partial$ and $(Q^T)^T = Q$ for all $Q \in A \langle \partial/r \rangle$.
\end{lem}
\begin{proof} We apply Lemma \ref{BanachUnivProp} with $A = \cO(X)$, $\delta = \partial_x \in \cB(A)$, $B :=A \langle \partial/r \rangle^{\op}$, $b = -\partial$ and $f : A \to B$ the natural inclusion. Writing $\cdot$ for the product in the $K$-Banach algebra $B$, we compute that for all $a \in A$ we have
\[ b \cdot f(a) - f(a) \cdot b = a(-\partial) - (-\partial)a = [\partial, f(a)] = f(\delta(a)),\]
so the condition (b)(i) in Lemma \ref{BanachUnivProp} holds. The condition (b)(ii) is clear since $|b^\ell| = |(-\partial)^\ell| = r^\ell$ holds in $B$ for all $\ell \geq 0$ by the definition of the norm on $B$ found at Definition \ref{AdelR}. So by Lemma \ref{BanachUnivProp} we obtain a bounded $K$-Banach algebra homomorphism $(-)^T : A\langle \partial/r \rangle \to B$ such that $a^T = a$ for all $a \in A$ and $\partial^T = -\partial$. 

It remains to check that this map is self-inverse. To this end, note that the map $(-)^{TT} : A\langle \partial/r \rangle \to A\langle \partial/r \rangle$ fixes $A$ and $\partial$ pointwise. Since it is also a bounded $K$-Banach algebra homomorphism and since $A$ and $\partial$ generate a dense $K$-subalgebra of $A \langle \partial/r\rangle$, we see that $(-)^{TT}$ is the identity map on all of $A \langle \partial/r\rangle$.
\end{proof}

\subsection{The sheaf on \ts{\cD_r} on \ts{\partial_x/r}-admissible subdomains of \ts{\A}} \label{DrSect}
%
In this section, on we will be interested in affinoid subdomains of the rigid-analytic affine line $\A$. 
Let $\partial_x$ denote the derivation $\frac{d}{dx}$ of $\cO_\A$.
\begin{defn} \label{dxradm} Let $X$ be an affinoid subdomain of $\A$.
\be \item We define \emph{the spectral radius of $X$} to be $r(X) := |\partial_x|_{\sp, \cO(X)}$.
\item Let $r \in \mathbb{R}_{>0}$. We say that $X$ is \emph{$\partial_x/r$-admissible} if and only if
\[  r > r(X) .\]
\ee\end{defn}
We emphasise that this notion of spectral radius $r(X)$ depends on the choice of global vector field $\partial_x$ on $\A$, i.e. on the choice of coordinate $x$ on $\A$. 

\begin{prop}\label{Gtop}  Let $X, Y$ be two affinoid subdomains of $\A$. Then 
\[ r(X \cap Y) \leq \max \{r(X), r(Y)\}.\]
For any $r \in \mathbb{R}_{>0}$, if $X$ and $Y$ are $\partial_x/r$-admissible, then so is $X \cap Y$.
\end{prop}
To prove this result, we will need the following elementary statements from $p$-adic functional analysis that we were unable to find in the literature.
\begin{prop}\label{SpectralNorm} Let $V, W$ be two $K$-Banach spaces, let $S : V \to V$ and $T : W \to W$ be bounded $K$-linear maps.
\be \item Let $U : V \h\otimes W \to V \h\otimes W$ be the bounded $K$-linear map $U := S \h{\otimes} 1 + 1 \h{\otimes} T$.
\[ |U|_{\sp, V \h\otimes W} \quad \leq \quad \max\{|S|_{\sp, V}, |T|_{\sp,W}\}.\]
\item Suppose that $W \leq V$ is closed and that $T$ is the restriction of $S$ to $W$. Let $\overline{S} : V/W \to V/W$ be the induced bounded $K$-linear map. Then
\[ |\overline{S}|_{\sp, V/W} \leq |S|_{\sp, V}.\]
\item Suppose that $K'$ is a complete field extension of $K$ and that $V$ is of countable type. Let $V' := V \h\otimes K'$ and let $S' := S \h\otimes 1 : V' \to V'$. Then 
\[ |S|_{\sp,V} = |S'|_{\sp,V'}.\]
\ee \end{prop}
\begin{proof} (a) Write $s := |S|_{\sp, V}$ and $t := |T|_{\sp, W}$. It follows from the definition of spectral radius that there are real constants $A, B > 1$ such that
\[ |S^n(v)| \leq A s^n \qmb{and} |T^n(w)| \leq B t^n \qmb{for all} n \geq 0, v\in V, w \in W.\]
Given $v \in V$ and $w \in W$, we compute that  
\[ U^n(v_i \otimes w_i) = \left(\sum_{m=0}^n \binom{n}{m} S^m \otimes T^{n-m}\right)(v \otimes w) = \sum_{m=0}^n S^m(v) \binom{n}{m} T^{n-m}(w).\]
Hence, given an element $u \in V \otimes W$, whenever we can write $u = \sum_{i=1}^r v_i \otimes w_i$ with $v_i \in V$ and $w_i \in W$, we can calculate that
\[ \begin{array}{lll} |U^n(u)| &\leq & \max\limits_{1 \leq i \leq r} \max\limits_{0 \leq m \leq n} |S^m(v_i)| \cdot |\binom{n}{m}| \cdot |T^{n-m}(w_i)| \leq \\
&\leq &  A s^m \cdot B t^{n-m} \cdot \max\limits_{1 \leq i \leq r} |v_i| |w_i| = \\
&=& AB \max\limits_{1 \leq i \leq r} |v_i| |w_i| \cdot \max\{s,t\}^n.  \end{array}\]
Taking the infimum over all ways of writing $u$ in this form, and using the density of $V \otimes W$ in $V \h\otimes W$, we conclude that
\[ |U^n|_{V \h\otimes W} \leq AB \max\{s,t\}^n\]
for all $n \geq 0$. Extracting $n^{\rm{th}}$-roots and letting $n \to \infty$ gives the result.

(b) Write $s := |S|_{\sp,V}$ and choose a real number $A > 1$ such that $||S^n|| \leq A s^n$ for all $n \geq 0$. Then
\[ |\overline{S}^n(v + W)|_{V/W} = \inf\limits_{w \in W} |S^n(v + w)| \leq \inf\limits_{w \in W} A s^n |v + w| = A s^n |v + W|.\]
We conclude that $||\overline{S}^n|| \leq A s^n$ for all $n \geq 0$. Again, extracting $n^{\rm{th}}$-roots and letting $n \to \infty$ gives the result.

(c) Because $V$ is of countable type, by \cite[Proposition 1.2.1(3)]{FvdPut} we can find a $K$-Banach space isomorphism $\varphi : c_0(K) \stackrel{\cong}{\longrightarrow} V$. We identify $c_0(K')$ with $c_0(K) \h\otimes K'$; this gives us an induced isomorphism $\varphi' : c_0(K') \stackrel{\cong}{\longrightarrow} V'$ given by $\varphi' = \varphi \h\otimes 1$. Let $U : c_0(K) \to c_0(K)$ be the bounded $K$-linear endomorphism defined by $U := \varphi^{-1} \circ S \circ \varphi$, and let $U' : c_0(K') \to c_0(K')$ be defined by $U' = (\varphi')^{-1} \circ S' \circ \varphi'$.  Then we can find positive real constants $a,b,a',b' > 0$ such that for all $n \geq 0$ we have
\[ a ||U^n|| \leq ||S^n|| \leq b||U^n|| \qmb{and} a' ||U'^n|| \leq ||S'^n|| \leq b' ||U'^n||.\]
This quickly implies that $|S|_{\sp, V} = |U|_{\sp,c_0(K)}$ and $|S'|_{\sp, V'} = |U'|_{\sp,c_0(K')}$. 

Let $\{e_1,e_2,\cdots\}$ be the standard orthonormal basis for the $K$-Banach space $c_0(K)$; then it also forms an orthonormal basis for $c_0(K')$ as a $K'$-Banach space. Since $S' = S \h\otimes 1$ and $\varphi' = \varphi \h\otimes 1$, we see that $U' = U \h\otimes 1$, and we can calculate
\[ ||U^n|| = \sup\limits_{m \geq 0} |U^n(e_m)| = \sup\limits_{m \geq 0} |U'^n(e_m)| = ||U'^n|| \quad\mbox{for all} \quad n \geq 0.\]
Hence $|U|_{\sp, c_0(K)} = |U'|_{\sp,c_0(K')}$ and the result follows.
\end{proof}

\begin{proof}[Proof of Proposition \ref{Gtop}] We can give $\cO(X)\h\otimes \cO(Y)$ the structure of a $K[\partial_x]$-module by making $\partial_x$ act by $\partial_x\h\otimes 1+1\h\otimes\partial_x$ as in Proposition \ref{SpectralNorm}(a). Then there is a bounded surjective $K[\partial_x]$-linear map
\[\cO(X) \h\otimes \cO(Y) \twoheadrightarrow  \cO(X \cap Y)\] induced by the map $\cO(X)\times \cO(Y)\to \cO(X\cap Y)$ that sends $(f,g)$ to $f|_{X\cap Y}g|_{X\cap Y}$.  
Since the kernel of this map is closed and $\partial_x$-stable, by Proposition \ref{SpectralNorm} we have
\[\begin{array}{lllll} r(X \cap Y) &=& |\partial_x|_{\sp, \cO(X \cap Y)} \\
&\leq& |\partial_x \h\otimes 1 + 1 \h\otimes \partial_x|_{\sp, \cO(X) \h\otimes \cO(Y)} \\
&\leq& \max \{|\partial_x|_{\sp, \cO(X)}, |\partial_x|_{\sp, \cO(Y)}\}  \\
&=& \max\{r(X), r(Y)\}.\end{array}\]
The second statement is an easy consequence of the first.
\end{proof}

\begin{defn} Let $r \in \mathbb{R}_{>0}$. The \emph{$\partial_x/r$-admissible $G$-topology} on $\A$ has the $\partial_x/r$-admissible affinoids as its admissible sets, and finite coverings as its admissible coverings. We denote this $G$-topology by $\A(\partial_x/r)$. 
\end{defn}

It follows easily from Proposition \ref{Gtop} that $\A(\partial_x/r)$ is indeed a $G$-topology in the sense of \cite[Definition 9.1.1/1]{BGR}. 

\begin{defn}\label{DrX} For any affinoid subdomain $X$ of $\A$, let $\cD_r(X)$ denote $K$-Banach space $\cO(X) \langle \partial/r\rangle$ from Definition \ref{AdelR}.
\end{defn}

\begin{prop}\label{DbRing} Let $Y \subseteq X$ be $(\partial_x/r)$-admissible affinoid subsets of $\A$.
\be \item $\cD_r(X)$ carries an associative, unital $K$-Banach algebra structure.
\item The function $\cD_r(X) \to \cD_r(Y)$ that sends $\sum\limits_{n=0}^\infty a_n \partial^n$ to $\sum\limits_{n=0}^\infty (a_n)_{|Y} \partial^n$ is a bounded homomorphism of $K$-Banach algebras.
\ee\end{prop}
\begin{proof} Since $r(X) < r$, it follows from Definition \ref{dxradm}(a) that 
\[\lim\limits_{\ell \to \infty} ||\partial_x^\ell||_{\cO(X)} / r^\ell= 0.\]
Now Lemma \ref{BanachStarProd} gives the required associative, unital, $K$-Banach algebra structure on $\cD_r(X) = \cO(X) \langle \partial / r \rangle$, proving part (a). 

For part (b), we apply Lemma \ref{BanachUnivProp} with $A = \cO(X)$, $B = \cD_r(Y)$ and $b = \partial \in B$ to obtain a bounded $K$-algebra homomorphism $\cD_r(X) \to \cD_r(Y)$ which extends the restriction map $\cO(X) \to \cO(Y)$ and sends $\partial \in \cD_r(X)$ to $\partial \in \cD_r(Y)$. 
\end{proof}

In this way, $\cD_r$ becomes a presheaf of $K$-Banach algebras on $\A(\partial_x/r)$. Evidently, we are thinking of $\cD_r(X)$ as a particular Banach completion of the usual ring of finite-order differential operators $\cD(X)$ on $X$. The following Lemma makes this more precise.

\begin{lem}\label{TheMapJ} There is an injective homomorphism $j : \cD \to \cD_r$ from the restriction of $\cD$ to $\A(\partial_x/r)$ to $\cD_r$. The image of $j(X)$ is dense for all $X \in \A(\partial_x/r)$.
\end{lem}
\begin{proof} Let $X \in \A(\partial_x/r)$. Then $\cD(X)$ is the skew-polynomial ring $\cO(X)[\partial;\partial_x]$; this gives us an injective left $\cO(X)$-linear map $j(X) : \cD(X) \to \cD_r(X)$ which sends $\partial^n \in \cD(X)$ to $\partial^n \in \cD_r(X)$ for all $n \geq 0$. Lemma \ref{multbypartial} implies that $j(X)$ is a ring homomorphism. The image of $j(X)$ is dense by the definition of $\cD_r(X)$, and it is clear that the maps $j(X)$ commute with the restriction maps in the (pre)sheaves $\cD$ and $\cD_r$.
\end{proof}

\begin{prop}\label{NCTate} Suppose that $r \in \sq{K}$. Then $\cD_r$ is a sheaf of $K$-Banach algebras on $\A(\partial_x/r)$, with vanishing higher \v{C}ech cohomology.
\end{prop}
\begin{proof} By Proposition \ref{DbRing}, $\cD_r$ is a presheaf of $K$-Banach algebras on the $G$-topology $\A(\partial_x/r)$; here we regard it only as a presheaf of $K$-Banach spaces. Let $\D_r$ denote the closed disc of radius $r$: the condition we imposed on $r$ ensures that $\D_r$ is a $K$-affinoid variety. Then for every $X \in \A(\partial_x/r)$, there is an isomorphism of $K$-Banach spaces 
\[\cD_r(X) = \cO(X) \langle \partial / r \rangle \cong \cO(X \times \D_r)\]
which is functorial in $X$. Let $\cU := \{X_1,\cdots, X_m\}$ be a finite affinoid covering in $\A(\partial_x/r)$. Then because $\D_r$ is itself an affinoid variety, $\cV := \{X_1 \times \D_r, \cdots, X_m \times \D_r\}$ is a finite affinoid covering of $X \times \D_r$, so by Tate's Acyclicity Theorem  \cite[Theorem 4.2.2]{FvdPut}, the natural map $\cO(X \times \D_r) \to \check{H}^0(\cV, \cO)$ is an isomorphism, and $\check{H}^j(\cV, \cO) = 0$ for $j > 0$. The result follows.
\end{proof}

We will also need an overconvergent version of these definitions.

\begin{defn}\label{DagSite} Let $r \in \mathbb{R}_{>0}$. 
\be
\item An affinoid subdomain $X$ of $\A$ is \emph{$(\partial_x/r)^\dag$-admissible} if and only if 
\[ r \geq r(X).\]
\item The \emph{$(\partial_x/r)^\dag$-admissible $G$-topology} on $\A$ has the $(\partial_x/r)^\dag$-admissible affinoids as its admissible sets, and finite coverings as the admissible coverings. 
\item We denote this $G$-topology by $\A(\partial_x/r)^\dag$. 
\item For each $X \in \A(\partial_x/r)^\dag$, we define
\[ \cD^\dag_r(X) := \underset{c>r}{\colim}{}  \hsp \cD_c(X).\]
\ee\end{defn}
Note that this colimit is computed in the category of associative unital $K$-algebras. The connecting maps $\cD_c(X) \to \cD_{c'}(X)$ appearing in this colimit are all injective, so one should think of $\cD^\dag_r(X)$ as the union of the $\cD_c(X)$ as $c$ runs over all real numbers strictly bigger than $r$.

\begin{thm}\label{NCDagTate} For any $r \in \bR_{>0}$, $\cD^\dag_r$ is a sheaf of $K$-algebras on $\A(\partial_x/r)^\dag$ with vanishing higher \v{C}ech cohomology.
\end{thm}
\begin{proof} It is clear from the definitions that $\A(\partial_x/r)^\dag \subset \A(\partial_x/c)$ whenever $c > r$. Hence $\cD^\dag_r(X)$ is a $K$-algebra for every $X \in \A(\partial_x/r)^\dag$ by Proposition \ref{DbRing}. The restriction maps in $\cD^\dag_r$ respect the $K$-algebra structure, so $\cD^\dag_r$ is a presheaf of $K$-algebras on $\A(\partial_x/r)^\dag$. We can find a decreasing sequence $c_0 > c_1 > \cdots$ in $\sq{K}$ converging to $r$ from above, so that as presheaves on $\A(\partial_x/r)^\dag$ we have
\[ \cD^\dag_r = \underset{n \to \infty}{\colim}{} \hsp \cD_{c_n}.\]
The sheaf property and the vanishing of higher \v{C}ech cohomology of $\cD^\dag_r$ now follow from Corollary \ref{NCTate} and the exactness of direct limits.
\end{proof}

For future use, we record here how the functors $\cD_r$ and $\cD_r^\dag$ behave with respect to finite base change of the ground field. We leave the easy proof to the reader.
\begin{lem}\label{DrBC} Let $X$ be an affinoid subdomain of $\bA$ and let $K'$ be a finite field extension of $K$. 
\be \item For all $r > r(X)$, $\cD_r(X_{K'}) \cong \cD_r(X) \otimes K'$ as $K'$-Banach algebras.
\item For all $r \geq r(X)$, $\cD_r^\dag(X_{K'}) \cong \cD_r^\dag(X) \otimes K'$ as $K'$-algebras.
\ee\end{lem}

The sheaf of rings $\cD_r$ acts naturally on $\cO$; more precisely, we have the following Lemma which will be frequently useful in performing explicit calculations.

\begin{lem}\label{ActionOnO} Let $r > 0$ and $X \in \bA(\partial_x/r)$. 
\be \item The natural action of $\cD(X)$ on $\cO(X)$ induces a bounded $K$-algebra map
 \[ \sigma_r : \cD_r(X) \to \cB(\cO(X))\]
such that $\sigma_r(a)(b) = ab$ for all $a,b \in \cO(X)$, and $\sigma_r(\partial) = \partial_x$.
\item The map $\sigma_r$ is injective.
\ee\end{lem}
\begin{proof}(a) The map $\sigma_r :  \cO(X) \to \cB(\cO(X))$ which sends $a \in \cO(X)$ to the operator of multiplication by $a$ is a bounded $K$-algebra homomorphism. Definition \ref{dxradm} directly implies that $||\partial_x^k||_{\cO(X)} / r^k \to 0$ as $k \to \infty$. Hence $\sup\limits_{\ell\geq 0}||\partial_x^\ell||_{\cO(X)} / r^\ell <\infty$, so by Lemma \ref{BanachUnivProp}, $\sigma_r$ extends uniquely to a bounded $K$-algebra homomorphism $\sigma_r : \cD_r(X) = \cO(X)\langle \partial/r\rangle \to \cB(\cO(X))$ with the required properties.

(b) Suppose that $Q = \sum_{n\geq 0} a_n \partial^n\in \cD_r(X)$ with $a_n\in \cO(X)$ not all zero. Let $m\geq 0$ be least such that $a_m\neq 0$. Then $x^m\in \cO(X)$ and \[ \sigma_r\left(\sum_{n\geq 0}a_n \partial^n\right)(x^m)= m!a_m.\] Therefore $Q \notin \ker \sigma_r$ and $\sigma_r$ is injective.
\end{proof}

We now turn to the question of calculating, at least in theory, the spectral radius $r(X)$ of an affinoid subdomain $X$ of $\bA$. See \cite[\S 4.1]{ArdWad2023} for the terminology and notation. The spectral radius behaves well with respect to base change:

\begin{lem}\label{BaseChR} Let $K'$ be a complete field extension of $K$ and let $X$ be an affinoid subdomain of $\bA$. Then 
\[ r(X) = r(X_{K'}).\]
\end{lem}
\begin{proof} The $K$-affinoid algebra $\cO(X)$ is a quotient of a Tate algebra, and is therefore of countable type as a $K$-Banach space. Now apply Proposition \ref{SpectralNorm}(c) to the bounded $K$-linear operator $\partial_x : \cO(X) \to \cO(X)$, noting that $\cO(X_{K'}) = \cO(X) \h\otimes K'$ by definition of $X_{K'}$.\end{proof}

In view of \cite[Theorem 4.1.8]{ArdWad2023},  we finish $\S \ref{DrSect}$ by explicitly calculating the spectral radius of affinoid subdomains of the affine line which are finite unions of cheeses.

\begin{defn}\label{MinRad}\hsp \be
\item We define $\rho(C(\mathbf{\alpha}, \mathbf{s})) := \min\limits_{1 \leq i \leq g} |s_i|$. 
\item If $X = X_1 \cup \cdots \cup X_m$ is a disjoint union of cheeses $X_i$, we define
\[\rho(X) := \min\limits_i \rho(X_i).\]
\ee\end{defn}

Recall that $\partial_x^{[n]} := \partial_x^n/n!: \cO(X) \to \cO(X)$ denotes the \emph{$n$th divided power} of $\partial_x$.

\begin{lem}\label{NormDivPow} Let $X$ be an affinoid subdomain of $\A$ that is split over $K$. Then 
\[||\partial_x^{[n]}||_{\cO(X)} = \frac{1}{\rho(X)^n} \qmb{for all} n \geq 0.\]
\end{lem}
\begin{proof} If $X_1, \ldots, X_m$ are the connected components of $X$, then $\cO(X)$ is the orthogonal direct sum of $K$-Banach spaces $\cO(X_1) \oplus \cdots \oplus \cO(X_m)$, and the differential operators $\partial_x^{[n]}$ respects this decomposition. Therefore $||\partial_x^{[n]}||_{\cO(X)} = \max\limits_{1 \leq i \leq m} ||\partial_x^{[n]}||_{\cO(X_i)}$. On the other hand, by definition we have that $\rho(X)^{-n} = \max\limits_{1 \leq i \leq m} \rho(X_i)^{-n}$, and this reduces us to the case where $X = C(\mathbf{\alpha}, \mathbf{s})$ is a cheese. 

For all $n, \ell \geq 0$ and all $i = 1,\ldots, g$ we have the estimate
\[ \left\vert \partial_x^{[n]} \left( \frac{s_i}{x - \alpha_i} \right)^\ell \right\vert_{X} = \left\vert (-1)^n \binom{\ell-n+1}{n} \left(\frac{-1}{s_i}\right)^n \left(\frac{s_i}{x-\alpha_i}\right)^{\ell + n} \right|_{X} \leq |s_i|^{-n},\]
as well as\[ \left\vert \partial_x^{[n]} \left( \frac{x-\alpha_0}{s_0} \right)^\ell \right\vert_{X} = \left\vert \binom{\ell}{n} \left(\frac{1}{s_0}\right)^n \left(\frac{x-\alpha_0}{s_0}\right)^{\ell - n} \right|_{X} \leq |s_0|^{-n}.\]
Since $\rho(X) = \min\limits_{1 \leq i \leq g} |s_i|$ by definition and since $ \min\limits_{1 \leq i \leq g} |s_i| \leq |s_0|$ by \cite[Definition 4.1.1]{ArdWad2023}, we can apply the explicit description of $\cO(C(\alpha,\mathbf{s}))$ given in \cite[Proposition 2.4.8(a)]{LutJacobians} to obtain $||\partial_x^{[n]}||_{\cO(X)} \leq \rho(X)^{-n}$.  For the reverse inequality, say $\rho(X) = |s_i|$ for some $1 \leq i \leq g$. Then taking $\ell = 1$ in the first estimate above shows that 
\[||\partial_x^{[n]}||_{\cO(X)} \geq \left\vert (-1)^n \left(\frac{-1}{s_i}\right)^n \left( \frac{s_i}{x - \alpha_i} \right) \right\vert_X \left/\right. \left\vert\left( \frac{s_i}{x - \alpha_i} \right)\right\vert_X = |s_i|^{-n} = \rho(X)^{-n}. \qedhere\] 
\end{proof}

Recall from $\S \ref{ConvNotn}$ that $|p| := 1/p$ and that $\varpi := p^{-\frac{1}{p-1}} \in \mathbb{R}_{>0}$. 

\begin{lem}\label{nFacVarpi} For all $n \geq 0$, we have $1 \leq \frac{|n!|}{\varpi^n} \leq p n$.
\end{lem}
\begin{proof} $v_p(n!) = \frac{n - s_p(n)}{p-1}$ where $s_p(n)$ is the sum of the $p$-adic digits of $n$, so 
\[\frac{|n!|}{\varpi^n} = |p|^{v_p(n!) - \frac{n}{p-1}} = |p|^{-\frac{s_p(n)}{p-1}} = p^{\frac{s_p(n)}{p-1}}.\] 
Now $0 \leq s_p(n) \leq (p-1)(\log_p(n) + 1)$, so $1 \leq \frac{|n!|}{\varpi^n} \leq p^{\log_p(n) + 1} = p n $.
\end{proof}

\begin{cor}\label{Heisenberg} For $X$ as in Lemma \ref{NormDivPow}, we have
\[r(X) = \varpi / \rho(X).\]
For $r \in \mathbb{R}_{>0}$, $X$ is $\partial_x/r$-admissible if and only if $r > \varpi / \rho(X)$.
\end{cor}
\begin{proof} By Lemma \ref{NormDivPow} we have $||\partial_x^n||_{\cO(X)} = |n!| \cdot ||\partial_x^{[n]}||_{\cO(X)} = \frac{ |n!| }{\rho(X)^n}$. Take $n^{\rm{th}}$-roots, let $n \to \infty$ and apply Lemma \ref{nFacVarpi} to obtain the first statement. The second one follows directly from Definition \ref{dxradm}(b).
\end{proof}

It will be useful to know that for a \emph{general} affinoid subdomain $X$ of $\A$, the spectral radius $r(X)$ is in fact completely determined by $||\partial_x||_{\cO(X)}$; for this we need the following elementary

\begin{lem}\label{FiniteBCforOpNorm} Let $V$ be a $K$-Banach space, let $K'$ be a finite field extension of $K$ and let $V' := V \h\otimes K'$. Let $T : V \to V$ be a bounded $K$-linear map and let $T' := T \h\otimes 1 : V' \to V'$. Then
\[ ||T'|| = ||T||.\]
\end{lem}
\begin{proof} Because $[K':K] < \infty$ we can find a basis $\{a_1,\ldots, a_n\}$ for $(K')^\circ$ as a $K^\circ$-module, where $a_1 = 1$. Then $V' = \bigoplus_{i=1}^n V \h\otimes a_i$, and by \cite[Lemma 17.2]{SchNFA} the tensor product norm on $V'$ is given by
\[ \left\vert \sum\limits_{i=1}^n v_i \h\otimes a_i\right\vert = \max\limits_{1 \leq i \leq n} |v_i| |a_i|. \]
Therefore
\[ \left\vert T'\left(\sum\limits_{i=1}^n v_i \h\otimes a_i\right)\right\vert = \max\limits_{1 \leq i \leq n} |T(v_i) | |a_i| \leq ||T|| \cdot \max\limits_{1 \leq i \leq n} |v_i | |a_i| = ||T|| \hsp \left\vert\sum\limits_{i=1}^n v_i \h\otimes a_i \right\vert\]
which implies that $||T'|| \leq ||T||$. For the reverse inequality, note that for any $v \in V$ we have $|T(v)| = |T'(v \otimes a_1)| \leq ||T'|| |v \otimes a_1| = ||T'|| \hsp |v|$; hence $||T|| \leq ||T'||$.
\end{proof}

\begin{cor}\label{rXcalc} Let $X$ be an affinoid subdomain of $\A$. Then
\[ r(X) = \varpi \hsp ||\partial_x||_{\cO(X)}.\]
\end{cor}
\begin{proof} Both sides of the equation are invariant under passing to a finite field extension, by Lemma \ref{BaseChR} and Lemma \ref{FiniteBCforOpNorm}. By \cite[Theorem 4.1.8]{ArdWad2023}, we may then assume that $X$ splits over $K$. But now $r(X) = \varpi / \rho(X)$ by Corollary \ref{Heisenberg}, whereas $1 / \rho(X) = ||\partial_x||_{\cO(X)}$ by Lemma \ref{NormDivPow}.
\end{proof}

\subsection{Twisting-automorphisms of $\cD_r$}\label{DivTwDer}In this section we study the line bundles with connectionn that arise from certain Kummer-\'etale coverings of $X$, and we investigate when the action of $\cD$ on these line bundles extends to an action of $\cD_r$. 
\begin{lem}\label{zTwist} Let $X$ be a smooth rigid $K$-analytic variety, let $u \in \cO(X)^\times$ and let $d$ be a non-zero integer. There is an $\cO(X)$-linear ring automorphism 
\[\theta_{u,d} : \cD(X) \to \cD(X)\] such that
\begin{equation}\label{ThetaTwist} \theta_{u,d}(\delta	) = \delta - \frac{1}{d} \frac{\delta(u)}{u} \qmb{for all} \delta \in \cT(X).\end{equation}
\end{lem}
\begin{proof} By considering an affinoid covering of $X$, we quickly reduce to the case where $X$ is itself affinoid. Let $Z \to X$ be any \'etale map such that $Z$ is affinoid and $\cO(Z)$ contains a unit $z$ such that $z^d = u$; for example we could take $Z := \Sp \cO(X)[T] / (T^d - u)$. There is a canonical $K$-algebra homomorphism $\cD(X) \to \cD(Z)$; it is injective, and we will identify $\cD(X)$ with its image in $\cD(Z)$.

Since $z$ is a unit in $\cD(Z)$, conjugation by $z$ is defines an $\cO(Z)$-linear ring automorphism $\theta_{u,d} : \cD(Z) \to \cD(Z)$ with inverse $\theta_{u^{-1},d}$.  Let $\delta \in \cT(X)$. Then $\delta(u) = \delta(z^d) = d z^{d-1} \delta(z) = d u z^{-1} \delta(z)$, and $0 = \delta(z z^{-1}) = z \delta(z^{-1}) + \delta(z) z^{-1}$ together show that 
\begin{equation} \label{zdz} z \delta(z^{-1}) = - \delta(z) z^{-1} = - \frac{1}{d} \frac{\delta(u)}{u}.\end{equation}
Therefore for all $f \in \cO(X)$ we have
\[ \theta_{u,d}(\delta)(f) = (z \delta z^{-1})(f) = z \delta(z^{-1} f) = z \delta(z^{-1}) f + \delta(f) = \delta(f) - \frac{1}{d} \frac{\delta(u)}{u},\]
which shows that $\theta_{u,d}$ preserves the image of $\cO(X) + \cT(X)$ inside $\cD(Z)$. Since $X$ is smooth, $\cD(X)$ is generated by $\cO(X)$ and $\cT(X)$, so $\theta_{u,d}$ also preserves the entire image of $\cD(X)$ in $\cD(Z)$, and therefore defines an endomorphism $\theta_{u,d}$ of $\cD(X)$. Since the same is true for $\theta_{u^{-1},d}$, this endomorphism must be bijective. \end{proof}

As our notation suggests, $\theta_{u,d}$ only depends on $u$ and $d$ and not on any particular choice of the \'etale map $Z \to X$. It could have been also defined without introducing the covering space $Z$ of $X$. For future use, we record the following useful observation, obtained by taking $Z \to X$ to be the identity map.

\begin{lem}\label{ThetaDthPower} Let $z \in \cO(X)^\times$. Then $\theta_{z^d,d}(Q) = z Q z^{-1}$ for all $Q \in \cD(X)$.
\end{lem}

Our main aim of the remainder of this section is to show that when $X$ is an $\partial_x/r$-admissible affinoid subdomain of $\bA$ then $\theta_{u,d}$ extends to a bounded $K$-algebra automorphism of $\cD_r(X)$ and thus that $\theta_{u,d}$ extends to a $K$-algebra automorphism of $\cD_r^\dag(X)$ whenever $X$ is $(\partial_x/r)^\dag$-admissible. To do this we will use Lemma \ref{BanachUnivProp} and so we must estimate the norms of the images of the elements $\theta_{u,d}(\delta^\ell)$ under the natural inclusion $\cD(X)\to \cD_r(X)$.

\begin{lem}\label{TwistsMultiply} For all $u, v \in \cO(X)^\times$ we have
\[\theta_{uv,d} = \theta_{u,d} \circ \theta_{v,d}.\]
\end{lem}
\begin{proof} This is a straightforward calculation on the covering space of $X$ obtained by adjoining a $d^{th}$-root of both $u$ and $v$, thinking of $\theta_{uv,d}$ as being conjugation by the $d^{th}$ root of $uv$. It can also be checked directly using the formula (\ref{ThetaTwist}).\end{proof}

Our next result tells us how to rewrite $\theta_{u,d}(\partial^{[n]})$ in standard form as a differential operator of finite order.

\begin{prop}\label{PnStdForm}For every $n \geq 0$ and every $\delta \in \cT(X)$ we have
\[ \theta_{u,d}(\delta^{[n]})  = \sum\limits_{\alpha=0}^n \theta_{u,d}(\delta^{[n-\alpha]})(1)\hsp \delta^{[\alpha]}.\]
\end{prop}
\begin{proof} Choose any \'etale map $Z \to X$ as in the proof of Lemma \ref{zTwist}. For each $a \in \cD(Z)$, let $\ell_a$ and $r_a$ denote the operations of left, respectively, right, multiplication by the element $a$, and let $\ad_a := \ell_a - r_a$. Then 
\[\ad_{\delta}^{[k]}(f) = \delta^{[k]}(f)\]
for each $k \geq 0$ and each $f \in \cO(Z)$. Using the binomial theorem, we now have
\[\begin{array}{lll} \delta^{[n]} z^{-1} &=& \ell_{\delta}^{[n]}(z^{-1}) = \left( r_{\delta} + \ad(\delta) \right)^{[n]}(z^{-1})  \\
&=& \sum\limits_{\alpha=0}^n r_{\delta}^{[\alpha]} \ad_{\delta}^{[n - \alpha]}  (z^{-1}) = \sum\limits_{\alpha=0}^n \delta^{[n - \alpha]}(z^{-1}) \delta^{[\alpha]} . \\  
  \end{array}\]
Hence $\theta_{u,d}(\delta^{[n]}) = z \hsp \delta^{[n]} z^{-1} = \sum\limits_{\alpha=0}^n z \hsp \delta^{[n-\alpha]}(z^{-1}) \delta^{[\alpha]}$. Evaluating both sides of this equation at $1 \in \cO(X)$ shows that $z \delta^{[n]}(z^{-1}) = \theta_{u,d}(\delta^{[n]})(1)$.
\end{proof}

After these generalities, we return to the affine line $\A$ with its local coordinate $x \in \cO_{\A}$. \textbf{We fix an affinoid subdomain $X$ of $\bA$ until the end of $\S \ref{DivTwDer}$}.

\begin{defn}\label{HunDefn} For every $n \geq 0$, we define 
\[h_{u,d}^{[n]} :=  \theta_{u,d}(\partial_x^{[n]})(1) \in \cO(X).\]
\end{defn}
Note that for any \'etale map $Z \to X$ as in the proof of Lemma \ref{zTwist}, we have 
\[ h_{u,d}^{[n]} = z \hsp \partial_x^{[n]}(z^{-1}).\]
Our next task will be to obtain an upper bound for $|h_{u,d}^{[n]}|_X$. Using Proposition \ref{PnStdForm} this will enable us to estimate the norms $|\theta_{u,d}(\partial_x^n)|_{\cD_r(X)}$.

\begin{cor}\label{DivPow2} Suppose that $\alpha\in K$ such that $x-\alpha\in \cO(X)^\times$ and $Y\to X$ is an \'etale cover such that there is $y\in \cO(Y)$ with  $y^d = x - \alpha$. Then
\[ \partial_x^{[m]}(y^k) = \binom{\frac{k}{d}}{m} y^k (x - \alpha)^{-m} \qmb{ for all}  k \in \Z \qmb{and} m\geq 0.\]
\end{cor}
\begin{proof} The case $m=0$ is trivial. Since $1=\partial_x(y^d)=d\partial_x(y)y^{d-1}$ we see that $\partial_x(y)=\frac{1}{dy^d}y=\frac{1}{d(x-\alpha)}y$
which implies by the Leibniz rule that $\partial_x(y^k) = \frac{k}{d} \frac{y^k}{x-\alpha}$ as required for the case $m = 1$. Now we proceed by induction on $m$: 
\begin{eqnarray*} \partial_x^{m+1}(y^k) &=& \partial_x\left( m! \binom{\frac{k}{d}}{m} \frac{y^k}{(x-\alpha)^m} \right) \\ 
&=& m! \binom{\frac{k}{d}}{m} \left(\frac{k}{d}\frac{y^k}{x-\alpha} \cdot (x-\alpha)^{-m} - m y^k (x - \alpha)^{-m-1}\right) \\
&=& m! \binom{\frac{k}{d}}{m}\frac{y^k}{(x-\alpha)^{m+1}} \left(\frac{k}{d} - m\right) \\
&=& (m+1)! \binom{\frac{k}{d}}{m+1} \frac{y^k}{(x-\alpha)^{m+1}}.\end{eqnarray*} 
This completes the induction.\end{proof}

Here is our upper bound.
\begin{prop}\label{HunEst}  Let $X$ be an affinoid subdomain of $\bA$, let  $u \in \cO(X)^\times$ and suppose that $p \nmid d$. Then 
\[ |h_{u,d}^{[n]}|_{X} \leq \left(\frac{r(X)}{\varpi}\right)^n \qmb{for all} n \geq 0.\]
\end{prop}
\begin{proof} Suppose first that $X = C(\mathbf{\alpha},\mathbf{s})$ is a cheese. Using \cite[Proposition 2.4.8(b)]{LutJacobians}, write $u = \lambda \cdot \nu \cdot (x-\alpha_1)^{k_1} \cdots (x - \alpha_g)^{k_g}$ for some $\lambda \in K^\times, \nu \in  \cO(X)^{\times\times}$ and $k_i \in \Z$. We form the \'etale cover $Y\to X$ given by 
\[\cO(Y)=\cO(X)[T_1^{\pm 1},\ldots,T_g^{\pm 1}]/\left\langle T_1^d-(x-\alpha_1), \cdots, T_1^d - (x-\alpha_d)\right \rangle.\] 
For each $i = 1,\ldots, g$, let $z_i$ denote the image of $T_i$ in $\cO(Y)$ so that $z_i^d = x - \alpha_i$. Since $p \nmid d$, by \cite[Lemma 4.3.2(a)]{ArdWad2023} we can find $\epsilon \in \cO(X)^{\times\times}$ such that $\nu = \epsilon^d$, so that $u = \lambda \cdot (z_1^{k_1} \cdots z_g^{k_g} \epsilon)^d$. Note that $\theta_{\lambda,d}$ is the identity map because $\partial(\lambda) = 0$; using Lemma \ref{TwistsMultiply} we see that we can assume $\lambda = 1$ so that $z := z_1^{k_1} \cdots z_g^{k_g} \epsilon$ satisfies $z^d = u$. Using the Leibniz rule, and Corollary \ref{DivPow2}, we obtain
\begin{equation}\label{MuCalc}\begin{array}{lll} z \hsp \partial_x^{[n]}(z^{-1}) &=& \epsilon \hsp z_1^{k_1} \cdots z_g^{k_g} \hsp \partial_x^{[n]}(z_1^{-k_1}\cdots z_g^{-k_g} \epsilon^{-1})  \\
&=& \sum\limits_{|m|=n} \left(\prod\limits_{i=1}^g  \binom{-\frac{k_i}{d}}{m_i} (x - \alpha_i)^{-m_i} \right) \cdot \epsilon \hsp \partial_x^{[m_{g+1}]}(\epsilon^{-1})
\end{array}\end{equation}
where the sum runs over all $m \in \N^{g+1}$ such that $|m| := \sum\limits_{i=1}^{g+1}m_i$ equals $n$. Now
\[ | \epsilon  \hsp \partial_x^{[m_{g+1}]}(\epsilon^{-1})|_{X} \leq |\epsilon|_X \cdot ||\partial_x^{[m_{g+1}]}||_{\cO(X)} \cdot |\epsilon^{-1}|_X = \rho(X)^{-m_{g+1}}\]
by Lemma \ref{NormDivPow} because $|\epsilon|_X = 1$. Using $(\ref{MuCalc})$ together with Corollary \ref{Heisenberg}, we now obtain the required estimate
\[ |h_{u,d}^{[n]}|_{X} = |z \hsp \partial_x^{[n]}(z^{-1})|_{X} \leq \max\limits_{|m|=n} \prod_{i=1}^g |s_i|^{-m_i} \cdot \rho(X)^{-m_{g+1}} \leq \rho(X)^{-n} = \left(\frac{r(X)}{\varpi}\right)^n.\]
Next we consider the case where $X$ splits over $K$; that is every connected component $X_i$ of $X$ is a cheese. Once again, choose any \'etale map $Z\to X$ as in the proof of Lemma \ref{zTwist}.  Let $u_i$ denote the restriction of $u$ to $X_i$ and let $z_i$ denote the restriction of $z$ to $Z_i := Z \times_X X_i$. Then using Definition \ref{MinRad}(b) together with Corollary \ref{Heisenberg}, we have
\[ \begin{array}{lllllll} |h_{u,d}^{[n]}|_X &=& |z \hsp \partial_x^{[n]}(z^{-1})|_X &=& \max\limits_i |z_i \hsp \partial_x^{[n]}(z_i^{-1})|_{X_i} &&\\
&\leq & \max\limits_i \rho(X_i)^{-n} &=&  \rho(X)^{-n} &=& \left(\frac{r(X)}{\varpi}\right)^n.\end{array}\]
Finally, suppose that $X$ is arbitrary. By \cite[Theorem 4.1.8]{ArdWad2023}, we can find a finite extension $K'$ of $K$ such that $X$ splits over $K'$. Since $r(X) = r(X_{K'})$ by Lemma \ref{BaseChR}, we obtain
\[ |h_{u,d}^{[n]}|_X = |h_{u,d}^{[n]}|_{X_{K'}} = \left(\frac{r(X_{K'})}{\varpi}\right)^n = \left(\frac{r(X)}{\varpi}\right)^n \qmb{for all} n \geq 0. \qedhere\]
\end{proof}
Our next main result, Theorem \ref{ThetaU}, requires the following elementary estimate.
\begin{lem}\label{RhoEst} Let $0 < \rho < 1$. Then $\max\limits_{t \geq 0} \hsp t \hsp \rho^t = ( e \log(\rho^{-1}) )^{-1}$.
\end{lem}
\begin{proof} Write $\lambda := \log(\rho^{-1}) > 0$, and consider the function $f : \bR_{\geq 0} \to \bR$ defined by $f(t) := t \hsp \rho^t = t e^{- \lambda t}$. Then $f'(t) = e^{-\lambda t} - \lambda t e^{-\lambda t}$ vanishes precisely when $t = \lambda^{-1}$, and $f''(t) = -\lambda e^{-\lambda t} - \lambda(1 - \lambda t) e^{-\lambda t}$ is strictly negative at $t = \lambda^{-1}$. So $t = \lambda^{-1}$ a global maximum of $f(t)$ with value $f(\lambda^{-1}) = \lambda^{-1} e^{-1} = ( e \log(\rho^{-1}) )^{-1}$.\end{proof}

\begin{thm}\label{ThetaU}Let $r > 0$, let $X$ be an $\partial_x/r$-admissible affinoid subdomain of $\bA$, let  $u \in \cO(X)^\times$ and suppose that $p \nmid d$. Then $\theta_{u,d} \in \Aut_K \cD(X)$ extends to a bounded $K$-algebra automorphism $\theta_{u,d}$ of $\cD_r(X)$.
\end{thm}
\begin{proof} By Lemma \ref{TheMapJ}, there is an injective $\cO(X)$-linear $K$-algebra map $j : \cD(X) \to \cD_r(X)$ such that $j(\partial_x) = \partial$. We will show that the $K$-algebra homomorphism 
\[f := j \circ \theta_{u,d} : \cO(X) \to \cD_r(X)\]
extends to a bounded $K$-algebra homomorphism 
\[\theta_{u,d} :  \cD_r(X) = \cO(X)\langle \partial/r\rangle \to \cD_r(X)\]
which sends $\partial \in \cD_r(X)$ to $b := j(\theta_{u,d}(\partial_x)) = \partial - \frac{1}{d} \frac{\partial_x(u)}{u}\in \cD_r(X)$. This follows from Lemma \ref{BanachUnivProp}, provided we can verify conditions (i) and (ii) in this Lemma. 

(i) Let $a \in \cO(X)$. Then because $\theta_{u,d}$ and $j$ are $K$-algebra homomorphisms,
\[[b, f(a)] = [j(\theta_{u,d}(\partial_x)), j(\theta_{u,d}(a))] = j (\theta_{u,d}([\partial_x, a])) = f ( \partial_x(a) ).\]

(ii) Let $\ell \geq 0$. Then by Proposition \ref{PnStdForm} we have
\[ b^\ell = f(\partial_x^\ell) = \ell! \sum\limits_{\alpha=0}^\ell h_{u,d}^{[\ell - \alpha]} \partial^{[\alpha]} = \sum\limits_{\alpha=0}^\ell h_{u,d}^{[\ell - \alpha]} (\ell - \alpha)! \binom{\ell}{\alpha} \partial^{\alpha}\]
inside $\cD_r(X)$. Using Proposition \ref{HunEst} we can now estimate
\[ \begin{array}{lll} |b^\ell| / r^\ell &\leq&  \sup\limits_{0 \leq \alpha \leq \ell} \left\vert h_{u,d}^{[\ell-\alpha]}\right\vert_X \cdot |(\ell - \alpha)!| \cdot r^{\alpha - \ell}  \leq \\
&\leq &  \sup\limits_{0 \leq \alpha \leq \ell} \left(\frac{r(X)}{\varpi}\right)^{\ell - \alpha} \cdot \frac{|(\ell - \alpha)!|}{r^{\ell - \alpha}} = \\
&=&\sup\limits_{0 \leq n \leq \ell} \left(\frac{r(X)}{r}\right)^n \cdot \frac{|n!|}{\varpi^n}.\end{array}\]
Since $X$ is $\partial_x/r$-admissible, the ratio $\rho := r(X) / r$ is strictly less than $1$. Lemma \ref{nFacVarpi} together with Lemma \ref{RhoEst} now show that
\[ \sup\limits_{\ell \geq 0} |b^\ell| / r^\ell \leq \sup\limits_{\ell \geq 0} \left(p \hsp \sup\limits_{n \geq 0} n \rho ^n \right) < \infty.\]
Lemma \ref{BanachUnivProp} now gives the required $K$-algebra endomorphism $\theta_{u,d}$ of $\cD_r(X)$, which is in fact bijective because $\theta_{u,d} \circ \theta_{u^{-1},d} = \theta_{u^{-1},d} \circ \theta_{u,d} = 1_{\cD_r(X)}$.
\end{proof}

\begin{rmk} It can be shown that the operator norm of the automorphism $\theta_{u,d}$ of $\cD_r(X)$ satisfies
\[ ||\theta_{u,d}|| \leq  p \left[ e  \hsp \log\left(\frac{r}{r(X)}\right) \right]^{-1}.\]
We omit the proof as we do not need this estimate.
\end{rmk}
Inspecting the definition of $\theta_{u,d} : \cD(X) \to \cD(X)$ given in the statement of Lemma \ref{zTwist} shows that it commutes with the restriction maps $\cD(X) \to \cD(Y)$ for any affinoid subdomain $Y$ of $X$. Therefore these automorphisms assemble to give an automorphism $\theta_{u,d} : \cD_X \to \cD_X$ of the sheaf of $K$-algebras $\cD_X$ for any admissible open subspace $X$ of $\A$ and any $u \in \cO(X)^\times$. If $X$ is an admissible open subset of $\bA$ we define $X(\partial_x/r)^\dag$ to be the $G$-topology obtained by restricting the $G$-topology $\bA(d_x/r)^\dag$ from Definition \ref{DagSite} to those affinoid open sets contained in $X$.

\begin{cor}\label{ThetaUDag} Let $X$ be an admissible open subspace of $\A$, let $u \in \cO(X)^\times$ and suppose that $p \nmid d$. Then for every $r > 0$, the restriction of $\theta_{u,d} : \cD \to \cD$ to the $G$-topology $X(\partial_x/r)^\dag$ extends to an automorphism $\theta_{u,d} : \cD^\dag_r \to \cD^\dag_r$ of sheaves of $K$-algebras on $X(\partial_x/r)^\dag$.
\end{cor}
\begin{proof} This follows immediately from Theorem \ref{ThetaU}.
\end{proof}

For future use, we record the fact that the automorphisms $\theta_{u,d}$ are \emph{inner} whenever $u$ happens to be a power of $d$ \emph{in $\cO(X)^\times$}. More precisely, we have the following
\begin{lem}\label{NearlyInner} Let $r > 0$, let $X$ be an $\partial_x/r$-admissible affinoid subdomain of $\bA$, let $v \in \cO(X)^\times$ and suppose that $p \nmid d$. Then
\[ \theta_{z^d, d} (Q) = z \hsp Q \hsp z^{-1} \qmb{for all} Q \in \cD_r(X).\]
\end{lem}
\begin{proof} Conjugation by $z$ is a bounded $K$-linear automorphism of $\cD_r(X)$, which agrees with $\theta_{z^d,d}$ on $\cD(X)$ by Lemma \ref{ThetaDthPower}. Also, $\theta_{z^d, d}$ is bounded and $K$-linear on $\cD_r(X)$ by Theorem \ref{ThetaU}. The result now follows since $\cD(X)$ is dense in $\cD_r(X)$.\end{proof}

\subsection{Affine transformations of \ts{\bA}}\label{AffTrans}

In this section we consider how the action of the group of affine transformations of $\bA$ relates to various constructions we have made. Let $\bB := \{g \in \mathbb{GL}_2 : g_{21}=0\}$ be the subgroup scheme of upper-triangular matrices in $\mathbb{GL}_2$.

\begin{defn} \label{AffA} Let $\varrho\colon \bB(K)\to K^\times$ be the character given by
\[ \varrho \left(\begin{pmatrix} a & b \\ 0 & d\end{pmatrix}\right) \mapsto \frac{a}{d}. \] 
\end{defn}
For the necessary background on equivariant structures and equivariant sheaves, we refer the reader to \cite[\S 2.3]{EqDCap}.
\begin{lem}\hfill \be \label{AutAOD}
	
	\item $\cO_{\bP^1}$ has a natural $\mathbb{GL}_2(K)$-equivariant structure as a sheaf of $K$-algebras given by $(g\cdot f)(z)=f(g^{-1}z)$.
	\item The action of $\mathbb{GL}_2(K)$ on $\bP^1$ preserves affinoid subdomains.
	\item If $X$ is an affinoid subdomain of $\bP^1$ and $G\leq \mathbb{GL}_2(K)$ stabilises $X$ then there is a natural group homomorphism $\rho\colon G\to \cB(\cO(X))^\times$. 
	\item The $\mathbb{GL}_2(K)$-equivariant structure on $\cO_{\bP^1}$ extends to an $\mathbb{GL}_2(K)$-equivariant structure as a sheaf of $K$-algebras on $\cD_{\bP^1}$ via \[ \begin{pmatrix}a&b\\c&d\end{pmatrix}\cdot \partial_x=\frac{(-cx+a)^2}{ad-bc}\partial_x.\]
	\item The $\mathbb{GL}_2(K)$-equivariant structure on $\cD_{\bP^1}$ restricts to a $\bB(K)$-equivariant structure on $\cD_{\bA}$ which satisfies 
	\[ g\cdot \partial_x= \varrho(g)\partial_x \qmb{and} g \cdot (x - z) = \varrho(g)^{-1}(x - g \cdot z) \qmb{for all} z \in \overline{K}.\]   	\ee
\end{lem}

\begin{proof} (a) The affine algebraic group $\mathbb{GL}_2$ acts on the scheme $\bP^1$ by M\"obius transformations. By \cite[Theorem 6.3.4]{EqDCap}, the action of $\mathbb{GL}_2(K)$ on the rigid analytic variety $\bP^1$ is continuous. In the proof of \cite[Lemma 3.4.3]{EqDCap} it is explained how the structure sheaf $\cO_X$ and the sheaf of finite-order differential operators $\cD_X$ on a smooth rigid analytic variety $X$ equipped with a continuous action of a $p$-adic Lie group $G$ can be endowed in a standard way with natural $G$-equivariant structures. These constructions do not require $G$ to be a $p$-adic Lie group, so they can also be applied in our setting to the continuous $\mathbb{GL}_2(K)$-action on $\bP^1$.
	
	(b) For any affinoid subdomain $X$ of $\bP^1$, each $g\in \mathbb{GL}_2(K)$ induces an isomorphism of rigid analytic $K$-varieties $(X,\cO_X)\to (g(X), \cO_{g(X)})$. Thus $g(X)$ is an affinoid subdomain of $\bP^1$.
	
	(c) is an immediate consequence of (a) together with \cite[Theorem 6.1.3/1]{BGR}.
	
	(d) Using \cite[Example 2.1.4 and Corollary 2.1.9]{EqDCap} we see that for all $g\in \mathbb{GL}_2(K)$ stabilising $X$, the $g$-action on $\cT(X)$ is given by
	\begin{eqnarray}
		\label{GactOandT}  (g\cdot \partial)(f)=g\cdot \partial(g^{-1}\cdot f)  \qmb{for all} \partial \in \cT(X), f \in \cO(X).
	\end{eqnarray} 
	This implies that $\left(\begin{pmatrix} a & b \\ c & d\end{pmatrix}\cdot \partial_x\right)(x)=\frac{(-cx+a)^2}{ad-bc}$, and (d) follows.
	
	(e) The first statement is now immediate given the formula in part (d). For the second one, we write $g = \begin{pmatrix} a & b \\ 0 & d \end{pmatrix}$ and compute
	 \[g \cdot (x-z) = \frac{dx-b}{a} - z = \frac{d}{a}\left(x - \left(\frac{az+b}{d}\right)\right) = \varrho(g)^{-1} (x - g \cdot z). \qedhere\]

\end{proof}
\begin{prop}\label{AutAandr} Let $X$ be an affinoid subdomain of $\bA$ and let $g\in \bB(K)$. Then \[r(g(X))=\frac{r(X)}{|\varrho(g)|}.\] 	
\end{prop}
\begin{proof}We note that the isomorphism of Banach algebras $g^\cO(X)\colon \cO(X)\to \cO(gX)$ is even an isometry since for every $f\in \cO(X)$ we have \[ |g\cdot f|_{gX}=\sup_{x\in X} f(g^{-1}gx)=|f|_X. \] 
	
	Using Lemma \ref{AutAOD}(e), we see that for any $f\in \cO(gX)$ and $n \geq 0$, we have \[|\partial_x^n(f)|_{g(X)}=|g^{-1}\cdot \left(\partial_x^n(f)\right)|_X=|\varrho(g)|^{-n} |\partial_x^n (g^{-1}\cdot f)|_X.\]
	
Hence $r(g(X))=|\partial_x|_{\sp,\cO(g(X))}=\frac{1}{|\varrho(g)|}|\partial_x|_{\sp, \cO(X)}=\frac{r(X)}{|\varrho(g)|}$ as required.  
	\end{proof}

In view of Definitions \ref{dxradm}(b) and \ref{DagSite}(a), Proposition \ref{AutAandr} and Lemma \ref{AutAOD}(b) together imply the following

\begin{cor}\label{gActsOnSites} Let $r\in \bR_{>0}$. Then every $g \in \bB(K)$ induces homeomorphisms
\[g : \bA(\partial_x/r) \stackrel{\cong}{\longrightarrow} \bA(|\varrho(g)|\partial_x/r) \qmb{and} g : \bA(\partial_x/r)^\dag \stackrel{\cong}{\longrightarrow}  \bA(|\varrho(g)|\partial_x/r)^\dag.\]
\end{cor}

\begin{lem}  \label{AffDr} Let $g\in \bB(K)$ and let $r \in \bR_{>0}$.
	\be \item There is an equivalence of categories \[ g_\ast\colon  \PreSh(\bA(\partial_x/r))\stackrel{\cong}{\longrightarrow} \PreSh(\bA(|\varrho(g)|\partial_x/r)) \] given by $(g_\ast \cF)(X)=\cF(g^{-1}X)$ for all $\cF \in \PreSh(\bA(\partial_x/r))$.
\item $g^\cD$ induces an isomorphism of presheaves of $K$-Banach algebras on $\bA(\partial_x/r)$  
\[ g_r\colon \cD_r\stackrel{\cong}{\longrightarrow} g^\ast \cD_{\frac{r}{\varrho(g)}}.  \] 
\ee\end{lem} 
\begin{proof}
(a) This follows immediately from Corollary \ref{gActsOnSites}.
	
(b) Let $X \in \bA(\partial_x/r)$. We will apply Lemma \ref{BanachUnivProp} with the following parameters: $A = \cO(X)$, $B = \cD_{\frac{r}{\varrho(g)}}(gX)$, $\delta = \partial_x \in \cB(A)$, $f : A \to B$ is the composition of $g^{\cO}(X) : \cO(X) \to \cO(gX)$ and the inclusion $\cO(gX) \hookrightarrow \cD_{\frac{r}{\varrho(g)}}(gX)$, and $b := \varrho(g) \partial \in B$. First, we must verify conditions (b)(i) and (b)(ii) of this Lemma.

\noindent (i) Let $a \in \cO(X)$. Using Lemma \ref{AutAOD}(e) and (\ref{GactOandT}), we compute in $B$ as follows:
\[ [b, f(a)] = [\varrho(g) \partial, g \cdot a] = \varrho(g) \partial_x( g \cdot a ) = (g \cdot \partial_x)( g \cdot a) = g \cdot (\partial_x \cdot a ) = f(\delta(a)).\]
(ii) Let $\ell \geq 0$; then $|\partial^\ell| = (r / |\varrho(g)|)^\ell$ in $B = \cD_{\frac{r}{\varrho(g)}}(gX)$ by Definition \ref{AdelR}. Hence 
\[ \sup\limits_{\ell \geq 0} |b^\ell|/r^\ell = \sup\limits_{\ell \geq 0}|\varrho(g)|^\ell |\partial^\ell| / r^\ell = 1.\]
Hence by Lemma \ref{BanachUnivProp}(b) $\Rightarrow$(a), the map $f : A \to B$ extends to a bounded $K$-Banach algebra homomorphism $g_r(X) : A \langle \partial / r \rangle = \cD_r(X) \to \cD_{\frac{r}{\varrho(g)}}(gX)$. By construction, this map makes the following diagram commutative:
\begin{equation}\label{DgjD} \xymatrix{ \cD(X) \ar[rr]^{g^{\cD}(X)}\ar[d]_{j_r(X)} && \cD(gX) \ar[d]^{j_{\frac{r}{|\varrho(g)|}}(gX)} \\ \cD_r(X) \ar[rr]_{g_r(X)} && \cD_{\frac{r}{|\varrho(g)|}}(gX) }\end{equation}
where the vertical maps come from Lemma \ref{TheMapJ}. Since the images of these maps are dense by Lemma \ref{TheMapJ}, the fact that $(g^{-1})^{\cD}(gX)$ is a two-sided inverse to $g^{\cD}(X)$ implies that $(g^{-1})_{\frac{r}{\varrho(g)}}(gX)$ is a two-sided inverse for $g_r(X)$. By construction, the maps $g_r$ commute with restriction maps in $\cD_r$ and $\cD_{\frac{r}{|\varrho(g)|}}$. Hence $g_r$ is an isomorphism of sheaves of $K$-Banach algebras as claimed.
\end{proof}

\begin{cor}\label{gDagTransform} Let $g\in \bB(K)$ and let $r\in \R_{>0}$. 
\be\item There is an equivalence of categories \[ g_\ast \colon \mathrm{PreSh}(\bA(\partial_x/r)^\dag)\to\mathrm{PreSh}(\bA(|\varrho(g)|\partial_x/r)^\dag)\] given by $(g_\ast\cF)(X)=\cF(g^{-1}X)$ for all $\cF \in \PreSh(\bA(\partial_x/r))$.
\item $g^\cD$ induces an isomorphism of sheaves of $K$-algebras on $\bA(\partial_x/r)^\dag$ \[ g^\dag_r \colon  \cD_r^\dag\stackrel{\cong}{\longrightarrow}  g^\ast \cD_{\frac{r}{\varrho(g)}}^\dag. \]
	\ee 
\end{cor}

\begin{proof}
(a) follows from Proposition \ref{AutAandr}. (b) follows from (a), Definition \ref{DagSite} and Lemma \ref{AffDr}(b), noting that $\cD_r^\dag$ and $\cD_{\frac{r}{\varrho(g)}}^\dag$ are in fact sheaves on $\bA(\partial_x/r)^\dag$ and $\bA(\varrho(g)\partial_x/r)^\dag$, respectively, by Theorem \ref{NCDagTate}.
\end{proof}

Finally, we record a purely algebraic calculation which tells us how the maps $g^{\cD} : \cD \to g^\ast \cD$ defining the $G$-equivariant structure on $\cD$ interact with the twisting automorphisms $\theta_{u,d}$ from $\S \ref{DivTwDer}$.
\begin{lem}\label{ThetaudG} Let $g \in \bB(K)$ and let $W$ be an affinoid subdomain of $\bA$. For each $u \in \cO(W)^\times$ and $d \geq 1$, the following diagram commutes:
\[\xymatrix{\cD(W) \ar[rr]^{\theta_{u,d}}\ar[d]_{g^{\cD}(W)} && \cD(W) \ar[d]^{g^{\cD}(W)} \\ \cD(gW) \ar[rr]_{\theta_{g \cdot u,d}} && \cD(gW).}\]
\end{lem}
\begin{proof}Because $\cD(W)$ is generated as a ring by $\cO(W)$ and $\partial_x$, it is enough to check that $g \cdot \theta_{u,d}(f) = \theta_{g \cdot u,d}(g \cdot f)$ when $f = \partial_x$ and when $f \in \cO(W)$. Now
\[ \begin{array}{lllll} g \cdot \theta_{u,d}(\partial_x)) &=& g \cdot (\partial_x - \frac{1}{d} \frac{\partial_x(u)}{u}) &=& g \cdot \partial_x - \frac{1}{d}\frac{ g \cdot \partial_x(u) }{g \cdot u} = \\
&=& g \cdot \partial_x - \frac{1}{d} \frac{ (g \cdot \partial_x) (g \cdot u)}{g \cdot u} &=& \theta_{g \cdot u,d}(g \cdot \partial_x),\end{array}\]
which gives the first equality, and the second one holds because the twisting automorphisms $\theta_{u,d}$ and $\theta_{g\cdot u,d}$ are $\cO$-linear.
\end{proof}

\section{Noetherianity of \ts{\cD_r} and flatness of connecting maps}\label{FlatnessChapter}

 In this section we will establish a different interpretation of $\cD_r(X)$, based on work of Berthelot \cite{Berth}, when $X$ is a $\partial_x/r$-admissible affinoid subdomain of the rigid-analytic affine line $\A$ and use it to prove some basic structural facts about these Banach algebras and the relationships between them. In particular we will prove that they are Noetherian whenever $r\in \sqrt{|K^\times|}$ and $X\in \bD(\partial_x/r)$, and that $\cD_s(X)\to \cD_r(Y)$ is flat on both sides whenever $s\geq r$ with $r,s\in \sqrt{|K^\times|}$ and $Y\subseteq X\in \bA(\partial_x/r)$. 

\subsection{Sections of \ts{\cD_r} as divided power algebras}

 We will discuss a construction essentially due to Berthelot involving level $m$ divided powers of $\partial_x$.

\textbf{Throughout $\S \ref{FlatnessChapter}$, we fix the non-negative integer $m$}. 

\begin{defn} We recall some notation from \cite[\S 1.1.2]{Berth}.\hfill \label{divpownot}
	
	For $k,k',k''\geq 0$ with $k'+k''=k$:
	\begin{enumerate}[(a)]
		
		\item $q_k := \lfloor k/p^m\rfloor$ so that $k=q_kp^m + r_k$ with $0\leq r_k<p^m$;
		\item $\binomb{k}{k'}:= \frac{q_k!}{q_{k'}!q_{k''}!} \in \N$; 
		\item $\binoma{k}{k'}:= \binom{k}{k'} \binomb{k}{k'} ^{-1}\in \Z_p$; and given a formal variable $\partial$,
		\item  $\partial^{\langle k\rangle} := q_k! \partial^{[k]}=\frac{q_k!}{k!}\partial^k$.
	\end{enumerate}
\end{defn} It is understood that all of these quantities depend on the parameter $m$, so it would be more correct to write $\langle k \rangle_m$ instead of $\langle k \rangle$ everywhere. However, following Berthelot, we suppress the parameter $m$ from this notation.

Now we suppose that $X$ is an affinoid variety and that $\partial$ is a derivation of $\cO(X)$.

\begin{lem}\label{levelmdivpowmult} The following relations hold in $\cD(X)$:
	\begin{enumerate}[(a)] 
		\item \label{divpow} for all $k,k'\geq 0$, we have \[ \partial ^{\langle k\rangle}\partial^{\langle k'\rangle}=\binoma{k+k'}{k} \partial^{\langle k+k'\rangle};\]
		\item \label{divpow2} for all $k\geq 0$ and $f\in \cO(X)$, we have\[ \partial^{\langle k\rangle}f= \sum_{k'+k''=k} \binomb{k}{k'} \partial^{\langle k'\rangle}(f) \partial^{\langle k''\rangle} . \]
	\end{enumerate}
\end{lem}
\begin{proof} This follows from \cite[p27]{GoodearlWarfield} together with an easy computation using Notation \ref{divpownot}. \end{proof}

\begin{defn}\label{defOXcircpartm} We define  $\cO(X)^{\circ}[\partial]^{(m)}$ to be the $\cO(X)^\circ$-subalgebra of $\cD(X)$ generated by the set $\left \{\partial^{\langle k\rangle}|k\geq 0\right \}$.  \end{defn} 

We note that in particular $\cO(X)^{\circ}[\partial]^{(0)}=\cO(X)^{\circ}[\partial]$. 

\begin{prop} \label{Levelmdivpow}\hfill \be \item If $\sup_{k\geq 1}||\partial^{\langle k\rangle}||_{\cO(X)}\leq 1$, then  $\{\partial^{\langle k\rangle}\mid k\geq 0\}$ is a free generating set for $\cO(X)^\circ[\partial]^{(m)}$ as a left $\cO(X)^\circ$-module; in particular $\cO(X)^\circ[\partial]^{(m)}$ is $p$-adically separated flat $K^\circ$-algebra.

\item If there is $\pi\in K^{\circ\circ}$ with $|\pi|\geq \max(|p|, \sup\limits_{k\geq 1}||\partial^{\langle k\rangle}||)$, then $\cO(X)^\circ[\partial]^{(m)}/(\pi)$ is a commutative $\cO(X)^\circ/(\pi)$-algebra of finite presentation.\ee 
\end{prop}
\begin{proof} 
	
	(a) That the set $S=\{\partial^{\langle k\rangle}\mid k\geq 0\}$ is $\cO(X)^\circ$-linearly independent in the left module $\cD(X)$ follows from its $\cO(X)$-linear independence. It thus suffices to show that the (free left) $\cO(X)^\circ$-submodule of $\cD(X)$ generated by $S$ is a subring of $\cD(X)$. This follows from Lemma \ref{levelmdivpowmult} and the assumption $\partial^{\langle k\rangle}(\cO(X)^\circ)\subset \cO(X)^\circ$ for all $k\geq 0$.

	(b) Now suppose that $\pi$ is as in the statement. As in \cite[Corollaire 2.2.5]{Berth} we see that if $k=\sum_{j=0}^{m-1}c_jp^j + cp^m$ with $0\leq c_j< p$ for all $j$ then 
\begin{equation}\label{d<k>formula} \partial^{\langle k\rangle}= u\left( \prod\limits_{j=0}^{m-1} (\partial^{\langle p^j\rangle})^{c_j}\right) (\partial^{\langle p^m\rangle})^c \end{equation} for some $u\in \Z_p^\times$. Thus $\cO(X)^{\circ}[\partial]^{(m)}/(\pi)$ is generated over $\cO(X)^\circ/(\pi)$ by the images of $\partial^{\langle p^j\rangle}$ for $0\leq j \leq m$. Moreover these generators commute with $\cO(X)^\circ/(\pi)$ by Lemma \ref{levelmdivpowmult}(b) together with the assumption on $\pi$ and they obviously commute with each other.
	
Now for any $0\leq j<m$ we have \[ (\partial^{\langle p^j\rangle})^p= \frac{p^{j+1}!}{(p^j!)^p}\partial^{\langle p^{j+1}\rangle} \in p\cO(X)^\circ[\partial]^{(m)}\subseteq \pi \cO(X)^\circ[\partial]^{(m)} .\]  Using this together with part (a) and equation (\ref{d<k>formula}), we conclude that
 \[ \cO(X)^\circ[\partial]^{(m)}/(\pi)\cong \cO(X)^\circ/(\pi)[t_0,\ldots,t_m]/(t_0^p,\ldots,t_{m-1}^p). \]
 Hence $\cO(X)^\circ[\partial]^{(m)}/(\pi)$ is finitely presented over $\cO(X)^\circ/(\pi)$, as required.
\end{proof}	

\begin{defn}\label{defDsmX} Suppose that $X$ is an affinoid subdomain of $\bA$. We write 
\[\cD_s^{(m)}(X):=\cO(X)^{\circ}[\partial_x/s]^{(m)}\] 
for any non-zero $s \in K$.
\end{defn}
This algebra is only well-behaved under specific restrictions on the parameter $s$.

\begin{cor}\label{DsmXfpsep} Suppose that $X$ is an affinoid subdomain of $\bA$ that is split over $K$. Let $s\in K^\times$ satisfy $|s|>1/\rho(X)$. 
\be \item $\cD_s^{(m)}(X)$ is a free left $\cO(X)^\circ$-module on $\{(\partial_x/s)^{\langle k\rangle}\mid k\geq 0\}$; in particular $\cD_s^{(m)}(X)$ is $p$-adically separated flat $K^\circ$-algebra.
	
\item Suppose further that there exists $\pi\in K^\times$ such that 
\[\max\left\{|p|,\frac{1}{s\rho(X)}\right\} \leq |\pi| < 1.\]
Then $\cD_s^{(m)}(X)/(\pi)$ is a finitely presented commutative $K^\circ/(\pi)$-algebra.  
\ee\end{cor}

\begin{proof} (a) By Lemma \ref{NormDivPow}, for all $k \geq 1$ we have \[ ||(\partial_x/s)^{\langle k\rangle}||_X\leq \frac{||\partial_x^{[k]}||_X}{|s|^k}\leq \left(\frac{1}{|s|\rho(X)}\right)^k \leq \frac{1}{|s|\rho(X)} < 1. \]
Now we can apply Proposition \ref{Levelmdivpow}(a).

(b) The above inequality shows that $|\pi| \geq \frac{1}{|s|\rho(X)} \geq \sup\limits_{k \geq 1} ||(\partial_x/s)^{\langle k \rangle}||_X$. Hence by Proposition \ref{Levelmdivpow}(b), $\cO(X)^\circ[\partial]^{(m)}/(\pi)$ is a commutative $\cO(X)^\circ/(\pi)$-algebra of finite presentation. It remains to prove that $\cO(X)^\circ/(\pi)$ is a finitely presented $K^\circ/(\pi)$-algebra.

Now, $\cO(X)^\circ$ is a topologically finitely generated $K^\circ$-algebra by \cite[2.4.8(a)]{Lut}. Thus $\cO(X)^\circ$ is a topologically finitely presented $K^\circ$-algebra by \cite[Proposition 1.1(c)]{BL1} and so $\cO(X)^\circ/(\pi)$ is a finitely presented $K^\circ/(\pi)$-algebra as required. 
\end{proof}

\begin{defn}\label{CompLevelmDivPow}  When $X$ and $s$ satisfy the hypotheses of Corollary \ref{DsmXfpsep} we let $\h{\cD}_s^{(m)}(X)$ denote the $p$-adic completion of $\cD_s^{(m)}(X)$ and \[ \h{\cD}_s^{(m)}(X)_K=K\otimes_{K^\circ} \h{\cD}_s^{(m)}(X).\]
\end{defn}
It follows from the presentation of $\cD_s^{(m)}(X)$ as a free $\cO(X)^\circ$-module in Proposition \ref{Levelmdivpow} that elements of $\h{\cD}_s^{(m)}(X)_K$ can be written uniquely as convergent sums 
\[\sum_{k\geq 0} f_k (\partial_x/s)^{\langle k\rangle} \mbox{ with }f_k\in \cO(X) \mbox{ and }|f_k|_X\to 0 \mbox{ as }k\to \infty.\]
We will view $\h{\cD}_s^{(m)}(X)_K$ as a $K$-Banach algebra with unit ball $\h{\cD}_s^{(m)}(X)$, so that the defining Banach norm on $\h{\cD}_s^{(m)}(X)_K$ is given by
\[\left|\sum_{k\geq 0} f_k (\partial_x/s)^{\langle k\rangle}\right|=\sup_{k\geq 0} |f_k|_X.\]

\begin{notn}\label{varpim} Let $\varpi_m:=(p^m)!^{\frac{1}{p^m}}\in \overline{K}$ so that $|\varpi_m|=p^{-\frac{p^m-1}{p^m(p-1)}}> \varpi$. \end{notn}
\begin{thm}\label{Adellevm} Let $X$ be an affinoid subdomain of $\bA$ that is split over $K$. Let $s\in K^\times$ be such that $|s|>1/\rho(X)$. There is an isomorphism of $K$-Banach algebras \[ \cD_{|\varpi_m s|}(X)  \cong \h{\cD}_{s}^{(m)}(X)_K. \]
\end{thm}

Before we prove Theorem \ref{Adellevm} we need another $p$-adic binomial estimate. 

\begin{lem} \label{binomest} For all $k\geq 0$, $1\leq \frac{|k!|}{|q_k!||\varpi_m|^k}\leq p^m$.\end{lem}

\begin{proof} Recalling from Notation \ref{divpownot}(a) that $k=q_kp^m+r_k$ with $0\leq r_k<p^m$, and that $s_p(n)$ denotes the sum of the $p$-adic digits of $n$, we compute directly
\begin{eqnarray*} (p-1)\log_p\left(\frac{|k!|}{|q_k!||\varpi_m|^k}\right) & = & (q_k-s_p(q_k)) - (k-s_p(k)) + \frac{k(p^m-1)}{p^m} \\ 
 & = & s_p(r_k) - \frac{r_k}{p^m}\leq (p-1)m  \\  
 \end{eqnarray*}
since $0\leq r_k<p^m$. Since also $s_p(r_k)- \frac{r_k}{p^m}\geq 0$ the result follows. 
\end{proof}

\begin{proof}[Proof of Theorem \ref{Adellevm}] 
Let $r=|\varpi_ms|$ so that $\cD_r(X) = \cO(X)\langle \partial/r\rangle$. Since $B:=\h{\cD}_{s}^{(m)}(X)_K$ is a $K$-Banach algebra we may use Lemma \ref{BanachUnivProp} to construct a $K$-Banach algebra homomorphism \[\phi\colon \cO(X)\langle \partial/r\rangle\to B:\] the inclusion $\iota\colon \cO(X)\to B$ is a $K$-Banach algebra homomorphism and $\partial_xf-f\partial_x=\partial_x(f)$, so there is a unique way to define $\phi$ that extends $\iota$ and sends $\partial$ to $\partial_x$ provided $\sup_{\ell\geq 0}\left(|\partial_x^\ell|_B/r^\ell\right)<\infty$. But \[ \left|\partial_x^\ell\right|_B= \left|\frac{\ell!}{q_\ell!}s^\ell\partial^{\langle \ell\rangle}\right|_B=\left|\frac{\ell!}{q_\ell!\varpi_m^\ell}\right|r^\ell \leq p^m r^\ell\] by Lemma  \ref{binomest}.  Now if $f_0,f_1,\ldots \in \cO(X)$ with $|f_k|r^k\to 0$ as $k\to \infty$, then \[ \phi\left(\sum\limits_{n\geq 0} f_k\partial^k\right)= \sum\limits_{k\geq 0} f_k \frac{k!}{q_k!}s^k\partial^{\langle k\rangle}. \] Therefore, because \[ \left|\sum\limits_{k\geq 0} f_k \frac{k!}{q_k!}s^k\partial^{\langle k\rangle}\right|_B=\sup\limits_{k \geq 0} \left(|f_k|r^k\left |\frac{k!}{q_k\varpi_m^k}\right| \right), \] 
we see, using Lemma \ref{binomest} again, that $\phi$ is a bijection with continuous inverse. Therefore it is an isomorphism of $K$-Banach algebras. 
\end{proof}

\begin{cor}\label{DrXNoeth} Suppose that $X$ is an affinoid subdomain of $\mathbb{A}$ and $r\in\sq{K}$ with $r>r(X)$. Then $\cD_r(X)$ is Noetherian. \end{cor}
\begin{proof} Using \cite[Theorem 4.1.8]{ArdWad2023}, choose a finite extension $K'$ of $K$ such that $X$ is split over $K'$. Then Lemma \ref{BaseChR} and Corollary \ref{Heisenberg} tell us that \[ r>r(X)=r(X_{K'})=\varpi/\rho(X_{K'}).\] Since $(|\varpi_m|)_{m=0}^\infty$ is a decreasing sequence converging to $\varpi$ from above, there is some $m\geq 0$ such that $r>|\varpi_m|/\rho(X_{K'})$. By enlarging $K'$ if necessary we may also assume that $r/|\varpi_m|=|s|$ for some $s\in K'^\times$. 

Now by Theorem \ref{Adellevm}, $\cD_r(X_{K'}) \cong \h{\cD}_s^{(m)}(X_{K'})_{K'}$ and the latter is Noetherian by Corollary \ref{DsmXfpsep} and \cite[Theorem 4.1.5]{EqDCap}, so the former is too. 

Lemma \ref{DrBC}(a) now gives an isomorphism of $K'$-Banach algebras \[  \cD_r(X_{K'})\cong K'\otimes \cD_r(X).\] Since $K'$ is faithfully flat over $K$, we deduce that $\cD_r(X)$ is also Noetherian.\end{proof}

\subsection{Discussion of flatness}\label{flatness}

Our main goal for the remainder of $\S \ref{FlatnessChapter}$ is to prove the following theorem about the rings introduced at Definition \ref{DagSite}(d).

\begin{thm} \label{FlatThm} If $s\geq r>0$ and $Y\subseteq X$ both lie in $\bA(\partial_x/r)^\dag$, then $\cD^\dag_r(Y)$ is a flat $\cD^\dag_s(X)$-module on both sides. 
\end{thm}

In this section will will perform some reductions and record some technical results from \cite{EqDCap} that we will use to prove Theorem \ref{FlatThm}.  First, we recall some very general results about flatness and direct limits. 

\begin{lem}\label{flatcolim1} Let $U$ be a ring and suppose $M_i$ is a directed system of flat $U$-modules. Then $\colim M_i$ is  a flat $U$-module.  \end{lem} 
\begin{proof} This is \cite[Proposition 5.4.6]{HGK1}.
\end{proof}

\begin{lem} \label{flatfinextn} Let $U$ be a $K$-algebra, $M$ a $U$-module and $K'$ a field extension of $K$. If $M_{K'}:=K'\otimes M$ is flat over $U_{K'}:=K'\otimes U$ then $M$ is flat over $U$. \end{lem}

\begin{proof} We assume that $M$ is a left module but the same argument works for right modules.  Let $N$ be any right $U$-module and write $N_{K'}=K'\otimes N$. For $i>0$,  \[ K'\otimes \Tor_i^U(N,M) \cong \Tor_i^{U_{K'}}(N_{K'}, M_{K'})=0 \]  since $M_{K'}$ is flat over $U_{K'}$. As $K'$ is faithfully flat over $K$ it follows that $\Tor_i^U(N,M)=0$ for all $i>0$ and so $M$ is flat as claimed. \end{proof}

The following lemma can be found for commutative rings at \cite[05UU]{stacks-project} with essentially the same proof. 
\begin{lem} \label{flatcolim2} Let $R_i$ be a directed system of rings and $R=\colim R_i$. Suppose that $M$ is an $R$-module that is flat over $R_i$ for each $i$. Then $M$ is flat as an $R$-module. \end{lem}

\begin{proof} Again we assume that $M$ is a left module but an identical argument works for right modules. 

Let $I$ be a finitely generated right ideal of $R$. By \cite[Proposition 5.4.11]{HGK1} it suffices to show that $I\otimes_R M\to M$ is an injection. Since $I$ is finitely generated we can find an index $i$ such that a set of its generators lies in the image of $R_i\to R$. That is we may choose a finitely generated right ideal $I_i$ of $R_i$ such that $I=I_iR$ (suitably interpreted). Then $I=\colim_{j\geq i} I_iR_j$ and $I\otimes_R M\to M$ is the colimit of the maps $I_iR_j\otimes_{R_j} M\to M$ over all $j\geq i$. Since $M$ is flat over each $R_j$ the maps in this directed family are all injective by \cite[Proposition 5.4.11]{HGK1} again. Since colimits commute with tensor products and colimits over directed sets are exact we're done. \end{proof}

We now reduce Theorem \ref{FlatThm} to proving the flatness of the restriction maps in the sheaves $\cD_r$ constructed in Proposition \ref{DbRing}(a).
\begin{prop}\label{reduction} Suppose that $\cD_r(Y)$ is a flat $\cD_s(X)$-module on both sides whenever $s,r\in \sqrt{|K^\times|}$ with $s\geq r$ and $Y\subseteq X$ both lie in $\bA(\partial_x/r)$. Then Theorem \ref{FlatThm} holds.
\end{prop}

\begin{proof} Suppose that whenever $s',r'\in \sqrt{|K^\times|}$ with $s'\geq r'$ and $Y\subseteq X$ both lie $\bA(\partial_x/r')$ then $\cD_{s'}(X)\to \cD_{r'}(Y)$ is flat on both sides. 

To prove Theorem \ref{FlatThm}, pick $s\geq r>0$, $Y\subseteq X$ in $\bA(\partial_x/r)^\dag$. Then for every $r'>r$, $X,Y\in \bA(\partial_x/r')$. Now Lemma \ref{flatcolim1} gives that $\cD_{s'}(X)\to \cD^\dag_r(Y)$ is flat on both sides whenever $s'\in \sqrt{|K^\times|}$ and $s'>r$, since $\cD^\dag_r(Y)=\colim_{r'>r} \cD_{r'}(Y)$ and we may view the colimit as only running over $r'\in (r,s')\cap \sqrt{|K^\times|}$. Then Lemma \ref{flatcolim2} gives that $\cD_s^\dag(X)\to \cD^\dag_r(Y)$ is flat on both sides, as required, since $\cD^\dag_s(X)=\colim_{s'>s} \cD_{s'}(X)$ and again we may view the colimit as only running over $s'\in \sqrt{|K^\times|}$.  \end{proof}

Finally, we recall a couple of results from \cite[\S4.1]{EqDCap} that will help us prove the flatness of the maps between the Banach algebras that appear in Proposition \ref{reduction}. For both, we suppose that $\pi\in K^{\circ\circ}$ such that:
\begin{itemize}
\item $\cU$ is a $\pi$-adically complete and separated and flat $K^\circ$-algebra,
\item $\cU / \pi \cU$ is a commutative $K^\circ / \pi K^\circ$-algebra of finite presentation,
\item $U:=K\otimes_{K^\circ} \cU$.
\end{itemize}

\begin{prop}[Proposition 4.1.7 of \cite{EqDCap}] \label{FlatPrepProp} Suppose that $\cV$ is another $\pi$-adically complete, separated and flat $K^\circ$-algebra which contains $\cU$. Let $y \in \cV$, and suppose that the map $\cU / \pi \cU \to \cV / \pi \cV$ extends to an $K^\circ/\pi K^\circ$-algebra isomorphism $(\cU/\pi\cU)[Y] \cong \cV/\pi\cV$ which sends $Y$ to $y + \pi \cV$. Let $C$ be the centraliser of $y$ in $U$. Then
\be 
\item $V := K\otimes_{K^\circ} \cV$ is a flat $U$-module on both sides, and 
\item for every finitely generated $U$-module $M$ there is a natural isomorphism of $C\langle Y\rangle$-modules $\eta_M \colon M\langle Y\rangle \stackrel{\cong}{\longrightarrow} V \otimes_U M$.
\ee
\end{prop}

\begin{thm}[Theorem 4.1.8 of \cite{EqDCap}] \label{AbstractFlatness} Suppose that  $y \in U$ is such that $[y,\cU] \subseteq \pi \cU$, that $\cV$ is the $\pi$-adic completion of $\cU[Y;\ad(y)]$ and $V=K\otimes_{K^\circ}\cV$. Then $V / (Y-y)V$ is a flat $U$-module on both sides.
\end{thm}

\subsection{Flatness of divided power algebras under change of base} \label{flatbasesec}
Our goal for this section is to prove the following theorem.

\begin{thm} \label{changeofbase} If $r\in \sqrt{|K^\times|}$ and $Y\subseteq X$ both lie in $\bA(\partial_x/r)$, then $\cD_r(Y)$ is a flat $\cD_r(X)$-module on both sides.
\end{thm}

Let $s \in K^\times$ and consider the following two affinoid subdomains of $X$:
\[X_1 := X\left(\frac{x}{s}\right) \qmb{and} X_2 := X\left(\frac{s}{x}\right). \]
Thus, we might have something like the following picture.
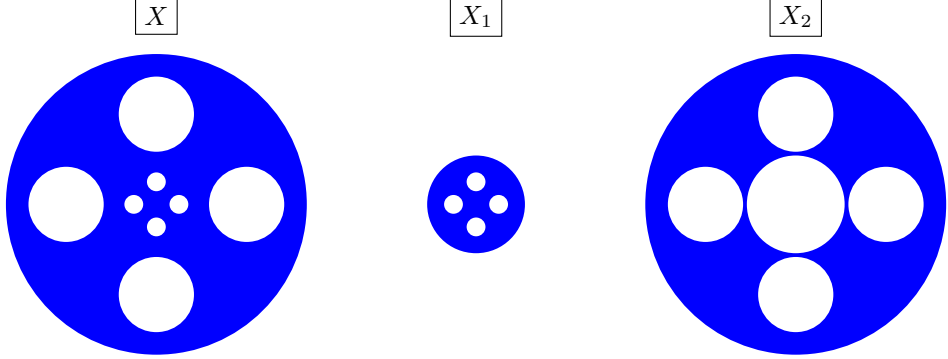
\begin{center}
\begin{tikzpicture}
\node[draw] at (1,2.5) {$X$};
\fill[blue] (1,0) circle (2cm);
\node[draw] at (5.25,2.5) {$X_1$};
\fill[blue] (5.25,0) circle (0.65cm);
\node[draw] at (9.5,2.5) {$X_2$};
\fill[blue] (9.5,0) circle (2cm);

\fill[white] (-.2,0) circle (0.5cm);   
\fill[white] (1.0,-1.2) circle (0.5cm);   
\fill[white] (0.7,0) circle (0.125cm);
\fill[white] (1.3,0) circle (0.125cm);
\fill[white] (1,-0.3) circle (0.125cm);
\fill[white] (1,0.3) circle (0.125cm);
\fill[white] (1,1.2) circle (0.5cm);   
\fill[white] (2.2,0) circle (0.5cm);

\fill[white] (4.95,0) circle (0.125cm);
\fill[white] (5.55,0) circle (0.125cm);
\fill[white] (5.25,-0.3) circle (0.125cm);
\fill[white] (5.25,0.3) circle (0.125cm);

\fill[white] (8.3,0) circle (0.5cm);   
\fill[white] (9.5,-1.2) circle (0.5cm);   
\fill[white] (9.5,1.2) circle (0.5cm);   
\fill[white] (10.7,0) circle (0.5cm);   
\fill[white] (9.5,0) circle (0.65cm);
\end{tikzpicture} 
\end{center}
Define $\rho := \min(\rho(X), |s|)$ and note that $\rho \leq \min \rho(X_1), \rho(X_2)$. Recall from Notation \ref{varpim} that $\varpi_m=(p^m)!^{\frac{1}{p^m}}$. We will first deal with the special case where the following Hypothesis holds; we will see in the proof of Theorem \ref{changeofbase} that the general case can be reduced to this one.
\begin{hypn}\label{rXmt} \hspace{1cm}
\begin{itemize} \item  $r\in |K^\times|$; 
\item $X\in \bA(\partial_x/r)$ is split over $K$; 
\item $Y$ is non-empty, and $Y = X_1 = X(\frac{x}{s})$ or $Y = X_2 = X(\frac{s}{x})$ for some $s \in K^\times$,
\item there exists $m\in \bN$ such that $\varpi_m\in K^\times$ and $|\varpi_m| / r < \rho$.
\end{itemize}
\end{hypn}
\textbf{We will assume that Hypothesis \ref{rXmt} holds throughout $\S \ref{flatbasesec}$}. The above conditions imply that there exists an element $t\in K^\times$ such that $|t| = r / |\varpi_m|$, which we fix from now on. Then since $|t| > 1 / \rho(X)$ and $|t| > 1 / \rho(X_i)$ for $i=1,2$, Theorem \ref{Adellevm} gives us isomorphisms
\[ \cD_r(X) \cong \h{\cD}_t^{(m)}(X)_K \qmb{and} \cD_r(X_i) \cong \h{\cD}_t^{(m)}(X_i)_K \qmb{for} i=1,2.\]
Thus proving the flatness of $\cD_r(X)\to \cD_r(X_i)$ for $i=1,2$ amounts to proving the flatness of $\h{\cD}_t^{(m)}(X)_K\to \h{\cD}_t^{(m)}(X_i)_K$.
 
Let $Z:=X\times \bD$ with a coordinate $y$ on $\bD$. Then $\cO(Z)=\cO(X)\langle y\rangle$, and the projection map $Z\to X$ induces $K$-Banach algebra isomorphisms $\cO(X_1)\cong \cO(Z)/(sy-x)$ and $\cO(X_2)\cong \cO(Z)/(xy-s)$ via maps $p_i\colon \cO(Z)\to \cO(X_i)$ for $i=1,2$.

Let $\delta_1,\delta_2$ be the bounded $K$-linear derivations of $\cO(Z)$ that extend $\partial_x$ on $\cO(X)$ and which satisfy $\delta_1(y)=1/s$ and $\delta_2(y)=-y^2/s$. Thus $\delta_1(sy-x)=0$ and $\delta_2(xy-s)=-(xy-s)y/s$, so \[\cD(X_1)\cong \cO(Z)[\delta_1]/(sy-x) \quad\mbox{ and }\quad \cD(X_2)\cong \cO(Z)[\delta_2]/(xy-s). \]

\begin{lem}\label{estdeltalem} Let $i=1$ or $i=2$. Then for all $n \geq 0$ we have
\begin{equation}\label{estdeltai} ||\delta_i^{\langle n \rangle}||_Z\leq \rho^{-n}.\end{equation}
\end{lem}
\begin{proof} Since $\delta_i^{\langle n\rangle}=q_n!\delta_i^{[n]}$ with $q_n! \in \mathbb{N}$, it suffices to show that $||\delta_i^{[n]}||_Z\leq \rho^{-n}$. We compute that for all $k,m\geq 0$, we have
\[ \delta_1^{[k]}(y^m)= \binom{m}{k} y^{m-k}/s^k\]  and \[\delta_2^{[k]}(y^m)= \binom{m+k-1}{k} y^{m+k}/(-s)^k.\]
This means that $|\delta_i^{[k]}(y^m)|_Z \leq 1/|s|^k$ for all $k,m\geq 0$. Let $f\in \cO(X)$ and $m, n \geq 0$. Using Lemma \ref{NormDivPow} and the fact that $\rho = \min(\rho(X), |s|)$, we have
\[ \left|\delta_i^{[n]}(fy^m)\right|_Z=\left|\sum_{j+k=n} \partial_x^{[j]}(f) \delta_i^{[k]}(y^m)\right|_Z\leq \sup_{j+k=n} \frac{|f|_X}{\rho(X)^j} \cdot \frac{1}{|s|^k} \leq \frac{|f|_X}{\rho^n}.\]
Hence $||\delta_i^{[n]}||_Z \leq \rho^{-n}$ for all $n \geq 0$ as required. \end{proof}
The following Lemma will enable us to perform the construction of Definition \ref{NegLaurPoly} with $A=\cD_t^{(m)}(X)$ and $\delta=\ad(x/s)$. 

\begin{lem}\label{adxovers} 
	$\ad(\frac{x}{s})\colon \cD_t^{(m)}(X)\to \cD_t^{(m)}(X)_K$ is locally nilpotent and has image contained in $\frac{1}{ts}\cD_t^{[m]}(X)$.
\end{lem}

\begin{proof}
	Suppose $k\geq 1$. Then by Lemma \ref{levelmdivpowmult}\ref{divpow2}, we have 
	\begin{eqnarray*} \left[(\partial_x/t)^{\langle k\rangle}, \frac{x}{s}\right] & = & \sum_{i=1}^{k} \binomb{k}{i}(\partial_x/t)^{\langle i\rangle}\left(\frac{x}{s}\right) (\partial/t)^{\langle k-i\rangle}\\ & = & \frac{1}{ts} \binomb{k}{1}  (\partial_x/t)^{\langle k-1\rangle} \in  \frac{1}{ts}\cD_t^{(m)}(X).\\\end{eqnarray*}  
By Proposition \ref{Levelmdivpow}(a) applied with $\partial = \partial_x/t$, we know that $\{(\partial_x/t)^{\langle k\rangle}\mid k\geq 0\}$ is a free generating set for $\cD_t^{(m)}(X) = \cO(X)^\circ[\partial_x/t]^{(m)}$ as a left $\cO(X)^\circ$-module. Now use the fact that $\ad(\frac{x}{s})$ is an $\cO(X)^\circ$-linear derivation.\end{proof}

\begin{notn} The following notation will be used in the remainder of $\S \ref{flatbasesec}$: \begin{itemize} \item $\cD=\cD_t^{(m)}(X)$ and $D=\h{\cD}_K$; 
		\item 
		$\cD_i=\cD_t^{(m)}(X_i)$ and $D_{i}=(\h{\cD_{i}})_K$ for $i=1,2$;
		\item $\cU_{i}=\cO(Z)^\circ[\delta_i/t]^{(m)}$ and $U_{i}=(\h{\cU_i})_K$ for $i=1,2$;
		\item $\cV_{1}=\cD[T; \ad(x/s)]$ and $V_{1}=(\h{\cV_{1}})_K$;
		\item $\cV_{2}=\cD[T^{-1}; \ad(x/s)]$ and $V_{2}=(\h{\cV_{2})}_K$.
	\end{itemize}
\end{notn}
We will establish that, for $i=1,2$, there is an isomorphism of $K$-Banach algebras $V_i \stackrel{\cong}{\longrightarrow} U_i$, and also that $D_i\cong V_i/(v_i)$ for explicit central elements $v_i$ of $V_i$. Together with Propositions \ref{FlatPrepProp} and \ref{AbstractFlatness}, this will allow us to prove that $D_i$ is flat over $D$. 

\begin{prop} \label{Dbasegood} The $K^\circ$-algebras $\cD$, $\cD_1$, $\cD_2$, $\cU_1$ and $\cU_2$ are all $p$-adically separated and flat $K^\circ$-algebras. There exists $\pi\in K$ with $0 < |\pi| < 1$, such that their reductions mod $\pi$ are all finitely presented commutative algebras over $K^\circ/(\pi)$. \end{prop}
\begin{proof} Since $X$ is split over $K$, we know that $\rho(X) \in |K^\times|$. Since also $s,t \in K^\times$, we can find an element $\pi \in K$ such that
\begin{equation} \label{FlatBasePi} |\pi| = \max \left\{|p|, \frac{1}{\rho |t|} \right\}.\end{equation}
Since $|p| < 1$ and $\frac{1}{\rho |t|} < 1$ by Hypothesis \ref{rXmt}, we see that $0 < |\pi| < 1$. Now, because $X$, $X_1$ and $X_2$ are split over $K$ and 
\[|t|>\rho^{-1}>\max \{\rho(X)^{-1},\rho(X_1)^{-1},\rho(X_2)^{-1}\},\] 
the cases $\cD$, $\cD_1$ and $\cD_2$ follow from Corollary \ref{DsmXfpsep}. To deal with $\cU_i$ for $i=1$ or $i=2$, we can apply Lemma \ref{estdeltalem} to see that
\[\sup\limits_{n \geq 1} ||(\delta_i/t)^{\langle n\rangle}|| \leq \sup\limits_{n \geq 1}\frac{1}{(\rho |t|)^n} \leq \frac{1}{|\rho|t} < 1.\]
Now we can apply Proposition \ref{Levelmdivpow} with $X = Z$ and $\partial = \delta_i/t$.\end{proof}

\begin{lem}  \label{UVbasis} For $n>0$, $i=1,2$ and $\pi$ given by Proposition \ref{Dbasegood}: \be 
\item $\cO(Z)^{\circ}/(\pi^n)$ is a free $\cO(X)^{\circ}/(\pi^n)$-module on $\{y^\ell+(\pi^n)\st \ell\geq 0\}$;
\item $\cU_{i}/(\pi^n)$ is a free $\cO(X)^\circ/(\pi^n)$-module on $\{\partial_i^{\langle k\rangle}y^\ell+(\pi^n)\st k,\ell\geq 0\}$;
\item $\cV_{1}/(\pi^n)$ is a free $\cO(X)^\circ/(\pi^n)$-module on $\{\partial_x^{\langle k\rangle}T^\ell+(\pi^n)\st k, \ell\geq 0\}$;
\item $\cV_{2}/(\pi^n)$ is a free $\cO(X)^\circ/(\pi^n)$-module on $\{\partial_x^{\langle k\rangle}T^\ell+(\pi^n)\st k\geq 0, \ell\leq 0\}$. \ee

\end{lem}

\begin{proof}
(a) $\cO(Z)=\cO(X)\langle y\rangle$, so $\cO(Z)^\circ=\cO(X)^\circ\langle y\rangle$ and the result follows easily. 

(b) Since $\partial_i^{\langle k\rangle}y^\ell\equiv y^\ell\partial_i^{\langle k\rangle} \mod \pi$ this follows from (a) together with Proposition \ref{Levelmdivpow} and \cite[I.2.3.17]{Laz1965}.

(c) and (d) follow from Proposition \ref{Levelmdivpow} together with the definitions of $A[T;\delta]$ and $A[T^{-1};\delta]$ from \S\ref{skewLaurent}.
\end{proof}

\begin{lem}\label{univV2m} For each $k\geq 0$ we have the following equality in $\cU_{2}$: \[ y(\delta_2/t)^{\langle k\rangle}= \sum_{n\geq 0} (-\ad (x/s))^n\left((\delta_2/t)^{\langle k\rangle}\right) y^{n+1}. \]
\end{lem}

\begin{proof} By Lemma \ref{adxovers}, $\ad(x/s)$ acts locally nilpotently on $\cD$. Hence the sum on the right hand side of the formula is finite. Since $\cU_{2}$ is flat over $K^\circ$ we may invert $p$ and then after multiplication by a scalar in $K$ reduce to proving that \[ y\delta_2^k = \sum_{n\geq  0} (-\ad (x/s))^n(\delta_2^k) y^{n+1}\] holds inside $\cO(Z)[\delta_2]$. Now we recall that $\ell_{\delta_2}= \ad(\delta_2)+r_{\delta_2}$ so, using a computation from the proof of Lemma \ref{estdeltalem}, we have
\begin{eqnarray*} y\delta_2^k & = & (\ell_{\delta_2} - \ad(\delta_2))^k(y) \\
                           			& = & \sum_{n=0}^k \binom{k}{n} \delta_2^{k-n} (-\delta_2)^{n}(y) \\
 							& = & \sum_{n=0}^k\binom{k}{n}\delta_2^{k-n} \binom{n}{n}n! y^{n+1}/s^{n}. \end{eqnarray*}

On the other hand, we saw in the proof of Lemma \ref{adxovers} that \[ (-\ad(x/s))^n(\delta_2^k)= 1/s^n \binom{k}{n}n! \delta_2^{k-n}. \] The result now follows easily. 
\end{proof}

\begin{prop} \label{ViisoUi} There are $\cD(X)$-linear isomorphisms of $K$-Banach algebras $\phi_i\colon V_{i}\to U_{i}$ with $\phi_1(T)=y$ and $\phi_2(T^{-1})=y$.
\end{prop}

\begin{proof} Since $\delta_i|_{\cO(X)}=\partial_x$, the natural map $\cO(X)\to \cO(Z)$ arising from the projection $Z\to X$ extends to maps $\cD(X)\to \cD(Z)$ sending $\partial_x$ to $\delta_i$. It is easy to verify that these in turn restrict to ring homomorphisms $g_i\colon \cD\to \cU_{i}$.
	 
Since $y-x/s$ is central in $\cU_{1}$ it is easy to verify that for each $P\in  \cD$, \[ yg_1(P)-g_1(P)y=(x/s)g_1(P)-g_1(P)(x/s)= g_1(\ad(x/s)(P)). \] The universal property of $\cV_{1}$ given by Proposition \ref{Univskewore} thus provides that $g_1$ extends to a unique ring homomorphism $h_1\colon \cV_{1}\to \cU_{1}$ such that $h_1(T)=y$. This in turn induces ring homomorphisms $h_{1,n}\colon \cV_{1}/(\pi^n)\to \cU_{1}/(\pi^n)$ that we claim are isomorphisms. To see this is suffices to observe that $h_{1,n}$ is $\cO(X)^\circ/(\pi^n)$-linear and sends  $\partial_x^{\langle k\rangle}T^\ell+(\pi^n)$ to $\delta_1^{\langle k\rangle}y^\ell + (\pi^n)$; now apply Lemma \ref{UVbasis}(b,c).  It follows that $\h{h_1}\colon \h{\cV_{1}}\to \h{\cU_{1}}$ is an isomorphism and we may take $\phi_1=(\h{h_1})_K$.

Similary, since every element of $\cD$ is an $\cO(X)^\circ$-linear combination of the elements $(\partial_x/t)^{\langle k\rangle}$ by Proposition \ref{Levelmdivpow}(a), and since $\cO(Z)$ is commutative, it follows from Lemma \ref{univV2m} that inside $\cU_2$, for all $P \in \cD$ we have \[ yg_2(P) = \sum_{n\geq 0} g_2((-\ad(x/s))^n(P))y^{n+1}.\] Then the universal property of $\cV_{2}$ given by Proposition \ref{univTinverse} shows that $g_2$ extends to a unique ring homomorphism $h_2\colon \cV_{2}\to \cU_{2}$ sending $T^{-1}$ to $y$. Again this induces ring homomorphisms $h_{2,n}\colon \cV_{2}/(\pi^n)\to \cU_{2}/(\pi^n)$ that we claim are isomorphisms. To see this is suffices to observe that each map $h_{2,n}$ is $\cO(X)^\circ/(\pi^n)$-linear and sends  $\partial^{\langle k\rangle}T^{-\ell}+(\pi^n)$ to $\partial^{\langle k\rangle}y^\ell + (\pi^n)$ so we can appeal to Lemma \ref{UVbasis}(b,d) again and  then see that $\phi_2=(\h{h_2})_K$ is an isomorphism as before.\end{proof}

\begin{prop}\label{presDim} Let $i=1$ or $i=2$. There is an exact sequence of $V_{i}$-modules \[ 0\to V_{i}\stackrel{\cdot v_i} \longrightarrow V_{i}\to D_{i}\to 0, \] where the $v_i$ is the central element of $V_{i}$ given by 
\[ v_1=sT-x \qmb{and} v_2=xT^{-1}-s.\]
\end{prop}
\begin{proof} Let $u_1=sy-x$ and $u_2=xy-s$.  By Proposition \ref{ViisoUi} it suffices to show that the sequences \[ 0\to U_{1}\stackrel{\cdot u_1}{\longrightarrow} U_{1} \to D_{1}\to 0 \qmb{and} 0\to U_{2}\stackrel{\cdot u_2}{\longrightarrow} U_{2}\to D_{2}\to 0\] are exact. We may pick $n\geq 0$ such that $\pi^nu_1,\pi^nu_2\in \cO(Z)^\circ$. Since  \[ 0\to \cO(Z)\stackrel{\cdot u_i}\to \cO(Z) \stackrel{p_i}\to \cO(X_i)\to 0 \]  is exact, it follows that \[  0\to \cO(Z)^\circ\stackrel{\cdot \pi^nu_i}{\longrightarrow} \cO(Z)^\circ \to \cO(X_i)^\circ \to 0\]  has $\pi$-torsion cohomology. Moreover the Banach Open Mapping Theorem, \cite[Proposition 8.6]{SchNFA} tells us that the maps $\cdot u_i$ and $p_i$ are open onto their image and so there is some $N>0$ such that $\pi^N$ annnihilates the cohomology groups.  We can then use Proposition \ref{Levelmdivpow} to deduce that \[0\to \cU_i \stackrel{\cdot \pi^nu_i}\to \cU_i \to \cD_i \to 0 \] also has cohomology groups annihilated by $\pi^N$. Now apply \cite[Lemma 3.6]{DCapOne}.
\end{proof}

\begin{prop} \label{flatbase} Suppose that Hypothesis \ref{rXmt} holds. Then $D_1$ and $D_2$ are flat as $D$-modules on both sides.
\end{prop}
\begin{proof} Let $\pi$ be given by Proposition \ref{Dbasegood}. Since $\h{\cD}$ is a $\pi$-adically complete and separated $K^\circ$-algebra, flat over $K^\circ$ and $\h{\cD}/(\pi)$ is a commutative $K^\circ/(\pi)$-algebra of finite presentation we will apply the results from \S\ref{flatness} with $\cU:=\h{\cD}$. 

By Proposition \ref{presDim}, $D_1 \cong V_1 / (sT - x) V_1$. By Lemma \ref{adxovers}, $[\frac{x}{s}, \cD] \subseteq \frac{1}{ts} \cD$. By (\ref{FlatBasePi}) we have $|\pi t s| \geq |s| / \rho$. Since $\rho = \min(\rho(X), |s|) \leq |s|$, we see that $|\frac{1}{t s}| \leq |\pi|$ which implies that $\frac{1}{ts} \cD \subseteq \pi \cD$. We can now apply Theorem \ref{AbstractFlatness} with $y = x/s$ to deduce that $D_1 \cong V_1 / (T-\frac{x}{s})V_1$ is a flat $D$-module.

Next, by Proposition \ref{presDim}, we have a presentation $D_2 \cong V_2/v_2 V_2$ where $v_2=xT^{-1}-s\in V_2$. By \cite[Proposition 4.4]{DCapOne} applied to the ring map $U \to V_2$ and the element $v_2 \in V_2$, to prove the flatness of $D_2$ as a $D$-module it suffices to prove that $v_2$ is a regular element in $V_2$, $V_2$ is flat as a right $U$-module and that $V_2\otimes_{D} M$ is $v_2$-torsion-free for all finitely generated left $D$-modules $M$.

Note that in $\cV_{2}$ we can use Lemma \ref{adTinverse} to see that for all $P \in \cD$ we have \[ \ad(T^{-1})(P)= \sum_{n = 1}^\infty (-1)^n \ad(x/s)^n(P) T^{-n-1}. \]
We saw above that Lemma \ref{adxovers} implies that $\ad(x/s)^n(\cD) \subseteq \frac{1}{ts} \cD \subseteq \pi \cD$ for all $n \geq 1$. Hence $\ad(T^{-1})(\cD)\subseteq \pi \cV_2$, so the hypotheses of Proposition \ref{FlatPrepProp} are satisfied with $\cU = \h{\cD}$, $\cV = \h{\cV_2}$ and $y=T^{-1}$. The centraliser of $T^{-1}$ in $D$ certainly contains $\cO(X)$. Proposition \ref{FlatPrepProp} now tells us that $V_2$ is a flat $U$-module on both sides, and that for every finitely generated $D$-module $M$ there is a natural isomorphism of $\cO(X)\langle T^{-1}\rangle$-modules $M\langle T^{-1}\rangle \cong V_2\otimes_D M$. We can now apply \cite[Lemma 4.1(a)]{DCapOne} with $t = T^{-1}$ and $f = x/s$ to see that $V_2 \otimes_D M$ is $v_2$-torsion-free for every finitely generated $D$-module $M$. In particular, setting $M = D$ shows that $v_2$ is a regular element in $V_2$. 
\end{proof}

\begin{proof}[Proof of Theorem \ref{changeofbase}]
We may certainly assume that $Y \neq \emptyset$ as the result is trivial otherwise. By Lemma \ref{flatfinextn} and Lemma \ref{DrBC}(a) we may replace $K$ by any finite extension. In particular, by \cite[Theorem 4.1.8]{ArdWad2023} we may assume that $X$ and $Y$ are both split over $K$ and $r\in |K^\times|$. 

If $Y=\bigcup_{i=1}^n Y_i$ is a disjoint union of cheeses then $\cD_r(Y)=\bigoplus_{i=1}^n \cD_r(Y_i)$ and so by \cite[Proposition 5.4.2]{HGK1} we may further assume that $Y$ is a single cheese. 

Next we observe that as $r(Y)<r$, $\rho(Y)>\varpi/r$ by Corollary \ref{Heisenberg}. That is $Y=C_K(\alpha,\bf{s})$ for some $\alpha,\bf{s}$ satisfying the conditions of \cite[Definition 4.1.1]{ArdWad2023} and moreover $|s_i|>\varpi/r$ for all $i=0,\ldots, g$. Since $Y\subset X$ if we define $X_0=X(\frac{x-\alpha_0}{s_0})$ and $X_i=X_{i-1}(\frac{s_i}{x-\alpha_i})$ for $i=1,\ldots,g$ then we see that $X_g=Y$. Then, since a composition of flat morphisms is flat, an induction argument reduces us to the cases $Y=X(\frac{x-\alpha}{s})$ and $Y=X(\frac{s}{x-\alpha})$ where $\alpha\in K$ and $s\in K^\times$. By a change of coordinate $x \mapsto x - \alpha$, we may assume that $\alpha = 0$. 

Since $Y$ lies in $\bA(\partial_x/r)$ by assumption, applying Corollary \ref{Heisenberg} shows that 
\[|s| \geq \rho(Y) = \varpi / r(Y) > \varpi /r.\] 
Since $\varpi_m/r$ is a decreasing sequence converging to $\varpi/r$ from above, we can find a sufficiently large $m$ such that $\rho(Y)>|\varpi_m|/r$. Since $Y = X(x/s)$ or $X(s/x)$, it is clear from Definition \ref{MinRad} that $\rho(X) \geq \rho(Y)$. Hence 
\[\rho = \min (\rho(X), |s|) \geq \rho(Y) >  |\varpi_m| / r.\]
By further enlarging $K$ if necessary, we may assume that $\varpi_m\in K^\times$, and now all conditions of Hypothesis \ref{rXmt} are satisfied. 

Choose any $t \in K^\times$ with $|t| = r / |\varpi_m|$. Theorem \ref{Adellevm} tells us that \[\cD_r(X)\cong \widehat{ \cD_t^{(m)}(X) } = D \qmb{and} \cD_r(X_i)\cong D_i \qmb{for} i=1,2.\] 
Hence $\cD_r(X_1)$ and $\cD_r(X_2)$ are flat $\cD_r(X)$-modules on both sides by Proposition \ref{flatbase}. Since either $Y = X_1 = X(x/s)$ or $Y = X_2 = X(s/x)$, we conclude that $\cD_r(Y)$ is a flat $\cD_r(X)$-module on both sides. \end{proof}

\subsection{Flatness of divided power algebras in cotangent direction}\label{flatcotdirsec}

Our goal in this section is to prove the following Theorem and then use it to complete the proof of Theorem \ref{FlatThm}.

\begin{thm} \label{flatcotangent} Let $s,r\in \sqrt{|K^\times|}$ be given with $s\geq r$, and suppose that $X$ lies in $\bA(\partial_x/r)$. Then $\cD_r(X)$ is a flat $\cD_s(X)$-module on both sides.
\end{thm}

As in $\S \ref{flatbasesec}$, we can enlarge $K$ to reduce to the case where $X$ is split over $K$. We can also assume that $|K|^\times$ contains any particular finite set of elements in $\sqrt{|K|^\times}$. We will first deal with the special case where the following Hypothesis holds; we will see in the proof of Theorem \ref{flatcotangent} given at the end of $\S \ref{flatcotdirsec}$ below that the general case can be reduced to this one.
\begin{hypn}\label{NotCotDirFlat} \hspace{0.1cm}
\begin{itemize} \item $X$ is an affinoid subdomain of $\bA$ split over $K$;
	\item $m\in \bN$ is such that $p^{-1/p^m}\in K^\times$; 
	\item $s,r\in K^\times$ are such that $|s|>|r|>1/\rho(X)$ and $|s/r|<|p|^{-1/p^m}$;  
	\item $\pi\in K^\times$ is such that $\max\left\{|p|,\frac{1}{|r|\rho(X)}\right\} \leq |\pi| < 1$.
	\end{itemize}
\end{hypn}
\textbf{We will assume that Hypothesis \ref{NotCotDirFlat} holds until the end of the proof of Corollary \ref{presenthKUi} below.}
\begin{defn}\label{UiDefn} Let $\cU^{0}=\cD_s^{(m)}(X)$. For each $i=0,\ldots,m$, define inductively 
\[\cU^{i+1} := \cU^i[(\partial_x/r)^{\langle p^i \rangle}]\quad \subset\quad \cD^{(m)}_{r}(X) \]
so that $\cU^0 \subset \cU^1 = \cD_s^{(m)}(X) [ \partial_x/r ] \subset \cU^2 \subset \cdots \subset \cU^{m+1}={\cD}_{r}^{(m)}(X)$.  
\end{defn}

We will show that for each $i = 0,\cdots,m$, $\h{\cU^{i+1}_K}$ is flat as a $\h{\cU^i_K}$-module on both sides. Then we will appeal to transitivity of flatness to deduce that $\h{\cD}_{r}^{(m)}(X)_K$ is flat as a $\h{\cD}_{s}^{(m)}(X)_K$-module on both sides. Then Theorem \ref{Adellevm} will imply that $\cD_{|r\varpi_m|}(X)$ is flat over $\cD_{|s\varpi_m|}(X)$. 

First we prove an important generalisation of Proposition \ref{Levelmdivpow}. 

\begin{prop} \label{Uifree} Let $i=1,\ldots, m$.
\be
\item  $\cU^i$ is a free $\cO(X)^\circ$-module with basis \[ T_i := \left\{ (\partial_x/r)^{\langle \alpha\rangle}(\partial_x/s)^{\langle p^{i}\beta\rangle} : 0\leq \alpha<p^i, \beta\geq 0\right\}.\] 
\item $\cU^i/(\pi)$ is a finitely presented commutative $K^\circ/\pi K^\circ$-algebra. 
\ee
\end{prop}
\begin{proof} (a)  Let $k \geq 0$. Recall from $(\ref{d<k>formula})$ that when we write $k=\sum_{j=0}^{m-1}c_jp^j + cp^m$ for some $c_0,\ldots,c_{m-1}\in \{0,\ldots,p-1\}$ and $c \geq 0$, there is a unit $u \in \Z_p^\times$ such that
\begin{equation} \label{Divpowrel} \partial_x^{\langle k\rangle}= u\left( \prod\limits_{j=0}^{m-1} (\partial_x^{\langle p^j\rangle})^{c_j}\right) (\partial_x^{\langle p^m\rangle})^c. \end{equation}
Using this equation together with Proposition \ref{Levelmdivpow}, we see that $T_i$ is a generating set for $\cU^i$ as a $\cO(X)^\circ$-algebra and $T_i$ is linearly independent over $\cO(X)^\circ$. To show that $\cU^i = \cO(X)^\circ \cdot T_i$ it will suffice to prove that $\cO(X)^\circ\cdot T_i$ is closed under multiplication. Since $\cO(X)^\circ$ is closed under multiplication and $K^\circ\subset \cO(X)^\circ$, it will suffice to show that 
\begin{enumerate}
\item $T_i\cdot \cO(X)^\circ\subseteq \cO(X)^\circ\cdot T_i$, and 
\item $T_i\cdot T_i\subseteq K^\circ\cdot T_i$.
\end{enumerate}
(1) Let $0 \leq \alpha < p^i$ and $\beta \geq 0$ be given, so that $t_{\alpha,\beta} := (\partial_x/r)^{\langle \alpha\rangle}(\partial_x/s)^{\langle p^{i}\beta\rangle}$ is a typical element of $T_i$. We can of course rewrite $t_{\alpha,\beta} = (r/s)^{p^i \beta} (\partial_x/r)^{\langle \alpha\rangle}(\partial_x/r)^{\langle p^{i}\beta\rangle}$. By applying (\ref{Divpowrel}) to $\partial_x^{\langle \alpha \rangle}$, $\partial_x^{\langle p^i \beta \rangle}$ and $\partial_x^{\langle \alpha + p^i \beta \rangle}$ in turn, and using the fact that $0 \leq \alpha < p^i$, we find that
\begin{equation}\label{tab} t_{\alpha,\beta} =u_{\alpha,\beta} (r/s)^{p^i\beta} (\partial_x/r)^{\langle \alpha+p^i\beta\rangle}  \qmb{for some} u_{\alpha,\beta}\in \Z_p^\times.\end{equation}
Let $f\in \cO(X)^\circ$. Then Lemma \ref{levelmdivpowmult}\ref{divpow2} tells us that
\begin{equation}\label{ftab} f t_{\alpha,\beta} =u_{\alpha,\beta}(r/s)^{p^i\beta}\sum\limits_{k=0}^{\alpha+p^i\beta} \binomb{\alpha+p^i \beta}{k} (\partial_x/r)^{\langle \alpha+p^i\beta-k\rangle}(f) (\partial_x/r)^{\langle k\rangle}. \end{equation}
For each $0 \leq k \leq \alpha + p^i \beta$, write $k = \alpha_k + p^i \beta_k$ with $0 \leq \alpha_k < p^i$ and $\beta_k \geq 0$, so that necessarily we have $\beta_k \leq \beta$. Then using (\ref{tab}) we have
\[ (r/s)^{p^i \beta} (\partial_x/r)^{\langle k\rangle} = \frac{(r/s)^{p^i (\beta - \beta_k)}}{u_{\alpha_k,\beta_k}}t_{\alpha_k,\beta_k} \in K^\circ \cdot T_i.\]
This allows us to rewrite $(\ref{ftab})$ as follows:
\begin{equation}\label{ftab2} f t_{\alpha,\beta} =u_{\alpha,\beta}\sum\limits_{k=0}^{\alpha+p^i\beta} \binomb{\alpha+p^i \beta}{k} (\partial_x/r)^{\langle \alpha+p^i\beta-k\rangle}(f)  \hspace{0.1cm}  \frac{(r/s)^{p^i (\beta - \beta_k)}}{u_{\alpha_k,\beta_k}}t_{\alpha_k,\beta_k}. \end{equation}
We have chosen $\pi \in K^\times$ to satisfy $\max\left\{|p|,\frac{1}{|r|\rho(X)}\right\} \leq |\pi| < 1$, so $\cD^{(m)}_r(X) / (\pi)$ is a commutative algebra by Corollary \ref{DsmXfpsep}(b). This implies that 
\begin{equation}\label{dxrjf} (\partial_x/r)^{\langle j\rangle}(f)\in \pi\cO(X)^{\circ} \qmb{for all} j \geq 1.\end{equation}
Using $(\ref{ftab2})$, we can now see that $ft_{\alpha,\beta} \in \cO(X)^\circ T_i$, and hence $T_i\cdot \cO(X)^\circ\subseteq\cO(X)^\circ  \cdot T_i$.

(2) Let $0\leq \alpha,\alpha'<p^i$ be given. Using Lemma \ref{levelmdivpowmult}\ref{divpow} together with (\ref{Divpowrel}) we see that if $\alpha+\alpha'<p^i$ then 
\[  (\partial_x/r)^{\langle \alpha\rangle}(\partial_x/r)^{\langle \alpha' \rangle}  = \binoma{\alpha+\alpha'}{\alpha} (\partial_x/r)^{\langle \alpha+\alpha'\rangle}, \] whereas if $p^i\leq \alpha+\alpha'< 2p^i$ then  
 \[ (\partial_x/r)^{\langle \alpha\rangle}(\partial_x/r)^{\langle \alpha' \rangle}  =  u(s/r)^{p^i} \binoma{\alpha+\alpha'}{\alpha} (\partial_x/r)^{\langle \alpha+\alpha'-p^i\rangle}(\partial_x/s)^{\langle p^i\rangle}\] for some $u\in \Z_p^\times$. Similarly for $\beta,\beta'\geq 0$, we have \[ (\partial_x/s)^{\langle p^i\beta\rangle}(\partial_x/s)^{\langle p^i\beta'\rangle}=\binoma{p^i(\beta+\beta')}{p^i\beta}(\partial_x/s)^{\langle p^i(\beta+\beta')\rangle}. \] 
 Thus to see that $K^\circ T_i$ is closed under multiplication it suffices to show that  $(s/r)^{p^i} \binoma{\alpha+\alpha'}{\alpha}\in K^\circ$ whenever $0\leq \alpha,\alpha'<p^i$ but $\alpha+\alpha'\geq p^i$. 
 
Now $0\leq \alpha,\alpha'<p^i$ implies that  $\binomb{\alpha+\alpha'}{\alpha}=1$. If in addition we have $\alpha+\alpha'\geq p^i$, there must be at least one carry when adding $\alpha$ and $\alpha'$ $p$-adically. Therefore by Kummer's Theorem (Theorem \ref{Kum}) we have \[ v_p\left( \binoma{\alpha+\alpha'}{\alpha}\right)=v_p\left(\binom{\alpha+\alpha'}{\alpha}\right)\geq 1.\] 
However we have assumed that $|s|$ and $|r|$ are sufficiently close, in the sense that $|s/r|\leq |p|^{-1/p^m}$. This implies that $|(s/r)^{p^i}| \leq |p|^{-p^i/p^m} \leq |p|^{-1}$ as $i \leq m$, so $p (s/r)^{p^i} \in K^\circ$.  Hence  $(s/r)^{p^i}\binoma{\alpha+\alpha'}{\alpha}\in K^\circ$  as required.

(b) The commutativity of $\cU^i / (\pi)$ is clear given equations $(\ref{ftab2})$ and $(\ref{dxrjf})$. 

It follows from (a) that $\cU^i/(\pi)$ is a free $\cO(X)^\circ$-module on the image of $T_i$. Moreover writing $\tau_{\alpha,\beta}$ for the image of $t_{\alpha,\beta}$ in $\cU^i/(\pi)$ we see, by (\ref{d<k>formula}), that the $\cO(X)^\circ/(\pi)$-submodule $U$ generated by $\{ \tau_{0,p^{m-i}\gamma}\mid\gamma\geq 0\}$ is a polynomial algebra $U=\cO(X)^\circ/(\pi)[\tau_{0,p^{m-i}}]$. Moreover $\cU^i$ is a free $U$-module on the finite set $\{\tau_{\alpha,\beta}\mid 0\leq \alpha< p^i, 0\leq \beta<p^{m-i} \}$. 

By the proof of Corollary \ref{DsmXfpsep}(b), $\cO(X)^\circ/(\pi)$ is a finitely presented $K^\circ/(\pi)$-algebra. Thus, by \cite[00F4]{stacks-project}, $U$ is also a finitely presented $K^\circ/(\pi)$-algebra. Finally $\cU^i/(\pi)$ is a finitely presented $K^\circ/(\pi)$-algebra by \cite[0D46]{stacks-project}.
  \end{proof}

Recall that for each $u\in \cD(X)$, $\ad(u)\colon \cD(X)\to \cD(X)$ is a $K$-linear derivation. 

\begin{lem}\label{adyi} For any $i=0,\ldots, m$, we have $\ad\left((\partial_x/{r})^{\langle p^i\rangle}\right)(\cU^{i})\subseteq \pi\cU^{i}.$
\end{lem}
\begin{proof} Let $f\in \cO(X)^\circ$. By Lemma \ref{levelmdivpowmult}(b), we have
\[ \ad\left((\partial_x/r)^{\langle p^i\rangle}\right)(f)= \sum_{k=0}^{p^i-1} \binomb{p^i}{k} (\partial_x/r)^{\langle {p^i-k}\rangle}(f) \cdot (\partial_x/r)^{\langle k\rangle}.  \]
Now $(\partial_x/r)^{\langle {p^i-k}\rangle}(f) \in \pi \cO(X)^\circ$ by $(\ref{dxrjf})$ and $(\partial_x/{r})^{\langle k\rangle}\in \cU^{i}$ for all $k<p^i$ by (\ref{Divpowrel}), so the sum on the right hand side lies in $\pi \cU^i$. By Proposition \ref{Uifree}(a), we know that $\cU^i$ is a free left $\cO(X)^\circ$-module with basis $T_i$. The result follows because $[(\partial_x/r)^{\langle p^i \rangle}, f t ] = [(\partial_x/r)^{\langle p^i \rangle}, f ] t \in \pi \cU^i$ for all $t \in T_i$. \end{proof}
\begin{lem} \label{ppoweryi} For each $i=0,\ldots, m-1$, $\left((\partial_x/r)^{\langle p^i\rangle}\right)^p\in \cU^0$.
\end{lem}
\begin{proof} The assumption $i + 1 \leq m$ implies that $q_{p^i} = q_{p^{i+1}} = 0$. Hence $\partial^{\langle p^i \rangle} = \partial^{[p^i]}$ and  $\partial^{\langle p^{i+1} \rangle} = \partial^{[p^{i+1}]}$ by Definition \ref{divpownot}(d), so
\[\left((\partial_x/r)^{\langle p^i\rangle}\right)^p = \frac{p^{i+1}!}{(p^i!)^p}(\partial_x/r)^{\langle p^{i+1}\rangle}  =  \frac{p^{i+1}!}{(p^i!)^p}(s/r)^{p^{i+1}}(\partial_x/s)^{\langle p^{i+1}\rangle}.\]
Now as $v_p(p^i!)=\frac{p^i-1}{p-1}$, we can see that $v_p\left(\frac{p^{i+1}!}{(p^i!)^p}\right)=1$. Since $i \leq m-1$ and $|s/r|\leq |p|^{-1/p^m}$, we see that $|(s/r)^{p^{i+1}}|\leq |p|^{-1}$. Hence $\frac{p^{i+1}!}{(p^i!)^p}(s/r)^{p^{i+1}} \in K^\circ$, so $\left((\partial_x/r)^{\langle p^i\rangle}\right)^p \in K^\circ \cdot (\partial_x/s)^{\langle p^{i+1}\rangle} \subset \cD_s^{(m)}(X) = \cU^0$ as required.\end{proof}
\begin{notn}\label{yiVi} Let $i=0,\ldots, m$. \begin{itemize} \item Let $y_i=(\partial_x/{r})^{\langle p^i\rangle}$. \item Let $\delta_i$ be $\ad(y_i)$ viewed as a derivation of $\cU^{i}$.  \item Let $\cV^i$ be the skew-Ore extension $\cU^{i}[Y_i;\delta_i]$. \item Let $s_i=(r/s)^{p^i}\in K^\circ$. \end{itemize} \end{notn}
Note that $\delta_i : \cU^i\to\cU^i$ is well-defined by Lemma \ref{adyi}.

 \begin{prop}\label{presentUi} For each $0\leq i\leq m$ there is a surjective ring homomorphism $\phi_i\colon \cV^i\twoheadrightarrow \cU^{i+1}$ extending the inclusion map $\cU^{i}\to \cU^{i+1}$ and sending $Y_i$ to $y_i$. The kernel contains the central element $s_iY_i - (\partial_x/s)^{\langle p^i\rangle}$ and \[ s_i^{2p-1}\cdot\left(\ker \phi_i/\cV^i(s_iY_i - (\partial_x/s)^{\langle p^i \rangle})\right)=0.\] 
 \end{prop}
\begin{proof} We write $z_i=(\partial_x/s)^{\langle p^i\rangle}=s_iy_i\in \cU^{i}$. Note that $z_i$ and $Y_i$ commute in $\cV^i$. 
 The existence of $\phi_i$ follows from the universal property of the skew-Ore extension, Proposition \ref{Univskewore}. Note that $\phi_i$ is surjective because $\cU^{i+1}$ is generated by $\cU^i$ together with $y_i$, by Definition \ref{UiDefn}. The element $s_iY_i - z_i$ is central in the skew-Ore extension $\cV^i$ because $s_i \delta_i$, $\ad(s_i y_i)$ and $\ad(z_i)$ define the same derivation $\cU^i \to \cU^i$. Moreover, $\phi_i\left(s_iY_i-z_i\right)=s_iy_i-z_i= 0$, so $\cV^i(s_iY_i- z_i)\subseteq \ker \phi_i$.

We suppose first that $i\leq  m-1$.  Lemma \ref{ppoweryi} gives that $y_i^p\in \cU^i$, so that $(z_i/s_i)^p\in \cU^i$. Then, writing $n=qp+r$ with $0\leq r<p$ and $q\geq 0$ we see that 
\begin{equation}\label{sip1} s_i^{p-1}(z_i/s_i)^n= s_i^{p-1-r}z_i^r(z_i/s_i)^{qp} \in \cU^i \qmb{for all} n\geq 0.\end{equation} 
It follows that for any $n \geq 0$ we have 
\begin{equation}\label{sip2} s_i^{p-1}(Y_i-z_i/s_i)^n=s_i^{p-1}\sum\limits_{k=0}^n \binom{n}{k}(-z_i/s_i)^kY_i^{n-k}\in \cV^i. \end{equation}
Take an element $u=\sum\limits_{n=0}^l u_nY_i^n\in \ker \phi_i$ with $u_0, \cdots, u_l\in \cU^{i}$. Then
\begin{eqnarray*} s_i^{2p-1}u & = &\sum\limits_{n=0}^l s_i^{2p-1}u_n(Y_i-z_i/s_i + z_i/s_i)^n \\
	& = & \sum\limits_{n=0}^l \sum\limits_{k=0}^n \binom{n}{k}u_ns_i^{p-1}(z_i/s_i)^{n-k}s_i^p(Y_i-z_i/s_i)^k. 
\end{eqnarray*}
When $k \geq 1$, we can rewrite the summand on the right hand side as follows:
\[\left(\binom{n}{k}u_n \right) \cdot \left(s_i^{p-1}(z_i/s_i)^{n-k}\right) \cdot \left(s_i^{p-1} (Y_i - z_i/s_i)^{k-1}\right) \cdot (s_i Y_i - z_i).\]
The first factor lies in $\cU^i$, the second factor lies in $\cU^i$ by $(\ref{sip1})$ and the third factor lies in $\cV^i$ by $(\ref{sip2})$.  Hence
\[ s_i^{2p-1} u  \equiv \sum\limits_{n=0}^l  s_i^{2p-1}u_n(z_i/s_i)^n \quad \mod \cV^i(s_iY_i - z_i).\]
But since $u \in \ker \phi_i$, we have
\begin{equation}\label{phiisi} 0=\phi_i(s_i^{2p-1}u)=\sum\limits_{n=0}^l  s_i^{2p-1}u_ny_i^n=\sum\limits_{n=0}^l u_ns_i^{2p-1}(z_i/s_i)^n\end{equation} 
and we deduce that $s_i^{2p-1}u \in \cV^i(s_i Y_i - z_i)$ as claimed.

Now we consider the case $i=m$. Again we suppose that $u=\sum\limits_{n=0}^l u_nY_m^n\in \ker \phi_m$ with $u_0, \cdots, u_l\in \cU^{m}$. This time \[ 0=\phi_m(s_m^lu)=\sum\limits_{n=0}^l  s_m^{l}u_ny_m^n=\sum\limits_{n=0}^l s_m^{l-n}u_nz_m^n.\]

Thus \begin{eqnarray*} 
 s_m^l u  & = & \sum\limits_{n=0}^l s_m^{l-n}u_n(s_mY_m-z_m + z_m)^n \\
         & = & \sum\limits_{n=0}^l \sum\limits_{k=0}^n s_m^{l-n} \binom{n}{k}u_n(s_mY_m-z_m)^kz_m^{n-k} \\
       &\in &  \cV^m(s_mY_m-z_m)+ \sum_{n=0}^l  s_m^{l-n}u_nz_m^n  = \cV^m(s_mY_m-z_m).\\ 
\end{eqnarray*}

So for all $u\in\ker \phi_m$, we can find $t\in K^\circ$ and $v=\sum\limits_{n\geq 0} v_nY_m^n\in \cV^m$ such that $tu=v(s_mY_m-z_m)$. Expanding out this equation gives  \[\left (-tu_0+ \sum_{n\geq 1}(s_mv_{n-1}-tu_n)Y_m^n\right)=\sum_{n\geq 0} v_nY_m^nz_m\]
and equating coefficients of $Y_m^n$ we obtain \begin{eqnarray} \label{uvrecurse}
v_0z_m= -tu_0 \mbox{ and }v_nz_m=(s_mv_{n-1}-tu_n) \mbox{ for all }n\geq 1.
\end{eqnarray} 
Now we recall from Proposition \ref{Uifree}(a) that $\cU^m$ is a free $\cO(X)^\circ$-module on  \[T_m=\{(\partial_x/r)^{\langle \alpha\rangle}(\partial_x/s)^{\langle p^m \beta\rangle} : 0\leq \alpha<p^m, \beta\geq 0\}.\] 
Using $(\ref{d<k>formula})$, we see that $(\partial_x/s)^{\langle p^m \beta\rangle}$ only differs from $z_m^\beta = \left((\partial_x/s)^{\langle p^m \rangle}\right)^{\beta}$ by a $p$-adic unit, for any $\beta \geq 0$. Therefore
\[T'_m:=\{(\partial_x/r)^{\langle \alpha\rangle}z_m^\beta : 0\leq \alpha<p^m, \beta\geq 0\}\] 
is also a basis for $\cU^m$ as a $\cO(X)^\circ$-module. Now right multiplication by $z_m$ in $\cU^m$ maps $T'_m$ to $T'_m$ and so if $wz_m\in t\cU^m$ for some $w\in \cU^m$ then $w\in t\cU_m$. Using (\ref{uvrecurse}) we thus see inductively that $v_n\in t\cU_m$ for all $n\geq 0$. It follows that $v/t\in \cV^m$ and so $u=(v/t)(s_mY_m-z_m)\in \cV^m(s_mY_m-z_m)$. We have shown that $\ker \phi_m = \cV^m(s_mY_m-z_m)$, which is even stronger than what we need.
\end{proof}

\begin{cor} \label{presenthKUi} The homomorphism $\h{\phi_{i,K}}\colon \h{\cV^i_K}\to \h{\cU^{i+1}_K}$ induces an isomorphism $\h{\cV^i_K}/(Y_i-y_i)\stackrel{\cong}{\longrightarrow} \h{\cU^{i+1}_K}$. \end{cor}
\begin{proof} By Proposition \ref{presentUi}, the complex $0 \to \cV^i (s_i Y_i - z_i) \to \cV^i \stackrel{\phi_i}{\longrightarrow} \cU^{i+1} \to 0$ has cohomology killed by $s_i^{2p-1}$. In this situation, \cite[Lemma 3.6]{DCapOne} tells us that the functor $\widehat{()}_K$ turns this complex into a short exact sequence. Hence the result. \end{proof}
\begin{proof}[Proof of Theorem \ref{flatcotangent}]
We note that by Lemma \ref{flatfinextn} and Lemma \ref{DrBC}(a) it will suffice to prove this result after replacing $K$ by a finite extension. In particular by \cite[Theorem 4.1.8]{ArdWad2023} we may assume that $X$ splits over $K$ and $s,r\in |K^\times|$. Then $r>r(X) = \varpi/\rho(X)$ by Corollary \ref{Heisenberg}. We may also assume $s > r$.

Since the sequence $|\varpi_m|$ converges to $\varpi$ from above, we may choose $m$ sufficiently large, so that $r>|\varpi_m|/\rho(X)$. We fix this $m$ for the rest of this proof. Now, we enlarge $K$ if necessary, to ensure that $\varpi_m\in K^\times$ and $p^{-1/p^m}\in K^\times$. We can then choose $s',r'\in K$ such that $|\varpi_m s'|=s$ and $|\varpi_m r'|=r$. By enlarging $K$ further, we may assume that $K^\times$ contains an element $\pi$ such that $\max\left\{|p|,\frac{1}{|r'|\rho(X)}\right\} \leq |\pi| < 1$.

Now since $r<s$ we see that $1/\rho(X)<|r'|<|s'|$. Then by Theorem \ref{Adellevm}, 
\[\cD_s(X)\cong \h{\cD}_{s'}^{(m)}(X)_K \qmb{and} \cD_r(X)\cong \h{\cD}_{r'}^{(m)}(X)_K.\]
We first consider the case when $s$ and $r$ are sufficiently close, in the sense that 
\[1 < s/r<|p|^{-1/p^m}.\] 
After this assumption, since $|s'/r'| = s/r$, we note that all conditions of Hypothesis \ref{NotCotDirFlat} are satisfied for the elements $s'$ and $r'$ of $K^\times$ in place of $s$ and $r$, respectively.

We form $\cU^0 := \cD_{s'}^{(m)}(X)$, and inductively, $\cU^{i+1} := \cU^i[ (\partial_x/{r'})^{\langle p^i \rangle} ] \subset \cD_{r'}^{(m)}(X)$ as in Definition \ref{UiDefn}, for each $i=0,\cdots, m$. We let $\cV^i$ be the skew-Ore extension $\cU^i[Y_i; \delta_i]$ where $\delta_i : \cU^i \to \cU^i$ is the derivation $\ad(y_i)$ with $y_i := (\partial_x/r')^{\langle p^i\rangle}$, as in Notation \ref{yiVi}. Lemma \ref{adyi} ensures that $[y_i, \cU^i] \subseteq \pi \cU^i$. Now Theorem \ref{AbstractFlatness} together with Corollary \ref{presenthKUi} imply that $\h{\cU^{i+1}_K} \cong \h{\cV^i_K} / (Y_i - y_i)$ is a flat $\h{\cU^i_K}$-module on both sides. By the transitivity of flatness, we see that $\h{\cU^{m+1}_K} = \h{\cD}^{(m)}_{r'}(X)_K = \cD_r(X)$ is a flat $\h{\cU^0_K} = \h{\cD}^{(m)}_{s'}(X)_K = \cD_s(X)$-module on both sides, as claimed.

Returning to full generality, we can find a sequence $r=r_0<r_1<\ldots<r_N=s$, all in $|K^\times|$, such that $1<r_i/r_{i-1}\leq |p|^{-1/p^m}$ for each $1\leq i\leq N$. The transitivity of flatness now reduces us to the special case considered above.
\end{proof}

\begin{proof}[Proof of Theorem \ref{FlatThm}]
By Proposition \ref{reduction} it suffices to show that if $s,r\in \sqrt{|K^\times|}$ with $s\geq r$ and $Y\subseteq X$ both lie in $\bA(\partial_x/r)$, the map $\cD_s(X)\to \cD_r(Y)$ is flat on both sides. By Theorem \ref{flatcotangent} $\cD_s(X)\to \cD_r(X)$ is flat on both sides, and by Theorem \ref{changeofbase}, $\cD_r(X)\to \cD_r(Y)$ is flat on both sides. We're now done by the transitivity of flatness.
\end{proof} 
\section{Overconvergent \ts{\cD}-modules}\label{Dmod}
\subsection{The Huber space and overconvergent sheaves}\label{HuberSpace}
Recall from \cite[\S 5]{SchVdPut} that to every rigid analytic space $X$ we can associate a topological space $\tilde{X}$ which we call the \emph{Huber space} of $X$. The elements of $\tilde{X}$ are the prime filters on the admissible open subsets of $X$. The sets of the form $\tilde{U} := \{ p \in \tilde{X} : U \in p\}$ as $U$ ranges over the admissible open subsets of $X$ form a basis for the topology on $\tilde{X}$. If $X$ is quasi-compact and quasi-separated (qcqs), then $\tilde{X}$ is quasi-compact but in general not Hausdorff. There is a set-theoretic embedding $X \hookrightarrow \tilde{X}$ given by $x \mapsto \fr{m}_x$, where $\fr{m}_x$ is the principal maximal filter consisting of the admissible open subsets of $X$ containing $x$. We will identify $X$ with its image in $\tilde{X}$ via this embedding.  By \cite[\S 5]{SchVdPut} there is an equivalence of categories between the abelian sheaves on $X$ and the abelian sheaves on $\tilde{X}$. Given an abelian sheaf $\cF$ on $X$, we will denote the corresponding sheaf on $\tilde{X}$ by $\tilde{\cF}$; it follows from the proof of \cite[Theorem 1]{SchVdPut} that $\tilde{\cF}(\tilde{W}) = \cF(W)$ for any admissible open subset $W$ of $X$. 

\begin{defn}\label{iUdagDefn} Let $X$ be a rigid analytic space, let $U$ be an admissible open subset of $X$ and let $\cF$ be an abelian sheaf on $X$.
\be \item Let $i : \overline{\tilde{U}} \to \tilde{X}$ denote the inclusion of the closure of $\tilde{U}$ into $\tilde{X}$.
\item Let $\cF_{\overline{U}}$ be the abelian sheaf on $X$ defined by
\[ \widetilde{\cF_{\overline{U}}}= i_{\ast} i^{-1} \tilde{\cF}.\]
\ee\end{defn}

In the setting of Definition \ref{iUdagDefn}, $i^{-1}$ is left adjoint to $i_\ast$, and the counit of the adjunction gives a natural morphism $\cF\to \cF_{\overline{U}}$ called restriction.
 
Note that $\widetilde{\cF_{\overline{U}}}$ is denoted $\tilde{\cF}_{\overline{\tilde{U}}}$ in \cite[\S2.3]{KS}. This construction is local, in the following sense.

\begin{lem}\label{iDagLocal} Let $X$ be a rigid analytic space, let $U$ be an admissible open subset of $X$ and let $\cF$ be an abelian sheaf on $X$.  Let $Y$ be another admissible open subset of $X$. Then there is a natural isomorphism of abelian sheaves on $Y$
\[ (\cF_{\overline{U}})_{|Y} \cong (\cF_{|Y})_{\overline{U\cap Y}}.\]
\end{lem}
\begin{proof}Because $\tilde{Y}$ is open in $\tilde{X}$, we deduce from \cite[Chapter I, \S 1.6, Proposition 5]{BourGenTop} that $\tilde{Y} \cap \overline{\tilde{U}}$ is the closure in $\tilde{Y}$ of $\tilde{Y} \cap \tilde{U} = \widetilde{Y \cap U}$. Let $f : \tilde{Y} \hookrightarrow \tilde{X}$ be the open inclusion. Applying the formula \cite[(2.3.19)]{KS} with $Z := \overline{\tilde{U}}$ gives
\[ \widetilde{(\cF_{\overline{U}})_{|Y}} = f^{-1} (\tilde{\cF}_{\overline{\tilde{U}}}) = (f^{-1} \tilde{\cF})_{f^{-1}(\overline{\tilde{U}})} = (f^{-1} \tilde{\cF})_{\overline{\widetilde{U \cap Y}}^{\tilde{Y}}} = \widetilde{(\cF_{|Y})_{\overline{U \cap Y}}}.\]
Now use the equivalence of categories $\Ab(Y) \cong \Ab(\tilde{Y})$ mentioned above. \end{proof}
The following result will prove useful later.

\begin{prop} \label{ApproxOnClosed} Suppose that $\cU$ is an admissible cover of $X$ that is totally ordered by inclusion and $\cF$ is an abelian sheaf on $X$. Then the restriction maps $\cF\to \cF_{\overline{U}}$ for $U\in \cU$ induce an isomorphism \[\cF\stackrel{\cong}{\to}\lim\limits_{\longleftarrow}\cF_{\overline{U}}.\]
\end{prop}

\begin{proof} If $U,V\in \cU$ with $U\subset V$ then, by \cite[Proposition 2.3.6(iii)]{KS}  \[\cF_{\overline{U}}=(\cF_{\overline{V}})_{\overline{U}}\] and so it is easy to verify that there is a natural morphism $\cF\to \lim\limits_{\longleftarrow} \cF_{\overline{U}}$ with the connecting maps in the limit also given by restriction. 
	
	Moreover, by Lemma \ref{iDagLocal}, for $U\subset V$ in $\cU$ we can compute \[\cF_{\overline{V}}|_{U}=(\cF|_{U})_{\overline{U\cap V}}=\cF|_{U}.\] Since for all $W\in \cU$ there is $V\in \cU$ containing both $U$ and $W$, $\{V\in \cU\st U\subset V\}$ is cofinal in $\cU$, and it follows that $ \cF|_{U}\cong \left(\lim\limits_{\longleftarrow} \cF_{\overline{V}}\right)|_{U}$ for all $U\in \cU$. Since $\cU$ is a cover of $X$ the result follows. 
\end{proof}

We note that the proof above works whenever $\cU$ is a directed set under inclusion; that is whenever any two elements of $\cU$ have a common upper bound in $\cU$ under inclusion.

The operation $\cF \mapsto \cF_{\overline{U}}$ on $\cO_X$-modules behaves well with respect to tensor products.

\begin{lem} \label{DagTensor} If $\cF$ and $\cG$ are two $\cO_X$-modules and $U$ is an admissible open subset of $X$ then there is a natural isomorphism  \[ (\cF\underset{\cO_X}{\otimes}{}\cG)_{\overline{U}}\cong \cF_{\overline{U}}\underset{(\cO_X)_{\overline{U}}}{\otimes}{}\cG_{\overline{U}}.\]
\end{lem}
\begin{proof}
By \cite[Proposition 2.3.10]{KS}  we have
\[\cF_{\overline{U}}\cong (\cO_X)_{\overline{U}}\underset{\cO_X}{\otimes} \cF \qmb{and} \cG_{\overline{U}}\cong (\cO_X)_{\overline{U}}\underset{\cO_X}{\otimes} \cG.\]
Now tensor these together over $(\cO_X)_{\overline{U}}$ and contract tensors. 
\end{proof}

We will now restrict to the case where $X$ is an affinoid variety, and explain how $\cF_{\overline{U}}$ is related to the more well-known operations in rigid analytic geometry.

\begin{defn}\label{SlightlyLarger} Suppose that $X$ is affinoid, and $U = X\left(\frac{g_1}{g_0},\frac{g_2}{g_0}, \cdots, \frac{g_n}{g_0}\right)$ is a rational subdomain of $X$, where $g_0,\ldots,g_n \in \cO(X)$ generate the unit ideal in $\cO(X)$. Following \cite[Exercise 7.1.12(5)]{FvdPut}, for each $s \in \sq{K}$ such that $s > 1$, define 
\[ U(s) := \{a \in X : |g_i(a)| \leq s |g_0(a)| \qmb{for all} i=1,\ldots, n\}.\]
\end{defn}
These slightly larger rational subdomains of $X$ form a cofinal system of \emph{wide open neighbourhoods} of $U$ in $X$, by \cite[Exercise 7.1.12(5)(a)]{FvdPut}. Note that this notation is slightly misleading because the sets $U(s)$ depend in general on the choice of presentation $U = X\left(\frac{g_1}{g_0},\frac{g_2}{g_0}, \cdots, \frac{g_n}{g_0}\right)$ of $U$ as a rational subdomain of $X$. 

\begin{lem}\label{CofinalWideOpens} Let $V$ be an admissible open subset of $X$.  Then $\tilde{V}$ contains the closure $\overline{\tilde{U}}$ in $\tilde{X}$ if and only if $V$ contains $U(s)$ for some $s \in \sq{K}$ with $s> 1$.\end{lem}
\begin{proof} Let $r : \tilde{X} \to \sM(X)$ be the retraction map onto the Berkovich space $\sM(X)$ associated to $X$. We claim that $\overline{\tilde{U}} = r^{-1}(\sM(U))$. To see this, note that $\sM(U) = r(\tilde{U})$ is closed in $\sM(X)$ because $\sM(U)$ is compact and $\sM(X)$ is Hausdorff; since $r$ is continuous this gives $\overline{\tilde{U}} \subseteq r^{-1}(\sM(U))$. For the other inclusion, let $x \in r^{-1}(\sM(U))$ so that $r(x) = r(y)$ for some $y \in \tilde{U}$. Therefore $x \in r^{-1}(r(y)) = \overline{\{y\}}$, by \cite[Lemma 3.2ii]{SchVdPut}, and thus $x \in \overline{\{y\}} \subseteq \overline{\tilde{U}}$ as claimed. 

It follows that $\tilde{V}\supseteq \overline{\tilde{U}}$ if and only if $r^{-1}(a) \subset \tilde{V}$ for all $a \in \sM(U)$. Now, $\tilde{V} \supseteq \overline{\tilde{U}}$ implies that $V$ is a wide open neighbourhood of $U$ in the sense of \cite[Exercise 7.1.12(4)]{FvdPut}. Hence $V$ contains some $U(s)$ by \cite[Exercise 7.1.12(5)(a)]{FvdPut}. Conversely, if $V$ contains some $U(s)$ then $V$ is a wide open neighbourhood of $U$ so for all $a \in \sM(U)$ we can find an open neighbourhood $W_a$ of $r^{-1}(a)$ contained in $V$. But then $\overline{\tilde{U}} = r^{-1}(\sM(U)) = \bigcup\limits_{a \in \sM(U)} r^{-1}(a) \subseteq \bigcup\limits_{a \in \sM(U)} \tilde{W_a} \subseteq \tilde{V}$ as required.\end{proof}

\begin{lem}\label{Acyclic} Let $\cF$ be an abelian presheaf on an affinoid variety $X$, let $U$ be a rational subdomain of $X$, and let $\cF^\dag_U$ be the presheaf
\[\cF^\dag_U : Y \mapsto \underset{s \in \sq{K}, s > 1}{\colim}{} \hsp \cF(Y \cap U(s))\]
on the weak $G$-topology of $X$. Then $\cF^\dag_U$ is a sheaf whenever $\cF$ is a sheaf, and $\cF^\dag_U$ is acyclic whenever $\cF$ is acyclic.
\end{lem}
\begin{proof} If $\cY := \{Y_1, \ldots, Y_n\}$ is a finite affinoid covering of an affinoid subdomain $Y$ of $X$, and if $s \in \sq{K}$ and $s>1$, then $\cY_s := \{Y_1 \cap U(s), \ldots, Y_n \cap U(s)\}$ is a finite affinoid covering of $Y \cap U(s)$. Hence there is an isomorphism of augmented \v{C}ech complexes
\[C^\bullet_{\aug}(\cY, \cF^\dag_U) = \underset{s \in \sq{K},s>1}{\colim}{} \hsp C^\bullet_{\aug}(\cY_s, \cF).\]
The result follows, because cohomology commutes with filtered colimits.
\end{proof}

\begin{thm}\label{iDag} We have $\cF_{\overline{U}} = \cF^\dag_U$ for every abelian sheaf $\cF$ on an affinoid variety $X$ and every rational subdomain $U$ of $X$.
\end{thm}
\begin{proof} For every open subset $V$ of $C := \overline{\tilde{U}}$, let 
\[ (i_{\pre}^{-1} \tilde{\cF})(V) = \colim \tilde{\cF}(W)\]
where the colimit runs over all open subsets $W$ of $\tilde{X}$ containing $i(V) = V$. This defines a presheaf $i_{\pre}^{-1} \tilde{\cF}$ on $C$, and $i^{-1}\tilde{\cF}$ is the sheafification of this presheaf. We will first show that 
\begin{equation}\label{ipreForm} (i^{-1}_{\pre} \tilde{\cF})( \tilde{Z} \cap C ) = \underset{s \in\sq{K},s > 1}{\colim}{} \hsp \cF( Z \cap U(s) )\end{equation}
for every affinoid subdomain $Z$ of $X$. Consider the open subsets $W$ of $\tilde{X}$ containing $\tilde{Z} \cap C$. Clearly those $W$ already contained in $\tilde{Z}$ are cofinal; since $\tilde{Z} \cap C$ is the closure of $\tilde{Z} \cap \tilde{U} = \widetilde{Z \cap U}$ in $\tilde{Z}$, Lemma \ref{CofinalWideOpens} tells us that the subsets of the form $(Z \cap U)(s)^{\tilde{}}$ are cofinal amongst these. Now $(\ref{ipreForm})$ follows since $(Z \cap U)(s) = Z \cap U(s)$ for each $s$.

Because the functor $Z \mapsto \colim\limits_{s \in\sq{K},s > 1} \cF( Z \cap U(s) )$ satisfies the sheaf condition with respect to finite coverings of $Z$ by affinoid subdomains of $Z$ by Lemma \ref{Acyclic}, and because the $\tilde{Z} \cap C$ form a basis for the topology of $C$, $(\ref{ipreForm})$ now implies that the values of $i^{-1}\tilde{\cF}$ agree with those of $i^{-1}_{\pre}\tilde{\cF}$ on these basic open sets. 

Hence we can calculate that
\[ \cF_{\overline{U}}(Y) = (i_\ast i^{-1} \tilde{\cF})(\tilde{Y}) = (i^{-1} \tilde{\cF})(\tilde{Y} \cap C) = \underset{s \in\sq{K}, s > 1}{\colim}{} \hsp \cF( Y \cap U(s) )\]
for every affinoid subdomain $Y$ of $X$ as claimed.
\end{proof}
Theorem \ref{iDag} implies, in particular, that the sheaf $\cF^\dag_U$ in fact does \emph{not} depend on the choice of presentation of $U$ as a rational subdomain of $X$. We also see that we should think of $\cF_{\overline{U}}$ as being the sheaf of \emph{sections of $\cF$ overconvergent into the complement of $U$}.

\begin{cor}\label{DagTate} Let $U$ be a rational subdomain of the affinoid variety $X$ and let $\cF$ be a coherent $\cO_X$-module. Then the higher \v{C}ech cohomology of $\cF_{\overline{U}}$ is zero.
\end{cor}
\begin{proof} This follows from Theorem \ref{iDag}, Lemma \ref{Acyclic} and Tate's Acyclicity Theorem \cite[Theorem 4.5.2 and Theorem 4.2.2]{FvdPut}.
\end{proof}

\subsection{The \ts{\cD^\dag_{\varpi/t}}-module \ts{\cM(\cS,u,d)_{\overline{U_t}}} }\label{KummerSect} In this section we study a Kummer-type finite \'etale covering of the rigid-analytic affine line minus finitely many rational points. We recall that any $K$-scheme $X$ of locally finite type has an analytification $X^{\an}$ that is a rigid analytic space over $K$. The construction $X \mapsto X^{\an}$ is functorial in $X$, and we refer the reader to \cite[\S1.6]{LutJacobians} or \cite[\S 5.4]{BoschFRG} for more details.

We fix $\cS=\{a_1,\ldots,a_h\}$, a finite subset of $K$.
 
\begin{defn}\label{ZeroesPoles} Given $u\in K(x)$ a rational function with factorisation in $\overline{K}(x)$ 
\[ u = \lambda \prod_{i=1}^n (x - b_i)^{k_i}\in K(x) \]
for some $\lambda \in K^\times$, pairwise distinct $b_1,\ldots,b_n\in \overline{K}$ and non-zero integers $k_1,\ldots,k_n$. The \emph{set of finite zeroes/poles} of $u$ is 
\[\cS(u) :=\{b_i \st i=1,\ldots n \}.\] We write  $v_{b_i}(u)= k_i \mbox{ for }i=1,\ldots n\mbox{ and } v_{b}(u)=0 \mbox{ for }b\in \overline{K}\backslash\{b_1,\ldots,b_n\}$.

{\bf We fix $u$ such that $\cS(u)$ is a subset of $\cS$.} Thus \[ u=\lambda\prod_{a\in \cS}(x-a)^{v_a(u)}\] with the $v_a(u)$ integers that may be zero and some $\lambda\in K^\times$. We consider the open embedding of rigid $K$-analytic spaces
\[j := \A - \cS \hookrightarrow \A := (\Spec K\left[x\right])^{\an},\]
and we introduce the finite \'etale covering 
\[f : Z(u,d) \to \A-\cS\]
where
\begin{equation}\label{DefOfZud}Z(u,d) := \left(\Spec K\left[x, \frac{1}{x-a_1},\cdots,\frac{1}{x-a_h}\right][z] / (z^d - u)\right)^{\an}.\end{equation}
\end{defn}
Note that $f_\ast \cO_{Z(u,d)}$ is a free $\cO_{\A-\cS}$-module of rank $d$:
\[ f_\ast \cO_{Z(u,d)} = \bigoplus_{m=0}^{d-1} \cO_{\A-\cS} z^m.\]
Since $f : Z(u,d) \to \A-\cS$ is \'etale, the discussion in $\S \ref{DivTwDer}$ shows that $f_\ast \cO_{Z(u,d)}$ is in fact a $\cD_{\A-\cS}$-module, and that this is a decomposition of $f_\ast \cO_{Z(u,d)}$ into a direct sum of $d$ line bundles with flat connection on $\A - \cS$.  Applying (\ref{zdz}) with $u$ replaced by $u^{-m}$ shows that the $\cD_{\A- \cS}$-action on each summand in the above decomposition is given by 
\begin{equation}\label{GMconn} \partial_x \cdot z^m = \frac{m}{d} \left(\sum_{a\in \cS} \frac{v_a(u)}{x - a}\right) \hsp z^m.\end{equation}

Our main goal in $\S \ref{KummerSect}$ is to study the following $\cD_{\A}$-module.
\begin{defn}\label{UstMud} Define $\cM(\cS,u,d) :=j_\ast (\cO_{\A-\cS} z)$. 
\end{defn}
	
The other line bundles with connection $\cO_{\A - \cS}z^m$ can be dealt with by replacing $u$ by $u^m$ so no generality is lost by restricting to the case $m = 1$.

\begin{defn} \label{USt} For each $t \in \sq{K}$ we define 
\[U(\cS)_t:= U_t  := \{z \in \A : |z - a|  \geq t \qmb{for all} a\in \cS\}.\]
\end{defn}
Clearly $U_t$ is a subset of $\A-\cS$, we have $U_t \subseteq U_{t'}$ whenever $t' < t$, and 
\[  \A - \cS = \bigcup\limits_{t > 0} U_t.\]
We will study the $\cD$-module $\cM(\cS,u,d)$ through its approximations $\cM(\cS,u,d)_{\overline{U}_t}$, as $t$ shrinks to zero: one can use Proposition \ref{ApproxOnClosed} to see that \[\cM(\cS,u,d)=\lim\limits_{\stackrel{\longleftarrow}{t>0}}\cM(\cS,u,d)_{\overline{U}_t}.\] We will give a proof of another form of this statement in Lemma \ref{jLlimit}.

\textbf{We assume until the end of $\S \ref{KummerSect}$ that $t \in \sq{K}$}.

 Next let us see what the local sections of these objects $\cM(\cS,u,d)_{\overline{U}_t}$ look like.

\begin{lem}\label{OvConvSects} Let $\cF$ be an abelian sheaf on $\bA$. Then
\[(\cF_{\overline{U_t}})(X) = \underset{s \in \sq{K}, s < t}{\colim}{} \hsp \cF(X\cap U_s).\]
for every affinoid subdomain $X$ of $\bA$.\end{lem}
\begin{proof} 

It is easy to verify by looking at the $\bf{C}$-points that 
\[X \cap U_t=X\left(\frac{t}{x-a_1},\ldots \frac{t}{x-a_h}\right).\]
Recall now the wide open neighbourhoods $(X \cap U_t)(r)$ of the affinoid subdomain $X \cap U_t$ of $X$ from Definition \ref{SlightlyLarger}.  We can now apply Lemma \ref{iDagLocal} together with Theorem \ref{iDag} to compute
\[ (\cF_{\overline{U_t}})(X) = (\cF_{|X})_{\overline{X \cap U_t}}(X) = \underset{r \in \sq{K}, r > 1}{\colim}{} \hsp \cF((X\cap U_t)(r)).\]
But for any $r \in \sq{K}$ with $r > 1$, we have
\[ (X \cap U_t)(r)(\mathbf{C}) = \bigcap\limits_{i=1}^h \{y \in X(\mathbf{C}) : t \leq r |(y-a_i)| \} = X \cap U_{t/r},\]
so $(\cF_{\overline{U_t}})(X) = \underset{s \in \sq{K}, s < t}{\colim}{} \hsp \cF(X\cap U_s)$ as required.\end{proof}
\begin{cor}\label{iUjOSect} We have $\cM(\cS,u,d)_{\overline{U_t}} = (\cO_{\A})_{\overline{U_t}} \hsp z$ as sheaves of $\cO_{\A}$-modules.
\end{cor}
\begin{proof} Let $X$ be an affinoid subdomain of $\A$ and let $s \in \sq{K}$. Since $X \cap U_s \subset \bA \backslash \cS$, we see that $\cM(\cS,u,d)(X \cap U_s) = \cO(X \cap U_s)z$. So by Lemma \ref{OvConvSects}, 
\begin{equation}\label{OvConvMudSects}\cM(\cS,u,d)_{\overline{U_t}}(X) = \underset{s \in \sq{K}, s < t}{\colim}{\hsp} \cO(X \cap U_s)\hsp z. \qedhere\end{equation}
\end{proof}

Recall from Definition \ref{DagSite}(a) that $\A( t \partial_x/\varpi )^\dag$ is the $G$-topology on $\A$ which consists of those affinoid subdomains such that $r(X) \leq \varpi / t$. In view of Corollary \ref{Heisenberg}, we see that $X \in \A( t \partial_x/\varpi )^\dag$ if and only if $\rho(X) \geq t$. The following technical Lemma is necessary to prove Lemma \ref{Ddagactover} and Proposition \ref{OUrtuMod}.

\begin{lem}\label{XUs} Let $X$ be a $(t \partial_x/\varpi)^\dag$-admissible affinoid subdomain of $\bA$, and let $s \in \sq{K}$ be such that $s < t$. Then $r(X \cap U_s) \leq \varpi/s$.
\end{lem}
\begin{proof} In view of  Lemma \ref{BaseChR}, we can replace $K$ by some finite field extension and thereby assume that $s \in K^\times$. Choose a disc $\D := \Sp K \langle x/s_0\rangle$ with $s_0 \in K^\times$ large enough to contain $X$ and note that $\D \cap U_s = C(\alpha,\mathbf{s})$ where $s_1 := s_2 := \cdots := s_h := s$, $\alpha_0 = 0$ and $\alpha_i := a_i$ for each $i=1,\ldots, h$.  Then $X \cap U_s = X \cap C(\alpha,\mathbf{s})$ and $r(C(\alpha,\mathbf{s})) = \frac{\varpi}{\rho(C(\alpha,\mathbf{s}))} = \frac{\varpi}{s}$ by Corollary \ref{Heisenberg},  so
\[r(X \cap U_s) = r(X \cap C(\alpha,\mathbf{s})) \leq \max\{ r(X), r(C(\alpha, \mathbf{s}))\} = \max\{ r(X), \varpi/s\} \]
by Proposition \ref{Gtop}. Our assumptions on $X$, $s$ and $t$ imply that $r(X) \leq \varpi/t < \varpi/s$, so $\max\{ r(X), r(C(\alpha,\mathbf{s}))\} = \varpi/s$.
\end{proof}

\begin{lem}\label{Ddagactover} The $\cD$-action on the restriction of $(\cO_\A)_{\overline{U_t}}$ to $\A( t \partial_x/\varpi )^\dag$ extends to $\cD^\dag_{\varpi/t}$.
\end{lem}
\begin{proof} Let $X \in \A( t \partial_x/\varpi )^\dag$; we must show that the $\cD(X)$-action on $(\cO_\A)_{\overline{U_t}}(X)$ extends to an action of $\cD^\dag_{\varpi/t}(X)$.

Let $f \in \cO_{\overline{U_t}}(X)$ and $Q \in \cD^\dag_{\varpi/t}(X)$. Then by Lemma \ref{OvConvSects} and Definition \ref{DagSite}(d) we can find some $s'<t$ and $r > \varpi/t$ such that $f \in \cO(X \cap U_{s'})$ and $Q \in \cD_r(X)$. Choose $s \in \sq{K}$ such that 
\[\max\{\varpi/r, s'\} < s < t.\]
Then $f \in \cO(X \cap U_s)$ and $r > \varpi/s$, so $Q \cdot f \in \cO(X \cap U_s) \subset \cO_{\overline{U_t}}(X)$ by Lemma \ref{XUs} and Lemma \ref{ActionOnO}(a). Thus the $\cD(X)$-action on $(\cO_\A)_{\overline{U_t}}(X)$ extends to an action of $\cD^\dag_{\varpi/t}(X)$ as required.
\end{proof}

\textbf{We assume for the remainder of $\S \ref{KummerSect}$ that $p \nmid d$.}

\begin{lem}\label{DrActsOnOz} Let $r > 0$ and let $Y$ be a $\partial_x/r$-admissible affinoid subdomain of $\A - \cS$. Then the natural $\cD(Y)$-action on $\cO(Y) z$ extends to a $\cD_r(Y)$-action.
\end{lem}
\begin{proof} Since $\partial_x(z)=\frac{1}{d}\frac{\partial_x(u)}{u} z$, the $\cD(Y)$-action on $\cO(Y) z$ is given by \[ (P,fz)\mapsto \sigma_r(\theta_{u^{-1},d}(P))(f) z,\] where $\theta_{u^{-1},d}$ was defined in Lemma \ref{zTwist} and $\sigma_r$ was defined in Lemma \ref{ActionOnO}. Thus to extend it to a $\cD_r(Y)$-action it suffices to extend the automorphism $\theta_{u^{-1},d}$ of $\cD(Y)$ to an automorphism of $\cD_r(Y)$; that this can be done is guaranteed by Theorem \ref{ThetaU}.
\end{proof}

This result enables us to extend Lemma \ref{Ddagactover} in the following manner.

\begin{prop}\label{OUrtuMod} The $\cD$-action on the restriction of $\cM(\cS,u,d)_{\overline{U_t}}$ to $\A(t \partial_x/\varpi)^\dag$ extends to $\cD^\dag_{\varpi/t}$. \end{prop}
\begin{proof} 
Let $X \in \A(t \partial_x/\varpi)^\dag$ and suppose that $f \in \cM(\cS,u,d)_{\overline{U_t}}(X)$ and $Q \in \cD^\dag_{\varpi/t}(X)$. Then by (\ref{OvConvMudSects}) and Definition \ref{DagSite}(d) we can find some $s'<t$ and $r > \varpi/t$ such that $f \in \cO(X \cap U_{s'})z$ and $Q \in \cD_r(X)$. 

Choose $s \in \sq{K}$ such that 
\[\max\{\varpi/r, s'\} < s < t.\]

Then $f \in \cO(X \cap U_s)z$ so if the $\cD(X)$-action on  $\cO(X \cap U_s)z$ extends to $\cD_r(X)$ then $Q\cdot f\in \cM(\cS,u,d)_{\overline{U_t}}(X)$.

 Now $r(X \cap U_s) \leq \varpi/s$ by Lemma \ref{XUs} and $\varpi/s<r$, so $X \cap U_s$ is $\partial_x/r$-admissible. Hence the $\cD(X \cap U_s)$-action on $\cO(X \cap U_s)z$ extends to $\cD_r(X \cap U_s)$ by Lemma \ref{DrActsOnOz}. Since $X$ and $X \cap U_s$ are both $\partial_x/r$-admissible, by Proposition \ref{DbRing}(b) there is a ring homomorphism $\cD_r(X) \to \cD_r(X \cap U_s)$ extending $\cD(X) \to \cD(X \cap U_s)$, and we deduce that the $\cD(X)$-action on $\cO(X\cap U_s)z$ extends to $\cD_r(X)$ as required.
\end{proof}

In the remainder of this section we will establish an explicit presentation of $\cM(\cS, u,d)_{\overline{U_t}}$ as a $\cD^\dag_{\varpi/t}$-module, assuming $t\in \sq{K}$ is sufficiently small and a positivity condition on the exponents of $u$.

\begin{defn}\label{TheRelator} We define
\be \item $\Delta_\cS := \prod\limits_{a \in \cS} (x - a)$, and
\item $R_\cS(u,d) :=\Delta_{\cS} \partial_x - \frac{1}{d} \sum\limits_{a\in \cS} v_{a}(u) \prod\limits_{b\in \cS\backslash\{a\}}(x - b)= \Delta_{\cS}( \partial_x - \frac{1}{d} \dlog(u) ) $.
\ee\end{defn}
Note that $R_\cS(u,d) \in \cD(\bA)$. We will write $R(u,d)=R_{\cS(u)}(u,d)$ and $\Delta=\Delta_{\cS(u)}$. 

\begin{lem}\label{OUrtuPres} There is a complex of $\cD^\dag_{\varpi/t}$-modules on $\A(t \partial_x/\varpi)^\dag$
\[0 \to \cD^\dag_{\varpi/t} \stackrel{\cdot R_{\cS}(u,d)}{\longrightarrow}\cD^\dag_{\varpi/t} \stackrel{\cdot z}{\longrightarrow} \cM(\cS,u,d)_{\overline{U_t}} \to 0.\]
\end{lem}
\begin{proof} Note that the restriction of $\cM(\cS,u,d)_{\overline{U_t}}$ to $\A(t \partial_x/\varpi)^\dag$ is a $\cD^\dag_{\varpi/t}$-module by Proposition \ref{OUrtuMod}. The second non-zero arrow sends $Q \in \cD^\dag_{\varpi/t}$ to $Q \cdot z \in \cM(\cS,u,d)_{\overline{U_t}}$.  By definition, $R_\cS(u,d) = \Delta_\cS \partial_x - \frac{1}{d} \sum_{i=1}^h v_{a_i}(u) \prod_{j \neq i}(x - a_j) \in \cD$, so $R_\cS(u,d) \in \cD_{\varpi/t}^\dag$ as well. A direct calculation shows that $R_\cS(u,d) \cdot z = 0$.
\end{proof}

We will now carry out three local calculations that will be used to prove that the complex appearing in Lemma \ref{OUrtuPres} is exact, at least when the $a_i$ live in different holes of $U_t$, and thus provides a finite presentation for $\cM(u,d)_{\overline{U_t}}$ as a $\cD^\dag_{\varpi/t}$-module. The first shows that the restriction of the complex to $U_t$ is exact.

\begin{prop}\label{YPres} Let $Y$ be an $(t \partial_x/\varpi)^\dag$-admissible affinoid subdomain of $U_t$. Then the following complex is exact: 
\[0 \to \cD^\dag_{\varpi/t}(Y) \stackrel{\cdot R_\cS(u,d)}{\longrightarrow}\cD^\dag_{\varpi/t}(Y) \stackrel{\cdot z}{\longrightarrow} \cM(\cS,u,d)_{\overline{U_t}}(Y) \to 0.\]
\end{prop}
\begin{proof} Since $Y$ is contained in $U_t$, $u$ and $\Delta$ are both units in $\cO(Y)$ and it follows from Lemma \ref{iDagLocal} that $\cO(Y) z \to \cM(\cS,u,d)_{\overline{U_t}}(Y)$ is an isomorphism. By Corollary \ref{ThetaUDag}, $Q \mapsto z \hsp Q \hsp z^{-1}$ induces an automorphism $\theta_{u,d}$ of $\cD^\dag_{\varpi/t}(Y)$, then by Definition \ref{TheRelator}(b) and Lemma \ref{zTwist}
	\[\theta_{u,d}^{-1}(R_\cS(u,d))= \Delta\partial_x.\]

	Since $\Delta$ is a unit in $\cO(Y)$, we see that $\cD^\dag_{\varpi/t}(Y) \cdot R_\cS(u,d) = \cD^\dag_{\varpi/t}(Y) \cdot \theta_{u,d}(\partial_x)$. By construction of the rings $\cD_r(Y)$, there is a direct sum decomposition of $\cO(Y)$-modules: $\cD_r(Y) = \cO(Y) \oplus \cD_r(Y) \cdot \partial_x$ whenever $r > \varpi/t$. Therefore we also have the decomposition
	\[\cD^\dag_{\varpi/t}(Y)= \cO(Y) \oplus \cD^\dag_{\varpi/t}(Y) \theta_{u,d}(\partial_x)\]
	and the result now follows easily.
\end{proof}

We record an easy but important consequence of the proof of Proposition \ref{YPres}.

\begin{cor}\label{LnYsimple} Let $Y$ be a $(t \partial_x/\varpi)^\dag$-admissible subdomain of $U_t$. Then the $\cO(Y)$-module $\cM(\cS,u,d)_{\overline{U_t}}(Y)$ is free of rank $1$, and it is simple as a $\cD^\dag_{\varpi/t}(Y)$-module whenever $Y$ is connected as an affinoid domain.
\end{cor}
\begin{proof} The first statement follows from the proof of Proposition \ref{YPres}. 

	For the second statement it suffices to show that  $\cM(\cS,u,d)_{\overline{U_t}}(Y)$ is a simple $\cD(Y)$-module via restriction whenever $Y$ is connected. Given the first part this follows from \cite[Lemma 4.3.2]{ArdWad2023}, because $\cT(Y)$ is freely generated by $\partial_x$ as a $\cO(Y)$-module.
\end{proof}

Our next step is to prove that the complex in Lemma \ref{OUrtuPres} is exact in the case when $|\cS|=1$ under appropriate mild conditions on $u$. First we need a couple of technical results to enable us to a make a certain estimate. Together, they amount to a slightly more general version of \cite[Lemme 4.2.1]{Berth1990}.

\begin{lem}\label{BerthMMP} Suppose that $V$ is a a reduced $K$-affinoid variety and $f\in \cO(V)$. Then there is $s\in \sqrt{|K|^\times}$ with $s<t$ such that whenever $g \in \cO(V)$ vanishes on $V(t/f)$, $g$ also vanishes on $V(s/f)$.
\end{lem}

\begin{proof}
	We first consider the case where $V$ is irreducible so that $\cO(V)$ is an integral domain. Suppose first that $V(t/f)$ is non-empty. Then the restriction map $\cO(V)\to \cO(V(t/f))$ is injective by the proof of \cite[Proposition 4.2]{ABB}, so  any choice of $s$ will do. On the other hand, if $V(t/f)=\emptyset$, then $|f|_V<t$ by the maximum modulus principle, so choosing $|f|_V<s<t$ ensures that $V(s/f) = \emptyset$ and yields the result. 
	
	In the general case, let $V_1,\ldots,V_n$ be the irreducible components of $V$ and let $f_i := f_{|V_i}$ be the restriction of $f$ to $V_i$ for each $i=1,\cdots, n$. Now, for each $s < t$ with $s \in \sq{K}$ there is a commutative diagram \[ \xymatrix{ \cO(V) \ar[r] \ar[d] & \cO(V(s/f)) \ar[r] \ar[d] & \cO(V(t/f)) \ar[d] \\
		 \bigoplus\limits_{i=1}^n \cO(V_i) \ar[r] & \bigoplus\limits_{i=1}^n \cO(V_i(s/f_i)) \ar[r] & \bigoplus\limits_{i=1}^n \cO(V_i(t/f_i))  .
	}.\] The vertical arrow on the left is injective because $V$ is reduced. Hence the middle vertical arrow is injective, because the restriction map $\cO(V)\to \cO(V(s/f))$ is flat.
	     
 Since each $V_i$ is irreducible, we may find $s_i<t$ in $\sqrt{|K^\times|}$ such that every $g \in \cO(V_i)$ vanishing on $V_i(t/f_i)$ also vanishes on $V_i(s_i/f_i)$. By enlarging each $s_i$ if necessary, we may assume that $s_1 = \cdots = s_n =: s$, say.  
	     
	Suppose now that $g \in \cO(V)$ vanishes on $V(t/f)$. Then by the commutativity of the diagram, $g_{|V_i}$ vanishes on $V_i(t/f_i)$ for all $i$. Hence by the above paragraph, $g_{|V_i}$ also vanishes on $V_i(s/f_i)$ for all $i$. Using the fact that the middle arrow in the diagram is injective, we deduce that $g$ also vanishes on $V(s/f)$ as required. 
\end{proof}
  
\begin{lem}\label{Berth421} Let $V$ be a reduced $K$-affinoid variety and $f\in \cO(V)$. Suppose moreover that $(v_n)$ is a sequence in $\cO(V)$ such that \begin{enumerate}[(i)] \item $\sum\limits_{n=0}^\infty v_nT^n\in \cO(V)[[T]]$ has radius of convergence greater than $1/t$, and  \item the series $\sum\limits_{n=0}^\infty v_nf^{-n}$ converges to zero in $\cO(V(t/f))$.\ee Then, writing $w_m:=\sum\limits_{n=0}^m v_nf^{m-n}\in \cO(V)$ for each $m \geq 0$, the radius of convergence of $\sum\limits_{m=0}^\infty w_mT^m\in \cO(V)[[T]]$ is also greater than $1/t$.
\end{lem} 	
\begin{proof}  By condition (i) we may choose $r\in \sqrt{|K|^\times}$ such that $r<t$ and \[ \lim_{n\to \infty}|v_n|_V/r^n= 0.\] Thus we may define $L := \sup_{n \geq 0} |v_n|_V/r^n <\infty$ and $W:=V(r/f)$. Because $|f^{-1}|_{W} \leq 1/r$, we have $|v_nf^{-n}|_{W} \leq |v_n|_V /r^n \to 0$ as $n\to \infty$, so the series $\sum\limits_{n=0}^\infty v_nf^{-n}$ converges in $\cO(W)$ to an element $w$, say. Note that condition (ii) implies that $w$ vanishes on $W(t/f)$.

Using Lemma \ref{BerthMMP} we may choose $s\in \sqrt{|K|^\times}$ with $s<t$ such that whenever $g\in \cO(W)$ vanishes on $W(t/f)$, $g$ also vanishes on $W(s/f)$. By enlarging $s$ if necessary, we may also assume that $r < s$.  Thus, $w_{|W(s/f)} = 0$.

We cover $V$ by its two affinoid subdomains $V(f/s)$ and $V(s/f)$. Noting that $w_{|V(s/f)} = 0$ because $V(s/f) = W(s/f)$, we can estimate $|w_m|_{V(s/f)}$ as follows: \[ |w_m|_{V(s/f)}=\left|-\sum_{n>m}v_nf^{m-n}\right|_{V(s/f)} \leq \sup_{n> m}L r^{-n} s^{m-n} \leq Ls^m.\] 

To estimate $|w_m|_{V(f/s)}$, we note that $|v_n|_{V(f/s)} \leq |v_n|_V \leq L r^n$ for all $n \geq 0$, so 
\[ |w_m|_{V(f/s)}\leq \max_{0\leq n\leq m} |v_nf^{m-n}|_{V_1}\leq \max_{0\leq n\leq m}  Lr^{n} s^{m-n}\leq Ls^m \]
since $r<s$.  Therefore $|w_m|_V =\max \{|w_m|_{V(f/s)},|w_m|_{V(s/f)}\} \leq Ls^m$ and hence $\sum\limits_{m=0}^\infty w_mT^m$ has radius of convergence at least $1/s$. \end{proof} 	

We're now ready for our next local calculation, dealing with the critical case when $|\cS|=1$. Note that, in this case, if $X\in \bA(t\partial_x/\varpi)^\dag$ and $a_1\not\in X(K)$ then $X\subseteq U_t$ since if $X$ splits over $K'$ then $\rho(X_{K'})\geq t$. Thus, given Proposition \ref{YPres}, we might as well assume $a_1\in X(K)$. 

\begin{thm}\label{StdPres} Suppose that $X\in \bA(t\partial_x/\varpi)^\dag$ and $a_1\in X(K)$. Suppose also that $h=1$ and write $a:=a_1$ and $k=v_{a}(u)$ so that $u=(x-a)^k$. Then the complex
\[0 \to \cD^\dag_{\varpi/t}(X) \stackrel{\cdot R_\cS(u,d)}{\longrightarrow}\cD^\dag_{\varpi/t}(X) \stackrel{\cdot z}{\longrightarrow} \cM(\cS,u,d)_{\overline{U_t}}(X) \to 0\]
is exact whenever $\frac{k}{d} \notin \bN$. 
\end{thm}

\begin{proof} 
This is based on \cite[Proposition 5.1.2(ii)]{Berth1990}, which Berthelot attributes to Laumon. Recall that the sequence is a complex follows from Lemma \ref{OUrtuPres}. Note also that $R_\cS(u,d)=R(u,d)$ since $k\neq 0$. 

First we prove that the last non-zero map is surjective. To the end we consider a typical element $P=\sum\limits_{n=0}^\infty P_n\partial_x^n\in \cD^\dag_{\varpi/t}(X)$ with $P_n\in \cO(X)$. The required convergence condition is that there is some $r>\varpi/t$ such that $|P_n|r^n\to 0$ as $n\to \infty$. We furthermore compute that \begin{equation} P\cdot z=\sum\limits_{n=0}^\infty\binom{\frac{k}{d}}{n}n!P_n(x-a)^{-n} \hsp z. \label{Pz}\end{equation} 
	
Take an element $g \in \cM(\cS,u,d)(X)_{\overline{U_t}}$. Using $(\ref{OvConvMudSects}$), we can find some $s\in \sqrt{|K|^\times}$ with $s < t$ such that $g = f z$ for some $f \in \cO(X \cap U_{s})$. Since $\cO(X \cap U_s) = \cO(X) \langle s / (x-a) \rangle$, we can find a sequence $(b_n)$ in $\cO(X)$ such that
\[f=\sum\limits_{n=0}^\infty b_n (x-a)^{-n}  \qmb{with}  \lim\limits_{n \to \infty} |b_n|s^{-n}=0.\]
Thus to prove surjectivity, it suffices to show that if $|b_n|s^{-n}\to 0$ as $n\to \infty$ for some $s<t$, then $\left|\frac{b_n}{n!} \binom{\frac{k}{d}}{n}^{-1}\right| r^n\to 0$ as $n\to \infty$ for some $r>\varpi/t$.

Now as noted in \cite[Proposition 13.1.5, Corollary 13.1.7]{Kedlaya}, the rational number $\frac{k}{d}$ is $p$-adic non-Liouville and so the radius of convergence of the series
	\[\sum\limits_{n= 0}^\infty \frac{T^n}{n!}\binom{\frac{k}{d}}{n}^{-1}\]
 is at least $\varpi$; that is $\sup_{n\geq 0}\left|\frac{1}{n!}\binom{\frac{k}{d}}{n}^{-1}\right|w^n<\infty$ for all $w<\varpi$. Now, because $s < t$, we can choose $w$ such that $\varpi s / t < w < \varpi$. Then $r:=w/s>\varpi/t$ works.
	
	Next we prove the injectivity of the first non-zero map. From Definition \ref{TheRelator}(b) we can see that
\[ R(u,d) = (x-a)\partial_x - \frac{k}{d}.\]
It follows from this that for every $Q=\sum\limits_{n=0}^\infty Q_n\partial_x^n\in \cD^\dag_{\varpi/t}(X)$ we have \begin{equation} Q \cdot R(u,d)= \sum_{n= 0}^\infty \left((n-\frac{k}{d})Q_n+ (x-a)Q_{n-1}\right)\partial_x^n \label{QR} \end{equation}
and so if $Q\cdot R(u,d)=0$ then \[(n-\frac{k}{d})Q_n+ (x-a)Q_{n-1}=0 \qmb{for all} n \geq 0,\] where $a_{-1}:=0$. Since $\frac{k}{d} \notin \N$, an easy recurrence gives that $a_n=0$ for all $n$. 
	
	It remains to show that the kernel of $\cdot z$ is contained in the image of $\cdot R(u,d)$.  So we suppose that $P=\sum\limits_{n=0}^\infty P_n \partial_x^n\in \cD^\dag_{\varpi/t}(X)$ satisfies $P\cdot z=0$, and deduce using (\ref{Pz}) that
	\[\sum_{n=0}^\infty \binom{\frac{k}{d}}{n}n!P_n(x-a)^{-n}=0\in \cO(X)_{\overline{U_t}}\subset \cO\left(X\left(\frac{t}{x-a}\right)\right).\]
	
	Let $b_m :=\sum\limits_{n=0}^m\binom{\frac{k}{d}}{n}n!P_n(x-a)^{m-n}$. Because $\frac{k}{d} \notin \N$ by assumption, we can define
	\[Q:=\sum_{n=0}^\infty \frac{c_n}{n!}\partial_x^n \qmb{where}  c_n:=\left[\binom{\frac{k}{d}}{n}(n-\frac{k}{d}) \right]^{-1}  b_n \qmb{for all} n \geq 0. \] 
	
	We wish to show that $Q\in \cD^\dag_{\varpi/t}(X)$. Granted this, it is an easy computation, using (\ref{QR}), to check that $Q\cdot R(u,d)=P$ giving the result. 
	
	Since $P\in \cD^\dag_{\varpi/t}$, there is some $r>\varpi/t$ such that $|P_n|r^n\to 0$ as $n\to \infty$. Moreover, for any $s>\varpi$, $|n!|/s^n$ is bounded by Lemma \ref{nFacVarpi} and $\left|\binom{k/d}{n}\right|\leq 1$ for all $n$ since $p\nmid d$. Thus, choosing $s \in (\varpi,tr)$, we see that $\sum\limits_{n=0}^\infty\binom{k/d}{n}n!P_nT^n$ has radius of convergence at least $r/s>1/t$ in $\cO(X)[[T]]$. Now we may appeal to Lemma \ref{Berth421}, with $f=x-a$ and $v_n = \binom{k/d}{n}{n!}P_n$ to see that $\sum\limits_{m=0}^\infty b_mT^m$ has radius of convergence greater than $1/t$. Since $\frac{k}{d}$ is non-Liouville, $\sum\limits_{n=0}^\infty \frac{T^n}{\frac{k}{d}-n}$ has radius of convergence $1$ and as noted above $\sum\limits_{n=0}^\infty \frac{T^n}{n!}\binom{\frac{k}{d}}{n}^{-1}$ has radius of convergence $\varpi$. Combining all of these estimates, we see that there is some $u>\varpi/t$ such that $|\frac{c_n}{n!}|u^n\to 0$ as $n\to \infty$ as required. 
	 \end{proof}

Next, we collect several basic facts about the linear differential operator $R_\cS(u,d)$.

\begin{lem}\label{RudProps} \hsp 
	\be \item $R_\cS(u,d) = \Delta_\cS \theta_{u,d}(\partial_x)$.
	\item $R_\cS(u,d) = \theta_{u\Delta_\cS^d, d}(\partial_x)  \Delta_\cS$.
	\item Let $v \in K(x)$ with $\cS(v) \subset \cS$.  Then
	\[ R_\cS(uv, d) = \theta_{u,d}(R_\cS(v,d)).\]
	\ee\end{lem}
\begin{proof} (a) This follows directly from Definition \ref{TheRelator}(b) and Lemma \ref{zTwist}.
	
	(b) For any $v \in \cO(X)^\times$, $\theta_{v^d,d}$ is conjugation by $v$ in $\cD(X)$ by Lemma \ref{zTwist}. Applying Lemma \ref{TwistsMultiply} together with part (a), we see that
	\[\theta_{u \Delta_\cS^d, d}(\partial_x)  \Delta_\cS = \theta_{\Delta_\cS^d,d} \left(\theta_{u,d}(\partial_x)\right)  \Delta_\cS = \Delta_\cS \theta_{u,d}(\partial_x) = R_\cS(u,d).\]
	(c) We can compute 
	\[R_\cS(uv,d) = \Delta_\cS \theta_{uv,d}(\partial_x) = \Delta_\cS \theta_{u,d}(\theta_{v,d}(\partial_x)) = \theta_{u,d}(R_\cS(v,d))\]
	using part (a), Lemma \ref{TwistsMultiply} and the fact that $\theta_{u,d}$ is $\cO(X)$-linear.
\end{proof}

Here is our final local calculation.

\begin{prop}\label{BiPres} Let $Y\in \bA(t \partial_x / \varpi)^\dag$ containing at most one $a_i$. Then
	\[0 \to \cD^\dag_{\varpi/t}(Y) \stackrel{\cdot R_\cS(u,d)}{\longrightarrow}\cD^\dag_{\varpi/t}(Y) \stackrel{\cdot z}{\longrightarrow} \cM(\cS,u,d)_{\overline{U_t}}(Y) \to 0.\]
	is exact, provided that $v_{a_i}(u)/d \notin \bN$.
\end{prop}
\begin{proof} If none of the $a_i$ live in $Y$ then we're done by Proposition \ref{YPres}. So we may
	suppose $a_i\in Y(K)$ and define $a:=a_i$, $k:=v_{a_i}(u)$ and $v=u/(x-a)^{k}$ so that $v$ is a unit in $\cO(Y)$. In particular $\theta_{v,d}$ induces an automorphism of $\cD^\dag_{\varpi/t}(Y)$ by Corollary \ref{ThetaUDag}. Moreover by Lemma \ref{RudProps}(c), \[R_\cS(u,d)= \theta_{v,d}(R_\cS((x-a)^k,d)).\] Thus $\cD^{\dag}_{\varpi/t}(Y)R_\cS(u,d)=\cD^\dag_{\varpi/t}(Y)\theta_{v,d}(R_\cS((x-a)^k,d))$. Under $\theta^{-1}_{v,d}$ the $\cD^\dag_{\varpi/t}(Y)$-module $\cM(\cS,u,d)_{\overline{U_t}}(Y)$ corresponds to the $\cD^\dag_{\varpi/t}(Y)$-module $\cM(\cS,(x-a)^k,d)_{\overline{U_t}}(Y)$. Now we can apply Theorem \ref{StdPres}. 
	 \end{proof}


\begin{defn} For $i=1,\ldots,h$ let $W_i$ be the affinoid subdomain of $\bA$ whose $\mathbf{C}$-points are obtained by removing the open discs $\{z\in {\bf C}\mid |z-a_j|<t \}$ for $j\neq i$ from the closed disc $\{z\in {\bf C}\mid |z-a_i|\leq t \}$. 
\end{defn}
We note that the condition $t\leq\min\limits_{i \neq j}|a_i - a_j|$ means that each $W_i$ contains $a_i$ but does not contain $a_j$ for $i\neq j$. 

\begin{lem} \label{CheckCov} \hsp
	\be \item Each $W_i$ is $\A(t \partial_x / \varpi)^\dag$-admissible.
	\item Let $X \in \bA(t \partial_x/\varpi)^\dag$, and suppose that
	\[ t \leq \min \{ |a_i - a_j| : a_i,a_j \in X({\bf C}) \qmb{and} i \neq j \}.\]
Then $\{X \cap W_1,\ldots, X \cap W_h, X \cap U_t\}$ is an admissible covering of $X$.
	\ee\end{lem}
\begin{proof} (a) We choose a finite field extension $K'$ of $K$ such that $t \in |K'|$; then $W_i \times_K K'$ becomes a cheese such that $\rho(W_i \times_K K') = t$. Hence \[ r(W_i) = r(W_i \times_K K') = \varpi/t \] by Lemma \ref{BaseChR} and Corollary \ref{Heisenberg}.	

	(b) Applying Lemma \ref{BaseChR} and Corollary \ref{Heisenberg} again shows
\[\rho(X \times_K {{\bf C}}) = \varpi / r(X \times_K {\bf C}) \geq \frac{\varpi}{\varpi/t} = t.\]
Suppose that $a \in X({\bf C})$ and $|a - a_i| < t$ for some $i$. If $a_i \notin X({\bf C})$, then because $\rho(X \times_K {\bf{C}}) \geq t$, the intersection of $X({\bf C})$ with the open disc of radius $t$ around $a_i$ is empty. But $|a - a_i| <t$ and $a \in X({\bf C})$ --- contradiction --- so in fact $a_i \in X({\bf C})$ whenever $a \in X({\bf C})$ and $|a - a_i| < t$. We will now show that
\[X({\bf C}) = (X \cap U_t)({\bf C}) \cup (X \cap W_1)({\bf C}) \cup \cdots \cup (X \cap W_h({\bf C}).\]
Indeed, if $a \in X({\bf C}) - (X \cap U_t)({\bf C})$ then $a \in X(\bf C)$ and $|a - a_i| < t$ for some $i$, so that $a_i \in X({\bf C})$ by the above. But if also $|a - a_j| < t$ for some $j \neq i$, then $a_j \in X({\bf C})$ by the above as well, and then
\[|a_i - a_j| \leq \max \{|a - a_i|, |a - a_j|\} < t\]
contradicts our hypothesis on $t$. So in fact $|a - a_j| \geq t$ for all $j \neq i$ which forces $a \in (X \cap W_i)(\bf{C})$ and establishes our claim. Since  $\{X \cap U_t,  X \cap W_1,\ldots, X \cap W_h\}$ is a finite set of affinoid subdomains of $X$, which covers $X$ on ${\bf C}$-points, we conclude that it is an admissible covering of $X$.
\end{proof}

Here is the main result of $\S \ref{KummerSect}$.
\begin{thm}\label{SheafyPres} Let $X\in\A(t \partial_x/\varpi)^\dag$. Suppose that
\begin{itemize}
\item $t \in \sqrt{|K^\times|}$ and $p \nmid d$,
\item $t \leq \min \{ |a_i - a_j| : a_i,a_j \in X({\bf C}) \qmb{and} i \neq j \}$, and
\item $k_i/d\not\in \bN$ whenever $a_i \in X({\bf C})$.
\end{itemize}
Then the complex of $\cD^\dag_{\varpi/t}(X)$-modules
	\[0 \to \cD^\dag_{\varpi/t}(X) \stackrel{\cdot R_{\cS}(u,d)}{\longrightarrow}\cD^\dag_{\varpi/t}(X) \to \cM(\cS,u,d)_{\overline{U_t}}(X) \to 0\]
from Lemma \ref{OUrtuPres} is exact.
\end{thm}
\begin{proof}  Lemma \ref{CheckCov}(b) tells us that $\{X \cap U_t, X \cap W_1, \ldots, X \cap W_h\}$ is an $\A(t \partial_x / \varpi)^\dag$-admissible covering of $X$. By Theorem \ref{NCDagTate} and Theorem \ref{DagTate}, we know that $\cD^\dag_{\varpi/t}$ and $\cM(\cS,u,d)_{\overline{U_t}}$ are sheaves on $\A(t \partial_x/\varpi)^\dag$ with vanishing higher \v{C}ech cohomology. Therefore it suffices to show that the sequence
	\[0 \to \cD^\dag_{\varpi/t}(Y) \stackrel{\cdot R_\cS(u,d)}{\longrightarrow}\cD^\dag_{\varpi/t}(Y) \to \cM(\cS,u,d)_{\overline{U_t}}(Y) \to 0\]
	is exact whenever $Y$ is an intersection of any members of our covering; note that this $Y$ is still $(t \partial_x / \varpi)^\dag$-admissible. Since each member of our covering contains at most one $a_i$ the exactness follows from Proposition \ref{BiPres}.\end{proof}

Finally, we record a purely algebraic calculation which tells us how the the relators $R_{\cS(u)}(u,d)$ from Definition \ref{TheRelator} interact with the maps $g^{\cD} : \cD \to g^\ast \cD$ defining the $G$-equivariant structure on $\cD$. Recall from $\S \ref{AffTrans}$ that $\bB$ denotes the subgroup scheme $\{g \in \mathbb{GL}_2 : g_{21}=0\}$ of upper-triangular matrices in $\mathbb{GL}_2$. Recall also the isomorphism $g^\dag_r \colon  \cD_r^\dag\stackrel{\cong}{\longrightarrow}  g^\ast \cD_{\frac{r}{\varrho(g)}}^\dag$ from Corollary \ref{gDagTransform}(b).

\begin{lem}\label{ChiOfRelDag} Let $g \in \bB(K)$, let $W$ be an affinoid subdomain of $\bA$ and let $u \in K(x) \cap \cO(W)^\times$. 
\be \item We have $g \cdot R_{\cS(u)}(u,d) =  \varrho(g)^{1-|\cS(u)|} \hsp R_{\cS(g \cdot u)}(g \cdot u,d)$.
\item Suppose that $W \in \bA(\partial_x/r)^\dag$ for some $r \in \bR_{>0}$. Then
\[g^\dag_r(W)\left(R_{\cS(u)}(u,d)\right) =  \varrho(g)^{1-|\cS(u)|} \ R_{\cS(g \cdot u)}(g \cdot u,d).\]
\ee \end{lem}
\begin{proof} (a) By Lemma \ref{AutAOD}(e), we have $g \cdot (x-z) = \varrho(g)^{-1} (x - g \cdot z)$ for any $z \in \overline{K}$. Therefore $v_{g \cdot z}(g \cdot u) = v_z(u)$, so $\cS(g \cdot u) = g \cdot \cS(u)$ and 
\[g \cdot \Delta_u = \prod_{z \in \cS(u)} \varrho(g) (x - g \cdot z) = \varrho(g)^{-|\cS(u)|} \Delta_{g \cdot u}.\]
Because we also have $ g \cdot \partial_x = \varrho(g) \partial_x$ by Lemma \ref{AutAOD}(e), Lemma \ref{ThetaudG} gives
\[\begin{array}{lll} g \cdot R_{\cS(u)}(u,d) &=& g\cdot \Delta_u \hsp g \cdot \theta_{u,d}(\partial_x) = \varrho(g)^{-|\cS(u)|} \Delta_{g \cdot u} \hsp \theta_{g \cdot u,d}(g \cdot \partial_x) \\
&=& \varrho(g)^{1-|\cS(u)|} R_{\cS(g \cdot u)}(g \cdot u,d).   \end{array}\qedhere\]

(b) This follows from part (a) and the commutative square (\ref{DgjD}).
\end{proof}

\subsection{Line bundles with connection on the local Drinfeld space}\label{UpperHP} In $\S \ref{UpperHP}$, we work with the following data.

\begin{notn}\label{UpperHPhyp1} \hsp
\be \item $\Omega := \bP^{1,\an} - \bP^1(F)$ is the  the Drinfeld upper half plane,
\item $\Upsilon=\bD\cap \Omega$,
\item $j \colon \Upsilon \hookrightarrow \bD$ is the open embedding,
\item $[\sL] \in \PicCon( \Upsilon)[p']$,
\item $d$ is the order of $[\sL]$ in $\PicCon(\Upsilon)$. Thus $d$ is coprime to $p$.
\ee\end{notn}

We refer the reader to \cite[Definition 3.2.1]{ArdWad2023} for the terminology, and we fix a $\cD$-linear isomorphism $\psi \colon \sL^{\otimes d} \stackrel{\cong}{\longrightarrow} \cO_{\Upsilon}$.

\begin{defn}\label{UnCovering} Let $n \geq 0$ be an integer.

\be \item $V_n := \bA \backslash \bigcup\limits_{a \in \cO_F} \{|z - a| < |\pi_F|^n\}$, 
\item $\Upsilon_n:= V_n\cap \bD$,
\item $\sL_n := j_\ast(\sL_{\overline{\Upsilon_n}})$, and
\item $h_n := q^{n+1}$.
\ee\end{defn}

See Definition \ref{iUdagDefn} for the meaning of the sheaf $\sL_{\overline{\Upsilon_n}}$ on $\Upsilon$.

We note that if $\cS :=\{a_1,\ldots,a_{h_n}\} \subset \cO_F$ is any set of coset representatives for the group $\cO_F / \pi_F^{n+1} \cO_F$, then $V_n$ was denoted $U(\cS)_{|\pi_F|^n}$ in \S\ref{KummerSect}. 

We also note that $\Upsilon_n$ is the cheese obtained by removing $h_n$ open balls of radius $|\pi_F|^n$ from $\bD$ with centres at $a_1,\ldots,a_{h_n}$ and $\{\Upsilon_n\}_{n\geq 0}$ forms an admissible cover of $\Upsilon$ totally ordered by inclusion. 

\begin{lem}\label{jLlimit} The maps $j_\ast\sL\to \sL_n$ given by restriction induce an isomorphism $j_\ast \sL \cong \invlim \sL_n$ of $\cD$-modules on $\bD$.
\end{lem}
\begin{proof} Since $j_\ast$ is a right adjoint it commutes with limits, and so it suffices to show that the restriction maps induce an isomorphism $\sL\cong \invlim \sL_{\overline{\Upsilon_n}}$ of $\cD$-modules on $\Upsilon$. Since the $\Upsilon_n$ forms an admissible cover of $\Upsilon$ which is totally ordered by inclusion, by Proposition \ref{ApproxOnClosed}, the restriction maps induce such an isomorphism of abelian sheaves on $\Upsilon$. It is easy to see that all relevant maps are $\cD$-linear.
\end{proof}

Note that there is a natural bijection \[ h(\Upsilon_n)\backslash \{D_\infty\}\to \cO_F/\pi^{n+1}\cO_F \] sending $D$ to $D\cap \cO_F$ for each $D\in h(\Upsilon_n)\backslash\{D_\infty\}$. 

\begin{notn}\label{HoleNotn}For each $a\in \cO_F$ and $n\geq 0$, let $D_{a,n}$ be the element of $h(\Upsilon_n)\backslash\{D_\infty\}$ corresponding to $a+\pi_F^{n+1}\cO_F$ under this bijection. \end{notn}

We introduce the natural map abelian groups
\[ M_{n,d}\colon \PicCon(\Upsilon)[d]\to M_0(h(\Upsilon_n),\bZ/d\bZ)\]
obtained by composing the restriction map
\[\PicCon(\Upsilon)[d]\to \Con(\Upsilon_n)[d];\hspace{0.5cm} \sL\mapsto \sL|_{\Upsilon_n}\] 
with the isomorphisms $\theta_d$ and $\mu_{\Upsilon_n,d}$  from \cite[Definition 3.1.15 and Corollary 4.3.3]{ArdWad2023}:
\[\xymatrix{\Con(\Upsilon_n)[d] \ar[rr]_{\cong}^{\theta_d} && \frac{\cO(\Upsilon_n)^\times}{K^\times \cO(\Upsilon_n)^{\times d}} \ar[rr]_{\cong}^{\mu_{\Upsilon_n,d}} && M_0\left(h(\Upsilon_n),\frac{\bZ}{d\bZ} \right)}.\] 
For each $n \geq 0$, there is a unique $\bZ/d\bZ$-valued measure on $h(\Upsilon_n)$ of total value zero, whose value on each $D_{a,n}$ is $1$:
\[ \nu_n(D_{a,n}) = 1 + d \bZ \qmb{for all} a \in \cO_F, \qmb{and} \nu_n(D_\infty) = -q^{n+1} + d \bZ.\]

\begin{lem} \label{UpsMeasures}  Let $n \geq 0$.
	\be
	\item $M_0(h(\Upsilon_n), \bZ/d\bZ)^I$ is freely generated as a $\bZ/d\bZ$-module by $\nu_n$. 
	\item The restriction map $r_{n+1} : M_0(h(\Upsilon_{n+1}),\bZ/d\bZ)^I \to M_0(h(\Upsilon_n),\bZ/d\bZ)^I$ sends $\nu_{n+1}$ to $q\nu_n$.
	\item $r_{n+1}$ is an isomorphism of $\bZ/d\bZ$-modules.
	\item The restriction map $\PicCon(\Upsilon)^I[d] \to \Con(\Upsilon_n)^I[d]$ is an isomorphism.
	\item $\PicCon(\Upsilon)^I[d]$ is a free $\bZ/d\bZ$-module of rank $1$.
	\ee 
\end{lem}
\begin{proof} (a) The natural $I$-action on $\cO_F/\pi_F^{n+1}\cO_F\cup\{\infty\}$ has two orbits: $\cO_F / \pi_F^{n+1}\cO_F$ and $\{\infty\}$. Therefore the natural $I$-action on $h(\Upsilon_n)$ also has two orbits, namely $\{D_{b,n} : b \in \cO_F\}$ and $\{D_\infty\}$. If $\nu$ is an $I$-invariant measure on $h(\Upsilon_n)$, then it must be constant on these orbits. If, in addition, it has total value zero, then we easily see that $\nu = \nu(D_{0,n}) \nu_n$ as required.
	
	(b) Fix $b \in \cO_F$. In view of part (a), the map $r_{n+1}$ must send $\nu_{n+1}$ to some integer multiple of $\nu_n$. However the restriction map $h(\Upsilon_{n+1}) \to h(\Upsilon_n)$ sends $D_{b + \pi_F^{n+2} c, n+1}$ to $D_{b, n}$ for each $c \in \cO_F$, and the preimage of $D_{b,n}$ under this map consists of precisely $q$ such holes.  Hence $r_{n+1}(\nu_{n+1}) = q \nu_n$.
	
	(c) This follows immediately from parts (a) and (b) since $d$ and $q$ are coprime.
	
	(d) For each $m \geq n$ there is a natural commutative square
	\[\xymatrix{ \Con(\Upsilon_{m+1})^I \ar[rr]\ar[d]^\cong_{\mu_{\Upsilon_{m+1},d}\circ\theta_d}&& \Con(\Upsilon_m)^I \ar[d]_\cong^{\mu_{\Upsilon_m,d}\circ\theta_d} \\ M_0(h(\Upsilon_{m+1}), \bZ/d\bZ)^I \ar[rr]_{r_{m+1}} && M_0(h(\Upsilon_m), \bZ/d\bZ)^I }\]
	whose vertical arrows are isomorphisms by \cite[Corollary 4.3.4]{ArdWad2023}. Using part (b) we see that the top horizontal arrow is bijective. Hence the connecting maps in the inverse system $(\Con(\Upsilon_m)^I)_{m \geq n}$ are all isomorphisms, so the projection map
	\[\invlim \Con(\Upsilon_m)^I \to \Con(\Upsilon_n)^I\]
	is an isomorphism as well. Next, we have $\PicCon(\Upsilon_m) = \Con(\Upsilon_m)$ by \cite[Proposition 4.1.11]{ArdWad2023}. The family $(\Upsilon_m)_{m \geq n}$ forms an increasing admissible covering of $\Upsilon$ by geometrically connected affinoid subdomains, so \cite[Proposition 3.1.9]{ArdWad2023} implies that the restriction map
	\[ \PicCon(\Upsilon)^I  \longrightarrow \invlim \PicCon(\Upsilon_m)^I \]
	is an isomorphism. The composition of this restriction map with the projection $\PicCon(\Upsilon_m)^I \to \PicCon(\Upsilon_n)^I$ is the restriction map $\PicCon(\Upsilon)^I \to \PicCon(\Upsilon_n)^I$ in the statement of (d), and it is therefore an isomorphism as required.
	
	(e) Use part (d), \cite[Corollary 4.3.4]{ArdWad2023} and part (a).
\end{proof}
\begin{defn}\label{AlgGen} Let $X$ be an affinoid subdomain of $\bD$. A section $\dot{z} \in \sL_n(X)$ is said to be an \emph{algebraic generator} if there exists $u \in K(x)$ such that
\be \item $\sL_{n|X} = (\cO_X)_{\overline{X \cap V_n}} \dot{z}$, 
\item $\psi(\dot{z}^{\otimes d}) = u$,
\item $\cS(u) \subset F$ (see Definition \ref{ZeroesPoles} for this notation),
\item $|a-b|\geq |\pi_F|^n$ for all $a\neq b\in \cS(u) \cap X$, and
\item $\frac{v_{a}(u)}{d}\not \in \bN$ for all $a\in \cS(u) \cap X$.
\ee
We call $u$ the \emph{associated rational function}. 
\end{defn}
 
We will now prove that algebraic generators exist. 

\begin{thm}\label{AlgGensExist} Let $X$ be an affinoid subdomain of $\bD$.
\be \item Let $\cS\subset \cO_F$ be any set of coset representatives for the group $\cO_F/\pi^{n+1}\cO_F$ and let $(k_a)_{a\in \cS}$ be a family of integers such that for all $a \in \cS$ we have
\begin{equation} \label{MuThetaL} \frac{k_a}{d} \notin \bN \qmb{and} M_{n,d}([\sL])(D_{a,n})={k_a}+d\Z. \end{equation}
Then there exists an algebraic generator $\dot{z} \in \sL_n(X)$ with associated rational function of the form
\[ u = \lambda \prod_{a \in \cS} (x-a)^{k_a} \qmb{for some} \lambda \in K^\times.\] 
\item Algebraic generators always exist. 
\ee
\end{thm}
\begin{proof} (a) Using Lemma \ref{iDagLocal} we can see that if $\dot{z}\in \sL_n(\bD)$ is algebraic generator then $\dot{z}|_X$ is an algebraic generator with the same associated rational function and so we only need to consider the case $X = \bD$ itself. 

Let $u_0 := \prod\limits_{a \in \cS} (x-a)^{k_a} \in \cO(\Upsilon_n)^\times$. Then the assumption (\ref{MuThetaL}) tells us that
\[ \mu_{\Upsilon_n,d}( u_0 K^\times \cO(\Upsilon_n)^{\times d} ) = M_{n,d}([\sL|) = \mu_{\Upsilon_n,d}(\theta_d(\sL|_{\Upsilon_n})).\]
Since $\mu_{\Upsilon_n,d}$ is an isomorphism, we see that 
\begin{equation} \label{thetadLUpsnn} \theta_d(\sL|_{\Upsilon_n}) = u_0 K^\times \cO(\Upsilon_n)^{\times d}.\end{equation}
Write $\cS=\{a_1,\ldots, a_{h_n}\}$, choose an integer $\ell$ coprime to $p$ and define
\[V  := \Sp K \left\langle x, \frac{\pi_F^{\ell n+1}}{(x - a_1)^\ell}, \cdots, \frac{\pi_F^{\ell n+1}}{(x - a_{h_n})^\ell} \right\rangle.\]
This is a $K$-form of the $\bfC$-cheese $V_{\bfC}$ obtained from $\bD_{\bfC}$ by removing the open balls of radius $|\pi_F|^{n + \frac{1}{\ell}}$ around the $a_i$. Observe that $V$ is a wide open neighbourhood of $\Upsilon_n$ contained in $\Upsilon_{n+1}$. Since $V$ and $\Upsilon_n$ are geometrically connected, we have the following commutative square whose vertical arrows are given by restriction:
\[\xymatrix{ \Con(V)[d] \ar[rrrr]^{\theta_d}_{\cong}\ar[d] &&&& \cO(V)^\times/K^\times\cO(V)^{\times d}\ar[d]^{r_n} \\
 \Con(\Upsilon_n)[d] \ar[rrrr]_{\theta_d}^{\cong} &&&& \cO(\Upsilon_n)^\times/K^\times\cO(\Upsilon_n)^{\times d}. }\]
We claim that $r_n$ is injective. Assuming this, and noting that $u_0$ also lies in $\cO(V)^\times$, it follows from (\ref{thetadLUpsnn}) that $\theta_d([\sL|_V]) = u_0 K^\times \cO(V)^{\times d}$ inside $\cO(V)^\times/K^\times\cO(V)^{\times d}$. Using the construction of $\theta_d$ given at \cite[Proposition 3.1.14]{ArdWad2023}, we can find a section $\dot{z} \in \sL(V)$ and some $\lambda \in K^\times$ such that $\sL(V) = \cO(V) \dot{z}$ and $\psi(\dot{z}^{\otimes d}) = u := \lambda u_0$. Since $V$ is affinoid and $\sL$ is a coherent $\cO$-module, in fact we also have $\sL_{|V} = \cO_{|V} \dot{z}$. Since $V$ is a wide open neighbourhood of $\Upsilon_n$, i.e. $\widetilde{V}$ contains the closure of $\widetilde{\Upsilon_n}$, we may view $\dot{z}$ as an element of $\sL_n(\bD)$ and deduce that $\sL_{n|\bD} = (\cO_\bD)_{\overline{\Upsilon_n}} \dot{z}$. 

Finally, $v_a(u) = k_a$ for all $a \in \cS$. As none of the $k_a$ can be zero, given the assumption that $\frac{k_a}{d} \notin \bN$ for all $a \in \cS$,  we have $\cS(u)=\cS$ and this is contained in $\cO_F$. We now see that Definition \ref{AlgGen}(c,d,e) is satisfied and $\dot{z} \in \sL_n(\bD)$ is an algebraic generator, once we have checked the injectivity of $r_n$ claimed above.

Let $L = K(\pi_F^{\frac{1}{\ell}})$. Then in the notation of \cite[Definition 4.1.1]{ArdWad2023}, $V_L$ is isomorphic to the $L$-cheese $C_L(0, a_1,\cdots,a_{h_n}, 1, \pi_F^{n+\frac{1}{\ell}}, \cdots, \pi_F^{n+\frac{1}{\ell}})$ as an $L$-affinoid variety. The map $r_n$ appears in the following commutative diagram:
\[\xymatrix{ \cO(V)^\times/K^\times\cO(V)^{\times d} \ar[r]\ar[d]_{r_n} & \cO(V_L)^\times/L^\times\cO(V_L)^{\times d} \ar[r]^(0.55){\mu_{V_L,d}} & M_0(h(V_L), \frac{\bZ}{d\bZ}) \ar[d] \\
\cO(\Upsilon_n)^\times/K^\times\cO(\Upsilon_n)^{\times d} \ar[r] & \cO(\Upsilon_{n,L})^\times/L^\times\cO(\Upsilon_{n,L})^{\times d} \ar[r]_(0.55){\mu_{\Upsilon_{n,L},d}} & M_0(h(\Upsilon_{n,L}), \frac{\bZ}{d\bZ}). 
}\]
Since $|h(V_L)| = |h(\Upsilon_{n,L})| = h_n$, the vertical arrow on the right is bijective. Since $\mu_{V_L,d}$ is an isomorphism by \cite[Corollary 4.3.3(c)]{ArdWad2023}, it remains to show that the first horizontal arrow in the top row is injective, or equivalently, that
\[ L^\times \cO(V_L)^{\times d} \hsp \cap \hsp \cO(V)^\times \subseteq K^\times \cO(V)^{\times d}.\]
Let $S$ denote the subgroup of $\mathcal{O}(V)^\times$ generated by $x-a_1,\cdots,x-a_{h_n}$. View $S$ also as a subgroup of $\cO(V_L)^\times$. Applying \cite[Proposition 4.1.10]{ArdWad2023} to $V_L$, we have
\[ \cO(V_L)^\times = \cO(V_L)^{\times\times} \cdot L^\times \cdot S.\] 
Now $L^\times =L^{\times\times}\cdot K^\times \cdot \pi_F^{\frac{1}{\ell}\bZ}$ because $L = K(\pi_F^{\frac{1}{\ell}})$ is a totally ramified extension of $K$. Using \cite[Lemma 4.3.2(a)]{ArdWad2023} we then see that 
\[ L^\times \cO(V_L)^{\times d} = \pi^{\frac{d}{\ell}\bZ} \cdot \cO(V_L)^{\times\times} \cdot K^{\times d} \cdot S^d.\] 
Since $K^{\times d} S^d \subseteq K^\times \cO(V)^{\times d}$, by the modular law it remains to show that
\[ \pi_F^{\frac{d}{\ell}\bZ} \cO(V_L)^{\times\times} \cap \cO(V)^\times \subseteq K^\times \cO(V)^{\times d}.\]
Let $\zeta \in \overline{K}$ be a primitive $\ell^{th}$ root of unity and let $M := L(\zeta)$, a Galois extension of $K$ with Galois group $\cG$. Let $u \in \pi_F^{\frac{d}{\ell}\bZ} \cO(V_L)^{\times\times} \cap \cO(V)^\times$. Then $u$ is fixed by the natural action of $\cG$ on $\cO(V_M)$ and we can write $u = \pi_F^{\frac{m}{\ell}} (1 + \epsilon) $ for some $m \in d\bZ$ and $\epsilon \in \cO(V_L)^{\circ\circ}$. Choose $\sigma \in \cG$ which sends $\pi_F^{\frac{1}{\ell}}$ to $\zeta \pi_F^{\frac{1}{\ell}}$; then
\[ \sigma(\pi_F^{\frac{m}{\ell}}) (1 + \sigma(\epsilon)) = \pi_F^{\frac{m}{\ell}}(1 + \epsilon), \qmb{so} \zeta^m = \frac{\sigma(\pi_F^{\frac{m}{\ell}})}{\pi_F^{\frac{m}{\ell}}} = \frac{1 + \epsilon}{1 + \sigma(\epsilon)} \in \cO(V_M)^{\times\times}.\]
Since $p \neq \ell$ and $(\zeta^m)^\ell = 1$, it follows from \cite[Lemma 4.3.2(a)]{ArdWad2023}  that $\zeta^m = 1$. We conclude that $\ell \mid m$. Therefore $\pi_F^{\frac{m}{\ell}} \in K^\times$ and $\epsilon = \pi_F^{-\frac{m}{\ell}} u - 1 \in \cO(V)$. Since $\epsilon \in \cO(V_L)^{\circ\circ}$, we see that $\epsilon \in \cO(V)^{\circ\circ}$ and $u \in K^\times \cO(V)^{\times\times}$. Applying \cite[Lemma 4.3.2(a)]{ArdWad2023} again shows that $u \in K^\times \cO(V)^{\times d}$ and completes the proof.

(b) Choose any set $\cS$ of coset representatives for $\pi_F^{n+1} \cO_F$ in $\cO_F$, and choose any family $(\ell_a)_{a \in \cS}$ of integers such that $M_{n,d}([\sL])(a+\pi_F^{n+1}\cO_F)=\frac{\ell_a}{d}+\Z$ for all $a \in \cS$. Let $M$ be any positive integer such that $M > \max_{a \in \cS} \frac{\ell_a}{d}$ and let $k_a := \ell_a - Md < 0$ for all $a \in \cS$. Then $\frac{k_a}{d} \notin \bN$ and $\frac{\ell_a}{d} + \Z = \frac{k_a}{d} + \Z$ for all $a \in \cS$. We can now apply part (a) to the family of integers $(k_a)_{a \in \cS}$ to conclude.\end{proof}



 We recall the Definition \ref{UstMud} of $\cM(\cS,u,d)$ for any $u\in K(x)$ with $\cS(u)\subset \cS$. 

\begin{prop}\label{CleanLn} Let $X$ be an affinoid subdomain of $\bD$ and $\cS_n$ be a set of coset representatives for the group $\cO_F/\pi_F^{n+1}\cO_F$. Moreover, let $\dot{z}\in\sL_n(X)$ be an algebraic generator, with associated rational function $u_n$ such that $\cS(u_n)\subset \cS_n$.  There is a $\cD_X$-linear isomorphism
\[ \left(\cM(\cS_n, u_n, d)_{\overline{V_n}}\right)_{|X} \stackrel{\cong}{\longrightarrow} \sL_{n|X}\]
which sends $z$ to $\dot{z}$.
\end{prop}
\begin{proof} By Corollary \ref{iUjOSect} applied with $t=|\pi_F|^n$, and by Lemma \ref{iDagLocal}, we have
\[ \left(\cM(\cS_n,u_n,d)_{\overline{V_n}}\right)|_X = \left((\cO_ \bA)_{\overline{V_n}}\right)|_Xz = (\cO_X)_{\overline{V_n\cap X}} z.  \]
By Definition \ref{AlgGen}(a) we know that $\sL_{n|X} = (\cO_X)_{\overline{X \cap V_n}}\dot{z}$. Hence sending $z$ to $\dot{z}$ defines an $\cO_X$-linear isomorphism 
\[\left(\cM(\cS_n,u_n,d)_{\overline{V_n}}\right)_{|X} \stackrel{\cong}{\longrightarrow} \sL_{n|X}.\]  
The fact that $\psi(\dot{z}^{\otimes d}) = u_n$ implies that this isomorphism is in fact $\cD$-linear. 
\end{proof}

We are now in the situation of $\S \ref{KummerSect}$ and can begin to reap the rewards.

\begin{defn}\label{CurlyDn} Let $n \geq 0$.
\be \item $\bD_n$ will denote the $G$-topology $\bD_n := \bD(|\pi_F|^n \partial_x / \varpi)^\dag$ on $\bD$.
\item $\sD_n$ will denote the sheaf $\cD^\dag_{\varpi/|\pi_F|^n}$ on $\bD_n$. 
\ee
\end{defn}

In view of Corollary \ref{Heisenberg} and Lemma \ref{BaseChR}, we see that an affinoid subdomain $X$ of $\bD$ lies in $\bD_n$ if and only if $\rho(X_\mathbf{C}) \geq |\pi_F|^n$. We note the following flatness properties that follow from work done in \S\ref{FlatnessChapter}.

\begin{lem}\label{Flat} \hsp
	\be \item $\sD_n(X)$ is a flat right $\sD_{n+1}(X)$-module whenever $X \in \bD_{n+1}$.
	\item $\sD_n(Y)$ is a flat right $\sD_n(X)$-module whenever $Y \subseteq X$ in $\bD_n$.
	\ee
\end{lem}
\begin{proof} This is a consequence of Theorem \ref{FlatThm}.
\end{proof}

\begin{cor}\label{LnDn} The natural action of the sheaf of finite order differential operators $\cD$ on the restriction of $\sL_n$ to $\bD_n$ extends to $\sD_n$.
\end{cor}
\begin{proof} Let $\cS_n$ be any set of coset representatives for the group $\cO_F/\pi^{n+1}\cO_F$. Theorem \ref{AlgGensExist} and Proposition \ref{CleanLn} together give us a $\cD_{\bD}$-linear isomorphism 
\[\cM(\cS_n, u_n, d)_{\overline{V_n}} \stackrel{\cong}{\longrightarrow} \sL_n.\] 
We can now apply Proposition \ref{OUrtuMod} with $t = |\pi_F|^n$ to transport the $\sD_n$-action on $\cM(\cS_n, u_n, d)_{\overline{V_n}}$ to $\sL_n$ along this isomorphism. \end{proof}


Recall the relator $R_{\cS(u)}(u,d)$ from Definition \ref{TheRelator}(b).
\begin{cor}\label{MnPres} Let $X \in \bD_n$ and let $\dot{z}\in\sL_n(X)$ be an algebraic generator, with associated rational function $u_n$. Then $\sL_n(X)$ is a finitely presented $\sD_n(X)$-module:
\[\sL_n(X) = \sD_n(X) \cdot \dot{z} \cong \sD_n(X) / \sD_n(X) R_{\cS(u_n)}(u_n,d).\]
\end{cor}
\begin{proof} By Proposition \ref{CleanLn}, there is a $\cD_X$-linear isomorphism 
	\[ \left(\cM(\cS(u_n), u_n, d)_{\overline{V_n}}\right)_{|X} \stackrel{\cong}{\longrightarrow} (\sL_n)_{|X}\]
sending $z$ to $\dot{z}$. Note that $|a - b| \geq |\pi_F|^n$ for all $a \neq b \in \cS(u_n) \cap X$ by Definition \ref{AlgGen}(d) and $\frac{v_a(u_n)}{d} \notin \bN$ for all $a \in \cS(u_n) \cap X$ by Definition \ref{AlgGen}(e). Now we may apply Theorem \ref{SheafyPres} with $t = |\pi_F|^n$.
\end{proof}

\subsection{Local irreducibility} \label{LocIrred}
We assume throughout $\S \ref{LocIrred}$ that our  ground field $K$ is discretely valued, with uniformiser $\pi_K$ and residue field $k$. Let $X$ be an affinoid subdomain of the closed unit disc $\bD$ obtained by removing finitely many open discs of radius $1$ \emph{not} containing the point $x = 0$. Thus
\[ X = \Sp K \left\langle x, \frac{1}{x - \alpha_1}, \cdots, \frac{1}{x - \alpha_v}\right\rangle\]
for some $\alpha_1,\ldots, \alpha_v \in K$ with $|\alpha_1| = \cdots = |\alpha_v| = 1$. We also consider the case where $v = 0$ when $X = \bD$. The main result of $\S \ref{LocIrred}$ is the following

\begin{thm}\label{SimpleKIsoc} Let $\lambda \in K \backslash \Z$ and let $X$ be as above. Then the $\cD^\dag_{\varpi}(X)$-module $\cD^\dag_{\varpi}(X) / \cD^\dag_{\varpi}(X)(x \partial - \lambda)$ is  simple if $\lambda \in \Zp$, and zero otherwise.\end{thm}
Let $\sX := \Spf \cO(X)^\circ$ and note that $\sX$ is the formal completion of the punctured affine line $\Spec K^\circ \left[x, \frac{1}{x-\alpha_1},\cdots, \frac{1}{x - \alpha_d}\right]$ along its special fibre. Recall Berthelot's ring $\h{\sD}^{(m)}_{\sX,\bQ}(\sX)$ of level-$m$ arithmetic  differential operators on $\sX$ for some fixed integer $m \geq 0$ from \cite[p. 46, (2.4.1.4)]{Berth}.

\begin{defn} For each $\lambda \in K$, define 
\[\sM^{(m)}(\lambda) := \frac{\h{\sD}^{(m)}_{\sX,\bQ}}{\h{\sD}^{(m)}_{\sX,\bQ} (x \partial - \lambda)} \in \coh(\h{\sD}^{(m)}_{\sX,\bQ}).\]
\end{defn}

We will actually prove the following stronger statement.

\begin{thm}\label{SimpleDmModule} Let $\lambda \in K \backslash \Z$, let $\sX$ be as above and let $m \geq 0$. Then the $\h{\sD}^{(m)}_{\sX,\bQ}(\sX)$-module $\sM^{(m)}(\lambda)(\sX)$ is simple if $\lambda \in \bigcup\limits_{i=0}^{p^m-1} i + p^m K^\circ$, and zero otherwise.
\end{thm}
The following elementary statement may help to explain the sets appearing in Theorem \ref{SimpleDmModule}.
\begin{lem}\label{ZpBalls} We have $\Zp = \bigcap\limits_{m=0}^\infty \bigcup\limits_{i=0}^{p^m-1} i + p^m K^\circ$.
\end{lem}
\begin{proof} Certainly the forward inclusion holds. On the other hand, if $\lambda \notin \Zp$ then because $\Zp$ is compact and $K$ is Hausdorff, we can find $m \geq 0$ sufficiently large so that $|\lambda - \alpha| > |p^m|$ for all $\alpha \in \Zp$. Then $\lambda \notin \bigcup\limits_{i=0}^{p^m-1} i + p^m K^\circ$.
\end{proof}

\begin{proof}[Proof of Theorem \ref{SimpleKIsoc}] Write $D := \cD^{\dag}_\varpi(X)$ and $D_m := \h{\sD}^{(m)}_{\sX,\bQ}(\sX)$ for each $m \geq 0$. Let $r_m := |p^m!|^{\frac{1}{p^m}}$; then by Theorem \ref{Adellevm} there is a (not necessarily isometric) isomorphism of $K$-Banach algebras $D_m \cong \cD_{r_m}(X)$. Because the sequence of real numbers $r_m$ approaches $\varpi$ as $m \to \infty$, it follows from Definition \ref{DagSite}(d) that $D = \colim\limits_{m \geq 0} D_m$. 

Suppose that $\lambda \notin \Zp$. Then by Lemma \ref{ZpBalls}, $\lambda \notin \bigcup\limits_{i=0}^{p^m-1} i + p^m K^\circ$ for some $m \geq 0$, and then $M_m := \sM^{(m)}(\lambda)(\sX)$ is zero by Theorem \ref{SimpleDmModule}. Therefore $D / D(x \partial - \lambda) = D \otimes _{D_m} M_m$ is zero as well. 

Now suppose that $\lambda \in \Zp$.  Since colimits are exact, the $D$-module $M := D / D(x\partial - \lambda)$ is the colimit of the $D_m$-modules $M_m = D_m / D_m (x \partial - \lambda)$. Note that the image of the canonical generator $v_m$  of $M_m$ inside $M$ equals the canonical generator $v$ of $M$. Now given any non-zero $w \in M$, because $M = \colim\limits_{m \geq 0} M_m$ we can find $m \geq 0$ such that $w$ is the image in $M$ of some non-zero $w_m \in M_m$. Since $\lambda \in \Zp$,  $M_m$ is a simple $D_m$-module by Theorem \ref{SimpleDmModule} and Lemma \ref{ZpBalls}, we can find $Q_m \in D$ such that $Q_m \cdot w_m = v_m$ in $M_m$. Hence $Q_m \cdot w = v$ in $M$. So any non-zero $D$-submodule of $M$ contains $v$ and is hence equal to $M$. \end{proof}

We will now use Berthelot's theory of \emph{Frobenius descent} to reduce Theorem \ref{SimpleDmModule} to the case $m = 0$. Let $q := p^m$, let $x'$ be a new local coordinate and let $\sX'$ be the formal completion of the punctured affine line $\Spec K^\circ \left[x', \frac{1}{x'-\alpha_1^q},\cdots, \frac{1}{x' - \alpha_d^q}\right]$ along its special fibre. There is a natural lift of the relative Frobenius morphism
\[ F : \sX \to \sX'\]
which sends $x' \in \cO_{\sX'}$ to $x^q \in \cO_{\sX}$ and which is completely determined by this property as a morphism of formal $K^\circ$-schemes. In this situation, Berthelot's Frobenius descent theorem, \cite[Th\'eor\`eme 2.3.6]{Berth3}, tells us that the functors of \emph{Frobenius pullback} and \emph{Frobenius descent}
\[ \begin{array}{llllll} F^\ast &:& \coh(\sD^{(0)}_{\sX'})  &\longrightarrow & \coh(\sD^{(m)}_{\sX}), \quad &\sN \mapsto \cO_{\sX} \underset{\cO_{\sX'}}{\otimes}{} \sN, \qmb{and} \\
F^\natural &:& \coh(\sD^{(m)}_{\sX}) &\longrightarrow&  \coh(\sD^{(0)}_{\sX'}), \quad& \sM \mapsto \left(\sD^{(0)}_{\sX'} \underset{\cO_{\sX'}}{\otimes}{} \cO_{\sX}^\vee\right) \underset{\sD^{(m)}_{\sX}}{\otimes}{} \sM\end{array}\]
are mutually inverse equivalences of categories, where $\cO_{\sX}^\vee := \mathpzc{Hom}_{\cO_{\sX'}}(\cO_{\sX}, \cO_{\sX'})$. These equivalences extend in a natural way to $p$-adic completions
\[ F^\natural : \coh(\h{\sD}^{(m)}_{\sX}) \stackrel{\cong}{\longrightarrow} \coh(\h{\sD}^{(0)}_{\sX'}) \qmb{and} F^\natural : \coh(\h{\sD}^{(m)}_{\sX,\bQ}) \stackrel{\cong}{\longrightarrow} \coh(\h{\sD}^{(0)}_{\sX',\bQ})\]
which, abusing notation, we will denote by the same symbols. We are now in a position to be able to compute the Frobenius descent of $\sM^{(m)}(\lambda)$.

\begin{prop}\label{KummerFrobDesc}Let $\lambda \in K$ and set $\lambda_i := \frac{\lambda - i}{q}$ for each $i=0,\ldots,q-1$. Then
\[ F^\natural(\sM^{(m)}(\lambda)) \cong \bigoplus_{i=0}^{q-1} \frac{\h{\sD}^{(0)}_{\sX',\bQ}}{\h{\sD}^{(0)}_{\sX',\bQ} (x' \partial' - \lambda_i)}\]
where $\partial' \in \cT_{\sX'}$ is the unique derivation such that $\partial'(x') = 1$.
\end{prop}
\begin{proof} In \cite[Proposition 4.1.2(ii)]{Berth3}, Berthelot established a canonical isomorphism of $\h{\sD}^{(m)}_{\sX}$-bimodules
\[ \Phi : F^\ast F^\flat \h{\sD}^{(0)}_{\sX'} := \cO_{\sX} \underset{\cO_{\sX'}}{\otimes}{} \h{\sD}^{(0)}_{\sX'} \underset{\cO_{\sX'}}{\otimes}{} \cO_{\sX}^\vee \stackrel{\cong}{\longrightarrow} \h{\sD}^{(m)}_{\sX}.\]
We will use the following explicit description of $\Phi$ due to Garnier \cite{Garnier}; we will follow the clear exposition of Garnier's results found at \cite[\S 2.2]{AbeExplFrob}. Note that $\cO_{\sX}$ is a free $\cO_{\sX'}$-module with basis $\{1, x, \cdots, x^{q-1}\}$; let $\{\theta_0, \theta_1, \cdots, \theta_{q-1}\}$ be the dual basis for $\cO_{\sX}^\vee$ so that 
\[\cO_{\sX} \underset{\cO_{\sX'}}{\otimes}{} \h{\sD}^{(0)}_{\sX'} \underset{\cO_{\sX'}}{\otimes}{} \cO_{\sX}^\vee  = \bigoplus_{i,j=0}^{q-1} x^i \otimes \h{\sD}^{(0)}_{\sX'} \otimes \theta_j.\]
Then the isomorphism $\Phi$ is defined as follows:
\begin{itemize}
\item $\Phi( x^i \otimes Q \otimes \theta_j) = x^i Q^\circ H x^{-j}$ for any $Q \in \h{\sD}^{(0)}_{\sX'}$ and $0\leq i,j < q$, where
\item $H := \frac{1}{q}\sum\limits_{\zeta^q = 1} \sum\limits_{k=0}^\infty (\zeta - 1)^k x^k \partial^{[k]}\in \h{\sD}^{(m)}_{\sX}x^{q-1}$ is the \emph{Dwork operator}, and
\item $Q \mapsto Q^\circ$ is the non-unital ring homomorphism $\h{\sD}^{(0)}_{\sX'} \hookrightarrow \h{\sD}^{(m)}_{\sX}$ determined by the rules $1^\circ = H, (x')^\circ = x^q H$ and $(\partial')^{\circ} = (q x^{q-1})^{-1} \partial H$.
\end{itemize}
Note that $(x' \partial')^\circ = x^q (q x^{q-1})^{-1} \partial H = \frac{1}{q} x \partial H$. Therefore
\[ \Phi( x^i \otimes (x' \partial' - \lambda_i) \otimes \theta_i) = x^i \left(\frac{1}{q} x \partial H - \frac{\lambda - i}{q} H \right) H x^{-i} = \frac{1}{q}(x \partial - \lambda) x^i H x^{-i}\]
for any $i=0,\ldots, q-1$, where we used the relations $x^i(x \partial) = (x\partial - i) x^i$ and $H^2 = H$. Because $\sum\limits_{i=0}^{q-1} x^i H x^{-i} = 1$ in $\h{\sD}^{(m)}_{\sX}$ by \cite[Proposition 2.5.1]{Garnier}, we see that
\begin{equation}\label{PhiInvxdl} \Phi^{-1}\left( \frac{x \partial - \lambda}{q} \right) = \sum\limits_{i=0}^{q-1} x^i \otimes (x' \partial' - \lambda_i) \otimes \theta_i.\end{equation}
By the Frobenius descent theorem, \cite[Th\'eor\`eme 2.3.6]{Berth3},
\[F^\natural(\h{\sD}^{(m)}_{\sX,\bQ}) \quad \underset{\cong}{\overset{F^\natural(\Phi^{-1})}{{\longrightarrow}}} \quad F^\natural F^\ast F^\flat \h{\sD}^{(0)}_{\sX'} \quad \cong \quad F^\flat(\h{\sD}^{(0)}_{\sX',\bQ}) = \bigoplus\limits_{j=0}^{q-1} \h{\sD}^{(0)}_{\sX',\bQ} \otimes \theta_j\]
which is a free left $\h{\sD}^{(0)}_{\sX',\bQ}$-module of rank $q$. By definition, $\sM^{(m)}(\lambda)$ is the cokernel of right-multiplication by $\frac{x \partial - \lambda}{q}$ on $\h{\sD}^{(m)}_{\sX,\bQ}$. Using (\ref{PhiInvxdl}), we conclude that $F^\natural(\sM^{(m)}(\lambda))$ is the cokernel of the `diagonal' left $\h{\sD}^{(0)}_{\sX',\bQ}$-linear endomorphism
\[ \bigoplus_{i=0}^{q-1} \cdot (x' \partial' - \lambda_i) \otimes 1 : \bigoplus\limits_{j=0}^{q-1} \h{\sD}^{(0)}_{\sX',\bQ} \otimes \theta_j \to \bigoplus\limits_{j=0}^{q-1} \h{\sD}^{(0)}_{\sX',\bQ} \otimes \theta_j.\]
The result follows.\end{proof}

\begin{cor}\label{KummerIsocDesc} Let $\lambda \in K$. 
\be \item If $\lambda \notin \bigcup\limits_{i=0}^{q-1} i + q K^\circ$, then $x \partial - \lambda$ is a unit in $\h{\sD}^{(m)}_{\sX,\bQ}$ and $\sM^{(m)}(\lambda) = 0$.
\item Otherwise there exists a unique integer $i$ such that $0 \leq i < q$ and $\lambda \in i + q K^\circ$, and then
\[F^\natural(\sM^{(m)}(\lambda)) \quad \cong \quad  \frac{\h{\sD}^{(0)}_{\sX',\bQ}}{\h{\sD}^{(0)}_{\sX',\bQ} (x' \partial' - \frac{\lambda - i}{q})}.\]
\ee\end{cor}
\begin{proof} (a) In this case, $\lambda_i = \frac{\lambda - i}{q} \notin K^\circ$ for all $i = 0, \ldots, q-1$, so $|\lambda_i| > 1$ for all $i$. But then $x' \partial' - \lambda_i$ is a unit in $\h{\sD}^{(0)}_{\sX',\bQ}$, so $F^\natural(\sM^{(m)}(\lambda)) = 0$ by Proposition \ref{KummerFrobDesc}. So $\sM^{(m)}(\lambda) = 0$ and $x \partial - \lambda$ is a unit in $\h{\sD}^{(m)}_{\sX,\bQ}$ because $F^\natural$ is an equivalence of categories.

(b) Note that if $0 \leq i < j < q$, then $q$ does not divide $j - i$, so $|j - i| > |q|$. Hence $(i + q K^\circ) \cap (j + q K^\circ)  = \emptyset$, and if $\lambda \in i + qK^\circ$ then $\lambda \notin j + q K^\circ$ for any other $j$. This means that $|\lambda_j| > 1$ whenever $j \neq i$, and then $x' \partial' - \lambda_j$ is a unit in $\h{\sD}^{(0)}_{\sX',\bQ}$ as above. Now apply Proposition \ref{KummerFrobDesc}.
\end{proof}

We now assume that $m = 0$ and focus on the algebra $D := \h{\sD}^{(0)}_{\sX}(\sX)$, which is $\pi_K$-adically complete and separated. It is a lattice inside the Tate-Weyl algebra $D_K = \h{\sD}^{(0)}_{\sX,\bQ}(\sX)$. If $I$ is a left ideal in $D$, then we say that $I$ is \emph{$\pi_K$-closed} if the $K^\circ$-module $D / I$ is torsionfree, or equivalently, if $I = I_K \cap D$. 

\begin{lem}\label{AppNak} \hsp
\be \item The algebra $D$ is Noetherian. 
\item Let $J \leq I$ be two left ideals of $D$ and suppose that $I$ is $\pi_K$-closed. Then $I + \pi_K D = J + \pi_K D$ implies that $I = J$.
\ee \end{lem}
\begin{proof} (a) $\overline{D} := D / \pi_K D$ is the skew-polynomial ring $B[\overline{\partial}; \frac{d}{dx}]$ where 
\[B := k\left[x, \frac{1}{x - \overline{\alpha_1}}, \cdots, \frac{1}{x - \overline{\alpha_d}}\right] = \cO( \sX \times_{K^\circ} k)\]
and $\overline{\alpha_i}$ denotes the image of $\alpha_i$ in the residue field $k$ of $K^\circ$. Since $k$ is a field, both $B$ and $\overline{D}$ are Noetherian rings \footnote{This uses the assumption that $K^\circ$ is a discrete valuation ring.}. The associated graded ring of $D$ with respect to the $\pi_K$-adic filtration is isomorphic to a polynomial ring $\overline{D}[s]$ and is therefore is Noetherian. Hence $D$ is Noetherian by \cite[Chapter II, \S 1.2, Proposition 3(1)]{LVO}. 

(b) Note that $I \cap \pi_K D = \pi_K I$ since $I$ is $\pi_K$-closed. Hence, by the modular law
\[ I = I \cap (I + \pi_K D) = I \cap (J + \pi_K D) = J + (I \cap \pi_K D) = J + \pi_K I\]
Now $I$ is finitely generated by part (a), and $\pi_K$ lies in the Jacobson radical of $D$ because $D$ is $\pi_K$-adically complete, so $I = J$ by Nakayama's Lemma. \end{proof}

We will assume until the end of $\S \ref{LocIrred}$ that $\lambda \in K^\circ$.

\begin{lem}\label{EulerOps} Let $I$ be a $\pi_K$-closed left ideal of $D$ containing $t := x \partial - \lambda$. Then $I$ is stable under the operators $E_n := \binom{\ad(x \partial)}{n} : D_K \to D_K$ for all $n\geq 0$.
\end{lem}
\begin{proof}  We first consider the case where $I = D$ itself. For any $f \in \cO(\sX)$ and any $m \geq 0$ we calculate $\ad(x \partial)( f \partial^m ) = [x \partial, f \partial^m] = (x\partial - m)(f) \partial^m$. Therefore for any $n \geq 0$ we see that
\begin{equation}\label{BinAdEuler} \binom{\ad(x \partial)}{n}(f \partial^m) = \binom{x \partial - m}{n}(f) \partial^m = \sum\limits_{i=0}^n \binom{-m}{n-i} x^i\partial^{[i]}(f) \hsp \partial^m\end{equation}
which implies that $\binom{\ad(x \partial)}{n}$ preserves $D$ as required. Now, $\ad(x \partial) = \ad(t)$. Letting 
\[\ad(t)_{[n]} := \ad(t) (\ad(t) - 1) \cdots (\ad(t) - n+1) = n! \binom{\ad(t)}{n}\]
we see that $\ad(t)_{[n]}(D) \subseteq n! D$. Returning to the general case where $I$ is not necessarily equal to $D$, note that $\ad(t)(I) = [t,I] \subseteq I$ because $t \in I$ by assumption. Hence
\[ \ad(t)_{[n]}(I) \subseteq I \cap n! D = n! I\]
because $D/I$ is torsionfree. Dividing through by $n!$ we conclude that $I$ is stable under $\binom{\ad(t)}{n} = \binom{\ad(x\partial)}{n}$ as required.
\end{proof}

\begin{prop}\label{Ibar} Let $I$ be a $\pi_K$-closed left ideal of $D$ properly containing $Dt$. Then for some $n \geq 0$, either $x^n \in I + \pi_K D$ or $\partial^n \in I + \pi_K D$. \end{prop}
\begin{proof} Let $\overline{I} := (I + \pi_K D)/\pi_K D$. Since $I$ is $\pi_K$-closed, we see that $\overline{I}$ strictly contains $\overline{D t} = (D t + \pi_K D)/\pi_K D$ by Lemma \ref{AppNak}(b).  Now $\overline{I}$ is a left ideal in $\overline{D} = B[\overline{\partial}; \frac{d}{dx}]$ which is an Ore localisation of the Weyl algebra $W := k[x][\overline{\partial}; \frac{d}{dx}]$ at the Ore set consisting of powers of the polynomial $\prod_{i=1}^d (x - \overline{\alpha_i})$. Therefore, by \cite[Proposition 2.1.16(iii)]{MCR}, $\overline{I} = \overline{D} \cdot J$ where $J:= \overline{I} \cap W$; note that $J$ strictly contains $W \overline{t}$ because $\overline{I}$ strictly contains $\overline{D} \overline{t}$. 

Now, $W / W \overline{t}$ is a $k$-vector space with basis $\{v_m: m \in \Z\}$ where $v_m := \overline{x^m}+ W \overline{t}$ if $m \geq 0$ and $v_m := \overline{\partial^{-m}}+ W \overline{t}$ if $m < 0$. It follows from (\ref{BinAdEuler}) that
\[ E_n(v_m) = \binom{m}{n} v_m \qmb{for all} n \geq 0, m \in \Z.\]
Pick a non-zero element $u \in J / W \overline{t}$ ; then we can write $u$ as a finite sum 
\[u = \sum_{m \in \Z} u_m v_m\]
for some $u_m \in k$, not all zero. Suppose that in fact $u_r \neq 0$ for some $r \in \Z$. The binomial coefficients $\binom{z}{n}$ form a basis for the ring of locally constant $k$-valued functions on $\Zp$. We can find such a function, $\epsilon$ say, such that $\epsilon(r) = 1$ and $\epsilon(m) = 0$ for all $m \in \Z \backslash \{r\}$ such that $u_m \neq 0$. Write $\epsilon = \sum\limits_{n=0}^N \epsilon_n \binom{z}{n}$; then 
\[ \left(\sum\limits_{n=0}^N \epsilon_n E_n\right)(u) = \sum_{n=0}^N \sum_{m \in \Z} \epsilon_n u_m \binom{m}{n} v_m = \sum_{m \in \Z} u_m v_m \epsilon(m) = u_r v_r.\]
Now $E_n(J) \subseteq J$ for all $n \geq 0$ by Lemma \ref{EulerOps}, so $u_r v_r \in J / W \overline{t}$. Since $k$ is a field, we conclude that $v_r \in J/W\overline{t}$ and the result follows.
\end{proof}

\begin{cor}\label{PolysInI} Let $I$ be a $\pi_K$-closed left ideal of $D$ properly containing $Dt$. Then either
\be \item there exist $a_n, a_{n-1}, \cdots, a_0 \in K^\circ \langle \partial\rangle$ with $a_n \neq 0$ and 
\[ a_n x^n + a_{n-1} x^{n-1} + \cdots + a_1 x + a_0 \in I, \qmb{or}\]
\item there exist $a_n, a_{n-1}, \cdots, a_0 \in \cO(\sX)$ with $a_n \neq 0$ and 
\[ a_n \partial^n + a_{n-1} \partial^{n-1} + \cdots + a_1 \partial + a_0 \in I.\]
\ee
\end{cor}
\begin{proof} By Proposition \ref{Ibar}, we know that either $\overline{x}^n \in \overline{I}$ or $\overline{\partial}^n \in \overline{I}$ for some $n \geq 0$. Suppose first that $\overline{x}^n \in \overline{I}$, and let $S := \{1, x, \cdots, x^{n-1}\}$. Because the image of $S$ in $B / B \overline{x}^n$ spans it as a $k$-vector space, its image also generates $\overline{D} / \overline{D} \overline{x}^n$ as a $k[\partial]$-module; hence its image in $\overline{D} / \overline{I} = \overline{D/I}$ generates it as a $k[\partial]$-module. Since $D/I$ is $\pi_K$-adically separated by Lemma \ref{AppNak}(a) and \cite[Chapter II, \S 1.2, Theorem 10(1)]{LVO},  we conclude using \cite[Chapter I, Theorem 5.7]{LVO} that the image of $S$ in $D/I$ generates it as a $K^\circ\langle \partial\rangle$-module. Hence we may write $x^n$ as a $K^\circ \langle \partial \rangle$-linear combination vectors in this image, and hence (a) holds. Alternatively, if $\overline{\partial}^n \in \overline{I}$ then the same argument using $T := \{1, \partial, \cdots, \partial^{n-1}\}$ instead of $S$ and $\cO(\sX)$ instead of $K^\circ\langle \partial \rangle$ shows that (b) holds.
\end{proof}

\begin{prop}\label{ImeetsSubrings} Let $I$ be a $\pi_K$-closed left ideal of $D$ properly containing $Dt$ and suppose that $\lambda \in K^\circ\backslash \Z$. Then either $I \cap K^\circ \langle \partial \rangle \neq \{0\}$ or $I \cap \cO(\sX) \neq \{0\}$.
\end{prop}
\begin{proof} Let $v_\lambda$ be the canonical generator of $D / Dt$ and let $m \geq 1$. Then
\[\begin{array}{lll} \partial \cdot x^mv_\lambda &=& \partial x \cdot x^{m-1} v_\lambda = (x^{m-1}\partial x + [\partial x, x^{m-1}]) \cdot v_\lambda = \\
&=& x^{m-1} (x \partial + 1 + (m-1)) v_\lambda = (\lambda + m)x^{m-1} v_\lambda.\end{array}\]
It follows that for any  $a \in K^\circ \langle \partial \rangle$ and any $m \geq 1$, we have
\[ \partial \cdot a x^m \equiv a (\lambda + m)x^{m-1} \mod Dt.\]
Suppose now that $a_n x^n + a_{n-1} x^{n-1} + \cdots + a_1 x + a_0 \equiv 0 \mod I$, for some $a_i \in K^\circ \langle \partial \rangle$ with $a_n \neq 0$ and $n \geq 0$ \emph{least possible}. Suppose for a contradiction that $n \geq 1$. Then 
\[ 0 = \partial \cdot 0 \equiv \partial \cdot \sum\limits_{m=0}^n a_m x^m \equiv \sum\limits_{m=1}^n a_m (\lambda + m) x^{m-1} + \partial a_0 \quad \mod I.\]
Since $\lambda \notin \Z$ and since $K^\circ \langle \partial \rangle$ is a domain, we see that $a_m (\lambda+m) \neq 0$, so this congruence contradicts the minimality of $n$. Therefore $n = 0$ and $I \cap K^\circ \langle \partial \rangle \neq \{0\}$ as claimed. Otherwise, by Corollary \ref{PolysInI}, we have the congruence $a_n \partial^n + a_{n-1} \partial^{n-1} + \cdots + a_1 \partial + a_0 \equiv 0 \mod I$ for some $a_i \in \cO(\sX)$ with $a_n \neq 0$. A similar calculation shows that
\begin{equation}\label{xdm} x \cdot a \partial^m \equiv a (\lambda - m + 1) \partial^{m-1} \quad \mod Dt \qmb{for all} m \geq 1, a \in \cO(\sX)\end{equation}
and using this congruence we conclude using a similar argument that in fact we can choose $n = 0$ and therefore $I \cap \cO(\sX) \neq \{0\}$.
\end{proof}

\begin{cor}\label{Simple} If $\lambda \in K^\circ \backslash \Z$, then the only $\pi_K$-closed left ideal $I$ of $D$ properly containing $D (x \partial - \lambda)$ is $D$ itself.
\end{cor}
\begin{proof} Every non-zero ideal in the Tate algebra $K \langle \partial \rangle$ is generated by a non-zero element of $K[\partial]$. So, if $I \cap K^\circ \langle \partial \rangle$ is non-zero, then $I_K \cap K[\partial]$ is non-zero. Clearing denominators, we may even assume that $I \cap K^\circ[\partial]$ is non-zero. By applying an appropriate linear combination of the Euler operators $E_n$ as in the proof of Proposition \ref{Ibar} to a non-zero element of $I \cap K^\circ[\partial]$, and using the fact that $I$ is $\pi_K$-closed, we see that $\partial^n \in I$ for some $n \geq 0$. Choosing $n$ minimal possible and using the relation $(\ref{xdm})$ together with the fact that $\lambda \notin \Z$, we see that $n = 0$ and hence $1 \in I$.

Otherwise, by Proposition \ref{ImeetsSubrings}, we know that $I \cap \cO(\sX)$ is non-zero. Again, every ideal in the affinoid algebra $\cO(X) = \cO(\sX)_K$ is generated by a non-zero element of $K[x]$, and we similarly conclude that $1 \in I$.
\end{proof}

\begin{proof}[Proof of Theorem \ref{SimpleDmModule}] Corollary \ref{KummerIsocDesc}(a) tells us that $\sM^{(m)}(\lambda)(\sX)$ is zero if $\lambda \notin \bigcup\limits_{i=0}^{p^m-1} i + p^m K^\circ$, so we may assume that $\lambda$ \emph{does} lie in this union. Then we can apply  Corollary \ref{KummerIsocDesc}(b) to see that $F^\natural(\sM^{(m)}(\lambda))$ is isomorphic to $\frac{\h{\sD}^{(0)}_{\sX',\bQ}}{\h{\sD}^{(0)}_{\sX',\bQ} (x' \partial' - \frac{\lambda - i}{p^m})}$ for some uniquely determined $i \in \{0,1,\ldots, p^m - 1\}$. 

Since $F^\natural$ is an equivalence of categories by \cite[Th\'eor\`eme 2.3.6]{Berth3}, it is enough to prove that $F^\natural(\sM^{(m)}(\lambda))(\sX')$ is a simple $\h{\sD}^{(0)}_{\sX',\bQ}(\sX')$-module.  Note that $\frac{\lambda - i}{p^m} \notin \Z$ because $\lambda \notin \Z$ by assumption, so we have reduced the problem to the case where $m = 0$. But now $\sM^{(0)}(\lambda)(\sX) = D_K / D_K (x\partial - \lambda)$, which is a simple $D_K$-module by Corollary \ref{Simple}.
\end{proof}

\section{Group actions and equivariance}\label{CoadDrin}
\subsection{The crossed product and the secret action}\label{CPandSA}
We recall some constructions from \cite[\S 2.1]{EqDCap}. Let $L$ be a commutative ring, let $S$ be an $L$-algebra and let $G$ be a group acting on $S$ by $L$-algebra automorphisms from the left. We denote the result of the group action of $g \in G$ on $s \in S$ by $g\cdot s$. The \emph{skew-group ring} $S \rtimes G$ is a free left $S$-module with basis $G$, and its multiplication is determined by 
\begin{equation}\label{SkewGpRngMult} (sg) \cdot (s'g') = (s (g \cdot s') )(gg') \qmb{for all} s,s' \in S, \quad g,g' \in G.\end{equation}

The skew-group ring satisfies the following universal property. 

\begin{lem} \label{SkewGpRngUniv} Suppose that $U$ is an $L$-algebra, $\sigma\colon S\to U$ is an $L$-algebra homomorphisms and $\rho\colon G\to U^\times$ is a group homomorphism such that \[\rho(g)\sigma(s)\rho(g)^{-1}=\sigma(g\cdot s) \mbox{ for all }g\in G\mbox{ and }s\in S.\] There is a unique $L$-algebra homomorphism $\sigma\rtimes \rho\colon S\rtimes G\to U$ such that \[(\sigma\rtimes\rho)(se)=\sigma(s)\mbox{ for all }s\in S\mbox{ and  }(\sigma\rtimes \rho)(1g)=\rho(g)\mbox{ for all }g\in G.\] 
\end{lem}
	
\begin{proof} It is straightforward to verify that $\sigma\rtimes \rho$ must be the $L$-linear extension of the map sending $sg$ to $\sigma(s)\rho(g)$ and that by (\ref{SkewGpRngMult}) this is indeed an $L$-algebra homomorphism.
\end{proof}
	
\begin{defn}\label{Triv}  Let $S \rtimes G$ be a skew-group ring. A \emph{trivialisation} is a group homomorphism $\beta : G \to S^\times$ such that for all $g \in G$, the conjugation action of $\beta(g) \in S^\times$ on $S$ coincides with the action of $g \in G$ on $S$.
\end{defn}

\begin{defn}\label{DefnTriv} Let $N$ be a normal subgroup of $G$, and suppose that $\beta : N \to S^\times$ is a trivialisation of the sub-skew-group ring $S \rtimes N$.
\be \item We define \[S \rtimes_N^\beta G := \frac{ S \rtimes G }{ (S \rtimes G) \cdot \{\beta(n)^{-1} n - 1 : n \in N\} }.\]
\item We say that $\beta$ is \emph{$G$-equivariant} if 
\[\beta({}^gn) = g \cdot \beta(n) \qmb{for all} g \in G \qmb{and} n \in N.\]
\ee Here ${}^gn := gng^{-1}$ denotes the conjugation action of $G$ on $N$.\end{defn}
Note that $S \rtimes_N^\beta G$ is \emph{a priori} only a left $S \rtimes G$-module. 

\begin{lem}\label{RingSGN} Suppose that $N$ is a normal subgroup of $G$, and that $\beta : N \to S^\times$ is a $G$-equivariant trivialisation. Then 
\be \item $S \rtimes_N^\beta G$ is an associative ring.
\item $S \rtimes_N^\beta G$ is isomorphic to a crossed product $S \ast (G/N)$.
\ee \end{lem}
\begin{proof} This is \cite[Lemma 2.2.4]{EqDCap}.\end{proof}
We will now place ourselves in the following situation.

\begin{hypn}\label{MorphCP}\hsp
\be \item $G$ is a group, $S$ and $S'$ are $L$-algebras,
\item $G$ acts on $S$ and $S'$ by $L$-algebra automorphisms,
\item $f : S' \to S$ is a $G$-equivariant $L$-algebra homomorphism,
\item $H' \leq H$ are two normal subgroups of $G$,
\item $\beta: H \to S^\times$ is a $G$-equivariant trivialisation of the $G$-action on $S$,
\item $\beta': H' \to S'^\times$ is a $G$-equivariant trivialisation of the $G$-action on $S'$,
\item $f(\beta'(h')) = \beta(h')$ for all $h' \in H'$.
\ee\end{hypn}
It follows from \cite[Lemma 2.2.7]{EqDCap} that the ring homomorphism $f : S' \to S$ extends to a ring homomorphism $f \rtimes 1 : S' \rtimes_{H'}^{\beta'} G \to S \rtimes_H^\beta G$.  To aid legibility, we will now drop the trivialisations $\beta$ and $\beta'$ from the notation.

\begin{lem}\label{SecretAction} Let $M'$ be an $S' \rtimes_{H'} G$-module. Then there is an action of $H$ on $M := S \otimes_{S'} M'$ by left $S$-linear automorphisms, given by
\[h \star (s \otimes m') = s \hsp \beta(h)^{-1} \otimes h \cdot m' \qmb{for all} h \in H, s \in S, m' \in M'.\]
The restriction of this action to $H'$ is trivial.\end{lem}
\begin{proof} We will first check that the formula makes sense, that is $h \star (s s' \otimes m') = h \star (s \otimes s' m')$ holds for all $s' \in S'$, $s \in S$, $h \in H$ and $m' \in M'$. We have
\[\begin{array}{lll} (ss') \beta(h)^{-1} \otimes h \cdot m' &=& s \hsp \beta(h)^{-1} \hsp \hsp \beta(h) s' \beta(h)^{-1} \otimes h\cdot m' = \\
&=& s  \hsp\beta(h)^{-1}  (h \cdot s') \otimes h \cdot m' = \\
&=& s \hsp \beta(h)^{-1} \otimes (h \cdot s') (h\cdot m') = \\
&=& s \hsp \beta(h)^{-1} \otimes h \cdot (s' m') \end{array}\]
as required. It is straightforward to verify that the given formula defines a left $H$-action on $M$ by left $S$-linear automorphisms. Because $M'$ is a $S' \rtimes_{H'}G$-module, by Definition \ref{DefnTriv}(a) we may view it as an $S' \rtimes G$-module where $h' \cdot m' = \beta'(h') \cdot m'$ for all $h' \in H'$ and $m' \in M'$.  So if $h'$ lies in $H'$ then by Hypothesis \ref{MorphCP}(g),
\[s \beta(h')^{-1} \otimes h' \cdot m' = s f(\beta'(h')^{-1}) \otimes h'\cdot m' = s \otimes \beta'(h')^{-1} \cdot (h' \cdot m') = s \otimes m'\]
and hence $h' \star (s \otimes m') = s \otimes m'$ for all $h' \in H'$, $s \in S$ and $m' \in M'$ as claimed.\end{proof}

\begin{defn} Let $M'$ be an $S' \rtimes_{H'}G$-module. We call the $H$-action on the $S$-module $S \otimes_{S'} M$ defined in Lemma \ref{SecretAction} the \emph{secret $H$-action}. 
\end{defn}

It follows from Lemma \ref{SecretAction} that via the secret $H$-action, $S \otimes_{S'}M'$ is naturally a left $S[H/H']$-module. We will use the notation
\[ (S \otimes_{S'}M')_H := \frac{S \otimes_{S'}M'}{\{h \star v - v : v \in S \otimes_{S'}M', h \in H\}}.\]
to denote the module of left $H$-coinvariants of $S \otimes_{S'}M'$ under the secret $H$-action. 

\begin{thm}\label{IndModCoinv} For every $S' \otimes_{H'}G$-module $M'$, there is a natural isomorphism
\[(S \rtimes_H G) \underset{(S'\rtimes_{H'}G)}{\otimes}{} M' \quad \cong \quad (S \otimes_{S'} M')_H\]
of left $S$-modules.
\end{thm}
\begin{proof} Let $[v]$ denote the image of $v \in S \otimes_{S'}M'$ in $(S \otimes_{S'} M')_H$. Consider the map $\alpha : (S \rtimes G) \times M' \longrightarrow (S \otimes_{S'} M')_H$ given by $(s g, m') \mapsto [s \otimes g \cdot m']$ for $s \in S, g \in G$ and $m' \in M'$. Let $s \in S, g \in G, h \in H$ and $m' \in M'$. Since 
\[\begin{array}{lll} \alpha(sg \beta(h), m') &=& \alpha(s \beta({}^gh)g, m') = [s \beta({}^gh) \otimes g\cdot m'] = \\
&=& [s \otimes {}^gh \cdot (g\cdot m')] = [s \otimes (gh) \cdot m'] = \alpha( s g h, m' )\end{array}\]
we see that $\alpha$ is zero on $(S \rtimes G)\{\beta(h) - h : h \in H\} \hsp \times \hsp M'$ and therefore descends to a map $\overline{\alpha} : (S \rtimes_HG) \times M' \longrightarrow (S \otimes_{S'}M')_H.$ Recalling from \cite[\S 2.2(3)]{EqDCap} that $s \gamma(g)$ denotes the image of $s g \in S \rtimes G$ in the crossed product $S \rtimes_{H} G$, we see that $\overline{\alpha}$ is given by $\overline{\alpha}(s \gamma(g), m') = [s \otimes g \cdot m']$ for all $s \in S, g \in G$ and $m' \in M'$. 

Let $s \gamma(g) \in S \rtimes_HG$, $s'\gamma(g') \in S' \rtimes_{H'}G$ and $m' \in M'$. Then we compute
\[\begin{array}{lll} \overline{\alpha}(s \gamma(g) \cdot s'\gamma(g') , m') &=&\overline{\alpha}(s f(g\cdot s') \gamma(gg'), m') = [s f(g\cdot s') \otimes (gg' \cdot m')], \\
\overline{\alpha}(s \gamma(g), s'\gamma(g') \cdot m') &=& [s \otimes g \cdot (s'\cdot (g'\cdot m'))] = [s \otimes ((g \cdot s')g) \cdot (g'\cdot m')] = \\
&=& [s \otimes (g \cdot s') \cdot (gg' \cdot m')] = [s f(g \cdot s') \otimes (gg' \cdot m')].\end{array}\]
Hence $\overline{\alpha}$ is $S' \rtimes_{H'} G$-balanced, and therefore descends to a well-defined map
\begin{equation}\label{Theta} \theta : (S \rtimes_H G) \underset{(S'\rtimes_{H'}G)}{\otimes}{} M' \longrightarrow (S \otimes_{S'} M')_H\end{equation}
given by $\theta( s \gamma(g) \otimes m') = [s \otimes g\cdot m']$ for all $s \in S, g\in G$ and $m' \in M'$. It is clear that $\theta$ is $S$-linear. On the other hand, there is evidently a well-defined $S$-linear map
\[\varphi : S \otimes_{S'} M' \longrightarrow (S \rtimes_H G) \underset{(S'\rtimes_{H'}G)}{\otimes}{} M'\]
given by $\varphi( s \otimes m' ) = s \otimes m'$ for all $s \in S$ and $m' \in M'$. Since 
\[\begin{array}{lll} \varphi( h \star (s \otimes m') ) &=& \varphi(s \beta(h)^{-1} \otimes h \cdot m') = s \beta(h)^{-1} \otimes h \cdot m' = \\
&=&s \gamma(h)^{-1} \otimes h \cdot m' = s \otimes h^{-1} \cdot (h \cdot m') = s \otimes m', \end{array}\]
we see that $\varphi$ factors through $(S \otimes_{S'} M')_H$ and induces an $S$-linear map 
\[\psi : (S \otimes_{S'} M')_H \longrightarrow (S \rtimes_H G) \underset{(S'\rtimes_{H'}G)}{\otimes}{} M'\]
given by $\psi([s \otimes m']) = s \otimes m'$ for all $s \in S$ and $m' \in M'$. It is now straightforward to check that $\theta$ and $\psi$ are mutually inverse bijections.
\end{proof}

Now let $M'$ be an $S' \rtimes_{H'} G$-module, let $M$ be an $S \rtimes_H G$-module and suppose we're given an $S' \rtimes_{H'} G$-linear map 
\[\tau : M' \to M,\]
where we regard $M$ as an $S' \rtimes_{H'}G$-module via restriction along the ring homomorphism $f \rtimes 1 : S' \rtimes_{H'}^{\beta'} G \to S \rtimes_H^\beta G$ coming from \cite[Lemma 2.2.7]{EqDCap}. Our next goal is to establish a sufficient criterion for the induced $S \rtimes_H G$-linear map
\begin{equation}\label{TauTilde} \widetilde{\tau} : (S \rtimes_H G) \underset{(S'\rtimes_{H'}G)}{\otimes}{} M' \longrightarrow M\end{equation}
given by $\widetilde{\tau}(v \otimes m') = v \cdot \tau(m')$, where $v \in S \rtimes_H G$ and $m' \in M'$, to be injective.

\begin{lem}\label{EquivariantStar} Let $M'$ be an $S' \rtimes_{H'}G$-module. 
\be \item There is an $S$-\emph{semi}linear action of $G$ on $S \otimes_{S'} M'$ given by
\[g \cdot (s \otimes m') = (g\cdot s) \otimes (g \cdot m')\qmb{for all} g \in G, s \in S, m' \in M'.\]
\item For every $g \in G, h\in H$ and $v \in S \otimes_{S'}M'$ we have
\[ g \cdot (h \star ( g^{-1} \cdot v) ) = ({}^g h) \star v .\]
\ee\end{lem}
\begin{proof}  (a) We check that the formula is well-defined. Let $g \in G$, $s \in S$, $s' \in S'$ and $m' \in M'$. Because $f : S' \to S$ is $G$-equivariant, we have
\[ \begin{array}{lll} g \cdot (sf(s') \otimes m') &=& (g \cdot s f(s')) \otimes (g \cdot m') = (g\cdot s)(g \cdot f(s')) \otimes (g \cdot m') = \\
&=& (g \cdot s) f(g \cdot s') \otimes (g \cdot m') = (g  \cdot s) \otimes (g \cdot s')(g \cdot m') \\
&=& (g\cdot s) \otimes g \cdot (s'\cdot m') = g \cdot (s \otimes s' \cdot m')\end{array}\]
as required. It is straightforward to see that this defines a $G$-action on $S \otimes_{S'}M'$ which is $S$-semilinear in the sense that
\[ g \cdot (s \cdot v) = (g\cdot s) \cdot (g \cdot v) \qmb{for all} g \in G, s \in S, v \in S \otimes_{S'} M'.\]
(b) We may assume that $v = s \otimes m'$ with $s \in S$ and $m' \in M'$. Let $g \in G$ and $h \in H$; using Definition \ref{DefnTriv}(b), we compute
\[\begin{array}{lll} g \cdot( h \star (g^{-1} \cdot (s \otimes m'))) &=& g \cdot (h \star (g^{-1} \cdot s \otimes g^{-1}\cdot m')) = \\
&=& g \cdot ( (g^{-1} \cdot s) \beta(h)^{-1} \otimes h \cdot (g^{-1} \cdot m') )= \\
&=& g \cdot ((g^{-1} \cdot s) \beta(h^{-1})) \otimes g \cdot (h \cdot (g^{-1} \cdot m')) = \\
&=& s \hsp g \cdot \beta(h^{-1}) \otimes g h g^{-1} \cdot m' = \\
&=& s \beta( ({}^gh)^{-1} ) \otimes ({}^gh) \cdot m' = ({}^g h) \star (s \otimes m'). \qedhere
\end{array}\]
\end{proof}

We will now restrict our scope slightly, and assume until the end of $\S\ref{CPandSA}$ that
\begin{itemize}
\item $H'$ has finite index in $H$,
\item $L$ is a field of characteristic zero.
\end{itemize}
Let $M'$ be an $S' \otimes_{H'} G$-module. By Lemma \ref{SecretAction} and the above assumptions, we see that $S \otimes_{S'}M'$ is a module over the group ring $L[H/H']$ which is semisimple because we're assuming that our ground field $L$ has characteristic zero. Because the action of $L[H/H']$ on $S \otimes_{S'}M'$ commutes with the action of $S$, we obtain a canonical $S$-module decomposition
\begin{equation}\label{IndDecomp} S \otimes_{S'} M' \cong \bigoplus_{[V] \in \Irr_L(H/H')} e_V \star (S \otimes_{S'} M')\end{equation}
where $\Irr_L(H/H')$ is the set of isomorphism classes of simple $L[H/H']$-modules, and for each $[V] \in \Irr_L(H/H')$, $e_V$ is the primitive central idempotent of $L[H/H']$ such that $1 - e_V$ generates the annihilator of $V$ in $L[H/H']$ as a two-sided ideal. 

Let $V$ be an $L[H/H']$-module and let $g \in G$. Let $gV := \{[g] v : v \in V\}$ where $[g]$ is a formal symbol, and define an $H$-action on $gV$ by the rule $h \cdot [g]v = [g] h^gv$ for all $h \in H$ and $v \in V$. This is well defined because $H$ is normal in $G$, and since $H'$ acts trivially on $V$ and is also normal in $G$, we see that $[g]V$ is again an $L[H/H']$-module. In this way we obtain a permutation action of $G$ on $\Irr_L(H/H')$ given by $g \cdot [V] = [gV]$. 

\begin{cor}\label{outerGaction} Let $M'$ be an $S' \rtimes_{H'}G$-module. The $S$-semilinear $G$-action on $S \otimes_{S'}M'$ permutes the direct summands appearing on the right hand side of (\ref{IndDecomp}):
\[ g \cdot \left(e_V \star (S \otimes_{S'} M')\right) = e_{gV} \star (S \otimes_{S'} M') \qmb{for all} g\in G \qmb{and} [V] \in \Irr_L(H/H').\]
\end{cor}
\begin{proof} The conjugation action of $g \in G$ on $L[H/H']$ sends $e_V$ to $e_{gV}$. Now 
\[g \cdot (e_V \star v) = {}^ge_V \star (g \cdot v) \qmb{for all} g \in G, v \in S \otimes_{S'}M'\]
by Lemma \ref{EquivariantStar}(b) and the result follows.
\end{proof}

Given an $S$-module $M$, let $\ell(M)$ denote its length. We can now state and prove our criterion. 
\begin{thm}\label{Criterion}  Assume that Hypothesis \ref{MorphCP} holds, that $H'$ has finite index in $H$ and that $L$ is a field of characteristic zero. Let $M'$ be an $S' \rtimes_{H'} G$-module, let $M$ be a $S \rtimes_H G$-module, let $\tau : M' \to M$ be an $S' \rtimes_{H'} G$-linear map and consider the $S$-linear map
\[ \Phi = 1 \otimes \tau : S \otimes_{S'} M' \longrightarrow M\]
induced by $\tau$. Suppose further that the following conditions hold.
\be 
\item the secret $H$-action on $S \otimes_{S'}M'$ is non-trivial,
\item the $G$-action on $\Irr_L(H/H')$ has exactly two orbits, and
\item $\ell(\ker \Phi) \leq |\Irr_L(H/H')|-1$.
\ee
Then $\Phi$ factors through the coinvariants $(S \otimes_{S'} M')_H$ under the secret $H$-action on $S \otimes_{S'}M'$, and the induced map $\overline{\Phi} : (S \otimes_{S'} M')_H \longrightarrow M$ is injective.
\end{thm}
\begin{proof} Consider the $S$-module decomposition of $N := S \otimes_{S'}M'$ given by (\ref{IndDecomp}):
\[N = e_{\mathbbm{1}} \star N \quad\oplus\quad (H-1)\star N, \qmb{where} (H-1)\star N = \bigoplus\limits_{[V] \neq [\mathbbm{1}]} e_V \star N\]
and where $\mathbbm{1}$ denotes the trivial $L[H/H']$-module.  Note that 
\[\begin{array}{lll}\Phi(h \star (s \otimes m')) &=& \Phi( s \beta(h)^{-1} \otimes h\cdot m') = s \beta(h)^{-1} \cdot \tau(h \cdot m') = \\
&=& s \beta(h)^{-1} \cdot (h \cdot \tau(m')) = s \cdot \tau(m') = \Phi(s \otimes m')\end{array}\]
for all $h \in H, s \in S$ and $m' \in M'$, because the $G$-action on $M$ satisfies $h \cdot m = \beta(h) \cdot m$ as $M$ is an $S \rtimes_HG$-module. In other words: $(H - 1) \star N \subseteq \ker \Phi$. Hence
\[ \ell( (H - 1) \star N) \leq \ell(\ker \Phi) .\]
We know that $e_V \star N \neq 0$ for at least one non-trivial $[V] \in \Irr_L(H/H')$ by (a). Since $g[\mathbbm{1}] = [\mathbbm{1}]$ for all $g \in G$, Corollary \ref{outerGaction} together with (b) then implies that $e_V \star N \neq 0$ \emph{for all} non-trivial $[V] \in \Irr_L(H/H')$. This gives us the inequality
\[ \left\vert\Irr_L(H/H')\right\vert - 1 \leq \ell\left(\bigoplus\limits_{[V] \neq [\mathbbm{1}]} e_V \star N\right) = \ell( (H-1) \ast N).\]
On the other hand, (c) gives that  
\[\ell(\ker \Phi) \leq \left\vert\Irr_K(H/H')\right\vert- 1.\]
It follows that the inclusion $(H-1)\star N \subseteq \ker \Phi$ is in fact an equality. 

Hence $\Phi$ factors through the module of coinvariants $N_H$ of $N$ under the secret $H$-action, and induces an $S$-linear injection $\overline{\Phi} : N_H \stackrel{\cong}{\longrightarrow} M$. 
\end{proof}

\subsection{The Iwahori action on differential operators on the closed unit disc} \label{IwahoriAction}
In this section, we explore in detail the action of the $p$-adic Lie group $\mathbb{GL}_2(K)$ on the rigid $K$-analytic projective line $\bP^1 := \bP^{1,\an}$. Fix a coordinate $x \in \cO_{\bP^1}$. Because of the local nature of our constructions, we will focus specifically on the affinoid subdomains of the closed unit disc $\bD = \Sp K\langle x \rangle = \{z\in \bP^1 : |z| \leq 1\}$.
\begin{defn}\label{GenIwahori} The \emph{generalised Iwahori subgroup of $\mathbb{GL}_2(K)$} is 
\[ \cG := \left\{ \begin{pmatrix} a & b \\ c & d \end{pmatrix} \in \mathbb{GL}_2(K^\circ) : |c| < 1 \right\}.\]
\end{defn}

We note that if $\begin{pmatrix} a & b \\ c & d \end{pmatrix}\in \cG$ then $|a|,|d|=1$ since $|ad-bc|=1$, $|a|,|b|,|d|\leq 1$ and $|c|<1$. We also see that $\cG$ stabilises $\bD$ under the standard M\"obius action of $\mathbb{GL}_2(K)$ on $\bP^1$. In fact, the stabiliser of $\bD$ in $\mathbb{GL}_2(K)$ is $\cG Z$, where $Z$ is the centre of $\mathbb{GL}_2(K)$. By Lemma \ref{AutAOD} the $\cG$-action on $\bD$ induces $\cG$-actions on $\cO(\bD)$ and $\cD(\bD)$. The $\cG$-action on the divided powers $\partial_x^{[n]}\in  \cD(\bD)$ of $\partial_x$ is given as follows.

\begin{lem}\label{GdotDn} Let $g = \begin{pmatrix}a & b \\ c & d \end{pmatrix} \in \cG$. Then 
\be \item $g \cdot x = \frac{dx - b}{-cx +a}$, and 
\item $g \cdot \partial_x^{[n]} =  \sum\limits_{i=1}^n \frac{\binom{n-1}{i-1} (-cx+a)^{n+i} (-c)^{n-i}}{{(ad - bc)^n}} \partial_x^{[i]}$ for all $n \geq 0$.
\ee\end{lem}
\begin{proof}(a) This is straightforward. (b) The case $n = 1$ follows from Lemma \ref{AutAOD}(c). In general, let $y := -cx + a$ so that $\partial_x(y) = -c$ and $\partial_x = -c \partial_y$. Then the case $n=1$ can be rewritten as $g \cdot \partial_x = -\frac{c y^2}{\det(g)} \partial_y$, and it follows that $g\cdot \partial_x^{[n]} = \left(\frac{-c}{\det(g)}\right)^n (y^2 \partial_y)^{[n]}$. Now, by Vandermonde's identity, for $k\in \N$, \begin{eqnarray*}
 (y^2\partial_y)^{[n]}(y^k) & = & \binom{n+k-1}{n}y^{k+n}\\ & = & \sum\limits_{i=0}^{n} \binom{n-1}{n-i}\binom{k}{i}y^{k+n}\\ & = & \sum\limits_{i=1}^{n} \binom{n-1}{i-1}y^{n+i}\partial_y^{[i]}(y^k).\end{eqnarray*}
Thus as $K[y]$ is dense in $\cO(\bD)$ we see that 
\[ (y^2 \partial_y)^{[n]} = \sum\limits_{i=1}^n \binom{n-1}{i-1} y^{n+i} \partial_y^{[i]}\] from which the result quickly follows.
\end{proof}

Recall the natural inclusion $j : \cD(X) \hookrightarrow \cD_r(X)$ from Lemma \ref{TheMapJ}.

\begin{cor}\label{rnormofgdxn} Let $g = \begin{pmatrix}a & b \\ c & d \end{pmatrix} \in \cG$. Let  $r > 0$  and let $X$ be an $\partial_x/r$-admissible and $g$-stable affinoid subdomain of $\bD$. Then $|a - cx|_X = 1$, and for all $n\geq 1$, 
\[ | j(g \cdot \partial_x^n)| = \max\limits_{1 \leq i \leq n} \left|\binom{n-1}{i-1}\frac{n!}{i!}\right| \hsp  |c|^{n-i} \hsp r^i.\]
\end{cor}
\begin{proof} Note that $|x|_X \leq 1$ since $X \subset \bD$ by assumption. Since $|a| = 1$ and $|c|<1$ by Definition \ref{GenIwahori}, we see that $|a - cx|_X = 1$, as well as $|ad-bc| =  1$. We can now use Lemma \ref{GdotDn}(b) together with Definition \ref{AdelR} to obtain the result. \end{proof}

Recall that if $X \subseteq \bD$ is an affinoid subdomain, $\cG_X$ denotes its stabiliser in $\cG$.  

\begin{prop}\label{NormOfRhor} Let $X$ be a $\partial_x/r$-admissible affinoid subdomain of $\bD$ and let $g = \begin{pmatrix} a & b\\ c & d \end{pmatrix} \in \cG_X$. Then the $g$-action on $\cD(X)$ extends to a bounded $K$-algebra automorphism $\rho_r(g)$ of $\cD_r(X)$.
\end{prop}
\begin{proof} 
Fix $n \geq 0$. Using Corollary \ref{rnormofgdxn} we have
\[\frac{|j(g \cdot  \partial_x^n) |}{|r^n|} =  \max\limits_{1 \leq i \leq n} \left\vert \binom{n-1}{i-1}\frac{n!}{i!}\right\vert |c|^{n-i} r^{i-n}.\]
Now $\left\vert\binom{n-1}{i-1}\frac{n!}{i!}\right\vert = \left\vert(n-i)! \binom{n-1}{i-1} \binom{n}{i}\right\vert \leq |(n-i)!|$, so by Lemma \ref{nFacVarpi},
\[\frac{|j(g \cdot  \partial_x^n) |}{|r^n|} \leq \max\limits_{0 \leq i \leq n-1} |i!| \left( \frac{|c|}{r} \right)^i \leq p \max\limits_{0 \leq i \leq n-1}  i \left(\frac{|c|\varpi}{r}\right)^i.\]
By \cite[Theorem 4.1.8]{ArdWad2023}, there is a finite extension $K'$ of $K$ such that $X$ splits over $K'$.  Lemma \ref{BaseChR} and Corollary \ref{Heisenberg} now imply that
\[ r(X) = r(X_{K'}) = \varpi / \rho(X_{K'}) \geq \varpi,\] because $\rho(X_{K'}) \leq 1$ as $X_{K'} \subseteq \bD_{K'}$.  Since $X$ is $\partial_x/r$-admissible, we conclude that $r > r(X) \geq \varpi$, so that $|c|\varpi/r < |c| < 1$.  Lemma \ref{RhoEst} now tells us that
\[ \sup\limits_{n \geq 0} \frac{|j(g \cdot  \partial_x^n) |}{|r^n|} < \infty.\]
Recall the group homomorphism $\rho\colon \cG_X\to \cB(\cO(X))^\times$ given by Lemma \ref{AutAOD}(c). Using the above inequality, we can now apply Lemma \ref{BanachUnivProp} with $A=\cO(X)$, $\delta=\partial_x$, $B=\cD_r(X)$, $f=j\circ \rho(g)$ and $b=j(g\cdot \partial_x)$ to deduce that $\rho(g)$ extends to a bounded $K$-algebra endomorphism $\rho_r(g)$ of $\cD_r(X)$. It is easy to see that $\rho_r(gh) = \rho_r(g) \rho_r(h)$ for $g,h \in \cG_X$, so in fact each $\rho_r(g)$ is an automorphism.
\end{proof}
Note that the automorphism $\rho_r(g)$ in Proposition \ref{NormOfRhor} is only $\cO(X)$-semilinear. We now have at our disposal the skew-group ring
\[ \cD_r(X) \rtimes \cG_X\]
whenever $r > r(X)$.  

\begin{cor} Let $X$ be an affinoid subdomain of $\bD$.
\be \item For every $r > r(X)$, the $\cO(X)$-semilinear action of  $\cG_X$ on $\cD(X)$ extends to an $\cO(X)$-semilinear action on $\cD_r(X)$ by bounded $K$-linear automorphisms.
\item The $\cO(X)$-semilinear action of $\cG_X$ on $\cD(X)$ extends to an $\cO(X)$-semilinear action on $\cD^\dag_{r(X)}(X)$ by $K$-algebra automorphisms.
\ee\end{cor}
\begin{proof} (a) This follows immediately from Proposition \ref{NormOfRhor}. 

(b) This follows from part (a) in view of Definition \ref{DagSite}(d).\end{proof}
We now explain how $\rho(\cG_X)$ and $\sigma_r(\cD_r(X))$ interact inside $\cB(\cO(X))$.

\begin{lem}\label{gQg} Let $X$ be a $\partial_x/r$-admissible affinoid subdomain of $\bD$. The maps $\rho \colon \cG_X \longrightarrow \cB(\cO(X))^\times$ from Lemma \ref{AutAOD}(c) and $\sigma_r \colon  \cD_r(X) \to \cB(\cO(X))$ from Lemma \ref{ActionOnO} satisfy
\[ \rho(g) \circ \sigma_r(Q) \circ \rho(g)^{-1} = \sigma_r( g \cdot Q )\]
for all $Q \in \cD_r(X)$ and all $g \in \cG_X$, where $g \cdot Q = \rho_r(g)(Q)$. 
\end{lem}
\begin{proof} Suppose that $g\in \cG_X$ and $f,h\in \cO(X)$. Then 
\[ (\rho(g)\circ \sigma_r(f)\circ \rho(g)^{-1})(h) = g \cdot(f \hspace{0.1cm} g^{-1} \cdot h)  =  (g \cdot f) \hspace{0.1cm} h =  \sigma_r(g\cdot f)(h), \] and using (\ref{GactOandT}) we see that
 \[ (\rho(g)\circ \partial_x \circ \rho(g)^{-1})(h) = g\cdot (\partial_x(g^{-1} \cdot h)) = (g\cdot \partial_x)(h).\]Since $\sigma_r$ is a $K$-algebra homomorphism and $\rho_r(g)$ acts on $\cD_r(X)$ by $K$-algebra homomorphisms it follows that $\rho(g)\circ \sigma_r(Q)\circ \rho(g)^{-1}=\sigma_r(g\cdot Q)$ for all $Q\in \cD(X)$ and then for all $Q\in \cD_r(X)$ by continuity. 
\end{proof}
The universal property of skew group rings, Lemma \ref{SkewGpRngUniv}, gives the following
\begin{cor} Let $X$ be a $\partial_x/r$-admissible affinoid subdomain of $\bD$. The natural $\cD_r(X) \rtimes \cG_X$-action on $\cO(X)$ induces a $K$-algebra homomorphism $\sigma_r \rtimes \rho : \cD_r(X) \rtimes \cG_X \longrightarrow \cB(\cO(X))$.
\end{cor} 
We will now explain how to construct a trivialisation of a certain sufficiently large sub-skew-group ring of this skew-group ring.
\begin{defn}\label{Beta} For each real number $r > \varpi$, we define
\be\item a subgroup $\cG_r := \{g \in \cG : \quad |g\cdot x - x|_{\bD} < \varpi/r\}$ of $\cG$, and
\item a function $\beta : \cG_r \longrightarrow \cD_r(\bD)$ defined by
\[ \beta(g) := \sum\limits_{n=0}^\infty (g \cdot x - x)^n \partial^{[n]}. \]
\ee\end{defn}
Lemma \ref{nFacVarpi} shows that for any $g \in \cG_r$ we have
\[ \lim\limits_{n\to\infty} \left\vert\frac{(g\cdot x - x)^n}{n!}\right\vert_{\bD} \cdot r^n \leq \lim\limits_{n\to\infty} \left( \frac{|g\cdot x - x|_{\bD}}{\varpi/r} \right)^n = 0.\]
Hence $\beta(g)$ \emph{does} define an element of $\cD_r(\bD)$ whenever $g \in \cG_r$ by Definition \ref{AdelR}. The following explicit description of the group $\cG_r$ will be useful later.
\begin{lem}\label{GrCalc} $\cG_r = \left\{ \begin{pmatrix}a & b \\ c & d\end{pmatrix} \in \cG: \hsp |b| < \frac{\varpi}{r},  \hsp |c| < \frac{\varpi}{r}, \hsp |a-d|<\frac{\varpi}{r}\right\}$.
\end{lem}
\begin{proof} Let $g = \begin{pmatrix} a & b \\ c & d\end{pmatrix} \in \cG$. Since $|c| < 1$, the invertibility of $g$ in $M_2(K^\circ)$ forces $|a| = |d| = 1$. By Corollary \ref{rnormofgdxn}, we have $|a - c x|_{\bD} = 1$. Next, using Lemma \ref{GdotDn}(a) we have
\[g\cdot x - x = \frac{c x^2 + (d-a)x - b}{-cx + a}\quad.\]
Since $|\cdot|_{\bD}$ is multiplicative, we conclude that $|g \cdot x - x|_{\bD} = |cx^2 + (d-a)x -b |_{\bD} = \max\{|c|, |d-a|, |b|\}$. Hence $g \in \cG_r$ if and only if $|b|, |c|, |d-a|$ are all strictly less than $\varpi/r$ as claimed.
\end{proof}
The main reason for Definition \ref{Beta} comes from the following

\begin{prop}\label{SigmaRho} Let $X$ be a $\partial_x/r$-admissible affinoid subdomain of $\bD$. Then
\[ \sigma_r(\beta(g)) = \rho(g) \qmb{for all} g \in \cG_r \cap \cG_X.\]
\end{prop}
\begin{proof} Fix $g \in \cG_r \cap \cG_X$ and write $w := g\cdot x - x$. Then for every $a,b \in \cO(X)$, 
\[ \left(\sum\limits_{n=0}^\infty w^n \partial_x^{[n]}\right)(ab) = \sum\limits_{n=0}^\infty w^n\sum_{i=0}^n \partial^{[i]}_x(a) \partial_x^{[n-i]}(b) = \left(\sum\limits_{i=0}^\infty w^i \partial_x^{[i]}(a) \right) \left(\sum\limits_{j=0}^\infty w^j \partial_x^{[j]}(b)\right).\]
It follows that $\sigma_r(\beta(g)) : \cO(X) \to \cO(X)$ is a \emph{$K$-algebra homomorphism}, so for every integer $m \geq 0$ we have
\[\sigma_r(\beta(g))(x^m) = \sum\limits_{n=0}^\infty w^n \binom{m}{n} x^{m-n} = (w+x)^m = (g\cdot x)^m = \rho(g)(x^m).\]
Because $\sigma_r(\beta(g))$ and $\rho(g)$ are both continuous $K$-linear operators on $\cO(X)$, we conclude that $\sigma_r(\beta(g))\rho(g)^{-1} : \cO(X) \to \cO(X)$  is a $K$-algebra homomorphism which fixes $K\langle x \rangle = \cO(\bD) \subseteq \cO(X)$ pointwise. Since $\cO(X)$ contains a dense $\cO(\bD)$-subalgebra generated by elements of the form $a/b$ with $a,b \in \cO(\bD)$, the continuity of $\sigma_r(\beta(g))$ and $\rho(g)$ implies that $\sigma_r(\beta(g))\rho(g)^{-1}$ in fact fixes all of $\cO(X)$.
\end{proof}

\begin{thm}\label{GrGxTriv} Let $X$ be a $\partial_x/r$-admissible affinoid subdomain of $\bD$. Then
\[\beta : \cG_r \cap \cG_X \longrightarrow \cD_r(X)^\times\]
is a $\cG_X$-equivariant trivialisation.
\end{thm}
\begin{proof} According to Definition \ref{Triv} and \ref{DefnTriv}(b), we have to verify that
\be\item $\beta$ is a group homomorphism into $\cD_r(X)^\times$,
\item $\beta(h) \hsp a \hsp \beta(h)^{-1} = h \cdot a$ for all $h \in H := \cG_r \cap \cG_X$ and $a \in \cD_r(X)$,
\item $H$ is normal in $\cG_X$, and
\item $\beta(ghg^{-1}) = g \cdot \beta(h)$ for all $g \in \cG_X$ and $h \in H$.
\ee
(a) Using Proposition \ref{SigmaRho}, we note that
\[\sigma_r(\beta(g)\beta(h)) = \sigma_r(\beta(g))\sigma_r(\beta(h)) = \rho(g)\rho(h) = \rho(gh) = \sigma_r(\beta(gh))\]
for all $g,h \in H$. Since $\sigma_r$ is injective by Lemma \ref{ActionOnO}(b), we conclude that the map $\beta : \cG_r \to \cD_r(X)$ is multiplicative. Since $\beta(1) = 1$ it follows that the image of $\beta$ is contained in $\cD_r(X)^\times$.

(b) Using Proposition \ref{SigmaRho} and Lemma \ref{gQg} we see that  
\[\sigma_r(\beta(g) \hsp a \hsp \beta(g)^{-1}) = \rho(g) \hsp \sigma_r(a) \hsp \rho(g)^{-1} = \sigma_r(g \cdot a) \qmb{for all} a \in \cD_r(X).\] The result follows from the injectivity of $\sigma_r$ again.  

(c) It suffices to prove that $\cG_r$ is normal in $\cG$. For each $0 \neq s \in \overline{K}^\circ$, let $\cG(s) := \{g \in \cG : |g\cdot x - x| \leq |s|\}$; since $|\overline{K}^\times|$ is dense in $\bR$, we see that
\[ \cG_r = \underset{|s| < \varpi/r}{\bigcup}{} \cG(s).\]
Hence it suffices to see that each $\cG(s)$ is normal in $\cG$. However, the action of $\cG$ on $\overline{K}^\circ \langle x \rangle$ is $\overline{K}^\circ$-linear and therefore induces an action of $\cG$ on the polynomial ring $(\overline{K}^\circ / s \overline{K}^\circ)[x]$. Then $\cG(s)$ is the kernel of this action and is therefore normal in $\cG$.

(d) $\sigma_r(\beta(ghg^{-1})) = \rho(ghg^{-1}) = \rho(g)\circ\rho(h)\circ\rho(g)^{-1}= \rho(g) \circ\sigma_r(\beta(h)) \circ\rho(g)^{-1}$ by Proposition \ref{SigmaRho}. This last equals $\sigma_r(g \cdot \beta(h))$ by Lemma \ref{gQg}. Once again the result follows because $\sigma_r$ is injective. \end{proof}
It follows from Theorem \ref{GrGxTriv} and Lemma \ref{RingSGN} that whenever $X$ is a $\partial_x/r$-admissible affinoid subdomain of $\bD$, we may form the crossed product
\[ \cD_r(X) \hsp \overset{\hsp \beta}{\underset{\cG_r \cap \cG_X}{\rtimes}} \cG_X.\]
\begin{cor}\label{CPactsOnO} Let $X$ be a $\partial_x/r$-admissible affinoid subdomain of $\bD$. Then the action of $\cD_r(X) \rtimes \cG_X$ on $\cO(X)$ descends to an action of $\cD_r(X) \hsp \rtimes_{\cG_r \cap \cG_X}^{\beta} \cG_X$.
\end{cor}
\begin{proof} This follows immediately from Proposition \ref{SigmaRho}. \end{proof}

To finish $\S \ref{IwahoriAction}$, we record a formula which computes the effect of applying a twisting automorphism $\theta_{u,d}$ from Theorem \ref{ThetaU} to the special infinite-order differential operators of the form $\beta(g)$. Recall the functions $h_{u,d}^{[m]}$ from Definition \ref{HunDefn}.

\begin{prop}\label{TwistedCocycle} Let $X$ be a $\partial_x/r$-admissible affinoid subdomain of $\bD$, let $g \in \cG_r \cap \cG_X$, let $u \in \cO(X)^\times$ and suppose that $p \nmid d$.
\be \item The element $c_{u,d}(g) := \sum\limits_{m=0}^\infty (g \cdot x - x)^m h_{u,d}^{[m]}$ lies in $1 + \cO(X)^{\circ\circ}.$
\item Inside $\cD_r(X)$ we have
\[ \theta_{u,d}(\beta(g)) = c_{u,d}(g) \hsp \beta(g).\]
\item We have $c_{uv,d}(g) = c_{u,d}(g) \, c_{v,d}(g)$ for any other $v \in \cO(X)^\times$. 
\item We have $c_{u,d}(g)^d = \frac{u}{g \cdot u}$ in $\cO(X)^\times$.
\ee
\end{prop}
\begin{proof} (a) Write $w := g \cdot x - x$. Since $g \in \cG_r$ and $X \subseteq \bD$, we know that 
\[|w|_X \leq |w|_{\bD} < \varpi/r\]
by Definition \ref{Beta}(a). Also, $|h_{u,d}^{[m]}|_X \leq (r(X)/\varpi)^m$ by Theorem \ref{HunEst}, hence
\[ \left\vert w^m h_{u,d}^{[m]} \right|_X < \left(\frac{\varpi}{r}\right)^m \cdot \left(\frac{r(X)}{\varpi}\right)^m = \left(\frac{r(X)}{r}\right)^m \]
which tends to zero as $m \to \infty$ because $r(X) < r$. We conclude that $c_{u,d}(g)$ converges in $\cO(X)$ to an element of $1 + \cO(X)^{\circ\circ} \subset \cO(X)^\times$.

(b) By Definition \ref{Beta}(b), we have $\beta(g) = \sum\limits_{n=0}^\infty w^n \partial^{[n]}$. Therefore
\[\begin{array}{lll} \theta_{u,d}(\beta(g)) &=& \sum\limits_{n=0}^\infty w^n \sum\limits_{\alpha=0}^n h_{u,d}^{[n-\alpha]} \partial^{[\alpha]} = \\
&=& \sum\limits_{\alpha=0}^\infty w^\alpha \left( \sum\limits_{n=\alpha}^\infty w^{n-\alpha} h_{u,d}^{[n-\alpha]}\right) \partial^{[\alpha]} = \\
&=& \sum\limits_{\alpha=0}^\infty w^\alpha \left(\sum\limits_{m=0}^\infty w^m h_{u,d}^{[m]}\right) \partial^{[\alpha]} = \\
&=& c_{u,d}(g) \hsp \beta(g) \end{array}\]
by Theorem \ref{PnStdForm}, as claimed. 

(c) Lemma \ref{TwistsMultiply} implies that $\theta_{uv,d} = \theta_{u,d} \circ \theta_{v,d}$ as automorphisms of $\cD_r(X)$. Applying part (b) to the units $uv$, $v$ and $u$ in turn shows that
\[ c_{uv,d}(g) \beta(g) = \theta_{uv,d}(\beta(g)) = \theta_{u,d}( \theta_{v,d}(\beta(g)) = \theta_{u,d}( c_{v,d}(g) \beta(g) ) = c_{v,d}(g) c_{u,d}(g) \beta(g)\]
since $\theta_{u,d}$ is $\cO(X)$-linear. The claimed identity now follows because $\beta(g)$  is a unit in $\cD_r(X)$ by Theorem \ref{GrGxTriv}.

(d) By Lemma \ref{NearlyInner} together with Proposition \ref{SigmaRho}, inside $\cB(\cO(X))$ we have
\[\sigma_r(\theta_{u^d,d}(\beta(g)) = \sigma_r( u  \, \beta(g) \, u^{-1} ) = u \, \sigma_r(\beta(g)) \, u^{-1} = u \, \rho(g) \, u^{-1}. \]
But $\theta_{u^d,d}(\beta(g)) = c_{u^d,d}(g) \beta(g) = c_{u,d}(g)^d \beta(g)$ by (b,c) above, so
\[ u \, \rho(g) \, u^{-1} = \sigma_r( c_{u,d}(g)^d \beta(g) ) = c_{u,d}(g)^d \sigma_r (\beta(g)) = c_{u,d}(g)^d \rho(g)\]
again by Proposition \ref{SigmaRho}. Apply these operators to $1 \in \cO(X)$ to find
\[ c_{u,d}(g)^d = c_{u,d}(g)^d \rho(g) (1) = (u \, \rho(g) \, u^{-1})(1) = u \, \rho(g)(u^{-1}) = u / g \cdot u. \qedhere\]\end{proof}

\subsection{Equivariant line bundles with connection on local Drinfeld space}\label{CompSect}
Recall the generalised Iwahori subgroup $\cG$ of $\mathbb{GL}_2(K)$ from Definition \ref{GenIwahori}; it consists of matrices with coefficients in $K$. We will now focus on certain groups of matrices with entries in $F$. 

\begin{notn}\label{CongSubs} \,
\be\item Let $I=\cG\cap \mathbb{GL}_2(\cO_F)$ be the \emph{Iwahori subgroup}.
\item Let $J$ be a closed subgroup of $I$.
\item For each $n \geq 0$, write $I_n := \ker(I\to \mathbb{GL}_2(\cO_F / \pi_F^n \cO_F))$ and $J_n :=J\cap I_n$.
\item Let $\bD_n / J$ denote the $G$-topology on $\bD$ whose open sets are the $J$-stable members of $\bD_n$ and whose coverings are the finite coverings.
\ee\end{notn}
The $G$-topology $\bD_n$ on $\bD$ was introduced at Definition \ref{CurlyDn}.
\begin{defn}\label{UpperHPhyp2} \,
\be \item $[\sL] \in \PicCon^I( \Upsilon)_{\tors}$ is such that $\omega[\sL] \in \PicCon(\Upsilon)^I[p']$,
\item $d \geq 1$ is the order of $\omega[\sL]$ in $\PicCon(\Upsilon)^I$, 
\item $e \in d \bZ$ is the order of $[\sL]$ in $\PicCon^I(\Upsilon)$.
\ee \end{defn}
Thus $\sL$ is an $I$-equivariant line bundle with flat connection on $\Upsilon = \bD \cap \Omega$, and we can find an $I$-equivariant $\cD$-linear isomorphism 
\[ \psi' : \sL^{\otimes e} \stackrel{\cong}{\longrightarrow} \cO.\]

Clearly, whenever $J$ is as in Definition \ref{CongSubs}, $J_n$ is a normal subgroup of $J$; note also that $J_n$ is a pro-$p$ group whenever $n \geq 1$ because the groups $I_1 / I_{m+1}$ are all finite $p$-groups. Recall now the normal subgroups $\cG_r$ of $\cG$ from Definition \ref{Beta}.

\begin{lem}\label{JG} $I_{n+1} \leq \cG_r$ whenever $r < \varpi/|\pi_F|^{n+1}$.
\end{lem}
\begin{proof} If $g = \begin{pmatrix} g_{11} & g_{12} \\ g_{21} & g_{22} \end{pmatrix}\in I_{n+1}$ then $g_{11}-1, g_{12}, g_{21}, g_{22}-1$ all lie in $\pi_F^{n+1} \cO_F$, which implies that 
\[\max\{|g_{12}|,|g_{21}|,|g_{11}-g_{22}|\} \leq |\pi_F|^{n+1} <  \varpi/r,\] 
by our assumption on $r$. Hence $g \in \cG_r$ by Lemma \ref{GrCalc}.
\end{proof}

Recall $\beta\colon \cG_r\cap \cG_X\to \sD_r(X)^\times$ from Theorem \ref{GrGxTriv} for $X\in \bD_w$ and $r>r(X)$. 
\begin{prop}\label{SheafDnJ} Let $n \geq 0$ and let $H$ be a closed normal subgroup of $J$. Then there is a sheaf $\sD_n \rtimes_{H_{n+1}} J$ on $\bD_n/J$ whose sections over $X \in \bD_n/J$ are given by $\sD_n(X) \rtimes_{H_{n+1}}^{\beta}J$. This sheaf has vanishing higher \v{C}ech cohomology.
\end{prop}
\begin{proof} Let $X \in \bD_n/J$. Note that $\sD_n(X) = \cD^\dag_{\varpi/|\pi_F|^n}(X) = \bigcup\limits_{r > \varpi/ |\pi_F|^n} \cD_r(X)$ by Definition \ref{DagSite}(d). Let $r > 0$ satisfy
\[\varpi/|\pi_F|^n < r < \varpi/|\pi_F|^{n+1}.\]
Then because $r > \varpi/|\pi_F|^n \geq r(X)$, the $\cG_X$-equivariant trivialisation 
\[\beta : \cG_r \cap \cG_X \to \cD_r(X)^\times\]
is well-defined by Theorem \ref{GrGxTriv}. On the other hand $H_{n+1} \leq \cG_r$ by Lemma \ref{JG} because $r < \varpi/|\pi_F|^{n+1}$.  Since $X$ is $J$-stable, the restriction of $\beta$ to $H_{n+1}$ is a $J$-equivariant trivialisation, which means that we can form the crossed product $\cD_r(X) \rtimes_{H_{n+1}}^{\beta}J$ which is an associative ring by Lemma \ref{RingSGN}. Letting $r$ approach $\varpi/|\pi_F|^n$ from above and taking the colimit, we obtain the ring $\sD_n(X) \rtimes_{H_{n+1}}^\beta J$ which is still a crossed product of $\sD_n(X)$ with $J/H_{n+1}$.

This construction is functorial in $X \in \bD_n/J$, and in this way we obtain the presheaf $\sD_n \rtimes_{H_{n+1}} J$ on this $G$-topology. Because this presheaf is a free $\sD_n$-module on $J/H_{n+1}$, we conclude that it is a sheaf with vanishing higher \v{C}ech cohomology using Theorem \ref{NCDagTate}.
\end{proof}

Recall from Definition \ref{UnCovering} the truncation $\sL_n =  j_\ast(\sL_{\overline{\Upsilon_n}})$ of $\sL$ to $\Upsilon_n$, which is naturally a $\sD_n$-module on $\bD_n$ by Corollary \ref{LnDn}.

\begin{lem} \label{tensorLn}Suppose that $\sM$ is another $\cD_{\Upsilon}$-module satisfying the same hypotheses as $\sL$, and let $\sM_n=j_\ast(\sM_{\overline{\Upsilon_n}})$. 
\be \item $\sL_n\otimes_{(\cO_\bD)_{\overline{\Upsilon_n}}}\sM_n$ is naturally a $\sD_n$-module. 
\item Let $X \in \bD_n$, and $r>\varpi/|\pi_F|^n$. Then 
\[ \beta(g)\cdot(l\otimes v)=(\beta(g)\cdot l)\otimes (\beta(g)\cdot v)  \qmb{for all}  g\in \cG_r\cap \cG_X, l\in \sL_n(X), v\in \sM_n(X).\] 
\ee\end{lem}
\begin{proof} By Theorem \ref{AlgGensExist} and Proposition \ref{CleanLn} there is $\dot{w} \in \sM_n(V_n)$ such that $\sM_n=(\cO_\bD \dot{w})_{\overline{\Upsilon_n}}$. Thus by Lemma \ref{DagTensor}, $\sL_n\otimes_{(\cO_{\bD})_{\overline{\Upsilon_n}}}\sM_n\cong (\sL\otimes_{\cO_{\Omega\cap \bD}}\sM)_{\overline{\Upsilon_n}}$ so $\sD_n$ acts naturally on this by Corollary \ref{LnDn}. Now for each $m \geq 0$,
\[ \partial_x^{[m]}(l\otimes v)=\sum\limits_{i=0}^m\partial_x^{[i]}(l)\otimes \partial_x^{[m-i]}(v) \qmb{ for all } l\in \sL_n(X) \qmb{and} v\in \sM_n(X),\] 
so writing $w=g\cdot x-x$ as in Proposition \ref{SigmaRho} we see that 
\[ \beta(g)(l\otimes v)=\sum\limits_{m=0}^\infty w^m \sum\limits_{i=0}^m \partial_x^{[i]}(l)\otimes \partial_x^{[m-i]}(v)=\beta(g)(l)\otimes\beta(g)(v). \qedhere \]
\end{proof}

We may view $\sL_n$ as a \emph{$J$-equivariant} $\sD_n$-module on $\bD_n$, so its restriction to $\bD_n/J$ may be viewed as a $\sD_n \rtimes J$-module. 
\begin{thm}\label{DnJnJmodule} Let $H$ be a closed normal subgroup of $J$ and recall that $e$ is the order of $[\sL]$ in $\PicCon^I(\Upsilon)_{\tors}$. For each $n\geq v_{\pi_F}(e)$, the $\sD_n \rtimes J$-action on $\sL_n$ factors through $\sD_n \rtimes_{H_{n+1}} J$.
\end{thm}
\begin{proof} Since $H_{n+1}\leq J_{n+1}$ it suffices to consider the case $H=J$. We fix $n\geq v_{\pi_F}(e)$. Write $e=p^kf$ with $f$ coprime to $p$. Then there are tensor powers $\sL_p$ and $\sL_{p'}$ of $\sL$ such that $[\sL_p]$ has order $p^k$, $[\sL_{p'}]$ has order $f$ and \[ [\sL]=[\sL_p][\sL_{p'}]\in \PicCon^I(\Upsilon).\] 
Because $\omega[\sL]$ has order $d$ coprime to $p$ by Definition \ref{UpperHPhyp2}, the order of $\omega[\sL_p] = \omega[\sL] \cdot \omega[\sL_{p'}]^{-1}$ is also coprime to $p$. Hence $\omega [\sL_p] = [\cO_{\Upsilon}]$. Because $\Upsilon$ is geometrically connected and quasi-Stein, using \cite[Proposition 3.2.14]{ArdWad2023} we can find a character $\chi\in \Hom(I,\mu_{p^k}(K))$ such that $[\sL_p]=[\cO_{\chi}]$ in $\PicCon^I(\Upsilon)$. The assumption $n \geq v_{\pi_F}(e)$ implies that $I_n \leq I^{p^k}$, and therefore $\chi(I_n)=1$. It follows that the image of $[\sL]$ in $\PicCon^{I_n}(\Upsilon)$ under the restriction map has order $f$, and so we can choose an $I_n$-equivariant isomorphism of line bundles with connection \[\psi\colon \sL^{\otimes f}\stackrel{\cong}{\longrightarrow} \cO.\]  

	 Let $X \in\bD_n/ J$. Since $\Upsilon_n$ lies in $\bD_n/J$ as well by Corollary \ref{Heisenberg}, we see that $X \cap \Upsilon_n \in \bD_n/ J$ also. Let $g \in J_{n+1}$, so that $g \in \cG_r$ for any $r$ satisfying $\varpi/|\pi_F|^n < r < \varpi/|\pi_F|^{n+1}$ by Lemma \ref{JG}. Corollary \ref{CPactsOnO} implies that $\beta(g) - g \in \sD_n(X \cap \Upsilon_n) \rtimes J$ acts trivially on $\cO(X \cap \Upsilon_n)$; hence $\tilde{\beta}(g) := \beta(g)^{-1}g \in (\sD_n(X \cap \Upsilon_n) \rtimes J)^\times$ fixes $\cO(X \cap \Upsilon_n)$ pointwise. 

Let $A := \cO_{\overline{X \cap \Upsilon_n}}(X)$ and use Theorem \ref{AlgGensExist} to find an algebraic generator $\dot{z} \in \sL_n(X)$ with associated rational function $u = \psi(\dot{z}^{\otimes f})$; then $\sL_n(X) = A \cdot \dot{z}$ by Proposition \ref{CleanLn}. On the other hand, $\sL_n(X)$ is a $\sD_n(X)$-module by Corollary \ref{LnDn} and it is even a $\sD_n(X) \rtimes J$-module by the above remarks. Write $\tilde{\beta}(g) \cdot \dot{z} = c(g) \dot{z}$ for some $c(g) \in A$. Since $\psi$ is $J_n$-equivariant and since $\tilde{\beta}(g)$ fixes $A$ pointwise, we see via Lemma \ref{tensorLn} that for all $g \in J_{n+1}$,
\[  c(g)^{f} u = \psi( (c(g) \dot{z})^{\otimes {f}}) = \psi((\tilde{\beta}(g) \cdot \dot{z})^{\otimes {f}}) = \tilde{\beta}(g) \cdot u =u\]
which implies that $c(g)^{f} = 1$ for all $g \in J_{n+1}.$ Because $\beta : J_{n+1} \to \sD_n(X)^\times$ is a trivialisation, it follows from \cite[Lemma 2.2.2]{EqDCap} that $\tilde{\beta}(gh) = \tilde{\beta}(g)\tilde{\beta}(h)$ for all $g,h \in J_{n+1}$. Now, $\tilde{\beta}(g)$ fixes $c(h) \in A$ and therefore commutes with it in $\sD_n(X \cap \Upsilon_n) \rtimes J$. Therefore

\[\begin{array}{lllllll} c(gh) \dot{z} &=& \tilde{\beta}(gh)\dot{z} &=& \tilde{\beta}(g) \tilde{\beta}(h) \dot{z} &=& \tilde{\beta}(g) c(h) \dot{z}  \\
&=& c(h) \tilde{\beta}(g) \dot{z} &=& c(h) c(g) \dot{z} &=& c(g) c(h) \dot{z}.\end{array}\]
Hence $c : J_{n+1} \to A^\times$ is a group homomorphism such that $c(g)^f =  1$ for all $g \in J_{n+1}$. Since $J_{n+1}$ is a pro-$p$ group and since $p \nmid f$, for each $g \in J_{n+1}$ there exists $h \in J_{n+1}$ such that $g = h^f$. Hence $c(g) = c(h^f) = c(h)^f = 1$ for all $g \in J_{n+1}$ and we conclude that $\tilde{\beta}(g) \cdot \dot{z} = \dot{z}$. Hence $\tilde{\beta}(J_{n+1})$ fixes $\sL_n(X) = A \cdot \dot{z}$ pointwise and the result follows.
\end{proof}

Since $\Upsilon_{n+1}$ contains $\Upsilon_n$, there is a canonical $\sD_{n+1}(\bD)$-linear restriction map
\[\sL_{n+1}(\bD) \longrightarrow \sL_n(\bD).\]
Let $X \in \bD_n$. Post-composing $\sL_{n+1}(\bD) \longrightarrow \sL_n(\bD)$ with the restriction map $\sL_n(\bD) \to \sL_n(X)$, in view of Corollary \ref{LnDn} we obtain a morphism 
\begin{equation}\label{CompMap} \sD_n \underset{\sD_{n+1}(\bD)}{\otimes}{} \sL_{n+1}(\bD) \longrightarrow \sL_n\end{equation}
of sheaves of $\sD_n$-modules on $\bD_n$. For the moment, suppose further that $X$ is $J$-stable. Then we find ourselves in the setting of Hypothesis \ref{MorphCP} with $G := J$, $f : S' \to S$ the restriction map $S' := \sD_{n+1}(\bD) \longrightarrow S := \sD_n(X)$, $H' := J_{n+2}$, $H := J_{n+1}$, and the trivialisations $\beta'$ and $\beta$ used in Proposition \ref{SheafDnJ}. 

Then $M' := \sL_{n+1}(\bD)$ is an $S' \rtimes_{H'} G$-module by Theorem \ref{DnJnJmodule}, so by Lemma \ref{SecretAction} there is a secret $H = J_{n+1}$-action on 
\[S \underset{S'}{\otimes}M' = \sD_n(X) \underset{\sD_{n+1}(\bD)}{\otimes}{} \sL_{n+1}(\bD)\]
by $\sD_n(X)$-linear endomorphisms, given by
\[h \star (s \otimes m') = s \hsp \beta(h)^{-1} \otimes h \cdot m' \qmb{for all} h \in H, s \in S, m' \in M'.\]
Inspecting this formula, we see that in fact it makes sense for any $X \in \bD_n$, not necessarily $J$-stable, because we can view $\beta(h)$ as an element of $\sD_n(\bD)^\times$ and then consider its image under the $\sD_n(X)^\times$. We will now study the module of coinvariants $(S \otimes_{S'}M)_H$ with respect to this secret $H$-action.
\begin{prop} Whenever $n \geq v_{\pi_F}(e)$, $X \mapsto \left(\sD_n(X) \underset{\sD_{n+1}(\bD)}{\otimes}{} \sL_{n+1}(\bD) \right)_{J_{n+1}}$ is a sheaf on $\sD_n$-modules on $\bD_n$, with vanishing higher \v{C}ech cohomology.
\end{prop}
\begin{proof} Whenever $Y \subseteq X$ are members of $\bD_n$, the restriction maps 
\[\sD_n(X) \underset{\sD_{n+1}(\bD)}{\otimes}{} \sL_{n+1}(\bD) \longrightarrow \sD_n(Y) \underset{\sD_{n+1}(\bD)}{\otimes}{} \sL_{n+1}(\bD)\] 
are equivariant with respect to the secret action, so we do have a presheaf of left $\sD_n$-modules. By Theorem \ref{AlgGensExist} we may choose an algebraic generator $\dot{z}$ for $\sL_{n+1}(\bD)$ with associated rational function $u$. We write $r := R_{\cS(u)}(u,d)$ for brevity. By Corollary \ref{MnPres}, we have an exact sequence of $\sD_{n+1}(\bD)$-modules
\[ 0 \to \sD_{n+1}(\bD) \stackrel{\cdot r}{\longrightarrow} \sD_{n+1}(\bD) \to \sL_{n+1}(\bD)  \to 0.\]
Applying the functor $\sD_n \otimes_{\sD_{n+1}(\bD)}-$ to this sequence, we obtain the following sequence of presheaves of $\sD_n$-modules on $\bD_n$:
\[ 0 \to \sD_n \stackrel{\cdot r}{\longrightarrow} \sD_n \longrightarrow \sD_n \underset{\sD_{n+1}(\bD)}{\otimes}{} \sL_{n+1}(\bD)  \to 0\]
which is exact because $\sD_n(X)$ is a flat right $\sD_{n+1}(\bD)$-module by Lemma \ref{Flat}(a,b) for any $X \in \bD_n$. Let $\cX$ be a $\bD_n$-covering. Then
\[0 \to C^{\aug}(\cX, \sD_n) \stackrel{\cdot r}{\longrightarrow} C^{\aug}(\cX,\sD_n) \to C^{\aug}(\cX, \sD_n \underset{\sD_{n+1}(\bD)}{\otimes}{} \sL_{n+1}(\bD)) \to 0\]
is an exact sequence of augmented \v{C}ech complexes. Since the first two of these complexes is exact by Theorem \ref{NCDagTate}, so is the third. Now, $\sD_n \underset{\sD_{n+1}(\bD)}{\otimes}{} \sL_{n+1}(\bD)$ carries a secret $\sD_n$-linear $J_{n+1}$-action, and the operation of taking $J_{n+1}$-coinvariants is exact since $K$ is a field of characteristic zero, and the action of $J_{n+1}$ factors through the finite group $J_{n+1} /J_{n+2}$ by Lemma \ref{SecretAction}. Therefore the complex $C^{\aug}(\cX, (\sD_n \underset{\sD_{n+1}(\bD)}{\otimes}{} \sL_{n+1}(\bD))_{J_{n+1}})$ is exact, and the result follows.
\end{proof}

We denote this sheaf by $\left(\sD_n \underset{\sD_{n+1}(\bD)}{\otimes}{} \sL_{n+1}(\bD) \right)_{J_{n+1}}$.
\begin{lem}\label{SheafyCompMap} The comparison map $(\ref{CompMap})$ induces a morphism 
\[  \Xi : \left(\sD_n \underset{\sD_{n+1}(\bD)}{\otimes}{} \sL_{n+1}(\bD) \right)_{J_{n+1}} \longrightarrow \sL_n\]
of sheaves of $\sD_n$-modules on $\bD_n$, whenever $n \geq v_{\pi_F}(e)$.
\end{lem}
\begin{proof} Let $X \in \bD_n $. The comparison map from $(\ref{CompMap})$ factors through the $J_{n+1}$-coinvariants of $\sD_n(X) \otimes_{\sD_{n+1}(\bD)} \sL_{n+1}(\bD)$, because $\beta(h)^{-1} h$ fixes $\sL_n(X)$ for any $h \in J_{n+1}$ by Theorem \ref{DnJnJmodule}, applied with $J_{n+1}$ instead of $J$.
\end{proof}

Our goal is now to show that under a certain mild additional condition on the group $J$, the connecting map $\Xi$ is in fact an isomorphism. We will verify that this is the case \emph{locally} on the $G$-topology $\bD_n$. We will use a particular covering of $\bD$, similar to the covering of $\bA$ appearing in Lemma \ref{CheckCov} above.

\begin{defn}\label{Win} For each $a\in \cO_F$ define 
\[ B_{a,n} := C_K(a;\pi^{n})\qmb{and}W_{a,n} := C_K(a, b_1, \ldots, b_{q-1};\pi_F^n,\ldots, \pi_F^n)\]where $b_1,\ldots,b_{q-1}\in a+\pi_F^n\cO_F$ are chosen so that   \[\pi^n\cO_F/\pi^{n+1}\cO_F=\{\pi_F^{n+1}\cO_F, (b_1-a)+\pi^{n+1}_F\cO_F,(b_{q-1}-a) +\pi^{n+1}_F\cO_F\}.\]   

We also define $\cU_n=\{ \Upsilon_n, W_{a,n}\st a\in \cO_F\}$. 
\end{defn}

Recall here that $\Upsilon_n = \bD \backslash \bigcup\limits_{a \in \cO_F} \{|z - a| < |\pi_F|^n\}$ --- see Definition \ref{UnCovering}(a,b).

We note that $W_{a,n}$ does not depend on the choices of $b_1,\ldots,b_{q-1}$ and both $W_{a,n}$ and $B_{a,n}$ only depend on the image of $a$ in $\cO_F/\pi^{n+1}\cO_F$. Thus $|\cU_n|=h_n+1$. We also note that $W_{a,n}\cap \cO_F=a+\pi^{n+1}\cO_F$ by construction. 

\begin{lem} \label{CheckCov2} $\cU_n$ is an $\bD_n / I_{n+1}$-admissible covering of $\D$.
\end{lem}
\begin{proof} Let $\{a_1,\ldots,a_{h_n}\}$ be a set of coset representatives for $\cO_F/\pi^{n+1}\cO_F$ so that $\cU_n=\{W_{a_1},\ldots, W_{a_{h_n}}, \Upsilon_n\}$. Because $\min\limits_{i\neq j} |a_i- a_j| = |\pi_F|^n$, it follows from Lemma \ref{CheckCov} that $\cU_n$ is indeed a $\bD_n$-admissible covering of $\bD$, so it remains to verify that each of its members is $I_{n+1}$-stable. 
	
	Now $\Upsilon_n$ itself is $I_{n+1}$-stable: it is even stable under the Iwahori subgroup $\mathbb{GL}_2(\cO_F) \cap \cG$ because it permutes connected components of $\bD - V_n$. 
	
	Fix $g \in I_{n+1}$, and $a,b_1,\ldots,b_{q-1}$ in $\cO_F$ so that \[W_{a,n}= C_K(a, b_1, \ldots, b_{q-1};\pi_F^n,\ldots, \pi_F^n).\] Moreover let $z \in W_{a,n}(\overline{K})$. Then $|z - a| \leq |\pi_F|^n$ and $|z - b_j| \geq |\pi_F|^n$ for each $j=1,\ldots,q-1$. Because each entry of $g - 1$ lies in $\pi_F^{n+1} \cO_F$, we see that
	\[ |g \cdot z - z| \leq |\pi_F^{n+1}|.\]
	Hence $|g\cdot z - a| \leq |\pi_F|^n$ and $|g\cdot z - b_j| = |z - b_j| \geq |\pi_F|^n$ for each $j=1,\ldots q-1$. We conclude that $g\cdot z \in W_{a,n}(\overline{K})$ as well. So $W_{a,n}$ is $I_{n+1}$-stable.\end{proof}

\begin{defn}\label{DefOfN} We introduce the following standard subgroup of $\GL_2(\cO_F)$:
\[N := \begin{pmatrix} 1 & \cO_F \\ 0 & 1 \end{pmatrix} \quad \leq \quad \GL_2(\cO_F).\]
\end{defn}

We can now state the main technical result of this part of the paper. 

\textbf{We assume until the end of $\S \ref{CoadDrin}$ that the field $K$ is discretely valued.}
\begin{thm}\label{CompOnW} Fix $n \geq v_{\pi_F}(e)$, $a \in \cO_F$ and write $W := W_{a,n}$. Assume that $d \mid (q+1)$ and $N_{n+1} \leq J_{n+1}$. Then the following comparison map is an isomorphism:
\[\Xi(W) : \left(\sD_n(W) \underset{\sD_{n+1}(W)}{\otimes}{} \sL_{n+1}(W)\right)_{J_{n+1}} \stackrel{\cong}{\longrightarrow} \sL_n(W).\]
\end{thm}
The proof of Theorem \ref{CompOnW} will occupy the entirety of $\S \ref{PfOfCompOnW}$ and can be found at the end of $\S \ref{PfOfCOW}$. 

In the remainder of $\S\ref{CompSect}$ we will use Theorem \ref{CompOnW} to derive some important consequences. To do this, we need two more technical Lemmas.

\begin{lem}\label{GenByGS} Let $Y \subseteq X$ both lie in $\bD_n$. Then the action map 
\[\sD_n(Y) \underset{\sD_n(X)}{\otimes}{} \sM(X) \longrightarrow \sM(Y)\] is an isomorphism when $\sM$ is one of the following $\sD_n$-modules:
\be \item  $\sM = \sL_n$, or
\item $\sM = \left(\sD_n \underset{\sD_{n+1}(\bD)}{\otimes}{} \sL_{n+1}(\bD)\right)_{J_{n+1}}$ and $n \geq v_{\pi_F}(e)$.
\ee\end{lem}
\begin{proof} (a) Using Theorem \ref{AlgGensExist}, choose an algebraic generator $\dot{z}$ for $\sL_n(\bD)$ with associated rational function $u$, write $r := R_{\cS(u)}(u,d)$ and consider the following commutative diagram:
\[\xymatrix{\sD_n(Y) \underset{\sD_n(X)}{\otimes} \sD_n(X) \ar[r]^{\cdot r}\ar[d] & \sD_n(Y) \underset{\sD_n(X)}{\otimes} \sD_n(X) \ar[r]\ar[d] & \sD_n(Y) \underset{\sD_n(X)}{\otimes} \sL_n(X) \ar[r]\ar@{..>}[d] & 0 \\
\sD_n(Y) \ar[r]_{\cdot r} & \sD_n(Y) \ar[r] &  \sL_n(Y) \ar[r] & 0.
}\]
The rows are exact by Corollary \ref{MnPres}. The first two vertical maps are isomorphisms, so the third vertical map is also an isomorphism by the Five Lemma. This deals with the case $\sM = \sL_n$. 

(b) Now, using Theorem \ref{AlgGensExist} again, choose an algebraic generator $\dot{z}$ for $\sL_{n+1}(\bD)$ with associated rational function $u$. Write $D := \sD_n(X)$, $D' := \sD_n(Y)$, $r := R_{\cS(u)}(u,d)$ and $H := J_{n+1}$ for brevity, and let $\epsilon \in K[J_{n+1}/J_{n+2}]$ be the principal idempotent --- the average of all the group elements of the finite group $J_{n+1}/J_{n+2}$. Then the sequence $D/Dr \stackrel{1 - \epsilon}{\longrightarrow} D/Dr \longrightarrow (D/Dr)_H \to 0$ is exact, so the first row in the commutative diagram
\[\xymatrix{ D' \underset{D}{\otimes}{} \frac{D}{Dr} \ar[rr]^{1 \hsp \otimes \hsp (1-\epsilon)\star} \ar[d] && D' \underset{D}{\otimes}{} \frac{D}{Dr} \ar[r]\ar[d] & D'\underset{D}{\otimes}{} (\frac{D}{Dr})_H \ar[r]\ar@{..>}[d]  & 0 \\ 
\frac{D'}{D'r} \ar[rr]_{(1-\epsilon)\star} && \frac{D'}{D'r} \ar[r] & (\frac{D'}{D'r})_H \ar[r] & 0
}\]
is exact. The second row is exact for the same reason, and the first two vertical maps are isomorphisms. Hence the third map is an isomorphism by the Five Lemma.

By Corollary \ref{MnPres}, we have isomorphisms $\sM(X) \cong (D/Dr)_H$ and $\sM(Y) \cong (D'/D'r)_H$. Hence $\sD_n(Y) \underset{\sD_n(X)}{\otimes}{} \sM(X) \longrightarrow \sM(Y)$ is also an isomorphism. \end{proof}

\begin{lem}\label{RestrictLn} For each $X \in \bD_n$, the restriction map $\sL_{n+1}(\bD) \to \sL_{n+1}(X)$ induces a natural $\sD_n(X)$-linear isomorphism
\[ \sD_n(X) \underset{\sD_{n+1}(\bD)}{\otimes}{} \sL_{n+1}(\bD) \stackrel{\cong}{\longrightarrow} \sD_n(X) \underset{\sD_{n+1}(X)}{\otimes}{} \sL_{n+1}(X).\]
\end{lem}
\begin{proof} By Lemma \ref{GenByGS}(a) applied with $n$ replaced by $n+1$, the action map 
\[ \sD_{n+1}(X) \underset{\sD_{n+1}(\bD)}{\otimes}{} \sL_{n+1}(\bD) \longrightarrow \sL_{n+1}(X)\]
is an isomorphism. Now apply the functor $\sD_n(X) \underset{\sD_{n+1}(X)}{\otimes}{} -$ to this isomorphism, and use the associativity of tensor product.
\end{proof}

Next, we deal with the `easy' member of the covering $\cU_n$, namely $\Upsilon_n$.
\begin{prop}\label{CompOnUn} The comparison map
\[\Xi(\Upsilon_n) : \left(\sD_n(\Upsilon_n) \underset{\sD_{n+1}(\Upsilon_n)}{\otimes}{} \sL_{n+1}(\Upsilon_n)\right)_{J_{n+1}} \longrightarrow \sL_n(\Upsilon_n)\]
is an isomorphism for each $n \geq 0$.
\end{prop}
\begin{proof} The map $\Xi(\Upsilon_n)$ appears in the following commutative diagram:
\[ \xymatrix{ \sD_n(\Upsilon_n) \underset{\sD_{n+1}(\Upsilon_n)}{\otimes}{} \sL_{n+1}(\Upsilon_n) \ar[rr]^(0.6){\varphi} \ar@{->>}[d] & & \sL_n(\Upsilon_n) \\  \left(\sD_n(\Upsilon_n) \underset{\sD_{n+1}(\Upsilon_n)}{\otimes}{} \sL_{n+1}(\Upsilon_n)\right)_{J_{n+1}} \ar[rru]_{\Xi(\Upsilon_n)} & & } \]
Now, $\varphi$ is non-zero and $\sL_n(\Upsilon_n)$ is simple as a $\sD_n(\Upsilon_n)$-module by Corollary \ref{LnYsimple}, so $\varphi$ must be surjective. On the other hand, $\sL_n(\Upsilon_n)$ is a free $\cO(\Upsilon_n)$-module of rank $1$ by Corollary \ref{LnYsimple}, and examining the proof of Corollary \ref{LnYsimple} we see that $\sD_n(\Upsilon_n) \underset{\sD_{n+1}(\Upsilon_n)}{\otimes}{} \sL_{n+1}(\Upsilon_n)$ is a free $\cO(\Upsilon_n)$-module of rank one as well. Because a surjective $\cO(\Upsilon_n)$-linear map between two such modules must be injective, we conclude that $\varphi$ is an isomorphism. 

It follows that the vertical map is injective. Since it is clearly surjective, it is an isomorphism. Therefore $\Xi(\Upsilon_n)$ is also an isomorphism.
\end{proof}

The next two statements are the main results of $\S \ref{CompSect}$.

\begin{thm}\label{XiIso} Assume that $n \geq v_{\pi_F}(e)$, $N_{n+1} \leq J_{n+1}$ and $d \mid (q+1)$. Then the comparison map from Lemma \ref{SheafyCompMap} 
\[  \Xi : \left(\sD_n \underset{\sD_{n+1}(\bD)}{\otimes}{} \sL_{n+1}(\bD) \right)_{J_{n+1}} \longrightarrow \sL_n\]
is an isomorphism of sheaves of $\sD_n$-modules on $\bD_n$.
\end{thm}
\begin{proof} Since  $\cU_n$ is a covering of $\bD$ by Lemma \ref{CheckCov2} and since $\Xi$ is a map of sheaves, it is enough to show $\Xi_{|Y}$ is an isomorphism for each $Y \in \cU_n$. After Lemma \ref{GenByGS}, it is enough to check that $\Xi(Y)$ is an isomorphism for each $Y \in \cU_n$. However in view of Lemma \ref{RestrictLn}, this follows from Proposition \ref{CompOnUn} and Theorem \ref{CompOnW}.
\end{proof}

\begin{cor}\label{JnGenBySn} Assume that $n \geq v_{\pi_F}(e)$, $N_{n+1} \leq J_{n+1}$ and $d \mid (q+1)$. Then for each $X \in \bD_n/J$, the action map
\[ \sD_n(X) \underset{J_{n+1}}{\rtimes}{} J \quad \underset{ \sD_{n+1}(\bD) \underset{J_{n+2}}{\rtimes}{}  J  }{\otimes}{} \quad \sL_{n+1}(\bD) \quad \longrightarrow \quad \sL_n(X)\]
is an $\sD_n(X) \underset{J_{n+1}}{\rtimes}{} J$-linear isomorphism.
\end{cor}
\begin{proof} The action map in question is an instance of the map $\widetilde{\tau}$ from (\ref{TauTilde}); hence it is $\sD_n(X) \rtimes_{J_{n+1}} J$-linear. It appears as the diagonal map in the following commutative triangle, where the vertical arrow is the isomorphism from Theorem \ref{IndModCoinv}:
\[\xymatrix{ \sD_n(X) \underset{J_{n+1}}{\rtimes}{} J \quad \underset{ \sD_{n+1}(\bD) \underset{J_{n+2}}{\rtimes}{} J  }{\otimes}{} \quad \sL_{n+1}(\bD) \ar[rrdd]^{\widetilde{\tau}} \ar[dd]_{\cong} & & \\ & & \\
\left( \sD_n(X) \underset{\sD_{n+1}(\bD)}{\otimes}{} \sL_{n+1}(\bD)\right)_{J_{n+1}} \ar[rr]_(0.65){\Xi(X)} & & \sL_n(X).}\]
Since $\Xi(X)$ is an isomorphism by Theorem \ref{XiIso}, this diagonal map must also be an isomorphism. \end{proof}

\section{Microlocal analysis}\label{PfOfCompOnW}

\subsection{Some microlocal rings}\label{MicRings} We fix an affinoid subdomain $X$ of $\A$ in this section. Recall from Definition \ref{dxradm} that $r(X)$ denotes the spectral seminorm of $\partial_x \in \cB(\cO(X))$ and recall the notation $A\langle \partial/r, s/\partial \rangle$ from Definition \ref{AdelSR}.
\begin{defn}\label{Est} Define $\cE_{[s,r]}(X) := \cO(X) \langle \partial/r, s/\partial \rangle$ for every $r \geq s \geq r(X)$.
\end{defn}
By Theorem \ref{skewannulus}, $\cE_{[s,r]}(X)$ is an associative $K$-Banach algebra with the norm
\begin{equation}\label{MicroNorm} \left\vert \sum\limits_{n \in \Z} a_n \partial^n \right\vert = \sup\limits_{s \leq t \leq r} \sup_{n \in \Z} |a_n| t^n\end{equation}
which contains an inverse $\partial^{-1}$ for $\partial$. By Lemma \ref{BanachStarProd}, there is a natural isometric embedding of $K$-Banach algebras $\cD_r(X) \hookrightarrow \cE_{[s,r]}(X)$ given by regarding a power series in $\partial$ as a Laurent series with zero negative terms.

Now let $u \in \cO(X)^\times$ be a unit as in $\S \ref{DivTwDer}$, let $d$ be a non-zero integer coprime to $p$ and let $z$ satisfy the equation $z^d = u$. We will show that each $\cE_{[s,r]}(X)$ contains an inverse that we will denote by $\xi_{u,d}$ to the twisted derivation $\theta_{u,d}(\partial_x)$ from $\S \ref{DivTwDer}$. Recall the coefficients $h_{u, d}^{[n]} = \theta_{u,d}(\partial_x^{[n]})(1) = z \partial_x^{[n]}(z^{-1})$ from Definition \ref{HunDefn}.

\begin{lem}\label{rhz} For every $\ell \geq 0$ we have
\be \item $(\ell+1) h_{u, d}^{[\ell+1]} = \sum\limits_{n=0}^\ell h_{u,d}^{[n]} \partial_x^{[\ell-n]}(h_{u, d}^{[1]})$, and
\item $(\ell+1) h_{u, d}^{[\ell+1]} = \partial_x( h_{u, d}^{[\ell]} ) + h_{u, d}^{[1]} h_{u, d}^{[\ell]} $.
\ee\end{lem}
\begin{proof} (a) By the Leibniz rule, $\partial_x^{[\ell]}(ab) = \sum\limits_{n=0}^\ell \partial_x^{[n]}(a) \partial_x^{[\ell-n]}(b)$. Applying this with $a = z^{-1}$ and $b = z \partial_x(z^{-1}) = h_{u, d}^{[1]}$, we have
\[\begin{array}{lll} (\ell+1)h_{u, d}^{[\ell+1]} &=& z\partial_x^{[\ell]}(\partial_x(z^{-1}))  \\
&=& z \sum\limits_{n=0}^\ell \partial_x^{[n]}(z^{-1}) \hsp \partial_x^{[\ell-n]}(h_{u, d}^{[1]}) \\
&=& \hsp \sum\limits_{n=0}^\ell  h_{u,d}^{[n]} \partial_x^{[\ell-n]}(h_{u, d}^{[1]}).\end{array}\]
(b) We expand the right hand side:
\[\begin{array}{lll}\partial_x(h_{u, d}^{[\ell]}) + h_{u, d}^{[1]} h_{u, d}^{[\ell]} &=& \partial_x(z \partial_x^{[\ell]}(z^{-1})) + z \partial_x(z^{-1}) \cdot z \partial_x^{[\ell]}(z^{-1}) \\
&=& \partial_x(z) \partial_x^{[\ell]}(z^{-1}) + z \partial_x(\partial_x^{[\ell]}(z^{-1})) - z^{-1} \partial_x(z) \cdot z \partial_x^{[\ell]}(z^{-1}) \\
&=& (\ell+1) h_{u, d}^{[\ell+1]}. 
\end{array}\]
We have used here that $z \partial_x(z^{-1}) + z^{-1} \partial_x(z) = \partial_x(z^{-1}\cdot z) = \partial_x(1) = 0$. \end{proof}
\begin{defn}\label{XiZ} Define $\xi_{u,d} := \sum\limits_{n = -\infty}^\infty (\xi_{u,d})_n \partial^n \in \prod\limits_{n= -\infty}^{-1} \cO(X) \partial^n$, where
\[ (\xi_{u,d})_{-n} = \left\{ \begin{array}{cl}  (-1)^{n-1} (n-1)! h_{u, d}^{[n-1]}  & \mbox{if }  n \geq 1, \\ 0 & \mbox{otherwise.}  \end{array} \right.\]
\end{defn}

\begin{lem}\label{XiZConv} The element $\xi_{u,d}$ lies in $\cE_{[s,r]}(X)$ whenever  $r \geq s > r(X)$.
\end{lem}
\begin{proof} Since $(\xi_{u,d})_n = 0$ for any $n \geq 0$, by Definition \ref{AdelSR} it is enough to prove that for if $s \leq t \leq r$, then $|(\xi_{u,d})_{-n}| t^{-n} \to 0$ as $n \to +\infty$. Now for $n \geq 1$ we have
\[\frac{|(\xi_{u,d})_{-n}|}{t^n} = \frac{|(n-1)! h_{u, d}^{[n-1]}|}{t^n} \leq  \frac{|(n-1)!|}{t^n} \left(\frac{r(X)}{\varpi}\right)^{n-1}  \leq \frac{p (n-1)}{t} \left(\frac{r(X)}{t}\right)^{n-1} \]
by Lemma \ref{nFacVarpi} and Proposition \ref{HunEst}. This converges to zero as $t \geq s > r(X)$.
\end{proof}
\begin{prop}\label{MicroInverse} Suppose that $r \geq s > r(X)$. Then $\xi_{u,d}$ is a two-sided inverse to $\theta_{u,d}(\partial_x)$ in $\cE_{[s,r]}(X)$.
\end{prop}
\begin{proof} First note that $\xi_{u,d}$ does lie in $\cE_{[s,r]}(X)$ by Lemma \ref{XiZConv}. Let $a = \xi_{u,d}$ and $b = \theta_{u,d}(\partial_x)$ so that 
\[a_{-n} = \left\{ \begin{array}{lll} (-1)^{n-1} (n-1)! h_{u, d}^{[n-1]} & \qmb{if}  n \geq 1 \\ 0 & \qmb{otherwise,}  \end{array} \right. \quad b_j = \left\{ \begin{array}{lll} h_{u, d}^{[1]} & \qmb{if} j = 0, \\ 1 &\qmb{if} j=1, \\ 0 & \qmb{otherwise.} \end{array}\right.\]
Note that $\binom{-n}{m} = (-1)^m \binom{n+m-1}{m}$ for any $n,m \geq 0$. Let $\ell \in \Z$; then using $(\ref{MicroMult})$ we have
\[\begin{array}{lll} a \ast_{-1-\ell} b &=& \sum\limits_{n=1}^\infty a_{-n} \sum\limits_{m = 0}^\infty \binom{-n}{m}\partial_x^m(b_{m+n-1-\ell})  \\
&=& \sum\limits_{n=1}^\infty (-1)^{n-1} (n-1)! h_{u, d}^{[n-1]} \sum\limits_{m=0}^\infty (-1)^m \binom{n+m-1}{m} \partial_x^m(b_{m+n-1-\ell})  \\
&=& \sum\limits_{n=0}^\infty h_{u,d}^{[n]} \sum\limits_{m=0}^\infty (-1)^{n+m} (n+m)! \partial_x^{[m]}(b_{m+n-\ell}). \end{array}\]
Now since $b_j$ vanishes unless $j \in \{0,1\}$ and since $m + n \geq 0$ in this sum, we see that this expression vanishes whenever $\ell \leq -2$, i.e. $-1-\ell \geq 1$. When $\ell = -1$ (so that $-\ell-1 = 0$), $m + n - \ell \in \{0,1\}$ forces $m = n = 0$ and the expression collapses to give $1$. So we will assume that $\ell \geq 0$. Now we still have to have $m + n - \ell \in \{0,1\}$ otherwise $b_{n+m-\ell}$ vanishes. If $m + n - \ell= 1$ then $\partial_x^{[m]}(b_{m+n-\ell}) = \partial_x^{[m]}(1) = 0$ unless $m = 0$; so this happens only when $m = 0$ and $n = \ell + 1$; in this case we get a contribution of
\[h_{u, d}^{[\ell+1]} (-1)^{\ell+1} (\ell+1)!\]
to the big sum. On the other hand, if $m + n - \ell = 0$ then $m = \ell - n \geq 0$ forces $n \leq \ell$, and we obtain a contribution of
\[ \sum_{n=0}^\ell h_{u,d}^{[n]} (-1)^\ell \ell! \partial_x^{[\ell-n]}(h_{u, d}^{[1]})\]
to the big sum. This is $(-1)^\ell \ell! (\ell+1) h_{u, d}^{[\ell+1]} = (-1)^\ell (\ell+1)! h_{u, d}^{[\ell+1]}$ by Lemma \ref{rhz}(a) and it therefore cancels with the first term. This proves that
\[ a \ast_k b = \delta_{k,0} \qmb{for all} k \in \Z\]
and therefore $\xi_{u,d} \cdot \theta_{u,d}(\partial_x) = 1$. Showing that $\theta_{u,d}(\partial_x) \cdot \xi_{u,d} = 1$ using $(\ref{MicroMult})$ is similar but much easier, and reduces to Lemma \ref{rhz}(b). We leave the details to the reader.
\end{proof}

\begin{defn} For each affinoid subdomain $X$ of $\A$, we define the \emph{Robba ring}
\[ \overset{\circ}{\cE}(X) := \bigcup\limits_{r > r(X)} \bigcap\limits_{r(X) < s \leq r} \cE_{[s,r]}(X).\]
\end{defn}

\begin{cor}\label{RobbaInverts} $\theta_{u,d}(\partial_x)$ is a unit in $\overset{\circ}{\cE}(X)$ with inverse $\xi_{u,d}$.
\end{cor}
\begin{proof} We know that $u = z^d$ is a unit in $\cO(X)$. Now apply Proposition \ref{MicroInverse} to see that whenever $r(X) < s \leq r$, $\xi_{u,d} \in \cE_{[s,r]}(X)$ is a two-sided inverse to $\theta_{u,d}(\partial_x)$ in $\cE_{[s,r]}$. It follows that
\[\xi_{u,d} \in \bigcap\limits_{r(X)< s \leq r} \cE_{[s,r]}(X) \subset \overset{\circ}{\cE}(X). \qedhere\]
\end{proof}

\subsection{The characteristic cycle}\label{CharCycleSect}
We now specialise to the case where $X \subseteq \bD$ is a cheese with $\rho(X) = 1$, and the ground field $K$ is discretely valued, with uniformiser $\pi_K$ and residue field $k$. In this case, $\sX := \Spf \cO(X)^\circ$ is a smooth affine formal scheme, whose special fibre $\sX_0$ is the complement of finitely many points in the affine line over $k$. In \cite{AbeMicro}, Tomoyuki Abe defined several sheaves of microlocal rings
\[ \pi^{-1} \Berth{X}{m} \subset \Abe{X}{m,\dag} \subset \cdots \subset \Abe{X}{m,m+2} \subset  \Abe{X}{m,m+1} \subset \Abe{X}{m,m} =\Abe{X}{m}\]
of level $m$ on the cotangent bundle $T^\ast \sX_0$ of $\sX_0$; here $\Berth{X}{m}$ denotes Berthelot's sheaf of level-$m$ arithmetic differential operators on $\sX_0$ and $\pi : T^\ast \sX_0 \to \sX_0$ denotes the structure morphism of this cotangent bundle. The definition of these sheaves uses the microlocalization of $\pi^{-1} \gr \sD^{(m)}_{\sX}$ on the level-$m$ cotangent bundle $T^{(m)\ast}\sX_0$ of $\sX$, \cite[\S 2.2, 2.4]{AbeMicro}, together an identification of $T^{(m)\ast}\sX_0$ with $T^\ast \sX_0$ which is carried out in \cite[\S 3.3]{AbeMicro}.

Recall the numbers $\varpi_m = (p^m)!^{\frac{1}{p^m}} \in \overline{K}$ from Notation \ref{varpim}.
\begin{defn}\hsp
 \be \item For each $m \geq 0$, write $r_m := |\varpi_m| = |(p^m)!|^{\frac{1}{p^m}}$.
\item Let $\Cot{X} = T^\ast \sX_0 \backslash s(\sX_0)$ denote the complement of the zero-section $s(\sX_0)$ in the cotangent bundle $T^\ast \sX_0$ of $\sX_0$.
\ee\end{defn}

Our next result explains the relationship between the microlocal rings $\cE_{[s,r]}(X)$ that appeared in $\S \ref{MicRings}$ and Abe's rings.
\begin{prop}\label{RelWithAbe} Let $m'\geq m \geq 0$ and let $X \in \D(\partial_x/\varpi)^\dag$. Then the embedding $\cO(X) \hookrightarrow \Gamma( \Cot{X}, \Abe{X}{m,m'} )$ extends to an injective bounded $K$-algebra map
\[ \cE_{[r_{m'},r_m]}(X) \hookrightarrow \Gamma( \Cot{X}, \Abe{X}{m,m'} ).\]
This extension is functorial in $X \in \D(\partial_x/\varpi)^\dag$.
\end{prop}

To prove this, we will need to recall the explicit description of Abe's ring appearing on the right hand side as a set of Laurent series in $\partial$ satisfying a particular convergence condition. This, unfortunately, requires some detailed notation involving Berthelot's modified divided powers $\partial^{\langle n \rangle_{(m)}}$, as well as their \emph{inverses}.

\begin{notn} Let $m, n \in \N$.
\be \item Write $q_n := \myfloor{n/p^m}$ and $\partial^{\langle n \rangle_{(m)}} := \frac{q_n! }{n!} \partial^n$.
\item Let $i_n := \myceil{n/p^m}$ and $\ell_n = i_np^m - n$ and define 
\[\partial^{\langle -n \rangle_{(m)}} := \partial^{\langle \ell_n \rangle_{(m)}} \left( \partial^{\langle i_n p^m \rangle_{(m)}}\right)^{-1} = \frac{i_n!}{\ell_n! (i_np^m)!} \partial^{-n}\]
with the product understood to be carried out in $\Gamma( \Cot{X}, \Abe{X}{m} )$. 
\ee\end{notn}
This notation comes from \cite[p. 282]{AbeMarmora}.
\begin{lem}\label{AbeExplicit} Let $0 \leq m \leq m'$. We have the following explicit descriptions:
\[ \begin{array}{lll} \Gamma(\sX, \Berth{X}{m}) &=& \left\{ \sum\limits_{n \in \N} a_n \partial^{\langle n \rangle_{(m)}} : a_n \in \cO(X), \lim\limits_{n \to +\infty} a_n = 0\right\},\\
\Gamma(\Cot{X}, \Abe{X}{m}) &=& \left\{ \sum\limits_{n \in \Z} a_n \partial^{\langle n \rangle_{(m)}} : a_n \in \cO(X), \hsp \sup\limits_{n < 0} |a_n| < \infty, \hsp \lim\limits_{n \to +\infty} a_n = 0\right\},\\
\Gamma(\Cot{X}, \Abe{X}{m,m'}) &=& \left\{ \sum\limits_{n \in \Z} a_n \partial^n : \sum\limits_{n < 0} a_n\partial^n \in \Abe{X}{m'}, \sum\limits_{n \geq 0} a_n\partial^n \in \Abe{X}{m}\right\}.
\end{array}\]
\end{lem}
\begin{proof} This follows from \cite[(2.4.1.2)]{Berth} and \cite[\S 1.1.1, \S 1.1.2]{AbeMarmora}.
\end{proof}

\begin{lem}\label{Rescale} For each $m\in \N$ and $n \in \Z$, there exists $\epsilon^{(m)}_n \in \overline{K}$ such that
\begin{equation}\label{Epsilon} \left(\frac{\partial}{\varpi_m}\right)^n = \epsilon^{(m)}_n \partial^{\langle n \rangle_{(m)}} \qmb{and} 1 \leq |\epsilon^{(m)}_n| \leq |p|^{-m}. \end{equation}
\end{lem}
\begin{proof}  Fix $m \geq 0$. For every $n \geq 0$ we define
\[\epsilon^{(m)}_{-n} := \frac{\varpi_m^n \hsp \ell_n! \hsp i_n! }{ (i_n \hsp p^m)! } \qmb{and} \epsilon^{(m)}_n = \frac{ \hsp n!}{\varpi_m^n q_n!}.\] 
Note that $\partial^{\langle \ell_n \rangle_{(m)}}$ is simply $\frac{\partial^{\ell_n}}{\ell_n!}$ because $0 \leq \ell_n < p^m$. Now we can compute
\[ \begin{array}{lll} \epsilon^{(m)}_{-n} \partial^{\langle -n\rangle_{(m)}} &=& \epsilon^{(m)}_{-n}  \cdot \partial^{\langle \ell_n \rangle_{(m)}}  \cdot \left( \partial^{\langle i_n p^m \rangle_{(m)}}\right)^{-1} = \\ 

&=& \frac{\varpi_m^n \hsp \ell_n! \hsp i_n! }{ (i_n \hsp p^m)! }  \cdot \frac{\partial^{\ell_n}}{\ell_n!} \cdot \left( \frac{i_n! \partial^{i_n p^m}}{ (i_n p^m)! }\right)^{-1} = \\

&=& \varpi_m^n \partial^{\ell_n - i_n p^m}  = \left( \frac{\partial}{\varpi_m} \right)^{-n}, \qmb{and} \\
\epsilon^{(m)}_n \partial^{\langle n\rangle_{(m)}} &=& \frac{n!}{\varpi_m^n q_n!} \cdot \frac{q_n! \partial^n}{n!} = \left( \frac{\partial}{\varpi_m} \right)^n.
\end{array}\]
Therefore regardless of the sign of $n$, the first equality in $(\ref{Epsilon})$ holds. 

Next, using the fact that $(p-1)v_p(a!) = a - s_p(a)$ for any $a \in \N$ where $s_p(a)$ denotes the sum the $p$-adic digits of $a$, we compute
\[\begin{array}{lll} (p-1)v_p(\epsilon^{(m)}_{-n}) &=& \frac{n}{p^m}(p^m-1) + (\ell_n - s_p(\ell_n)) + (i_n - s_p(i_n)) - (i_n p^m - s_p(i_n p^m)) = \\
&=& -\frac{n}{p^m} + n + (\ell_n - i_n p^m) + i_n + s_p(i_n p^m) - s_p(i_n) - s_p(\ell_n) = \\
&=& \frac{\ell_n}{p^m} - s_p(\ell_n) \geq 0 - (p-1)m \end{array}\]
because $\ell_n - i_np^m = -n$, $0 \leq \ell_n < p^m$ and because $s_p(i_np^m) = s_p(i_n)$. 

Similarly, writing $n = \alpha_n + p^m q_n$ we compute
\[\begin{array}{lll} (p-1) v_p(\epsilon^{(m)}_n) &=& (n - s_p(n)) - \frac{n}{p^m}(p^m - 1) - (q_n - s_p(q_n)) = \\ 
&=& \frac{n}{p^m} - (s_p(\alpha_n) + s_p(q_n))  - q_n + s_p(q_n) = \\
&=& \frac{\alpha_n}{p^m}  - s_p(\alpha_n)  \geq 0 - (p-1)m
\end{array}\]
because $0 \leq \alpha_n < p^m$ and $s_p(n) = s_p(\alpha_n) + s_p(q_n)$. The result follows.
\end{proof}

\begin{proof}[Proof of Proposition \ref{RelWithAbe}] We will apply Proposition \ref{SkewLaurentUP}. Set $A = \cO(X) = \cO(\sX)_K$, $\delta = \partial_x \in \cB(\cO(X))$, $s := r_{m'} \leq r_m =: r$, $B := \Gamma(\Cot{X}, \Abe{X}{m,m'})$ and $b := \partial \in B$. Note that $B$ is a $K$-Banach algebra with a norm $|\cdot |_B$ extending the standard norm on $K$ and unit ball given by $\Gamma(\Cot{X}, \widehat{\sE}_{\mathscr{X}}^{(m,m')})$: more precisely this norm on $B$ is determined by the following relations:
\begin{equation}\label{Bnorm} |\partial^{\langle n \rangle_{(m)}}|_B = |\partial^{\langle -n \rangle_{(m')}}|_B = 1 \qmb{for all} n \geq 0.\end{equation}
We must verify the conditions of Proposition \ref{SkewLaurentUP}, of which (i) and (ii) are clear. Thus it remains to show that
\[\sup\limits_{n \geq 0} |\partial^n|_B/r_m^n < \infty \qmb{and} \sup\limits_{n \leq 0} |\partial^n|_B/r_{m'}^n < \infty.\]
Since $|\varpi_m| = r_m$, both of these facts follow from Lemma \ref{Rescale} together with equation (\ref{Bnorm}). Thus Proposition \ref{SkewLaurentUP} gives the required bounded $K$-algebra homomorphism $\cE_{[r_{m'},r_m]}(X) = \cO(X) \langle \partial/r_m, r_{m'}/\partial\rangle \to B$ extending the inclusion $\cO(X) \hookrightarrow B$. It is easy to see that this homomorphism is injective, as well as functorial for those affinoid subdomains $X$ of $\bD$ with $\rho(X) = 1$.\end{proof}

Now recall the notation $\S \ref{KummerSect}$: \[ u = \lambda \prod_{i=1}^h (x - a_i)^{k_i} \qmb{and} z^d = u\]
for some $k_1,\ldots,k_h \in \Z \backslash\{0\}$, some $\lambda \in K^\times$ and some pairwise distinct $a_1,\ldots, a_h \in X(K)$. Our next goal is to study the following coherent $\Berth{X}{0}$-module on $\sX$
\[ \sM^{(0)} := \Berth{X}{0} / \Berth{X}{0} R_{\cS(u)}(u,d).\]
We will now calculate its \emph{characteristic variety} $\Char^{(0)}(\sM^{(0)})$, in the sense of \cite[\S 5.2.5]{Berth2}, \cite[\S 5.14]{AW13} and \cite[\S 2.14]{AbeMicro}.
\begin{lem}\label{CharVar} If $p \nmid d$, then $\Char^{(0)}(\sM^{(0)}) = \sX_0 \cup T^\ast_{\overline{a_0}} \sX_0 \cup \cdots \cup T^\ast_{\overline{a_h}} \sX_0$.
\end{lem}
\begin{proof} By Definition \ref{TheRelator}(b), we know that
\[ R_{\cS(u)}(u,d) = \prod\limits_{i=1}^h (x - a_i) \partial_x \hsp - \hsp \frac{1}{d} \sum\limits_{i=1}^h k_i \prod\limits_{j \neq i}(x - a_j).\]
Because $p \nmid d$, this is an element of $D := \widehat{\sD}^{(0)}(\sX)$, and we see that the principal symbol of the image of $R_{\cS(u)}(u,d)$ in $D / \pi_K D$ is 
\begin{equation}\label{GrRz} \Gr R_{\cS(u)}(u,d) = \prod\limits_{i=1}^h (x - \overline{a_i})\cdot y\end{equation}
where we identify $\Gr D$ with the polynomial ring $\cO(\sX_0)[y]$. Hence
\[\Char^{(0)}(\sM^{(0)}) = \Supp \left(\frac{ \Gr D }{ \Gr D \cdot \prod\limits_{i=1}^h (x - \overline{a_i}) \cdot y} \right) = \sX_0 \cup T^\ast_{\overline{a_1}} \sX_0 \cup \cdots \cup T^\ast_{\overline{a_h}} \sX_0. \qedhere\]
\end{proof}

Write $\sY_i := T^\ast_{\overline{a_i}} \sX_0$ for each $i = 1,\ldots, h$. We assume without loss of generality that $\sY_1,\ldots,\sY_g$ are pairwise distinct and that every other $\sY_i$ equals $\sY_j$ for some $j \leq g$. Thus $|a_i - a_j| = 1$ whenever $1 \leq i \neq j \leq g$.
\begin{prop}\label{RzUnitOrNot} Let $0 \leq m \leq m'$ and write $\cE := \Abe{X}{m,m'}$. Then
\be \item $R_{\cS(u)}(u,d)$ is a unit in $\cE( T^\ast \sX_0 - \left(\sX_0 \cup \sY_1 \cup \cdots \cup \sY_h) \right)$,
\item $R_{\cS(u)}(u,d)$ is \emph{not} a unit in $\cE\left( T^\ast \sX_0 - (\sX_0 \cup \sY_1 \cup \cdots \cup \widehat{\sY_i} \cup \cdots \cup \sY_g) \right)$ for any $1 \leq i \leq g$, and
\item $R_{\cS(u)}(u,d)$ is \emph{not} a unit in $\cE ( T^\ast \sX_0 - \left(\sY_1 \cup \cdots \cup \sY_g\right))$.
\ee\end{prop}
\begin{proof} (a) Let $\sU = \sX - \{\overline{a_1},\cdots, \overline{a_h}\}$ so that 
\[T^\ast \sX_0 - \left(\sX_0 \cup \sY_1 \cup \cdots \cup \sY_h) \right) = \Cot{U}.\]
Now $\Delta_{\cS(u)} = \prod_{i=1}^h (x - a_i)$ is a polynomial in $x$ which does not vanish on $\sU_{\rig}$. Hence
\[\Delta_{\cS(u)} \in \cO(\sU_{\rig})^\times \subset \left(\Gamma(\sU, \Berth{X}{m})\right)^\times \subset \left(\Gamma(\Cot{U}, \cE)\right)^\times.\] 
On the other hand, $\theta_{u,d}(\partial_x)$ is a unit in $\cE_{[r_{m'},r_m]}(\sU_{\rig})$ by Proposition \ref{MicroInverse}. Using Proposition \ref{RelWithAbe}, we deduce that $\theta_{u,d}(\partial_x)$ is also a unit in $\cE(\Cot{U})$. Hence $R_{\cS(u)}(u,d) \in \cE(\Cot{U})^\times$ as claimed.

(b) Fix $i = 1,\ldots, g$ and let $\sV := \sU \cup \{\overline{a_i}\}$, so that 
\[T^\ast \sX_0 - \left(\sX_0 \cup \sY_1 \cup \cdots \cup \widehat{\sY_i} \cup \cdots \cup \sY_g \right) = \Cot{V}\]
because the $\sY_1,\ldots,\sY_g$ are pairwise distinct. Now, by part (a), $R_{\cS(u)}(u,d) \in \cE(\Cot{U})^\times$; we must show that $R_{\cS(u)}(u,d)^{-1}$ does not lie in the subring $\cE(\Cot{V})$ of $\cE(\Cot{U})$. Using Lemma \ref{RudProps}(b) and Proposition \ref{MicroInverse} we see that 
\[R_{\cS(u)}(u,d)^{-1} = \Delta_{\cS(u)}^{-1} \theta_{u \Delta_{\cS(u)}^d,d}(\partial_x)^{-1} = \Delta_{\cS(u)}^{-1} \xi_{u \Delta_{\cS(u)}^d,d}.\] 
By Definition \ref{XiZ}, the coefficient of $\partial^{-1}$ in this Laurent series is $\Delta_{\cS(u)}^{-1}$, which must lie in $\cO(\sV_{\rig})$ if $R_{\cS(u)}(u,d)^{-1} \in \cE(\Cot{V})$. But $\Delta_{\cS(u)}$ has a zero at $x = a_i \in \sV_{\rig}$ and therefore is not a unit in $\cO(\sV_{\rig})$ --- a contradiction.

(c) This time we have $T^\ast \sX_0 - \left(\sY_1 \cup \cdots \cup \sY_g\right) = T^\ast \sV_0$, and 
\[\cE(T^\ast \sV_0) = \Berth{X}{m}(\sU).\]
If $R_{\cS(u)}(u,d)$ is a unit in $\Berth{X}{m}(\sU)$ then since $\Delta_{\cS(u)} \in \cO(\sU_{\rig})^\times$, we see that $\theta_{u,d}(\partial_x)$ has to be a unit in $\Berth{X}{m}(\sU)$ as well. Choose $r \in \bR$ such that $\varpi < r < r_m$. Then $\Berth{X}{m}(\sU)$ embeds naturally into $\cD_r(\sU_{\rig})$, so $\theta_{u,d}(\partial_x) \in \cD_r(\sU_{\rig})$. But the automorphism $\theta_{u,d}$ of $\cD(\sU_{\rig})$ extends to a bounded $K$-algebra automorphism of $\cD_r(\sU_{\rig})$ by Theorem \ref{ThetaU}, and it follows that $\partial_x$ is a unit in $\cD_r(\sU_{\rig})$ which is not the case.
\end{proof}

\begin{cor}\label{MStable} For any integers $m' \geq m \geq 0$, we have
\[\Supp\left( \Abe{X}{m,m'} / \Abe{X}{m,m'} R_{\cS(u)}(u,d) \right) = \sX_0 \cup \sY_1 \cup \cdots \cup \sY_h = \Char^{(0)}(\sM^{(0)}).\]
\end{cor}
\begin{proof} This follows immediately from Lemma \ref{CharVar} and Proposition \ref{RzUnitOrNot}.
\end{proof}
We can now calculate the \emph{characteristic cycle} of the module that we are really interested in, namely
\[\sD^\dag_{\sX,\Q} / \sD^\dag_{\sX,\Q} R_{\cS(u)}(u,d).\]
In the papers by Abe \cite{AbeSwan} and Abe-Marmora \cite{AbeMarmora}, one finds at least two variants of the characteristic cycle, applicable for various kinds of modules. We are mainly interested in the cycle defined at \cite[\S 1.5.2]{AbeMarmora} for holonomic $\sD^\dag_{\sX,\Q}$-modules $\sM$, denoted $\Cycl(\sM)$. To calculate it, we will need the following notation. 
\begin{defn}\label{Multi} Let $u = \lambda \prod_{i=1}^h (x - a_i)^{k_i}$ be given, with $\lambda \in K^\times$, $a_1,\ldots, a_h \in X(K)$ pairwise distinct and $k_1,\ldots, k_h$ all non-zero. Arrange the $a_i$'s so that $a_1,\ldots, a_g$ satisfy the following conditions:
\be\item $|a_i - a_j| = 1$ for all $1 \leq i\neq j \leq g$, and 
\item for every $1 \leq \ell \leq h$, there is a unique $i_\ell \in \{1, \ldots, g\}$ such that $|a_\ell - a_{i_\ell}| < 1$.
\ee
Then define for each $i = 1,\ldots, g$ the \emph{$i$-th multiplicity} as follows:
\[ m_{\overline{a_i}}(u) := |\{1 \leq \ell \leq h : i_\ell = i\}|.\]
Recall also that $\sY_i$ denotes $T^\ast_{\overline{a_i}}\sX_0$ for each $i = 1,\ldots, g$.
\end{defn}

\begin{thm}\label{CharCycle} Suppose that $p \nmid d$ and that $u = \lambda \prod_{i=1}^h (x - a_i)^{k_i}$ is given.  Then 
\[\Cycl\left( \sD^\dag_{\sX,\Q} / \sD^\dag_{\sX,\Q} R_{\cS(u)}(u,d)\right) = [\sX_0] + \sum_{i=1}^g m_{\overline{a_i}}(u) [\sY_i].\]
\end{thm}
\begin{proof} Let $\sM^{(0)} := \Berth{X}{0} / \Berth{X}{0} R_{\cS(u)}(u,d)$ as above. It follows from Corollary \ref{MStable} that
\[ \Supp\left( \Abe{X}{m,m'} \underset{\pi^{-1} \Berth{X}{0}}{\otimes}{} \pi^{-1}\sM^{(0)} \right) = \Char^{(0)}(\sM^{(0)})\]
for any integers $m' \geq m \geq 0$. This means that the holonomic $\Berth{X}{0}$-module $\sM^{(0)}$ is \emph{stable} in the sense of \cite[\S 1.3.8]{AbeMarmora}.  Let $\sM := \sD^\dag_{\sX,\Q} \otimes_{\Berth{X}{0}} \sM^{(0)}$. Because we are assuming that each $a_i$ is a $K$-rational point of $\sX$, we can now apply \cite[Definition 1.5.2 and Proposition 1.4.3]{AbeMarmora} to see that
\[\Cycl(\sM) = \Cycl^{(0)}(\sM^{(0)})\]
where $\Cycl^{(0)}(\sM^{(0)})$ is the characteristic cycle of the holonomic $\Berth{X}{0}$-module $\sM^{(0)}$ defined at \cite[Definition 2.1.17]{AbeSwan}. Using \cite[\S 2.1.12]{AbeSwan}, we see that $\Cycl^{(0)}(\sM^{(0)})$ is the usual class of the coherent $\cO(T^\ast \sX_0)$-module $\cO(T^\ast \sX_0) / \Gr R_{\cS(u)}(u,d) \cdot \cO(T^\ast \sX_0)$ in the Grothendieck group of coherent $\cO(T^\ast \sX_0)$-modules of dimension at most $1$. The result now follows in view of equation $(\ref{GrRz})$ and Definition \ref{Multi}.
\end{proof}

\begin{rmk}\label{ZetaInK} Recall from $\S \ref{ConvNotn}$ that $q$ denotes the size of the residue field of $k$. Since $K$ contains the finite field extension $F$ of $\Qp$ whose residue field has size $q$ --- see $\S \ref{ConvNotn}$ --- it follows that $K$ contains a primitive $(q-1)$-th root of unity $\zeta$.
\end{rmk}

\begin{ex}\label{CyclRwd} Suppose that $0 \in X(K)$, and let 
\[ w := \frac{1}{ (x^q - \pi_F^{q-1} x)^k }\]
 for some non-zero integer $k$, so that
\[\cS(w) = \{a_1,\ldots,a_q\} = \{0, \pi_F, \pi_F\zeta, \pi_F\zeta^2, \cdots, \pi_F\zeta^{q-2}\} \subset K.\]
Because $|\pi_F| < 1$, all of these points map to the same point $\overline{0}$ in $\sX_0$ under the reduction map $X(K) \to \sX_0$. Therefore $g = 1$ and $m_{\overline{0}}(w) = q$. So, in this case Theorem \ref{CharCycle} tells us that
\[\Cycl\left( \frac{\sD^\dag_{\sX,\Q}}{\sD^\dag_{\sX,\Q} R_{\cS(w)}(w,d)}\right) = [\sX_0] + q [T^\ast_0\sX_0].\]
\end{ex}

\subsection{The differential equation \ts{\nabla(\zeta) = 1 - c}}\label{TheIntegral}
Recall the admissible open subspaces $V_n$ of $\bA$ from Definition \ref{UnCovering}. We are interested here in the following three affinoid subdomains of $\D = \Sp K\langle x \rangle$, one containing the next:
\begin{itemize}
\item $X = \Sp K \langle x, \frac{1}{x^{q-1}-1}\rangle$, 
\item $X \cap V_1 := X - \bigcup\limits_{a \in \pi_F \cO_F} \{|z - a| < |\pi_F|\}$, and 
\item $X \cap V_0 = \Sp K\langle x, \frac{1}{x^q - x}\rangle = X - \{|z| < 1\}$.
\end{itemize} 
Viewed as subsets of $\bP^{1,\an}$, $X$ has $q$ big holes (each of radius $1$), $X \cap V_0$ has $q+1$ big holes and $X \cap V_1$ has $2q$ holes, of which $q$ are big and $q$ are small (radius $|\pi_F|$).  Note that we have the following relations:
\[ X \cap V_0 \subset X \cap V_1 \subset X \qmb{and}  \cO(X \cap V_0) \supset \cO(X \cap V_1) \supset \cO(X).\]
Here is a drawing of these three affinoid domains when $q = 5$:
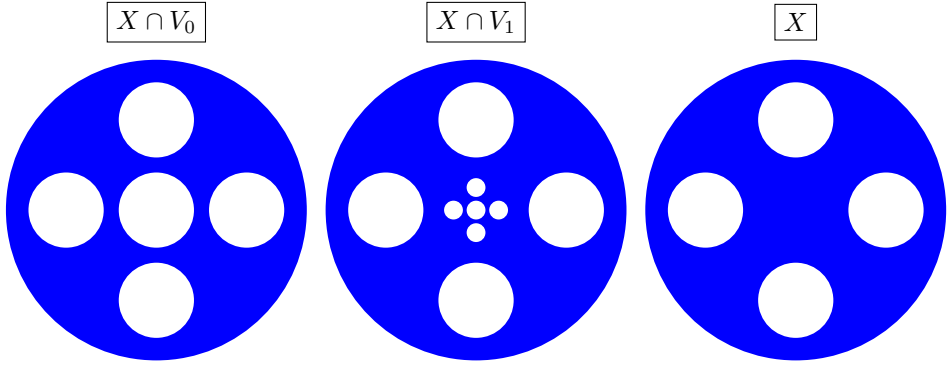
\begin{center}
\begin{tikzpicture}
\node[draw] at (1,2.5) {$X \cap V_0$};
\fill[blue] (1,0) circle (2cm);
\node[draw] at (5.25,2.5) {$X \cap V_1$};
\fill[blue] (5.25,0) circle (2cm);
\node[draw] at (9.5,2.5) {$X$};
\fill[blue] (9.5,0) circle (2cm);

\fill[white] (-.2,0) circle (0.5cm);
\fill[white] (1,-1.2) circle (0.5cm);
\fill[white] (1,0) circle (0.5cm);
\fill[white] (1,1.2) circle (0.5cm);
\fill[white] (2.2,0) circle (0.5cm);

\fill[white] (4.05,0) circle (0.5cm);   
\fill[white] (5.25,-1.2) circle (0.5cm);   
\fill[white] (4.95,0) circle (0.125cm);
\fill[white] (5.55,0) circle (0.125cm);
\fill[white] (5.25,-0.3) circle (0.125cm);
\fill[white] (5.25,0) circle (0.125cm);
\fill[white] (5.25,0.3) circle (0.125cm);
\fill[white] (5.25,1.2) circle (0.5cm);   
\fill[white] (6.45,0) circle (0.5cm);   

\fill[white] (8.3,0) circle (0.5cm);   
\fill[white] (9.5,-1.2) circle (0.5cm);   
\fill[white] (9.5,1.2) circle (0.5cm);   
\fill[white] (10.7,0) circle (0.5cm);   

\end{tikzpicture} 
\end{center}

For future use, we record the basic invariants $\rho(Y)$ and $r(Y)$ from Definition \ref{MinRad} and Definition \ref{dxradm} for these cheeses in the following table, using the basic relationship $r(Y) = \varpi / \rho(Y)$ given by Corollary \ref{Heisenberg}:

\[\begin{array}{c|ccc} Y & X \cap V_0 & X \cap V_1 & X\\ \hline \rho(Y) & 1 & |\pi_F| & 1\\ r(Y) & \varpi& \varpi/|\pi_F|& \varpi \end{array}\]

Fix an integer $d \geq 1$ which divides $q + 1$ and an integer $k$ such that $1 \leq k \leq d$. We choose the following unit on $X \cap V_1$ whose mod-$d$-divisor is appropriately invariant:
\[ w := \frac{1}{ (x^q - \pi_F^{q-1} x)^k }.\]
\begin{defn} Define $h = \begin{pmatrix} 1 & -\pi_F \\ 0 & 1 \end{pmatrix} \in \cG$.\end{defn}
Since $h\cdot x = x  + \pi_F$ by Lemma \ref{GdotDn}(a), it follows from Definition \ref{Beta} that $h \in \cG_r$ whenever $\varpi < r < \varpi/|\pi_F|$ and that 
\[\beta(h) = \sum\limits_{n=0}^\infty \pi_F^n \partial_x^{[n]} \in \cD^\dag_\varpi(X).\]
We begin with the first step of the proof of our main result of interest, namely Theorem \ref{BetaNotInSum}, as a way of motivating the calculations that follow.

Let $Z$ be an admissible open subspace of a rigid analytic space $Y$ and let $\cF$ be a sheaf on $Y$. We say that the local section $f \in \cF(Z)$ \emph{extends to $Y$} if it lies in the image of the restriction map $\cF(Y) \to \cF(Z)$. 

\begin{prop}\label{XiBetaZero} Suppose that $\beta(h) \in \cD^\dag_{\varpi/|\pi_F|}(X) + \cD^\dag_\varpi(X) R_{\cS(w)}(w,d)$. Then $(\xi_{w,d} \beta(h))_0 \in \cO(X \cap V_0)$ extends to $X \cap V_1$.
\end{prop}
\begin{proof}Suppose that $Q \in \cD^\dag_{\varpi/|\pi_F|}(X)$ is such that $\beta(h) - Q \in \cD^\dag_\varpi(X) R_{\cS(w)}(w,d)$.  Note that $\theta_{w,d}(\partial_x)$ is a unit in the Robba ring $\overset{\circ}{\cE}(X \cap V_0)$ with inverse $\xi_{w,d}$ by Corollary \ref{RobbaInverts}, because $w \in \cO(X \cap V_0)^\times$. Now, in $\overset{\circ}{\cE}(X \cap V_0)$ we have the equation 
\[\beta(h) \xi_{w,d} = (\beta(h) - Q) \xi_{w,d} + Q \xi_{w,d}\]
Comparing the constant term coefficients in this equation, we have
\begin{equation}\label{BetaSum}  (\beta(h)  \xi_{w,d})_0 =  ((\beta(h) - Q)  \xi_{w,d})_0 + (Q  \xi_{w,d})_0. \end{equation}
By Lemma \ref{zTwist} applied to $w \in \cO(X \cap V_0)^\times$, we have
\[R_{\cS(w)}(w,d) = \Delta_{\cS(w)} \theta_{w,d}(\partial_x)\]
which implies that $R_{\cS(w)}(w,d) \hsp \xi_{w,d} = \Delta_{\cS(w)}$ inside $\overset{\circ}{\cE}(X \cap V_0)$. Since $\beta(h) - Q \in \cD^\dag_\varpi(X) R_{\cS(w)}(w,d)$ by assumption, we conclude that
\begin{equation}\label{XBZ1} ((\beta(h) - Q) \xi_{w,d})_0 \in \cO(X).\end{equation}
Since $\xi_{w,d}$ also lies in $\overset{\circ}{\cE}(X \cap V_1)$ --- again by Corollary \ref{RobbaInverts} --- and since 
\[Q \in \cD^\dag_{\varpi / |\pi_F|}(X) \subseteq \cD^\dag_{\varpi/|\pi_F|}(X \cap V_1),\]
by comparing the multiplication in $\overset{\circ}{\cE}(X \cap V_0)$ and $\overset{\circ}{\cE}(X \cap V_1)$ we see that
\begin{equation}\label{XBZ2} (Q \xi_{w,d})_0 \in \cO(X \cap V_1).\end{equation}
Since $X \supset X \cap V_1$, (\ref{BetaSum}), (\ref{XBZ1}) and (\ref{XBZ2}) imply that $(\beta(h)  \xi_{w,d})_0$ extends to $X \cap V_1$.

Next, $\beta(h) \in \cD^\dag_\varpi(X \cap V_0)^\times \subset \overset{\circ}{\cE}(X \cap V_0)^\times$, so conjugation by $\beta(h)$ defines an automorphism $S$ of $\overset{\circ}{\cE}(X \cap V_0)$ given by $S(u) := \beta(h)^{-1}  \hsp u \hsp  \beta(h)$ for any $u \in \overset{\circ}{\cE}(X \cap V_0)$. Note that $S( \beta(h) \xi_{w,d} ) = \xi_{w,d} \beta(h)$. Since $\beta(h) = \sum\limits_{n=0}^\infty \pi_F^n \partial_x^{[n]}$ commutes with $\partial_x$, it follows from statement (b) in the proof of Theorem \ref{GrGxTriv} that
\[ S(u) = \beta(h)^{-1}  \hsp u \hsp \beta(h) = \sum_{i \in \Z} (h^{-1} \cdot u_i) \partial_x^i \qmb{for any} u = \sum_{i \in \Z} u_i \partial_x^i\in \overset{\circ}{\cE}(X \cap V_0).\]
Because both $X \cap V_1$ and $X \cap V_0$ are $h$-stable, we see that 
\[(\xi_{w,d} \beta(h))_0 = S(\beta(h) \xi_{w,d})_0 = h^{-1} \cdot (\beta(h) \xi_{w,d})_0\]
also extends to $X \cap V_1$, as desired. \end{proof}

The key to our argument is that the rigid analytic function $(\xi_{w,d} \beta(h))_0$ on $X \cap V_0$ can be explicitly computed, and shown that in fact it \emph{cannot} extend to $X \cap V_1$. To do this, we study the solutions of a certain differential equation satisfied by this function. Recall the units $c_{u,d}(g)$ from Proposition \ref{TwistedCocycle}.

\begin{lem}\label{XVzeroSol} Let $\nabla := \theta_{w,d}(\partial_x) \in \cD(X \cap V_0)$. Then 
\[ \nabla \left( \xi_{w,d} \beta(h))_0 \right) = 1 - c_{w,d}(h^{-1}).\]
\end{lem}
\begin{proof} The constant term of the product $\xi_{w,d} \beta(h)$ is given by formula $(\ref{MicroMult})$:
\[ (\xi_{w,d} \beta(h))_0 = \xi_{w,d} \ast_0 \beta(h) = \sum_{ i \in \Z} (\xi_{w,d})_i \sum_{m=0}^\infty \binom{i}{m} \partial_x^m(\beta(h)_{m-i}).\]
Since all coefficients of $\beta(h)$ are scalars, we see that $\partial_x^m(\beta(h)_{m-i})$ is non-zero only when $m = 0$. Changing the summation to $n = -i $ and using Definition \ref{XiZ}, this sum is equal to
\[\sum\limits_{n \in \mathbb{Z}} (\xi_{w,d})_{-n} \beta(h)_n = \sum\limits_{n=1}^\infty (-1)^{n-1} (n-1)! h_{w, d}^{[n-1]} \cdot \frac{\pi_F^n}{n!}= -\sum\limits_{n=1}^\infty \frac{1}{n} (-\pi_F)^n h_{w, d}^{[n-1]}.\]
Note that by Lemma \ref{HunEst}, $|h_{w,d}^{[n-1]}|_{X\cap V_0} \leq (r(X \cap V_0)/\varpi)^n = 1$ for all $n \geq 1$, so this sum converges in $\cO(X \cap V_0)$. 

On the other hand, since $h_{w,d}^{[n-1]} = \frac{1}{(n-1)!} \nabla^{n-1}(1)$ by Definition \ref{HunDefn}, we have
\[ \nabla( h_{w,d}^{[n-1]} ) = \nabla\left( \frac{1}{(n-1)!} \nabla^{n-1}(1)\right) = n \, h_{w,d}^{[n]} \qmb{for all} n \geq 1.\]
The twisted derivation $\nabla : \cO(X \cap V_0) \to \cO(X \cap V_0)$ is continuous, so 
\[ \nabla( (\xi_{w,d} \beta(h))_0 ) = - \nabla \left(\sum\limits_{n=1}^\infty \frac{1}{n} (-\pi_F)^n h_{w, d}^{[n-1]}\right) = - \sum\limits_{n=1}^\infty (-\pi_F)^n h_{w,d}^{[n]}.\]
Using Proposition \ref{TwistedCocycle}(a), we see that this is equal to $1-c_{w,d}(h^{-1})$ because $h^{-1} \cdot x = x - \pi_F$ by Lemma \ref{GdotDn}(a). \end{proof}
\begin{defn} Let $Y$ be the subset $\bP^{1,\an}$ obtained by removing all the open balls of radius $|\pi_F|$ around all points in $\pi_F \cO_F$:
\[ Y :=  \bP^{1,\an} \, \backslash\, \bigcup\limits_{a \in \pi_F \cO_F} \{|z-a| < |\pi_F|\}.\]
\end{defn}
Note that this $Y$ contains $X \cap V_1$ and therefore also $X \cap V_0$. It turns out that the differential equation from Lemma \ref{XVzeroSol} satisfied by $(\xi_{w,d}\beta(h))_0$ makes sense on this larger subset of $\bP^{1,an}$, namely $Y \supset X \cap V_0$. 

To understand $Y$ better, we consider the involution $s : \bP^{1,\an} \longrightarrow \bP^{1,\an}$ given by the M\"obius action of $\begin{pmatrix} 0 & \pi_F \\ 1 & 0 \end{pmatrix} \in \mathbb{GL}_2(F)$ on $\bP^{1,\an}$:
\[ s(a) = \pi_F/a \qmb{for all} a \in \bP^{1,\an}.\]
\begin{lem} $s$ swaps $X$ and $Y$. 
\end{lem}
\begin{proof} To show that $s$ sends $X$ onto $Y$, it is enough to do this on ${\bf C}$-points, and for this, it suffices to show that $s$ exchanges the complements of $X({\bf C})$ and $Y({\bf C})$ in $\bP^{1,\an}({\bf C})$. These complements are disjoint unions of $q$ open balls. Now, $s$ exchanges the polar hole $\{z \in \bP^1({\bf C}) : |z| > 1\}$ of $X$ with the open ball of radius $|\pi_F|$ around zero. Next, if $c \in \cO_F^\times$ and $z \in {\bf C}$, then using the ultrametric inequality, we see that
\[ |s(z) - \pi_F c| < |\pi_F| \Leftrightarrow |1 - c z| < |z| \Leftrightarrow |z - c^{-1}| < |z| \Leftrightarrow |z - c^{-1}| < |c^{-1}| = 1.\] 
This shows that $s$ exchanges the remaining $q-1$ open balls contained in the complements of $X$ and $Y$. 
\end{proof}

\begin{cor}\label{CoordRingOfY} $Y$ is an affinoid subvariety of $\bP^{1,\an}$ and there is an isomorphism of $K$-affinoid algebras\[ K \left\langle x, \frac{1}{x^{q-1}-1}\right\rangle = \cO(X) \stackrel{\cong}{\longrightarrow} \cO(Y)\]
which sends $x \in \cO(X)$ to $y := \pi_F/x \in \cO(Y)$.
\end{cor}
\begin{lem}\label{ExtendToY} \,
\be \item  The function $\left( \frac{y^q}{\pi_F^q(1 - y^{q-1})}\right)^k$ is an extension of $w \in \cO(X \cap V_0)^\times$ to $Y \backslash \{\infty\}$.
\item The differential operator  $- \frac{1}{\pi_F} \theta_{(1 - y^{q-1})^{-k},d}\left( y^2 \partial_y - \frac{qk}{d} y \right) \in \cD(Y)$ is an extension of $\nabla \in \cD(X \cap V_0)$ from Lemma \ref{XVzeroSol} to $Y$.
\item There is an extension of $c_{w,d}(h^{-1}) \in \cO(X \cap V_0)^\times$ to $Y$ of the form  \[c_{w,d}(h^{-1}) = \left(1 + p  y \frac{f(y)}{1-y^{q-1}} \right)^{\frac{1}{d}} \in \cO(Y)^{\times\times} \qmb{for some} f(y) \in \bZ[y].\]
\ee
\end{lem}
\begin{proof} (a) Recall that $w = 1/(x^q - \pi_F^{q-1} x)^k$ is a meromorphic function on $\bP^1$ with no zeroes or poles on $Y \backslash \{\infty\}$. Now substitute $x = \pi_F/y$ into $w$.
	
(b) First we note that $-\frac{1}{\pi_F} y^2 \partial_y$ is an extension of the derivation $\partial_x\in \cT(X\cap V_0)$ to $Y$ since $-\frac{1}{\pi_F} y^2 \partial_y(x)=1$. Since $w$ is a unit on $Y \backslash \{\infty\}$ by (a), $\theta_{w,d}(\partial_x) \in \cD(X \cap V_0)$ extends to $Y \backslash \{\infty\}$. Next, using Lemma \ref{zTwist} on $Y \backslash \{\infty\}$, we compute
\[ \theta_{y^{qk}, d}( y^2 \partial_y ) = y^2 \left(\partial_y  - \frac{qk}{d} \frac{1}{y} \right) = y^2 \partial_y - \frac{qk}{d} y.\]
So, this differential operator in fact extends to $Y$. Since $w = \pi_F^{-qk} (1 - y^{q-1})^{-k} y^{qk}$ and $(1 - y^{q-1})^{-k} \in \cO(Y)^\times$, we can now apply Lemma \ref{TwistsMultiply} on $Y \backslash \{\infty\}$ to find 
\[ \theta_{w,d}(y^2\partial_y) = \theta_{(1 - y^{q-1})^{-k},d}\left( \theta_{y^{qk}, d} ( y^2 \partial_y ) \right) \in \cD(Y). \]
Putting everything together, we see that
\[ - \frac{1}{\pi_F} \theta_{(1 - y^{q-1})^{-k},d}\left( y^2 \partial_y - \frac{qk}{d} y \right) \in \cD(Y)\]
extends $\nabla \in \cD(X \cap V_0)$ to $Y$ as claimed.

(c) We will show that there exists a unit $c \in \cO(Y)^{\times\times}$ such that 
\[w / h^{-1} \cdot w = c^d.\]
We can then use Proposition \ref{TwistedCocycle}(d) to see that $c|_{X \cap V_0}$ and $c_{w,d}(h^{-1})$ are both $d$th roots of $w / h^{-1}\cdot w \in \cO(X \cap V_0)$.  Now Proposition \ref{TwistedCocycle}(a) tells us that $c_{w,d}(h^{-1}) \in  \cO(X \cap V_0)^{\times\times}$, whereas $c|_{X \cap V_0} \in \cO(X \cap V_0)^{\times\times}$ because $c \in \cO(Y)^{\times\times}$. Since this group is $d$-divisible by \cite[Lemma 4.3.2(a)]{ArdWad2023}, we conclude that $c|_{X \cap V_0} = c_{w,d}(h^{-1})$. Hence $c \in \cO(Y)^{\times\times}$ extends $c_{w,d}(h^{-1})$ to $Y$.

It remains to produce the unit $c \in \cO(Y)^{\times\times}$ such that $c^d = w / h^{-1}\cdot w$. It is enough to consider the case $k = 1$. First, we compute
\begin{equation}\label{hinvy} h^{-1} \cdot y = \frac{\pi_F}{h^{-1} \cdot x} = \frac{\pi_F}{x - \pi_F} = \frac{y}{1-y}\end{equation}
which implies that $h$ stabilises $Y$. Using this together with part (a), we then have
\[ w = \pi_F^{-q} \frac{ y^{q+1} }{y - y^q} \quad\qmb{and}\quad h^{-1} \cdot w = \pi_F^{-q} \frac{ \left(\frac{y}{1-y}\right)^{q+1}}{\frac{y}{1-y} - \left(\frac{y}{1-y}\right)^q} \]
Dividing through,  $y^{q+1}\pi_F^{-q}$ cancels and we obtain
\[ \frac{w}{h^{-1} \cdot w} = \frac{y(1-y)^q - y^q(1-y)}{y-y^q}.\]
Since $y(1-y)^q  - y^q(1-y) \equiv y - y^q \mod p y^2 \bZ[y]$, we see that 
\[ \frac{w}{h^{-1} \cdot w} = 1 + p y \frac{f(y)}{1-y^{q-1}} \qmb{for some} f(y) \in \bZ[y].\]
Thus $w / h^{-1}\cdot w \in \cO(Y)^{\times\times}$ and we can use \cite[Lemma 4.3.2(a)]{ArdWad2023} to define
\begin{equation}\label{DefnOfC} c := \left(1 + p y  \frac{f(y)}{1-y^{q-1}} \right)^{\frac{1}{d}} \in \cO(Y)^{\times\times} .\end{equation}
Then $c^d = w/h^{-1}\cdot w$ as desired.
\end{proof}

Because of Lemma \ref{ExtendToY}, we can consider the differential equation on $Y$:
\[ \nabla(\zeta) = c - 1, \quad \zeta \in \cO_Y,\]
and its sheaf of solutions. More precisely, with $\nabla \in \cD(Y)$ the extension defined by Lemma \ref{ExtendToY}(b) and $c \in \cO(Y)^\times$ as defined by Lemma \ref{ExtendToY}(c), for each admissible open subspace $U$ of $Y$, we can consider the following subset $\cF(U)$ of $\cO(U)$:
\[ \cF(U) := \{\zeta \in \cO(U) : \nabla_{|U}(\zeta) = c_{|U}.\}\]
It is clear that $\cF$ is a \emph{subsheaf of sets} of $\cO_Y$. We begin the study of $\cF$ by looking at formal meromorphic solutions near the point $x = \infty$ or $y = 0$. We identify the completed local ring $\widehat{\cO_{Y,\infty}}$ of $\cO_Y$ at $y = 0$ with the ring of formal power series $K[[y]]$ and look for formal solutions in its field of fractions $K((y))$. Actually, it will be more convenient to work in a larger field, namely 
\[ K((t)) := K((y))[t] / \langle t^d - y \rangle.\]
For every admissible open subset $U$ of $Y$ containing $\infty$, we have natural $K$-algebra and $\cD(U)$-module homomorphisms 
\[\cO(U) \longrightarrow \cO_{Y,\infty} \hookrightarrow \widehat{\cO_{Y,\infty}} = K[[y]] \hookrightarrow K((y)) \hookrightarrow K((t)).\]
\begin{lem}\label{GermAtInfty} Let $U$ be a connected affinoid subdomain of $Y$ containing $\infty$, and let $\rho^U_\infty : \cO(U) \to \widehat{\cO_{Y,\infty}} = K[[y]]$ be the natural map. Then $\rho^U_\infty$ is injective and $\cD(U)$-linear. \end{lem}
\begin{proof} Only the injectivity requires proof. For this, note that the closed disc $\Sp K\langle y / \pi_F^n \rangle$ is contained in $U$ for sufficiently large $n$ since $\infty \in U$. Because $U$ is connected, the restriction map $\cO(U) \to K \langle y / \pi_F^n \rangle$ is injective by \cite[Proposition 4.2]{ABB}. The result follows, because $K\langle y / \pi_F^n \rangle$ embeds into $K[[y]]$.
\end{proof} 
For every connected affinoid subdomain $U$ of $Y$ containing $\infty$, we will also write
\[ \rho^U_\infty : \cO(U)[1/y] \longrightarrow K((y))\]
for the localisation at powers of $y$ of the map $\rho^U_\infty$ from Lemma \ref{GermAtInfty}. This map is also injective and $\cD(U)$-linear.

\begin{lem}\label{ExtendToKt} \,
\be \item The derivation $y^2 \partial_y : \cO(Y) \to \cO(Y)$ extends to a continuous $K$-linear derivation of $K((t))$. 
\item The action of $h^{-1}$ on $\cO(Y)$ extends to a continuous $K$-algebra automorphism $\eta$ of $K((t))$ via $\eta(t)=t(1-t^d)^{-\frac 1 d }$.
\item The resulting $K$-linear operators on $K((t))$ commute.
\ee\end{lem}
\begin{proof} (a) Consider the $K$-linear derivation $\frac{1}{d} t^{d+1} \partial_t : K[[t]] \to K[[t]]$. It respects the $t$-adic filtration on $K[[t]]$ and restricts to $y^2 \partial_y$ on $K[[y]] \subset K[[t]]$ because 
\[ \frac{1}{d} t^{d+1} \partial_t(y) = t^{d+1} t^{d-1} = y^2 = y^2\partial_y(y).\]
Now extend this derivation to the field of fractions $K((t))$ of $K[[t]]$.

(b) Define $\eta : K[[t]] \to K[[t]]$ by setting 
\begin{equation}\label{DefnEta}\eta( f(t) ) = f( t (1 - t^d)^{-\frac{1}{d}} ) \qmb{for all} f(t) \in K[[t]].\end{equation}
This is a $K$-algebra automorphism which respects the $t$-adic filtration on $K[[t]]$. Its unique extension $\eta$ to the field of fractions $K((t))$ of $K[[t]]$ is then given by the same formula. Because $y = t^d$, its restriction to $K[[y]] \subset K[[t]]$ satisfies
\[ \eta(y) = ( t ( 1 - t^d)^{-\frac{1}{d}})^d = y / (1 - y).\]
On the other hand, the automorphism $h^{-1} : \cO(Y) \to \cO(Y)$ sends $y$ to $\frac{y}{1-y}$ by (\ref{hinvy}), so it preserves the ideal $y \cO(Y)$ and therefore extends uniquely to a $y$-adically continuous $K$-algebra automorphism of $\h{\cO_{Y,\infty}} = K[[y]]$ which sends $y$ to $\frac{y}{1-y}$. This extension is then the restriction of $\eta$ to $K[[y]]$. 

(c) It remains to see that the derivation $D = \frac{1}{d} t^{d+1} \partial_t$ commutes with $\eta$. It is enough to show that the restriction of $D$ to $K((y))$ agrees with the restriction of $\eta D \eta^{-1}$ to $K((y))$, because then the derivation $D - \eta D \eta^{-1}$ of $K((t))$ is $K((y))$-linear and therefore must vanish on the finite \'etale extension $K((t))$ of $K((y))$. Now $D(x) = y^2 \partial_y(x) = - \pi_F$ and $\eta(x) = h^{-1} \cdot x = x - \pi_F$, so $D(\eta(x)) = D(x - \pi_F) = D(x) = \eta(D(x))$. Therefore $\eta D \eta^{-1}$ and $D$ agree on $K[y]$, but both are $y$-adically continuous and $K[y]$ is $y$-adically dense in $K[[y]]$, so they must be equal.
\end{proof}

Now we define $\epsilon := (1 - y^{q-1})^{\frac{k}{d}} \in K[[y]]$ and $z := \epsilon^{-1} t^{qk} \in K((t))$. Using Lemma \ref{ExtendToY}(a), we have
\[z^d = (1 - y^{q-1})^{-k} y^{qk} = \pi_F^k\rho^Y_\infty(w).\]
\textbf{We assume from now on until the end of $\S \ref{TheIntegral}$ that $k \neq d$.} Under this assumption, we see that $qk \equiv -k \neq 0 \mod d$, so that $a - \frac{qk}{d} \notin \bZ$ for any integer $a$. This allows us to make the following
\begin{defn}\label{DefnOfJ} For each $m \geq 0$, define $\alpha_m := \frac{(-1)^m \binom{\frac{k}{d}}{m}}{(q-1)m - \frac{qk}{d} - 1}$, and set
\[ J := t^{-qk} y^{-1} \sum\limits_{m=0}^\infty  \alpha_m y^{(q-1)m} \in t^{-qk} K((y)).\]
\end{defn}

\begin{lem}\label{Jintegral} We have $(y^2 \partial_y)(J) = t^{-qk} \epsilon = z^{-1}$.
\end{lem}
\begin{proof} Let $f(t) \in K((t))$. Since $y = t^d$, using the Leibniz rule we have
\[ \begin{array}{lll} (y^2 \partial_y) (t^{-qk} y^{-1} f(t)) &=& f(t) \left(\frac{1}{d} t^{d+1} \partial_t\right)( t^{-qk-d} )  + t^{-qk} y^{-1} \, (y^2 \partial_y)(f(t)) \\ &=& t^{-qk} \left( y \partial_y - \frac{qk}{d} - 1 \right)(f(t)).\end{array}\] 
Since $y^2 \partial_y$ is $t$-adically continuous on $K((t))$, it remains to show that
\[ (y \partial_y - \frac{qk}{d} - 1) \left( \sum\limits_{m=0}^\infty  \alpha_m y^{(q-1)m} \right) = \epsilon.\]
Using the binomial expansion inside $K[[y]]$, we have
\[ \epsilon = (1 - y^{q-1})^{\frac{k}{d}} = \sum\limits_{m=0}^\infty (-1)^m \binom{\frac{k}{d}}{m} y^{(q-1)m}.\]
The result follows, because 
\[ (y \partial_y - \frac{qk}{d} - 1)(y^{(q-1)m}) = \left((q-1)m - \frac{qk}{d} - 1\right) y^{(q-1)m}\]
and $((q-1)m - \frac{qk}{d} - 1)\alpha_m = (-1)^m \binom{\frac{k}{d}}{m}$ for all $m \geq 0$.
\end{proof}
In other words, we think of $J$ informally as the integral 
\[\boxed{J = \int^{y} \left(\frac{1 - y^{q-1}}{y^q}\right)^{\frac{k}{d}} \frac{dy}{y^2}}\] 
\begin{defn} Define $\zeta_{w,d} := -\pi_F z \, (\eta(J) - J) \in K((t))$. 
\end{defn}

The operator $\nabla$ also extends to $K((t))$, by Lemma \ref{ExtendToKt} and Lemma \ref{ExtendToY}(b). Because $z^d = \pi_F^k\rho^W_\infty(w)$ holds in $K((t))$,  Lemma \ref{ExtendToY}(b) implies that
\[ \nabla = -\frac{1}{\pi_F} z (y^2 \partial_y) z^{-1}.\]
\begin{prop}\label{ZetaSolvesDE} We have $\nabla(\zeta_{w,d}) = \rho^Y_\infty(c) - 1$.
\end{prop}
\begin{proof} Since $\eta$ and $y^2 \partial_y$ commute by Lemma \ref{ExtendToKt}(c), Lemma \ref{Jintegral} tells us that
\[ \begin{array}{lll} z^{-1} \nabla(\zeta_{w,d}) &=& z^{-1} \nabla( -\pi_F z (\eta(J) - J) ) \\ &=& (y^2 \partial_y) ( \eta(J) - J ) \\ &=& \eta( (y^2 \partial_y)(J) ) - (y^2 \partial_y)(J) \\ &=& \eta(z^{-1}) - z^{-1}.\end{array}\]
Next, recall from the proof of Lemma \ref{ExtendToY}(c) that $c^d = \frac{w}{h^{-1} \cdot w}$ in $\cO(Y)$. Hence
\[\rho^Y_\infty(c)^d = \frac{\rho^W_\infty(w)}{ \eta(\rho^W_\infty(w))} =  \left(\frac{z}{\eta(z)}\right)^d.\]
The definition of $\eta$ given at (\ref{DefnEta}) shows that $\eta$ respects the $t$-adic filtration on $K((t))$ and induces the trivial automorphism on the associated graded. Hence $\frac{z}{\eta(z)} \equiv 1 \mod t K[[t]]$. Looking at $(\ref{DefnOfC})$, we see that $\rho^Y_\infty(c) \in K((t))$ also satisfies $\rho^Y_\infty(c) \equiv 1 \mod t K[[t]]$. Since $1 + t K[[t]]$ is $d$-divisible, we conclude that
\[ \rho^Y_\infty(c) = \frac{z}{\eta(z)}.\]
Therefore $\nabla(\zeta_{w,d}) = z( \eta(z^{-1}) - z^{-1} ) = \rho^Y_\infty(c) - 1$ as desired.\end{proof}

\begin{prop} The element $t^{qk}(\eta(J) - J)$ lies in $K[[y]]$.
\end{prop}
\begin{proof}Let $r$ be an integer; then since $t/\eta(t) = (1-y)^{\frac{1}{d}} \in K[[y]]$ and $y/\eta(y) = 1-y \in K[[y]]$, using the binomial expansion for $(1 - y)^\alpha$ we see that
\[t^{qk}(\eta(t^{-qk} y^{r-1}) - t^{-qk} y^{r-1}) = y^{r-1} \left( \frac{t^{qk} y^{1-r}}{\eta(t^{qk} y^{1-r})} - 1\right) = y^r \left( \frac{ (1 - y)^{1 +\frac{qk}{d} - r} - 1}{y} \right)\]
lies in $K[[y]]$, provided that $r \geq 0$. Using Definition \ref{DefnOfJ}, and the above with $r = m(q-1)$ where $m \geq 0$, we conclude that
\begin{equation}\label{DivJ} t^{qk}(\eta(J)-J) = \sum\limits_{m=0}^\infty \alpha_m y^{m(q-1)} \left( \frac{ (1-y)^{1+ \frac{qk}{d} - m(q-1)} - 1}{y} \right)\end{equation}
lies in $K[[y]]$ as desired. \end{proof}
\begin{cor}\label{ZetaSeries} The element $\zeta_{w,d}$ lies in $K[[y]]$. In fact, we have
\[\zeta_{w,d} = \pi_F (1 - y^{q-1})^{-\frac{k}{d}} \sum_{m=0}^\infty (-1)^m \binom{\frac{k}{d}}{m} y^{(q-1)m} \left[ \frac{ (1 - y)^{1 + \frac{qk}{d} - (q-1)m} - 1}{\left(1 + \frac{qk}{d} - (q-1)m\right)y} \right].\]
\end{cor}
\begin{proof} Recall that $z = \epsilon^{-1} t^{qk}$ where $\epsilon = (1 - y^{q-1})^{\frac{k}{d}} \in K[[y]]^\times$. Then
\[\zeta_{w,d} = -\pi_F z (\eta(J)-J) = -\pi_F (1-y^{q-1})^{-\frac{k}{d}} t^{qk} (\eta(J)-J).\]
Substituting $\alpha_m = \frac{(-1)^m \binom{\frac{k}{d}}{m}}{(q-1)m - \frac{qk}{d} - 1}$ into (\ref{DivJ}), gives the formula for $\zeta_{w,d}$.\end{proof}
For future use, we record the following 
\begin{prop}\label{UniqueSol} $\zeta_{w,d}$ is the \emph{unique} solution to the differential equation \newline $\nabla(\zeta) = \rho^Y_\infty(c) - 1$ with $\zeta \in K((y))$.
\end{prop}
\begin{proof} Proposition \ref{ZetaSolvesDE} and Corollary \ref{ZetaSeries} tell us that $\zeta_{w,d}$ is \emph{a} solution of $\nabla(\zeta) = \rho^Y_\infty(c)-1$, so we need to show it is unique. It is enough to show that the map $\nabla : K((y)) \to K((y))$ is injective. Now, $\nabla = -\frac{1}{\pi_F} \epsilon^{-1}  (y^2 \partial_y - \frac{qk}{d} y) \epsilon$ by Lemma \ref{ExtendToY}(b), so it is enough to show that $y \partial_y - \frac{qk}{d} : K((y) \to K((y))$ is injective. Since 
\[ (y \partial_y - \frac{qk}{d}) \left(\sum_n a_n y^n \right) = \sum_n \left(n - \frac{qk}{d}\right) a_n y^n\]
this is an immediate consequence of the fact that $\frac{qk}{d} \notin \bZ$. \end{proof}

Next, we introduce the following affinoid subdomain of $\bP^{1,\an}$:
\[ W := \Sp K \langle y / \pi_F \rangle.\]
We identify $Y$ with $\Sp K \langle y, \frac{1}{y^{q-1}-1}\rangle$ using Corollary \ref{CoordRingOfY}, and note that $W$ is contained in $Y$ because $1 - y^{q-1} = 1 + \pi_F^{q-1} (y/\pi_F)^{q-1} \in \cO(W)^{\times\times}$.
\begin{lem}\label{ROCzeta} $\cF(W) \neq \emptyset$.
\end{lem}
\begin{proof} Recall \cite[Definition 13.1.1]{Kedlaya} that the \emph{type} of a number $\lambda \in K$ is the radius of convergence of the formal power series $\sum\limits_{n=0}^\infty \frac{t^n}{\lambda - n}$. If $\lambda$ happens to lie in $\Zp \cap \mathbb{Q}$, then $\lambda$ has a \emph{recurrent} $p$-adic expansion, and now \cite[Proposition 13.1.4]{Kedlaya} implies that $\lambda$ has type $1$. In other words, for any real number $0 < r < 1$, we have
\begin{equation}\label{RatTypeOne} \lim\limits_{n \to \infty} \frac{r^n}{|\lambda - n|} = 0 \qmb{for any} \lambda \in \Zp \cap \mathbb{Q}.\end{equation}
Now, the spectral norm $|\cdot|_W$ on $\cO(W) = K \langle y / \pi_F\rangle$ satisfies $|y|_W = |\pi_F| < 1$. Then for any $m \geq 0$, we have
\[\left\vert (-1)^m \binom{\frac{k}{d}}{m} y^{(q-1)m} \left[ \frac{ (1 - y)^{1 + \frac{qk}{d} - (q-1)m} - 1}{\left(1 + \frac{qk}{d} - (q-1)m\right)y} \right] \right\vert_W \leq \frac{ |\pi_F|^{(q-1)m-1} }{|\frac{qk}{d} - ((q-1)m - 1)| }\]
because $|(-1)^m \binom{\frac{k}{d}}{m}| \leq 1$ and $|(1-y)^{1 - \frac{k}{d} - (q-1)m} - 1|_W \leq 1$.  Because $\frac{qk}{d} \in \Zp \cap \mathbb{Q}$, (\ref{RatTypeOne}) implies that the partial sums
\[\zeta_n := \pi_F (1 - y^{q-1})^{-\frac{k}{d}} \sum_{m=0}^n (-1)^m \binom{\frac{k}{d}}{m} y^{(q-1)m} \left[ \frac{ (1 - y)^{1 + \frac{qk}{d} - (q-1)m} - 1}{\left(1 + \frac{qk}{d} - (q-1)m\right)y} \right]\]
of the power series $\zeta_{w,d}$ converge with respect to the spectral norm $|\cdot|_W$ to an element $\zeta' \in \cO(W)$. Because $\zeta_n - \zeta_{n-1}$  converges to zero $y$-adically in $\cO(W)$, it follows that $\zeta_n$ also converges $y$-adically to $\zeta'$ in $\cO(W)$. 

Recall the $\cD(W)$-linear and injective map $\rho^W_\infty : \cO(W) \to \widehat{\cO_{Y,\infty}} = K[[y]]$ from Lemma \ref{GermAtInfty}. Since $\rho^W_\infty$ is continuous with respect to the $y$-adic topologies on $\cO(W)$ and $K[[y]]$, we deduce that 
\[\rho^W_\infty(\zeta') = \rho^W_\infty(\lim\limits_{n\to \infty} \zeta_n) = \lim\limits_{n\to\infty} \rho^W(\zeta_n) = \zeta_{w,d}.\]
Since $\rho^W_\infty$ is $\cD(W)$-linear, $\rho^W_\infty(\nabla(\zeta')) = \nabla(\zeta_{w,d}) = \rho^W_\infty(c_{|W}-1)$ by Proposition \ref{UniqueSol}, so $\nabla(\zeta') = c_{|W}-1$ because $\rho^W_\infty$ is injective. Thus $\zeta' \in \cF(W)$ as desired. \end{proof}

Our next result relies on a long computation of the $p$-adic valuations of certain binomial coefficients, which we perform in $\S \ref{BinEsts}$.

\begin{thm}\label{ZetaUnbounded} The series $\zeta_{w,d} \in K[[y]]$ does \emph{not} have bounded coefficients.
\end{thm}
\begin{proof} Using the binomial theorem, we compute that for any $0 \neq \mu \in \Zp$, 
\[ \frac{ (1-y)^\mu - 1 }{\mu y} = \frac{1}{\mu y} \sum_{n=1}^\infty \binom{\mu}{n} (-y)^n  = -\frac{1}{\mu} \sum_{n=0}^\infty \binom{\mu}{n+1} (-y)^n.\]
Let $s := y^{q-1}$, and let $\Phi : K[[y]] \twoheadrightarrow K[[s]]$ be the $K[[s]]$-linear projection operator, which sends $y^i$ to $0$ if $q-1 \nmid i$, and which is the identity on $K[[s]] \subset K[[y]]$. Using the elementary identity $\frac{1}{\mu} \binom{\mu}{n+1} = \frac{(\mu-1)\cdots(\mu-n)}{(n+1)!} = \frac{1}{n+1}\binom{\mu-1}{n}$ we compute that
\begin{equation}\label{PhiMuY}\Phi \left(\frac{ (1-y)^\mu - 1 }{\mu y}\right) = - \sum_{\ell = 0}^\infty \binom{\mu-1}{\ell(q-1)} \frac{ (-s)^{\ell} }{\ell(q-1)+1}.\end{equation}
Set $\mu_r := 1 + \frac{qk}{d} - (q-1)r$ and note that $\mu_r \neq 0$ for any $r \in \N$ because $\frac{k}{d} \notin \Z$ by our choice of $k$ and $d$. Take the series appearing in Corollary \ref{ZetaSeries}, change the dummy variable from $m$ to $r$, apply $\Phi$ and substitute in $(\ref{PhiMuY})$ to obtain
\[ \begin{array}{lll} \frac{1}{\pi_F} (1-s)^{\frac{k}{d}} \Phi(\zeta_{w,d}) &=& \sum\limits_{r=0}^\infty  \binom{\frac{k}{d}}{r} (-s)^r \cdot \left( - \sum\limits_{\ell = 0}^\infty \binom{\mu_r-1}{\ell(q-1)} \frac{ (-s)^{\ell} }{\ell(q-1)+1} \right) \\
&=& -\sum\limits_{n=0}^\infty \left( \sum\limits_{r=0}^n  \frac{ \binom{\frac{k}{d}}{r} \binom{\frac{qk}{d}-(q-1)r}{(n-r)(q-1)} }{(n-r)(q-1)+1} \right) s^n.\\
\end{array}\]
We observe that the formula on the right hand side only involves $\frac{k}{d}$. Since $d \mid (q+1)$ and $1 \leq k < d$, after multiplying both $k$ and $d$ by $\frac{q+1}{d}$ if necessary, we may assume that $d = q+1$ and $1 \leq k \leq q$.

By Corollary \ref{SumEstimate} below, the $p$-adic valuation of the coefficient of $s^n$ appearing in the above expansion is not bounded below as $n$ varies: in other words, this series does not have bounded coefficients. Since the operator $\Phi$ sends $K^\circ[[y]]_K$ to $K^\circ[[s]]_K$ we conclude that $\zeta_{w,d}$ also does not have bounded coefficients.
\end{proof}

\begin{cor}\label{NoSolsOnY} We have $\cF(Y) = \emptyset$.
\end{cor}
\begin{proof} Suppose for a contradiction that $\zeta \in \cO(Y)$ satisfies $\nabla(\zeta) = c - 1$. Now $\cO(Y) = K\langle y,\frac{1}{y^{q-1}-1}\rangle$ by Corollary \ref{CoordRingOfY}, so 
\[\rho^Y_\infty\left(\cO(Y)\right) \subseteq K^\circ[[y]]_K.\] 
Since $\rho^Y_\infty : \cO(Y) \to K[[y]]$ is $\cD(Y)$-linear by Lemma \ref{GermAtInfty}, we have $\nabla(\rho^Y_\infty(\zeta)) = \rho^Y_\infty(c) - 1$, so $\rho^Y_\infty(\zeta) = \zeta_{w,d}$ by Proposition \ref{UniqueSol}. But then $\zeta_{w,d} \in K^\circ[[y]]_K$ which contradicts Theorem \ref{ZetaUnbounded}.\end{proof}

\begin{lem}\label{XcapV0sols} We have $|\cF(X \cap V_0)| \leq 1$.
\end{lem}
\begin{proof} It is enough to show that $\nabla = \theta_{w,d}(\partial_x)$ is injective on $\cO(X \cap V_0)$. Now,
\[w = \frac{1}{(x^q - \pi_F^{q-1} x)^k} = \frac{x^{-qk}}{(1 - (\pi_F/x)^{q-1})^k} \equiv x^{-qk} \, \mod \cO(X \cap V_0)^{\times\times}.\]
Since $\cO(X \cap V_0)^{\times\times}$ is $d$-divisible by \cite[Lemma 4.3.2(a)]{ArdWad2023}, we can now use Lemma \ref{TwistsMultiply} and Lemma \ref{ThetaDthPower} to see that
\[\nabla = u  \hsp \theta_{x^{-qk}, d}(\partial_x) \hsp u^{-1} \qmb{for some} u \in \cO(X \cap V_0)^\times.\]
Thus it remains to show that $\theta_{x^{-qk},d}(\partial_x) = \partial_x + \frac{qk}{d}\frac{1}{x}$ is injective on $\cO(X \cap V_0)$. Suppose that $\partial_x(v) = -\frac{qk}{d} \frac{v}{x}$ for some non-zero $v \in \cO(X \cap V_0)$. Then $\partial_x(v^dx^{qk}) = 0$ so $v^d = \lambda x^{-qk}$ for some $\lambda \in K$. Then $\lambda \neq 0$ since $v \neq 0$, so $v \in \cO(X \cap V_0)^\times$. Let $D := \{|x| < 1\}$ and apply the map $\mu_{X \cap V_0}$ from \cite[Proposition 4.3.1]{ArdWad2023} to see that $d \cdot \mu_{X \cap V_0}(v)(D) = -qk \cdot \mu_{X \cap V_0}(x)(D) = -qk$. This contradicts $-\frac{qk}{d} \notin \bZ$.\end{proof}
We can finally prove the main result of $\S \ref{TheIntegral}$.
\begin{thm}\label{BetaNotInSum} Let $w = \frac{1}{(x^q - \pi_F^{q-1} x)^k}$, where $d \mid (q+1)$ and $1 \leq k \leq d$. Then \[\beta(h) \notin \cD^\dag_{\varpi/|\pi_F|}(X) + \cD^\dag_\varpi(X) R_{\cS(w)}(w,d).\]
\end{thm}
\begin{proof} First, we treat the case when $k\neq d$. Suppose for a contradiction that $\beta(h) \in \cD^\dag_{\varpi/|\pi_F|}(X) + \cD^\dag_\varpi(X) R_{\cS(w)}(w,d)$. Then we can find $\zeta \in \cO(X \cap V_1)$ such that $\zeta_{|X \cap V_0} = -(\xi_{w,d} \beta(h))_0$ by Proposition \ref{XiBetaZero}. Now, $-(\xi_{w,d} \beta(h))_0 \in \cF(X \cap V_0)$ by Lemma \ref{XVzeroSol}. Because the restriction map $\cO(X \cap V_1) \to \cO(X \cap V_0)$ is injective, we deduce that $\zeta \in \cF(X \cap V_1)$. 

Now, since $W({\bf C}) = \{a \in {\bf C} : |a| \geq 1\} \cup \{\infty\}$ in the $x$-coordinate, we see that $X \cap W = X \cap V_0$. Since also $V_0 \subseteq V_1$, we have
\[(X \cap V_1) \cap W = X \cap V_0.\]
On the other hand, if $a \in {\bf C}$ and $|a - \zeta^i| < 1$ for some $i = 0 ,\cdots, q-2$, then $|a| = 1$ and $a \in W({\bf C})$. Therefore $W$ contains all of the $q$ `large' holes of $X \cap V_1$. Hence
\[ (X \cap V_1) \cup W = Y\]
and $\{W, X \cap V_1\}$ is an affinoid covering of $Y$. 

By Lemma \ref{ROCzeta}, we can find some $\zeta' \in \cF(W)$. Then both $\zeta|_{X \cap V_0}$ and $\zeta'|_{X \cap V_0}$ lie in $\cF(X \cap V_0)$, so they are equal by Lemma \ref{XcapV0sols}. Finally, Tate's Acyclicity Theorem implies that the local solutions $\zeta \in \cF(X \cap V_1)$ and $\zeta' \in \cF(W)$ glue to give an element in $\cF(Y)$, which contradicts Corollary \ref{NoSolsOnY}.

The case $k = d$ requires a special, but much easier, argument. First note that when $k = d$, $w = (x^q - \pi_F^{q-1}x)^{-d}$ is a $d$-th power,  and in fact we that we have
\[ w = \Delta_{\cS(w)}^{-d}.\]
Hence we can apply Lemma \ref{RudProps}(b) to see that
\[R_{\cS(w)}(w,d) = \theta_{w \Delta_{\cS(w)}^d, d}(\partial_x) \Delta_{\cS(w)} = \partial_x \Delta_{\cS(w)},\]
because $\theta_{1,d}$ is the identity map in view of equation (\ref{ThetaTwist}).  Therefore $R_{\cS(w)}(w,d)$ lies in the left ideal $\cD^\dag_\varpi(X) x$, and it will be enough to show that
\[ \beta(h) \notin \cD^\dag_{\varpi/|\pi_F|}(X)  + \cD^\dag_\varpi(X)  x.\]
Suppose for a contradiction that $\beta(h) = P + Q x$ for some $P \in \cD^\dag_{\varpi/|\pi_F|}(X)$ and $Q \in \cD^\dag_\varpi(X)$. Hence $P \in \cD_r(X)$ for some $r > \varpi/|\pi_F|$ and $Q \in \cD_s(X)$ for some $s > \varpi$, by Definition \ref{DagSite}(d). Since $r > \varpi/|\pi_F| > \varpi$, we may shrink $s$ if necessary to ensure that $r \geq s > \varpi$. Now, applying the involutive transpose automorphism $(-)^T : \cD_s(X) \to \cD_s(X)$ from Lemma \ref{Transpose} shows that
\begin{equation}\label{TranspBetaH} \beta(h)^T = P^T + x Q^T.\end{equation}
By Definition \ref{AdelR} we can find a sequence of elements $(a_n)_{n=0}^\infty$ in $\cO(X)$, such that
\[ \lim\limits_{n\to\infty} |a_n|_Xr^n = 0 \qmb{and} P^T = \sum\limits_{n=0}^\infty a_n \partial^n \in \cD_r(X).\]
Next, we use the fact that $(-)^T$ is continuous and $\partial^T = -\partial$ to compute
\[ \beta(h)^T = \left(\sum\limits_{n=0}^\infty \frac{\pi_F^n}{n!} \partial^n\right)^T = \sum\limits_{n=0}^\infty \frac{\pi_F^n}{n!} (-\partial)^n = \beta(-h).\]
Setting $x = 0$ in equation (\ref{TranspBetaH}) and equating coefficients shows that
\[ \frac{ (-\pi_F)^n }{n!} = a_n(0) \qmb{for all} n \geq 0.\]
Finally, Lemma \ref{nFacVarpi} tells us that $\frac{1}{|n!|} \geq \frac{pn}{\varpi^n}$. Hence
\[ pn \cdot \left(\frac{|\pi_F|r}{\varpi}\right)^n \leq \frac{|\pi_F|^nr^n}{|n!| }|a_n(0)| \leq |a_n|_X r^n \to 0 \qmb{as} n \to \infty\]
which is impossible because $r > \varpi / |\pi_F|$ implies that $\frac{|\pi_F|r}{\varpi}>1$.
\end{proof}

\subsection{Estimates of \ts{p}-adic valuations of binomial coefficients}\label{BinEsts}

The goal of this section is to estimate the $p$-adic growth rate of the coefficient \begin{equation} \label{coeff}\sum_{r=0}^n \frac{ \binom{\frac{k}{d}}{r} \binom{\frac{qk}{d}-(q-1)r}{(n-r)(q-1)} }{(n-r)(q-1)+1} \end{equation} that arises in the proof of Theorem \ref{ZetaUnbounded} for $1\leq k\leq q$. In particular we will  see that for certain special values of $n$, the term corresponding to the value of $r$ that maximises the $p$-adic valuation of the denominator $(n-r)(q-1)+1$ will dominate this sum and moreover the $p$-adic norm of this term will be unbounded as $n$ increases through these special values. To achieve this we will appeal to a Theorem of Kummer to estimate the $p$-adic valuation of the binomial coefficients that appear in the sum.

For the remainder of the section we {\bf fix $1\leq k\leq q$ and define $f=v_p(q)$.}

Let $\lambda \in \Zp$ and $n \in \N$ be given, and consider their $p$-adic expansions
\[ \lambda = \lambda_0 + p \lambda_1 + p^2 \lambda_2 + \cdots, \quad n = n_0 + p n_1 + p n_2 + \cdots \]
where the $p$-adic digits $\lambda_i$ and $n_i$ all lie in $\{0,1,...,p-1\}$. Since $n$ is assumed to be a natural number, we know that $n_i = 0$ for all sufficiently large $i$. 

\begin{defn}\label{CarryDefn} Let $\lambda \in \Zp$ and $n \in \N$ be as above.
\be \item Let $i \geq 0$. The  \emph{$i$-th carry function} $\gamma_i(\lambda, n)$ is defined as follows:
\[\gamma_0(\lambda,n) = \left\{ \begin{array}{lll} 1 &\qmb{if} & \lambda_0 + n_0 > p-1 \\ 0 & \qmb{if} & \lambda_0 + n_0 \leq p-1\end{array}\right.\]
and for $i \geq 1$, it is defined recursively by
\[ \gamma_i(\lambda,n) = \left\{ \begin{array}{lll} 1 &\qmb{if} & \lambda_i + n_i + \gamma_{i-1}(\lambda,n) > p-1 \\ 0 & \qmb{if} & \lambda_i + n_i + \gamma_{i-1}(\lambda,n) \leq p-1.\end{array}\right.\]
\item We define
\[L(\lambda, n) := \inf\{ i \geq 0 : \gamma_j(\lambda,n)=0 \qmb{for all} j\geq i\}.\]
\item For each $m \geq 0$, the \emph{$m$-th non-carry function} is defined to be
\[ \cN_m(\lambda, n) := |\{0 \leq i < m : \gamma_i(\lambda,n) = 0\}|.\]
\item We define $\langle \lambda | n \rangle$ to be the binomial coefficient $\binom{\lambda+n}{n}$.
\ee\end{defn}
Thus $L(\lambda,n)$ records the position after which there are no further carries when performing the addition of $\lambda$ and $n$ in $\Zp$; it is possible that $L(\lambda,n) = \infty$ but this happens only when $\lambda_i = p-1$ for all sufficiently large $i$. More precisely we have the following Lemma.

\begin{lem}\label{CarStop} $L(\lambda,n)=\infty$ only when $\lambda$ is a negative integer and $n\geq -\lambda$. 
\end{lem}
\begin{proof}
Since $n\in \bN$ there is $m\geq 1$ such that $n_j=0$ for all $j\geq m$. 

Suppose that $L(\lambda,n)=\infty$. Then by definition $\gamma_j(\lambda,n)=1$ for infinitely many values of $j$. But, by \ref{CarryDefn}(a), if $\gamma_j(\lambda,n)=1$ for $j\geq m$ then $\gamma_{j-1}(\lambda,n)=1$ and $\lambda_{j}=p-1$. It follows that $\gamma_j(\lambda,n)=1$ for all $j\geq m-1$ and $\lambda_j=p-1$ for all $j\geq m$ and $\lambda$ is a negative integer. Moreover $(\lambda+n)_j=0$ for all $j\geq m$; that is $\lambda+n\in bN$ and so $n\geq -\lambda$ as claimed.
\end{proof}
In our applications, $L(\lambda,n)$ will always be finite. The function $\cN_m(\lambda,n)$ records the number of non-carries that occurred in the first $m$ digits when performing the addition of $\lambda$ and $n$ in $\Zp$.

We have now developed enough language to precisely state Kummer's theorem on the $p$-adic valuation of binomial coefficients.

\begin{thm}[Kummer, 1852]\label{Kum} Let $\lambda \in \Zp$ and $n \in \N$ be given, and suppose that $L(\lambda,n) < \infty$. Then 
\[ v_p \Bin{\lambda}{n} = m - \cN_m(\lambda,n) \qmb{for all} m \geq L(\lambda,n).\]
\end{thm}

Under the assumption that $L(\lambda,n) < \infty$, the carrying will stop in finite time in the sense that eventually all $\gamma_i(\lambda,n)$ are zero. Then for sufficiently large $m$, the quantity on the right hand side is independent of $m$ and returns the total number of carries that occurred during the addition of $\lambda$ and $n$ in $\Zp$ --- this is perhaps a more usual way of formulating Kummer's theorem. We also note that if $L(\lambda,n)=\infty$ then $\langle \lambda|n\rangle=0$ so $v_p(\langle \lambda|n\rangle)=\infty$ in that case.

\begin{proof}[Proof of Theorem \ref{Kum}] A proof in the case that $\lambda\in \bN$ appeared in \cite{Kummer1852}. A modern account of the proof in that case can be found in \cite[Theorem 3.7]{Granville}. 
	
Now let $\mu\in \bN$ be such that $0\leq \mu< p^{L(\lambda,n)}$ and $v_p(\lambda-\mu)\geq L(\lambda,n)$; that is the $p$-adic expansion of $\mu$ is the truncation of the $p$-adic expansion of $\lambda$ at the $L(\lambda,n)$th digit. Then it is easy to verify that $L(\lambda,n)=L(\mu,n)$ and $\cN_m(\lambda,n)=\cN_m(\mu,n)$ for all $m$. Moreover $v_p\Bin{\lambda}{n}=v_p\Bin{\mu}{n}$ since $v_p(\lambda+i)=v_p(\mu+i)\leq L(\lambda,n)$ for all $1\leq i\leq n$. 
\end{proof}
In most cases of interest to us, the possibility of carrying process \emph{not} stopping in finite time is ruled out by the following 


\begin{notn} Given $\alpha_0,\alpha_1,\cdots, \alpha_{j-1} \in \{0,1,\cdots, q-1\}$, write
	\[ [\alpha_0, \alpha_1, \cdots, \alpha_{j-1}]_j := \alpha_0 q^0 + \alpha_1 q^1 + \cdots + \alpha_{j-1} q^{j-1}.\]
	We call $\alpha_i$ the  \emph{$i$th $q$-adic digit} of this \emph{$q$-expansion}. \end{notn}

\begin{lem}\label{nocarry} Suppose that $\lambda\in \bZ_p$, $n,j\in \bN$ with \[\lambda\equiv [\alpha_0,\ldots,\alpha_{j-1}]_j \mod q^j \mbox{ and  }n=[\beta_0,\ldots, \beta_{j-1}]_j. \]\be \item If $\beta_i+\alpha_i<q-1$ for some $0<i<j$ then $\gamma_{(i+1)f-1}(\lambda,n)=0$. 
 \item If $\beta_{j-1}+\alpha_{j-1}<q-1$ then $L(\lambda,n)<(j+1)f$.  
\ee\end{lem}

\begin{proof} (a) Under the hypothesis $\beta_i+\alpha_i<q-1$ \[\gamma_{(i+1)f-1}(\lambda,n)=\begin{cases} \gamma_{f-1}(\alpha_{i},\beta_i) & \mbox{if }\gamma_{if-1}(\lambda,n)=0 \\ \gamma_{f-1}(\alpha_i+1,\beta_i) & \mbox{if }\gamma_{if-1}(\lambda,n)=1\end{cases}\] and since $\alpha_i+1+\beta_i<q$ this is zero in either case. 

(b) Since $n<q^j$, we have $n_i=0$ whenever $i\geq jf$. The result now follows from part (a) and Definition \ref{CarryDefn}(a).
\end{proof}

\begin{defn} For $n\geq 1$, let $M=M_n\in \bN$ be largest possible such that $1+q+\cdots+q^M\leq n$. We also define \[s:=s_n:=n- (1+q+\cdots + q^M)\geq 0.\]
\end{defn}

It is easy to see that $s_n<q^{M_n+1}$ since otherwise we could make $M$ larger.

\begin{lem} \label{Arch} Let $0 \leq r \leq n$ be an integer. Then 
	\[ v_p(r-s_n) < (M_n+1)f\]
	whenever $r \neq s_n$. 
\end{lem}
\begin{proof}
	Since $s<q^{M+1}$ and $n-s=1+q+\cdots + q^M< q^{M+1}$, \[s - q^{M+1} < 0 \leq s \leq n < s + q^{M+1},\] which implies that 
	\[\{0,1,\cdots, n\} \cap (s + q^{M+1} \bZ) = \{s\}.\]
	Suppose that $v_p(r-s) \geq (M+1)f$. Then $q^{M+1}$ divides $r - s$, so $r \in s + q^{M+1} \bZ$. Since $0 \leq r \leq n$ by assumption, we conclude that $r = s$.\end{proof}

\begin{cor}\label{vprs} Let $0 \leq r \leq n$ be an integer. Then
	\[ v_p((n-r)(q-1) + 1) = \left\{ \begin{array}{lll} (M_n+1)f &\qmb{if}& r = s_n, \\ v_p(r - s) &\qmb{if} & r \neq s_n. \end{array} \right. \]
\end{cor}
\begin{proof} By definition of $s$, $v_p((n-s)(q-1)+1) = v_p(q^{M+1}) = (M+1)f$. Now
	\[ (n-r)(q-1) + 1 = (q-1)(s - r) \quad + \quad (n-s)(q-1) + 1\]
	and $m := v_p(r - s)$ is strictly less than $v_p(q^{M+1}) = (M+1)f$ whenever $r \neq s$ by Lemma \ref{Arch}. Since $v_p(q-1) = 0$, the result now follows from the non-Archimedean triangle inequality. \end{proof}

Thus, combining the last two results, we see that $s_n$ has been chosen to maximise $p$-adic value of the denominator of (\ref{coeff}) for given $n$.
 
We are also interested in the following complicated-looking binomial coefficient.
\begin{defn}\label{SNR} For each $n, r \in \N$ with $0 \leq r \leq n$, define
\[S_{n,r} := \Bin{ \frac{qk}{q+1} - n(q-1) \hsp }{\hsp (n-r)(q-1)} = \binom{\frac{qk}{q+1}-r(q-1)}{(n-r)(q-1)}.\]
\end{defn}

In order to control this, until Corollary \ref{SumEstimate} below, we \textbf{fix} a possibly very large positive integer $N \geq 0$, and define $n_N$ to be the smallest positive integer such that 
\[n_N\equiv \frac{qk}{q^2-1}\mod q^{N}\] 

This expression obviously depends on $N$, but because of its frequent appearance we will abbreviate it to $n := n_N$ until Corollary \ref{SumEstimate}. The reason for this choice of the form of $n$ is the following result which will enable us to ignore the $S_{n,r}$ in the numerator of (\ref{coeff}) when estimating its $p$-adic value. 

\begin{prop}\label{SecondVal}  $v_p(S_{n,r})\geq 0$ for all $0\leq r\leq n$ and $v_p(S_{n,s_n}) = 0$.
\end{prop}
\begin{proof}The first part is well-known since $qk/(q+1)$ is a $p$-adic integer.
	
	We define $\alpha := \frac{qk}{q+1} - n(q-1)\in \bZ_p$ so that $\alpha  \equiv 0 \mod q^{N}$.
	That is, the $p$-adic digits $\alpha_j$ of $\alpha$ are  zero whenever $j < Nf$. Similarly, since $M<N$, 
	\[\beta := (n-s)(q-1) \leq q^{M+1}-1 < q^{N}\]
	so $\beta_j = 0$ for all $j \geq Nf$. It follows that $\gamma_j(\alpha,\beta) = 0$ for all $j \geq 0$, and therefore $v_p(S_{n,s}) = v_p \Bin{\alpha}{\beta} = 0$ by Theorem \ref{Kum}.
\end{proof}

 For reasons that will become apparent later {\bf we now assume \[\begin{cases} N \mbox{ is even } & \mbox{if } 1\leq k\leq q-1 \mbox{ or } k=2=q \\ N \mbox{ is odd } & \mbox{if } k=q>2.\end{cases}\]} It is also convenient to suppose that $N\geq 6$. 

We can now compute the precise value of $M_n$ in terms of $N$ depending on the values of $k$ and $q$.

\begin{lem}\label{valM} \hfill \[ M_n = \begin{cases}    N-1 & \mbox{if }  1\leq k \leq q-2 \mbox{ or }k=q>2,\\ 
	N-2 & \mbox{if }  \left(k=q-1 \mbox{ and } q>2\right) \mbox{ or } k=q=2,\\ 
	N-3 & \mbox{if }k=1, q=2. \end{cases}\]	
\end{lem}
\begin{proof}
	Since $n\leq q^N$ it is immediate that $M\leq N-1$ in all cases. 	 
		 
	We see that \[ n\equiv 1+ \frac{(q-k-1)q+(q-1)}{1-q^2}=2+ \frac{(2q-k-1)q+(q-2)}{1-q^2} \mod q^{N} \] so, for $ 1\leq  k\leq q-1$, since $N$ is even, the $q$-adic expansion of $n$ is  
	\[ n= 0 + (q-k)q + (q-1)q^2 + (q-1-k)q^3  + \cdots + (q-1)q^{N-2} + (q-1-k)q^{N-1},  \] 
 For $k=q>2$, since $N$ is odd, the $q$-adic expansion of $n$ is 
 \[ n = 0 + 0\cdot q + (q-1)q^2 + (q-1)q^3 + (q-2)q^4
 + \cdots +  (q-1)q^{N-2} + (q-2)q^{N-1},   \] and for $k=q=2$, since $N$ is even, 
\[ n = 0 + 0\cdot q + (q-1)q^2 + (q-1)q^3 + (q-2)q^4
+ \cdots +  (q-2)q^{N-2} + (q-1)q^{N-1},   \]
 where in all cases, the $q$-digits after the $3$rd digit repeat with period $2$ until the last $q$-digit.

If $1\leq k\leq q-2$ then $q \geq 3$, and it is now easy to see that \[ 1+ q + q^2 + \cdots + q^{N-1} < 2q^{N-2} + q^{N-1}< n\] so $M=N-1$. 

Similarly if $k=q-1$ and if $q \geq 3$, then \[1+q+q^2 + \cdots + q^{N-2}<2q^{N-2}<n<q^{N-1} \] 
and so $M=N-2$, and if $k=q$ and $q \geq 3$, then
 \[ 1+q+q^2+\cdots + q^{N-1}<n<q^N\] 
so $M = N-2$.

The two cases where $q=2$ can be treated in a similar manner noting that when $q=2$, $1+q+q^2+\cdots +q^M=q^{M+1}-1$. 
\end{proof}

Our next job will be to compute the $q$-adic expansions of $\frac{k}{q+1} - s_n$ and $s_n$ up to the $f(M+1)$st $p$-adic digit in all cases.

\begin{lem}\label{qExpCalc} Write $\lambda := \frac{k}{q+1} \in \Zp \backslash \N$.

\be \item Suppose that $1 \leq k \leq q-2$. Then $M$ is odd and
\[ \begin{array}{rclcccccr} s &=& [q-1, & q-k-2, &q-2, & q-k-2, &\cdots, & q-2, & q-k-2]_{M+1} \\
\lambda-s &\equiv& [k+1, & 1, & k+1, & 1, &\cdots, &k+1 , &1]_{M+1}
\end{array}\]
\item Suppose that $k = q-1$ and $q>2$. Then $M$ is even and 
\[ \begin{array}{rclccccr} s &=& [q-1, & q-1, &q-3, &\cdots, &q-1, &q-3]_{M+1} \\
\lambda-s &\equiv& [0, &2, &0, &\cdots, &2, &0]_{M+1}\end{array}\]

\item Suppose that $k=1$ and $q=2$. Then $M$ is odd and 
\[ \begin{array}{rclcccccr} s & = & [1, & 1 , & 1 ,& 0, & \cdots, & 1 , & 0]_{M+1} \\
	\lambda-s &\equiv& [0, &0, &1, &0, & \cdots , &1, &0 ]_{M+1}
	\end{array}\]
\item Suppose that $k=2$ and $q=2$. Then $M$ is even and
\[ \begin{array}{rclccccccr} s & = & [ 1, & 0, & 1 ,& 1, & 0, & \cdots, & 1 , & 0]_{M+1} \\
\lambda-s &\equiv& [1, &0, &0, & 1, & 0 & \cdots, &1, &0]_{M+1}\end{array} \]
\item Suppose that $k = q$ and $q>2$. Then  $M$ is even and
\[ \begin{array}{rclccccr} s &=& [q-1, &q-2, &q-3, &\cdots, &q-2, &q-3]_{M+1} \\
\lambda-s &\equiv& [1, &2, &1, &\cdots, &2, &1]_{M+1}
\end{array}\]\ee
where $\equiv$ denotes congruence modulo $q^{M+1}$.

Note that all of the displayed expansions are recurrent with period $2f$ but not purely periodic: the two $q$-digits just preceeding the $\cdots$ sign are the repeating ones. \end{lem}
\begin{proof} The parity of $M$ can be deduced in each case from Lemma \ref{valM} and the parity of $N$. 
	
	 Since $q^{M+1}$ divides $q^N$ in all cases, \begin{eqnarray*} n &\equiv &\frac{-qk}{1-q^2} \mod q^{M+1} \mbox{ so} \\  s &\equiv& \frac{-qk}{1-q^2}-\frac{1}{1-q} \mod q^{M+1}\\ & \equiv &\frac{-qk-q-1}{1-q^2} \mod q^{M+1}\\& \equiv & 1 + \frac{(q-2)+ (q-k-2)q}{1-q^2} \\ &\equiv & 2 + \frac{(q-3)+(2q-k-2)q}{1-q^2}, \mbox{ and }\\ 
		\lambda-s & \equiv & \frac{k}{q+1}+\frac{qk+q+1}{1-q^2}=\frac{k+q+1}{1-q^2} \mod q^{M+1}.\end{eqnarray*}	
All the calculations can now be done in a straightforward manner. For the case where $k = q$, it is helpful to note that $s \leq n < q^{2N+1}$ by definition of $n$, so that the last two $q$-digits in the $q$-expansion of $s$ are $q-3$ and $q-2$. For the cases where $q=2$ it is helpful to remember that $s+2^{M+1}=n+1$ and to use the $q$-expansions of $n$ found in the proof of Lemma \ref{valM}.
\end{proof}

\begin{cor}\label{NoCarryInLastSpot} In all five cases appearing in Lemma \ref{qExpCalc}, the carrying in the sum $s + (\lambda - s)$ stops before the $(M+1)f$ position. More precisely, we have
\[L(\lambda-s,s) < (M+1)f.\]
 \end{cor}
We note that the parity of $N$ in each case was chosen to make this true. 
\begin{proof} In each case we can use the expressions for $s$ and $\lambda-s$ in Lemma \ref{qExpCalc} and Lemma \ref{nocarry}(b). 
\end{proof}

The remaining substantive result will be the following Proposition, whose proof highlights the importance of \emph{counting the }non\emph{-carry positions} in the sum $(\lambda - r) + r = \lambda$.
\begin{prop}\label{CountNoCarries} Let $\lambda \in \Zp \backslash \N$, $s \in \N$ and $\ell \geq 1$ be given, such that
\[L(\lambda-s,s) < \ell.\]
Then for any $r \in \N$ with $v_p(r-s) < \ell$, we have
\[ v_p \Bin{\lambda-r}{r} - v_p(r - s) \quad  > \quad v_p \Bin{\lambda-s}{s} - \ell.\]
\end{prop}
\begin{proof} Theorem \ref{Kum} implies that
\begin{equation}\label{vpBin} v_p\Bin{\lambda-s}{s} = \ell - \cN_\ell(\lambda-s,s).\end{equation}
Define the \emph{$m$-th truncation function} $t_m : \Zp \to \N$ for $m \geq 0$ by
\[ t_m \left( \sum_{i=0}^\infty \mu_i p^i \right) = \sum_{i=0}^{m-1} \mu_i p^i\]
where $\mu_i \in \{0,\ldots,p-1\}$. Note that the sum on the left is computed in $\Zp$ but the sum on the right is computed in $\N$.  Because $L(\lambda-s,s) < \ell$ by assumption, Definition \ref{CarryDefn}(b) implies that $\gamma_{\ell - 1}(\lambda-s,s) = 0$. Therefore
\begin{equation}\label{LastNoCarry} \cN_\ell(\lambda-s,s) > \cN_{\ell - 1}(\lambda-s,s) \geq \cN_m(\lambda-s,s) = \cN_m(t_m(\lambda-s),t_m(s)) \end{equation}
whenever $0 \leq m < \ell$.  This inequality applies in particular when $m := v_p(r - s)$ which is strictly less than $\ell$ by hypothesis. Note also that because $r \equiv s$ mod $p^m$, we know that $t_m(r) = t_m(s)$ and $t_m(\lambda-r) = t_m(\lambda-s)$. 

Next, we have the crude estimate
 \begin{equation}\label{Crude} \cN_j(\lambda-r,r) \leq \cN_m(t_m(\lambda-r), t_m(r)) + (j - m) \qmb{for all} j\geq m.\end{equation}

Finally, let $j = \max \{m, L(\lambda-r,r)\}$ which is finite by Lemma \ref{CarStop} because $\lambda \notin \N$. Then Theorem \ref{Kum} implies that
\begin{equation}\label{vpBin2} v_p\Bin{\lambda-r}{r} = j - \cN_j(\lambda-r,r).\end{equation}
 Applying $(\ref{vpBin2})$, $(\ref{Crude})$, $(\ref{LastNoCarry})$ and $(\ref{vpBin})$, we obtain 
\[\begin{array}{lll} v_p\Bin{\lambda-r}{r} - m &=& j-\cN_j(\lambda-r,r) - m   \\
&\geq & -\cN_m(t_m(\lambda-r), t_m(r)) \\
&=  & -\cN_m(t_m(\lambda-s), t_m(s)) \\
& > & - \cN_\ell(\lambda-s,s) \\
&=& v_p\Bin{\lambda-s}{s} - \ell
 \end{array}\]
as required. \end{proof}

We now put everything together to obtain our main estimate.
\begin{thm}\label{UniqueMin} Let $0 \leq r \leq n$ be given with $r \neq s$. Then 
\[ v_p \left( \frac{\Bin{\lambda-r}{r}}{(n-r)(q-1)+1} \right) > v_p \left( \frac{\Bin{\lambda-s}{s}}{(n-s)(q-1)+1} \right).\]
\end{thm}
\begin{proof} By Corollary \ref{vprs}, it is enough to show that
\[ v_p(\Bin{\lambda-r}{r}) - v_p(r-s) > v_p(\Bin{\lambda-s}{s}) - (M+1)f.\]
We set $\ell := (M+1)f$. Then Corollary \ref{NoCarryInLastSpot} tells us $L(\lambda - s, s) < \ell$, and Lemma \ref{Arch} tells us that $v_p(r-s) < \ell$. Now the required inequality follows from Proposition \ref{CountNoCarries}.
\end{proof}
It remains to estimate the carries in the sum $s+ (\lambda - s)$.

\begin{prop}\label{TheVal} For each $N\geq 6$ of suitable parity depending on $k$ and $q$
\[ v_p \left( \frac{\Bin{\lambda-s}{s}}{s(q-1)+1} \right) \leq  \frac{3-N}{2}  .\]
\end{prop}
\begin{proof} 
	We note that by Theorem \ref{Kum} and Corollary \ref{NoCarryInLastSpot}, \[v_p(\Bin{\lambda-s}{s})=(M+1)f-1-\cN_{(M+1)f-1}(\lambda-s,s).\] Considering Lemma \ref{nocarry}(a) and Lemma \ref{qExpCalc} we see that \[\cN_{(M+1)f-1}(\lambda-s,s)\geq \left\lfloor \frac{M-2}{2}\right\rfloor.\] Moreover by Corollary \ref{vprs} $v_p((n-s)(q-1)+1)=(M+1)f$. 
	
However, $M \geq N - 3$ by Lemma \ref{valM}. Hence \[ v_p \left( \frac{\Bin{\lambda-s}{s}}{s(q-1)+1} \right) \leq - \left\lfloor  \frac{M}{2}\right\rfloor\leq -\frac{N-3}{2} . \qedhere \]
\end{proof}

Here is the main result of $\S$\ref{BinEsts}. 

\begin{cor}\label{SumEstimate}  The $p$-adic valuation of the rational number  
\[ \sum_{r=0}^{n}  \frac{ \binom{\frac{k}{q+1}}{r} \binom{\frac{qk}{q+1}-(q-1)r}{(n-r)(q-1)}}{(n-r)(q-1)+1}\] is not bounded below as $n$ varies. 
\end{cor}
\begin{proof} Apply Proposition \ref{SecondVal}, Theorem \ref{UniqueMin} and Corollary \ref{TheVal}.
\end{proof}

\subsection{Proof of Theorem \ref{CompOnW}}\label{PfOfCOW}
We assume in this section that the hypotheses of Theorem \ref{CompOnW} hold, namely:
\begin{itemize}
\item $[\sL] \in \PicCon^I(\Upsilon)_{\tors}$ and $\omega[\sL] \in \PicCon(\Upsilon)[p']$,
\item the order $d$ of $\omega[\sL]$ in $\PicCon(\Upsilon)$ divides $q+1$, 
\item $e \in d \bZ$ is the order of $[\sL]$ in $\PicCon^I(\Upsilon)$,
\item $n \geq v_{\pi_F}(e)$ and $a \in \cO_F$,
\item $N_{n+1} \leq J_{n+1}$.
\end{itemize}

The map $M_{n,d}$ from $\S \ref{UpperHP}$ sends $\PicCon(\Upsilon)[d]^I$ into $M_0(h(\Upsilon_n),\bZ/d\bZ)^I$. Since $\omega[\sL] \in \PicCon(\Upsilon)[d]^I$, we can now use Lemma \ref{UpsMeasures}(a) to make the following
\begin{defn}\label{TheK} We let $k \in \{1,2,\cdots, d\}$ be the unique integer such that
\[ M_{n,d}(\omega[\sL]) = k \nu_n.\]
\end{defn}


Recall from Definition \ref{Win} and \cite[Definition 4.1.1]{ArdWad2023} the affinoid $W := W_{a,n}$:
\[W = \Sp K \left\langle \tau, \frac{1}{\tau^{q-1} - 1}\right\rangle \qmb{where} \tau:= \frac{x - a}{\pi_F^n}.\]
We fix a primitive $(q-1)^{th}$ root $\zeta$ of $1$ in $K^\times$ using Remark \ref{ZetaInK} and note the following basic properties of our affinoid domain $W$. 
\begin{lem}\label{WUn} \hsp
\be \item $W \cap V_n$ is obtained from $W$ by removing the open disc $\{|\tau| < 1\}$.
\item $W \cap V_{n+1}$ is obtained from $W$ by removing the $q$ open discs 
\[ \{|\tau - \pi_F| < |\pi_F|\}, \quad \{|\tau - \pi_F \zeta| < |\pi_F|\}, \quad \cdots \quad \{|\tau - \pi_F \zeta^{q-2}| < |\pi_F|\}, \quad \{|\tau| < |\pi_F|\}.\]
\ee\end{lem}
\begin{proof} (a) $V_n$ is obtained from $\bD$ by removing all open discs of radius $|\pi_F|^n$ around all points in $\cO_F$ with respect to the $x$-coordinate. Now $W$ is contained in a closed disc $\Sp K \langle \tau \rangle$ of radius $|\pi_F|^n$ which contains exactly $q$ such open discs. We have removed $q-1$ of these when forming $W$. Hence $W \cap V_n$ is obtained from $W$ by removing the last of these open discs, namely $\{ |\tau| < 1 \}$.

(b) This time $\Sp K \langle \tau \rangle$ contains $q^2$ open discs of radius $|\pi_F|^{n+1}$ with respect to the $x$-coordinate. Of these, $q(q-1)$ are contained in the $q-1$ open discs of radius $|\pi_F|^n$ which are removed when we form $W$. This shows that $W$ contains exactly $q$ more such open discs that have to be removed when forming $W \cap V_{n+1}$.
\end{proof}
When $n = 0, q = 5$, a picture of $W \supset W \cap V_{n+1} \supset W \cap V_n$ can be found in $\S \ref{TheIntegral}$. 

\begin{thm}\label{ExplicitMeasure} There exists $c \in K(\tau) \cap \cO(W)^\times$ with $\cS(c) \subseteq F$, and an algebraic generator $\dot{z} \in \sL_{n+1}(W)$ with associated rational function 
\[ \psi(\dot{z}^{\otimes d}) = u \hsp c \qmb{where} u := \frac{1}{(\tau^q - \pi_F^{q-1}\tau)^k}.\]
\end{thm}
\begin{proof} We define $\cS_1 := \{0\} \cup \{1,\zeta,\zeta^2,\cdots,\zeta^{q-2}\}$ and note that $\cS_1$ forms a set of coset representatives for $\pi_F \cO_F$ in $\cO_F$ containing $0$. 
Let $\cS$ be any set of coset representatives for $\pi_F^{n+2} \cO_F$ in $\cO_F$ containing $a + \pi_F^n \cS_1 = W(\overline{K}) \cap \cO_F$. 

Using Definition \ref{TheK} together with Lemma \ref{UpsMeasures}(b) we see that
\[r_{n+1}(M_{n+1,d}(\omega[\sL])) = M_{n,d}(\omega[\sL]) = k \nu_n = r_{n+1}(-k\nu_{n+1}).\]
Lemma \ref{UpsMeasures}(c) then implies that $M_{n+1,d}(\omega[\sL]) = -k \nu_{n+1}$, and hence
\[ M_{n+1,d}(\omega[\sL])(D_{b,n+1}) = - k + d \bZ \qmb{for all} b \in \cO_F.\]
Note that $-\frac{k}{d} \notin \bN$ because $k \geq 1$. We can now apply Theorem \ref{AlgGensExist}(a) with $n+1$ in place of $n$ with $k_b := -k$ for all $b \in \cS$, to choose an algebraic generator $\dot{y} \in \sL_{n+1}(W)$ with associated rational function $v$ such that
\[ v_b(v) = k \qmb{for all} b \in \cS.\]
Now, since $u = \frac{1}{(\tau^q - \pi_F^{q-1}\tau)^k}, \tau = \frac{x - a}{\pi^n_F}$ and $k\neq 0$, we see that $\cS(u) = a + \pi^n_F \cS_1$, and that $v_b(u) = v_b(v)$ for all $b \in \cS(u)$. Therefore the rational function $c := v/u$ satisfies $\cS(c) = \cS(v) \backslash \cS(u) = \cS \backslash (a + \pi^n_F \cS_1) \subseteq F$. This finite set does not meet $W(\overline{K})$ because $\cS \cap W(\overline{K}) = a + \pi^n_F \cS_1$. Hence $c \in \cO(W)^\times$ as claimed. \end{proof}  

\begin{prop}\label{LnWgen} Let $[\dot{z}]$ denote the image of $\dot{z} \in \sL_{n+1}(W)$ in $\sL_n(W)$ under the natural restriction map $\sL_{n+1}(W) \to \sL_n(W)$.
Then the element 
\[\dot{y} := (1 - (\pi_F/\tau)^{q-1})^{k/d} [\dot{z}] \in \sL_n(W)\] 
is an algebraic generator for $\sL_n(W)$ with associated rational function $c / \tau^{qk}$. 
\end{prop}
\begin{proof} Since $\dot{z}$ generates $\sL_{n+1}(W)$ as an $(\cO_W)_{\overline{W \cap V_{n+1}}}(W)$-module, we see that $[\dot{z}]$ generates $\sL_n(W)$ as an $(\cO_W)_{\overline{W \cap V_n}}(W)$-module. Note that $1 - (\pi_F/\tau)^{q-1} \in 1 + \cO(W)\langle \pi_F/\tau^2 \rangle^{\circ\circ}$, so by \cite[Lemma 4.3.2(a)]{ArdWad2023} it is a $d$th power in $1 + \cO(W)\langle \pi_F/\tau^2\rangle^{\circ\circ}$. Since $\cO(W)\langle \pi_F/\tau^2\rangle\subseteq \cO(W)_{\overline{W \cap V_n}}$ in view of Lemma \ref{WUn}(a), we conclude that $\dot{y}$ also generates $\sL_n(W)$ as an $\cO(W)_{\overline{W \cap V_n}}$-module as required by Definition \ref{AlgGen}(a). Using Theorem \ref{ExplicitMeasure}, we see that the associated rational function to $\dot{y}$ is
\[ \psi( \dot{y}^{\otimes d} ) = \left( 1 - (\pi_F/\tau)^{q-1} \right)^k \frac{c}{(\tau^q - \pi_F^{q-1}\tau)^k}  = \frac{c}{\tau^{qk}} .\]
Next, $\cS(c/\tau^{qk}) \subseteq \{a\} \cup \cS( c ) \subseteq K$ by Theorem \ref{ExplicitMeasure}, so Definition \ref{AlgGen}(c) is satisfied. 
Since $c \in \cO(W)^\times$ and $k \neq 0$, we see that $\cS(c / \tau^{qk}) \cap W = \{a\}$. Definition \ref{AlgGen}(d) is now satisfied vacuously by $\dot{y}$, whereas Definition \ref{AlgGen}(e) holds because $v_a(c / \tau^{qk}) / d = -qk/d \notin \bN$ as $k \geq 1$.
\end{proof}

After Theorem \ref{ExplicitMeasure} and Proposition \ref{LnWgen} and Corollary \ref{MnPres}, we have explicit presentations for the $\sD_n(W)$-module $\sL_n(W)$ and $\sD_{n+1}(W)$-module $\sL_{n+1}(W)$.  Next, we will make the connection with the theory developed earlier on in $\S \ref{PfOfCompOnW}$ by making a change of coordinates: recall from $\S \ref{TheIntegral}$ the affinoid domain 
\[X := \Sp K \left\langle x, \frac{1}{x^{q-1}-1} \right\rangle.\]
\begin{lem}\label{gan} Let $g_{a,n} := \begin{pmatrix} 1 & -a \\ 0 & \pi_F^n \end{pmatrix} \in \bB(K)$. Then $g_{a,n}(W) = X$. \end{lem}
\begin{proof} We have $g_{a,n} \cdot z = \frac{z - a}{\pi_F^n}$ for all $z \in \overline{K}$. Using the definitions of $W$ and $X$, we see that $z \in W(\overline{K})$ if and only if $g_{a,n} \cdot z \in X(\overline{K})$. Hence $g_{a,n}(W) = X$.
\end{proof}
\begin{cor} There is a $K$-algebra isomorphism
\[ \chi : \sD_n(W) \stackrel{\cong}{\longrightarrow} \sD_0(X)\]
such that $\chi(\tau) = x$ and $\chi(\pi_F^n \partial_x) = \partial_x$, and $\chi(\sD_{n+m}(W)) = \sD_m(X)$ for all $m \geq 0$.
\end{cor}
\begin{proof} Note that $\varrho(g_{a,n}) = \pi_F^{-n}$ by Definition \ref{AffA} and that $g_{a,n}(W) = X$ by Lemma \ref{gan}. The affinoid $W$ lies in $\bA( |\pi_F|^n \partial_x/\varpi)^\dag$ by Lemma \ref{CheckCov2}, so we can apply Corollary \ref{gDagTransform} with $r = \varpi/|\pi_F|^n$ to get the required $K$-algebra isomorphism
\[\chi := (g_{a,n})^\dag_{\varpi/|\pi_F|^n}(W) : \sD_n(W) = \cD^\dag_{\varpi/|\pi_F|^n}(W) \stackrel{\cong}{\longrightarrow} \cD^\dag_\varpi(X) = \sD_0(X).\]
Then $\chi(\tau) = g_{a,n} \cdot \tau = (g_{a,n} \cdot x - a)/\pi_F^n = ((a + \pi_F^n x) - a) / \pi_F^n = x$ and $\chi(\partial_x) = g_{a,n} \cdot \partial_x  = \varrho(g_{a,n}) \partial_x = \pi_F^{-n}\partial_x$ by Lemma \ref{AutAOD}(e). 

Finally, we can apply Corollary \ref{gDagTransform} again with $r = \varpi/|\pi_F|^{n+m}$ to see that $\chi$ maps $\sD_{n+m}(W)$ onto $\sD_m(X)$ for all $m \geq 0$.
\end{proof}

Next, we record the presentations for $\sL_n(W)$ and $\sL_{n+1}(W)$ after applying this change of coordinates as well as untwisting via the untwisting automorphism $\theta_{\chi(c),d}^{-1}$ of $\sD_0(X) = \cD^\dag_\varpi(X)$ from Corollary \ref{ThetaUDag}.   More precisely:

\begin{defn} \hspace{1cm}
\be \item Define $b :=  \chi(c) \in \cO(X)^\times$ and $w := (x^q - \pi_F^{q-1} x)^{-k} = \chi(u)$.
\item Let  $M := \sD_0(X) \underset{\sD_n(W)}{\otimes}{} \sL_n(W)$ where $\sD_0(X)$ is viewed as a $\sD_n(W)$-module via the isomorphism $\theta_{b,d}^{-1} \circ \chi : \sD_n(W) \stackrel{\cong}{\longrightarrow} \sD_0(X)$.
\item Let $M' := \sD_1(X) \underset{\sD_{n+1}(W)}{\otimes}{} \sL_{n+1}(W)$ where $\sD_1(X)$ is viewed as a $\sD_{n+1}(W)$-module via the isomorphism $\theta_{b,d}^{-1} \circ \chi : \sD_{n+1}(W) \stackrel{\cong}{\longrightarrow} \sD_1(X)$.
\ee\end{defn}

\begin{thm}\label{MandMdash} We have
\[M' \cong \frac{\sD_1(X)}{\sD_1(X) R_{\cS(w)}( w, d )}\qmb{and}M \cong \frac{\sD_0(X)}{\sD_0(X) R_{\cS(x^{-qk})}(x^{-qk}, d)} .\]
\end{thm}
\begin{proof} By Theorem \ref{ExplicitMeasure} and Corollary \ref{MnPres}, we have
\[ \sL_{n+1}(W) = \sD_{n+1}(W) \cdot{z} \cong \sD_{n+1}(W) / \sD_{n+1}(W) R_{\cS(uc)}(u c, d).\]
Now, applying Lemma \ref{ChiOfRelDag}(b) with $r = \varpi/|\pi_F|^n$ shows that
\[ \chi(R_{\cS(uc)}(uc,d)) \equiv R_{\cS(\chi(uc))}(\chi(uc),d) \quad \mod K^\times.\]
Next, $\chi(uc) = w b$, so Lemma \ref{RudProps}(c) gives
\[R_{\cS(\chi(uc))}(\chi(uc), d) = R_{\cS(wb)}(w b, d) = \theta_{b,d} ( R_{\cS(wb)}(w,d)).\]
Since $\cS(w) \cap \cS(b) = \emptyset$, $\cS(wb)$ is the disjoint union of $\cS(w)$ and $\cS(b)$, so $\Delta_{\cS(wb)} = \Delta_{\cS(w)} \Delta_{\cS(b)}$ and $R_{\cS(wb)}(w,d) = \Delta_{\cS(b)} R_{\cS(w)}(w,d)$. Therefore
\[ \chi(R_{\cS(uc)}(uc,d)) \equiv  \theta_{b,d}\left(\Delta_{\cS(b)} R_{\cS(w)}(w,d)\right) \quad \mod K^\times.\]
However, $b$ and hence $\Delta_{\cS(b)}$ are units in $\cO(X)$, so we have
\begin{equation}\label{ChiRwd} \theta_{b,d}^{-1} \chi \left(\sD_{n+1}(W) R_{\cS(uc)}(u c, d)\right) = \sD_1(X) R_{\cS(w)}(w,d).\end{equation}
This implies that $M' \cong \frac{\sD_1(X)}{\sD_1(X) R_{\cS(w)}( w, d )}$ as $\sD_1(X)$-modules. The second isomorphism is deduced similarly using Proposition \ref{LnWgen} instead of Theorem \ref{ExplicitMeasure}. \end{proof}

\textbf{We assume until the end of $\S \ref{PfOfCOW}$ that the field $K$ is discretely valued.}
\begin{cor}\label{LnWSimple} The $\sD_n(W)$-module $\sL_n(W)$ is simple if $k \neq d$.
\end{cor}
\begin{proof}  After Theorem \ref{MandMdash}, it is enough to check that $\frac{\sD_0(X)}{\sD_0(X) R_{\cS(x^{-qk})}(x^{-qk}, d)}$ is a simple $\sD_0(X)$-module. Now $R_{\cS(x^{-qk})}(x^{-qk}, d) = x \partial_x + \frac{qk}{d}$ by Definition \ref{TheRelator}(b), and $-\frac{qk}{d} \in \bZ_p \backslash \bZ$ because $q \equiv -1 \mod d$ and $k \neq 0 \mod d$ by our assumption on $k$. We can now apply Theorem \ref{SimpleKIsoc} to conclude.
\end{proof}

\begin{prop}\label{VarPhiIsOnto} The natural action map
\[\varphi : \sD_n(W) \underset{\sD_{n+1}(W)}{\otimes}{} \sL_{n+1}(W) \longrightarrow \sL_n(W)\]
given by $\varphi(Q \otimes v) := Q \cdot \tau(v)$ for all $Q \in \sD_n(W)$ and $v \in \sL_{n+1}(W)$, is surjective.
\end{prop}
\begin{proof} The restriction map $\tau : \sL_{n+1}(W) \to \sL_n(W)$ is non-zero, so $\varphi$ is a non-zero $\sD_n(W)$-linear map. If $k \neq d$ then $\sL_n(W)$ is a simple $\sD_n(W)$-module by Corollary \ref{LnWSimple}, so $\varphi$ is surjective. 

Suppose now that $k = d$. Inspecting the definition of the algebraic generator $\dot{y} \in \sL_n(W)$ from Proposition \ref{LnWgen}, we see that
\[ \tau^{q-1} \dot{y} = (\tau^{q-1} - \pi_F^{q-1}) [\dot{z}] \in \im \varphi.\]
Since $\sL_n(W) = \sD_n(W) \cdot \dot{y}$, it remains to show that $\dot{y} \in \sD_n(W) \cdot \tau^{q-1} \dot{y}$. Now, 
\[ M = \sD_0(X) \underset{\sD_n(W)}{\otimes}{} \sL_n(W) \cong \frac{\sD_0(X)}{\sD_0(X) R_{\cS(x^{-qd})}(x^{-qd}, d)}\]
as $\sD_0(X)$-modules by Theorem \ref{MandMdash}. However, $\cS(x^{-qd}) = \{0\}$, and
\[R_{\{0\}}(x^{-qd}, d) = x \partial_x + q = R_{\{0\}}(x^{-q}, 1),\]
so applying Theorem \ref{StdPres} with $u = x^{-q}$, $t = 1$ and $d = 1$, we see that
\begin{equation}\label{Mk=d} M \cong (j_\ast \cO_{\bA - \{0\}})_{\overline{U_0}}(X)\end{equation}
as $\sD_0(X)$-modules. Inside $(j_\ast \cO_{\bA - \{0\}})_{\overline{U_0}}(X)$, we compute 
\[x^{-q } = (-1)^{q-1}\partial_x^{[q-1]} \cdot x^{-1} \in \sD_0(X) \cdot x^{q-1} x^{-q},\]
and hence $\dot{y} \in \sD_n(W) \cdot \tau^{q-1} \dot{y}$ because the isomorphism (\ref{Mk=d}) sends $1 \otimes \dot{y}$ to $x^{-q}$.  \end{proof}


Next, we need to obtain the following upper bound.
\begin{lem}\label{LengthUB}  We have $\ell(\ker \varphi) \leq q-1$.
\end{lem}
\begin{proof} After applying the change of coordinates made precise by Theorem \ref{MandMdash}, the map $\varphi$ induces a non-zero $\sD_0(X)$-linear map
\[ \psi : \frac{\sD_0(X)}{\sD_0(X) R_{\cS(w)}(w,d)} \cong \sD_0(X) \underset{\sD_1(X)}{\otimes}{} M' \longrightarrow M \cong \frac{\sD_0(X)}{\sD_0(X) R_{\cS(x^{-qk})}(x^{-qk}, d)}.\]
Now, by applying Theorem \ref{CharCycle} --- see also Example \ref{CyclRwd} --- we know that
\[\begin{array}{lll} \Cycl\left( \frac{\sD^\dag_{\Q}(\sX)}{\sD^\dag_{\Q}(\sX) R_{\cS(w)}(w,d)}\right) &=& [\sX_0] + q [T^\ast_0\sX_0], \quad \mbox{whereas} \\
\Cycl\left( \frac{\sD^\dag_{\Q}(\sX)}{\sD^\dag_{\Q}(\sX) R(x^{-qk},d)}\right) &=& [\sX_0] + [T^\ast_0\sX_0] \end{array}.\]
Since $\varphi$ is surjective by Proposition \ref{VarPhiIsOnto}, so is $\psi$. Therefore 
\[\Cycl(\ker \psi) = (q-1)[T_0^\ast \sX_0].\]
Since $|\Cycl|$ is additive on short exact sequences and $|\Cycl(V)| \geq 1$ for non-zero holonomic $\sD^\dag_{\Q}(\sX)$-modules $V$, we see that $\ell(V) \leq |\Cycl(V)|$ for every holonomic $\sD^\dag_{\Q}(\sX)$-module $V$. Hence $\ell(\ker \varphi) = \ell(\ker \psi) \leq |\Cycl(\ker \psi)| = q-1$. \end{proof}

\begin{defn} Define $g = \begin{pmatrix} 1 & -\pi_F^{n+1} \\ 0 & 1 \end{pmatrix} \in N_{n+1}$ and $h := \begin{pmatrix} 1 & -\pi_F \\ 0 & 1 \end{pmatrix} \in N_1$. \end{defn}
It follows from Lemma \ref{JG} that $\beta(g) \in \cD_r(\bD)^\times$ whenever $r < \varpi / |\pi_F|^{n+1}$. Since $\varpi/|\pi_F|^n < \varpi/|\pi_F|^{n+1}$, $\beta(g)$ lies in $\sD_n(\bD)^\times$ and hence can also be viewed as an element of $\sD_n(W)^\times$.

\begin{prop} \label{BetaNotInSumOnW} $\beta(g) \notin \sD_{n+1}(W) + \sD_n(W) R_{\cS(uc)}(u c,d).$
\end{prop}
\begin{proof} Suppose for a contradiction that $\beta(g) \in \sD_{n+1}(W) + \sD_n(W) R_{\cS(uc)}(uc,d)$. Then we can use equation $(\ref{ChiRwd})$ to deduce that
\[ \theta_{b,d}^{-1}\chi(\beta(g)) \in \sD_1(X) + \sD_0(X) \cdot \sD_1(X)R_{\cS(w)}(w,d).\] 
Since $\chi(\pi_F^n\partial_x) = \partial_x$ and $g \cdot x - x = \pi_F^{n+1}$ by Lemma \ref{GdotDn}(a), we see that
\[\chi(\beta(g)) = \chi\left(\sum\limits_{m=0}^\infty \pi_F^{(n+1)m} \partial_x^{[m]} \right) = \sum\limits_{m=0}^\infty \pi_F^m \partial_x^{[m]} =\beta(h).\]
By Proposition \ref{TwistedCocycle}, we can find an element $c_{b^{-1},d}(h) \in \cO(X)^\times$ such that
\[\theta_{b,d}^{-1} \chi(\beta(g)) = \theta_{b^{-1},d}(\beta(h)) = c_{b^{-1},d}(h) \beta(h).\]
Since $\sD_1(X) \subseteq \sD_0(X)$, we conclude that $c_{b^{-1},d}(h) \beta(h) \in \sD_1(X) + \sD_0(X)R_{\cS(w)}(w,d)$.  Since $c_{b^{-1},d}(h)$ is a unit in $\cO(X)$, this implies that $\beta(h) \in \sD_1(X) + \sD_0(X) R_{\cS(w)}(w,d)$.  This contradicts Theorem \ref{BetaNotInSum}.
\end{proof}
Note that since $g \in N_{n+1} \leq J_{n+1}$ by our assumption on $J$, and $W$ is $N_{n+1}$-stable by Lemma \ref{CheckCov2}, the $\sD_{n+1}(W)$-module carries a compatible action of $g$.

\begin{cor}\label{SAnontriv} The secret $g$-action from Lemma \ref{SecretAction} on 
\[\sD_n(W) \underset{\sD_{n+1}(W)}{\otimes} \sL_{n+1}(W)\]
is non-trivial.
\end{cor}
\begin{proof} Write $S:= \sD_n(W)$ and $S' := \sD_{n+1}(W)$. Choose an algebraic generator $\dot{z} \in \sL_{n+1}(W)$ as in Theorem \ref{ExplicitMeasure} and write $r := R_{\cS(uc)}(u c, d)$. By Corollary \ref{MnPres}, we can choose $Q \in S'$ such that $Q \cdot \dot{z} = g \cdot \dot{z}$. If the secret action was trivial, then
\[ 1 \otimes \dot{z} = g \star (1 \otimes \dot{z}) = \beta(g)^{-1} \otimes g \cdot \dot{z} = \beta(g)^{-1} \otimes Q \cdot \dot{z} = \beta(g)^{-1} Q \otimes \dot{z}.\]
Now $\ann_{S'}(\dot{z}) = S' r$ by Corollary \ref{MnPres}, so $\ann_S(1 \otimes \dot{z}) = S r$ because $S$ is a flat $S'$-module by Lemma \ref{Flat}(a), Therefore $\beta(g)^{-1} Q - 1 \in S r$. Hence $Q - \beta(g) \in S r$ so $\beta(g) \in S' + S  r$, which contradicts Proposition \ref{BetaNotInSumOnW}.
\end{proof}

Next, we need some elementary group theory. Let $k_F$ denote the residue field of our $p$-adic field $F$ and let $k_F^\times = k_F - \{0\}$. We view $k_F$ as a finite elementary abelian $p$-group of order $q$, and $k_F^\times$ as a cyclic group of order $q-1$. As in $\S \ref{CPandSA}$, let $\Irr_{\overline{K}}(k_F)$ be the set of isomorphism classes of simple $\overline{K}[k_F]$-modules, which is naturally in bijection with the set $\widehat{k_F}$ of $\overline{K}^\times$-valued characters of $k_F$. The group $k_F^\times$ acts on $k_F$ by multiplication; since this action is by abelian group homomorphisms, we obtain an induced action of $k_F^\times$ on $\widehat{k_F}$ given by 
\[(t \cdot \chi)(a) = \chi(t^{-1} a) \qmb{for all} \chi \in \widehat{k_F}, a \in k_F, t \in k_F^\times.\]
\begin{lem}\label{BactsOnN} $k_F^\times$ acts transitively and freely on $\widehat{k_F} \backslash \{0\}$.
\end{lem} 
\begin{proof} Suppose $t \in k_F^\times$ fixes $0 \neq \chi \in \widehat{k_F}$ under the above action. Then $(t^{-1} - 1)a \in \ker \chi$ for all $a \in k_F$. Since $k_F$ is a \emph{field}, $(t^{-1} - 1)k_F = k_F$ unless $t = 1$. But $\chi \neq 0$ forces $\ker \chi < k_F$, so indeed $t = 1$. The result follows since $|k_F^\times| = |\widehat{k_F} \backslash \{0\}|$. 
\end{proof}

\begin{cor}\label{GmOnGa} The action of $k_F^\times$ on $\Irr_K(k_F)$ has exactly two orbits.
\end{cor}
\begin{proof} When $K = \overline{K}$, this follows directly from Lemma \ref{BactsOnN} because the bijection $\Irr_{\overline{K}}(k_F) \cong \widehat{k_F}$ respects the $k_F^\times$-action on both sides. In general, because our ground field $K$ has characteristic zero, we may identify $\Irr_K(k_F)$ with the set of primitive idempotents of the group ring $K[k_F]$, and every such idempotent is the sum of a $\cG_K$-orbit of primitive idempotents of $\overline{K}[k_F]$, where $\cG_K = \Gal(\overline{K}/K)$. Since the $\cG_K$-action on $\overline{K}[k_F]$ commutes with the $k_F^\times$-action, we deduce from the case $K = \overline{K}$ that any two non-principal idempotents in $K[k_F]$ lie in the same $k_F^\times$-orbit. 
\end{proof}

We now spell out the application of these ideas that we will need.

\begin{cor}\label{Baction} Let $B := \begin{pmatrix} 1 & a \\ 0 & 1 \end{pmatrix} \begin{pmatrix} \cO_F^\times & 0 \\ 0 & 1 \end{pmatrix} \begin{pmatrix} 1 & -a \\ 0 & 1 \end{pmatrix}$. Then 
\be \item $B$ stabilises $W = \Sp K \langle \tau, \frac{1}{\tau^{q-1} - 1}\rangle$ where $\tau = \frac{x - a}{\pi_F^n}$. 
\item $B$ normalises $N_{n+1} = \begin{pmatrix} 1 & \pi_F^{n+1} \cO_F \\ 0 & 1 \end{pmatrix}$  and $N_{n+2} = \begin{pmatrix} 1 & \pi_F^{n+2} \cO_F \\ 0 & 1 \end{pmatrix}$.
\item the $B$-action on $\Irr_K(N_{n+1} / N_{n+2})$ has exactly two orbits.
\ee\end{cor}
\begin{proof} (a) This is an easy calculation using Lemma \ref{GdotDn}(a). 

(b,c) We compute that $\begin{pmatrix} \alpha & 0 \\ 0 & 1 \end{pmatrix}$ conjugates $\begin{pmatrix} 1 &  \pi_F^{n+1} a \\ 0 & 1\end{pmatrix}$ to $\begin{pmatrix} 1 &  \pi_F^{n+1} a \alpha^{-1} \\ 0 & 1\end{pmatrix}$, so that the conjugation $B$-action on $N_{n+1}/N_{n+2}$ is completely determined by the multiplication action of $k_F^\times$ on $k_F$. Now apply Corollary \ref{GmOnGa}.
\end{proof}

We can finally put everything together and prove Theorem \ref{CompOnW}.

\begin{proof}[Proof of Theorem \ref{CompOnW}] Consider the commutative triangle
\[\xymatrix{ \left(\sD_n(W) \underset{\sD_{n+1}(W)}{\otimes}{} \sL_{n+1}(W)\right)_{N_{n+1}}  \ar[rr] \ar@{->>}[d] & & \sL_n(W) \\
\left(\sD_n(W) \underset{\sD_{n+1}(W)}{\otimes}{} \sL_{n+1}(W)\right)_{J_{n+1}}\ar[rru] & & }\]
where we must show that the diagonal map is an isomorphism. Because $J_{n+1}$ contains $N_{n+1}$ by assumption, the vertical map is surjective. Proposition \ref{VarPhiIsOnto} implies that the diagonal map is surjective. Hence, to show that it is an isomorphism it is enough to check that the horizontal map is injective. By enlarging $K$ if necessary, we will assume that $K$ contains a primitive $p^{th}$-root of unity $\zeta_p$. 

We will apply Theorem  \ref{Criterion} with the following parameters: $S = \sD_n(W)$, $S' = \sD_{n+1}(W)$, $M = \sL_n(W)$, $M' = \sL_{n+1}(W)$, $G$ is the subgroup $B$ from Lemma \ref{Baction}, $H = N_{n+1}$ and $H' = N_{n+2}$. Note that $M$ is an $S \rtimes B$-module and $M'$ is an $S' \rtimes B$-module, and these actions descend to $S \rtimes_H B$ and $S' \rtimes_{H'} B$ respectively by Theorem \ref{DnJnJmodule}. Note also that the restriction map $M' \to M$ is $B$-equivariant because our line bundle $\sL$ is assumed to be $I$-equivariant in Definition \ref{UpperHPhyp2}. All points of Hypotheses \ref{MorphCP} clearly hold, except possibly for points (e,f) of Hypothesis \ref{MorphCP} which, after decoding the notation, follow from Theorem \ref{GrGxTriv} since $G = B \leq \cG_W$. 

Next, we verify the conditions (a-c) of Theorem \ref{Criterion}.

(a) This is Corollary \ref{SAnontriv}.

(b) This is Corollary \ref{Baction}(c).

(c) Since $K$ is assumed to contain $\zeta_p$ and since $H/H' = N_{n+1}/N_{n+2} \cong k_F$ is an elementary abelian $p$-group of order $q$, we see that $\left\vert \Irr_K(H/H')\right \vert = q$. Now the required inequality follows from Lemma \ref{LengthUB}. 

Theorem \ref{Criterion} now tells us that the comparison map $(S \otimes_{S'} M')_H \to M$ is injective as desired. \end{proof}

\section{Coadmissibility and applications}\label{Applications}
\subsection{Quasi-coherent modules over a tower of rings}\label{QcohTower}
We begin with some algebraic preliminaries. Recall that a \emph{tower of rings} is a diagram 
\[S_\bullet = S_0 \gets S_1 \gets S_2 \gets \cdots \]
of associative, unital rings, and unital ring homomorphisms. A \emph{morphism} of towers of rings $f_\bullet : T_\bullet \to S_\bullet$ is a commutative diagram of the form
\[ \xymatrix{ T_0 \ar[d]^{f_0} & T_1 \ar[d]^{f_1}\ar[l] & T_2 \ar[d]^{f_2}\ar[l] & \cdots\ar[l] \\ 
S_0 & S_1 \ar[l] & S_2 \ar[l] & \cdots\ar[l] }\]
An \emph{$S_\bullet$-module} is a diagram $M_\bullet := M_0 \gets M_1 \gets M_2 \gets \cdots$ where each $M_n$ is an $S_n$-module, and every connecting map $\mu^M_{n+1,n} : M_{n+1} \to M_n$ is an $S_{n+1}$-linear map for $n \geq 1$, where $M_n$ is regarded as an $S_{n+1}$-module via restriction of scalars along the ring homomorphism $\mu^S_{n+1,n} : S_{n+1} \to S_n$. A \emph{morphism of $S_\bullet$-modules} is defined in the obvious way. The space of \emph{global sections} of $M_\bullet$ is 
\[\Gamma(M_\bullet) := \invlim M_n,\]
this is naturally a module over $\Gamma(S_\bullet) = \invlim S_n$. We say that the $S_\bullet$-module $M_\bullet$ is \emph{quasi-coherent} if the following map is an isomorphism for all $n \geq 0$:
\[ 1 \otimes \mu^M_{n+1,n} : S_n \underset{S_{n+1}}{\otimes}{} M_{n+1} \longrightarrow M_n.\]
Given a morphism $f_\bullet : T_\bullet \to S_\bullet$ and an $S_\bullet$-module $M_\bullet$, we can form the \emph{pullback} $T_\bullet$-module $T_\bullet \otimes_{S_\bullet} M_\bullet$ defined by 
\[ (T_\bullet \otimes_{S_\bullet} M_\bullet)_n = T_n \otimes_{S_n} M_n \qmb{for all} n \geq 0,\]
whose connecting maps for $n \geq m$ are given by 
\[\mu^{T \otimes_S M}_{n,m} = \mu^T_{n,m} \otimes_{\mu^S_{n,m}} \mu^M_{n,m} : T_n \otimes_{S_n} M_n \longrightarrow T_m \otimes_{S_m} M_m.\]
\begin{lem}\label{QCohBC} If $M_\bullet$ quasi-coherent, then so is $T_\bullet \otimes_{S_\bullet} M_\bullet$.
\end{lem}
\begin{proof} Consider the following commutative square:
\[ \xymatrix{ T_n \underset{T_{n+1}}{\otimes}{} \left(T_{n+1} \underset{S_{n+1}}{\otimes}{} M_{n+1}\right) \ar[rr]^(0.6){1 \otimes \mu^{T \otimes_S M}_{n+1,n}}  & & T_n \underset{S_n}{\otimes}{} M_n \\ 
T_n \underset{S_{n+1}}{\otimes}{} M_{n+1} \ar[rr]_{\cong}\ar[u]^{\cong} & & T_n \underset{S_n}{\otimes}{} (S_n \underset{S_{n+1}}{\otimes}{} M_{n+1}). \ar[u]_{1 \otimes \mu^M_{n+1,n}}}\]
The vertical arrow on the left and the bottom horizontal arrow are isomorphisms by the associativity of tensor product, whereas the vertical arrow on the right is an isomorphism because $M_\bullet$ is quasi-coherent. So the top horizontal arrow is also an isomorphism.
\end{proof}

For each integer $c \geq 0$, we can form the \emph{shifted tower} $S_\bullet[c]$:
\[S_\bullet[c] = S_c \gets S_{c+1} \gets S_{c+2} \gets \cdots\]
so that $S_\bullet[c]_n = S_{n+c}$ for all $n\geq 0$. There is a \emph{shift functor} $M_\bullet \mapsto M_\bullet[c]$ from $S_\bullet$-modules to $S_\bullet[c]$-modules, given by $M_\bullet[c]_n = M_{n+c}$ for all $n \geq 0$. The connecting maps $\mu^S_{n+c,c} : S_{n+c} \to S_n$ in $S_\bullet$ induce a morphism of towers of rings
\[ \iota_{S}^c : S_\bullet[c] \longrightarrow S_\bullet\]
and there is a similar morphism of towers of abelian groups $\iota_{M}^c : M_\bullet[c] \to M_\bullet$.
\begin{lem}\label{ShiftLim} Let $c \geq 0$ and let $M_\bullet$ be an $S_\bullet$-module.
\be \item The map $\Gamma(\iota_{S}^c) : \Gamma(S_\bullet[c]) \to \Gamma(S_\bullet)$ is a ring isomorphism.
 \item The map $\Gamma(\iota_{M}^c) : \Gamma(M_\bullet[c]) \to \Gamma(M_\bullet)$ is an isomorphism of abelian groups.
\ee\end{lem}
\begin{proof} (a) Because $c \geq 0$, we have the projection map $\invlim S_n \to \invlim S_{n+c}$, and it is a two-sided inverse to $\Gamma(\iota^c_S) = \invlim \mu^S_{n+c,n} : \invlim S_{n+c} \to \invlim S_n$. 

(b) This is entirely similar.\end{proof}
In fact, $\Gamma(\iota_M^c)$ is a $\Gamma(S_\bullet[c])$-linear isomorphism, when its codomain is regarded as a $\Gamma(S_\bullet[c])$-module via the map $\Gamma(\iota^c_S)$.

\begin{prop}\label{TStowers} Let $c \geq 0$ and let $f_\bullet : T_\bullet \to S_\bullet$ and $g_\bullet : S_\bullet[c] \to T_\bullet$ be morphisms of towers of rings such that the diagram
\[\xymatrix{T_\bullet[c] \ar[d]_{f_\bullet[c]}\ar[rr]^{\iota^c_{T}} & & T_\bullet \ar[d]^{f_\bullet}\\ S_\bullet[c] \ar[rr]_{\iota^c_{S}}\ar[urr]_{g_\bullet} && S_\bullet} \]
is commutative. Then:
\be \item $\Gamma(f_\bullet) : \Gamma(T_\bullet) \to \Gamma(S_\bullet)$ is a ring isomorphism.
\item Let $M_\bullet$ be a quasi-coherent $S_\bullet$-module and let $L_\bullet := T_\bullet \otimes_{S_\bullet[c]} M_\bullet[c]$. Then there is a $\Gamma(T_\bullet)$-linear isomorphism $\Gamma(L_\bullet) \to \Gamma(M_\bullet)$, where the $\Gamma(S_\bullet)$-module $\Gamma(M_\bullet)$ is regarded as a $\Gamma(T_\bullet)$-module via $\Gamma(f_\bullet)$.
\ee 
\end{prop}
\begin{proof} (a) Applying $\Gamma$ gives the commutative diagram
\[\xymatrix{\Gamma(T_\bullet[c]) \ar[d]_{\Gamma(f_\bullet[c])}\ar[rr]^{\Gamma(\iota^c_{T})} & & \Gamma(T_\bullet) \ar[d]^{\Gamma(f_\bullet)}\\ \Gamma(S_\bullet[c]) \ar[rr]_{\Gamma(\iota^c_{S})}\ar[urr]_{\Gamma(g_\bullet)} && \Gamma(S_\bullet).} \]
Since the horizontal arrows are isomorphisms by Lemma \ref{ShiftLim}(a), we conclude that the diagonal arrow has both a right inverse and a left inverse. Hence it is an isomorphism, and consequently so is $\Gamma(f_\bullet)$.

(b) Noting that $L_m = T_m \otimes_{S_{m+c}} M_{m+c}$ for each $m \geq 0$, we define additive maps
\[  M_{m+c} \quad \stackrel{\alpha_m}{\longrightarrow} \quad L_m \quad \stackrel{\beta_m}{\longrightarrow}  \quad M_m\]
as follows: for $x,y \in M_{m+c}$ and $t \in T_m$ we set
\begin{equation}\label{AlphaBeta} \alpha_m(x) = 1 \otimes x \qmb{and} \beta_m(t \otimes y) = f_m(t) \cdot \mu^M_{m+c,m}(y).
\end{equation}
Now let $n\geq m \geq 0$ and consider the following diagrams:
\[ \xymatrix{ M_{n+c} \ar[r]^{\alpha_n}\ar[d]_{\mu^M_{n+c,m+c}} & L_n \ar[r]^{\beta_n}\ar[d]_{\mu^L_{n,m}} & M_n \ar[d]^{\mu^M_{n,m}} \\
M_{m+c} \ar[r]_{\alpha_m} & L_m \ar[r]_{\beta_m} & M_m} \quad\quad \xymatrix{L_{m+c} \ar[d]_{\beta_{m+c}}\ar[rr]^{\mu^L_{m+c,c}} & & L_m \ar[d]^{\beta_m}\\ M_{m+c} \ar[rr]_{\mu^M_{m+c,m}}\ar[urr]_{\alpha_m} && M_m.} 
\]
The verification that these diagrams commute is straightforward; for example, to show that $\alpha_m \circ \beta_{m+c} = \mu^L_{m+c,c}$, let $t \in T_{m+c}$ and $y \in M_{m+2c}$ and use the fact that $g_m \circ f_{m+c} = \mu^T_{m+c,m}$ to calculate as follows:
\[\begin{array}{lll} \alpha_m(\beta_{m+c}(t \otimes y)) &=& \alpha_m(f_{m+c}(t) \cdot \mu^M_{m+2c,m+c}(y)) \\
&=& 1 \otimes f_{m+c}(t) \cdot \mu^M_{m+2c,m+c}(y) \\
&=& g_m(f_{m+c}(t))\otimes \mu^M_{m+2c,m+c}(y)\\
&=& \mu^T_{m+c,m}(t) \otimes \mu^M_{m+2c,m+c}(y) \\
&=& \mu^L_{m+c,m}(t \otimes y).\end{array}\]
Thus we have defined morphisms of towers of abelian groups $\alpha_\bullet : M_\bullet[c] \to L_\bullet$ and $\beta_\bullet : L_\bullet \to M_\bullet$, and they fit into the following commutative diagram:
\[\xymatrix{L_\bullet[c] \ar[d]_{\beta_\bullet[c]}\ar[rr]^{\iota^c_{L}} & & L_\bullet \ar[d]^{\beta_\bullet}\\ M_\bullet[c] \ar[rr]_{\iota^c_{M}}\ar[urr]_{\alpha_\bullet} && M_\bullet.} \]
 Considering the $T_\bullet$-module $L_\bullet$ as an $S_\bullet[c]$-module via the morphism $g_\bullet : S_\bullet[c] \to T_\bullet$, and the $S_\bullet$-module $M_\bullet$ as a $T_\bullet$-module via the morphism $f_\bullet : T_\bullet \to S_\bullet$, we can verify that $\alpha_\bullet$ becomes a morphism of $S_\bullet[c]$-modules, whereas $\beta_\bullet$ becomes a morphism of $T_\bullet$-modules. Apply the global sections functor $\Gamma$ to this diagram. Using Lemma \ref{ShiftLim}(b) we deduce as in part (a) that $\Gamma(\beta_\bullet) : \Gamma(L_\bullet) \to \Gamma(M_\bullet)$ is an isomorphism of $\Gamma(T_\bullet)$-modules.
\end{proof}

\subsection{An alternative presentation for the algebra \ts{\w\cD(X,J)}}\label{CompareSandT}
We now return to the setting of $\S \ref{CoadDrin}$ so that $\bD \subset \bP^1$ is the closed unit disc, with local coordinate $x$. Recall the $G$-topology $\bD_m$ on $\bD$ from Definition \ref{CurlyDn}.

\begin{defn}\label{mX}  For each affinoid subdomain $X$ of $\bD$, let $m_X$ be the least integer $m$ such that $X \in \bD_m$.
\end{defn}

\begin{lem}\label{PiMstable} Let $X$ be an affinoid subdomain of $\bD$ and let $m \geq 0$. The derivation $\pi_F^m \partial_x$ of $\cO(X)$ preserves $\cO(X)^\circ$ if and only if $m \geq m_X$.
\end{lem}
\begin{proof} By Definitions \ref{CurlyDn}(a) and \ref{DagSite}(a), $m \geq m_X$ if and only if $r(X) \leq \varpi / |{\pi_F}|^m$. By Corollary \ref{rXcalc}, this is equivalent to $||\partial_x||_{\cO(X)} \leq 1 / |{\pi_F}|^m$. This, in turn, is equivalent to $\pi_F^m \partial_x$ preserving $\cO(X)^\circ$ inside $\cO(X)$.
\end{proof}

\begin{notn} We will write $\cA := \cO(X)^\circ$ and $\cL := \cA \partial_x$ until the end of $\S \ref{CompareSandT}$.
\end{notn}
We refer the reader to \cite[Definition 6.1]{DCapOne} for the definition of Lie lattices.

\begin{cor}\label{U=D} Let $X$ be an affinoid subdomain of $\bD$. For every $m \geq m_X$,
\be \item $\pi_F^m \cL$ is an $\cA$-Lie lattice in $\cT(X)$, and
\item there is a natural isometric isomorphism of $K$-Banach algebras
\[ j : \hK{U(\pi_F^m \cL)}  \stackrel{\cong}{\longrightarrow}  \cD_{|{\pi_F}|^{-m}}(X).\]
\ee\end{cor} 
\begin{proof} (a) Because $m \geq m_X$, Lemma \ref{PiMstable} implies that $\pi_F^m \cL$ is a sub $(K^\circ, \cA)$-Lie algebra of $\cT(X)$. Now apply \cite[Definition 6.1]{DCapOne}.

(b) There is a unique $\cA$-linear map $j_{\pi_F^m\cL} : \pi_F^m \cL \to \cD_{|{\pi_F}|^{-m}}(X) =: D$ that sends $\pi_F^m \partial_x$ to $\pi_F^m \partial$. Let $j_{\cA}$ denote the inclusion of $\cA$ into $D$. It is easy to verify that $j_{\pi_F^m \cL}$ respects Lie brackets and satisfies $[j_{\pi_F^m \cL} (v), j_{\cA} (a)] = j_{\pi_F^m \cL}(v(a))$ for all $v \in \pi_F^m \cL$ and $a \in \cA$. The universal property \cite[\S 2.1]{DCapOne} of the enveloping algebra $U(\pi_F^m \cL)$ implies that $j_{\pi_F^m\cL}$ extends to an $K^\circ$-algebra homomorphism $j : U(\pi_F^m \cL) \to D$. Because the image of this map by construction lands in the unit ball of the $K$-Banach algebra $D$, it extends further to a $K$-Banach algebra homomorphism $j : \hK{U(\pi_F^m \cL)} \to D$. Finally, Rinehart's Theorem \cite[Theorem 3.1]{Rinehart} implies that $\{\pi_F^{mn} \partial_x^n : n \geq 0\}$ is an orthonormal basis for the $\cO(X)$-Banach module $\hK{U(\pi_F^m \cL)}$ in the sense of \cite[\S 1.2, p.7]{FvdPut}, whereas it follows from the construction of $D$ that $\{\pi_F^{mn} \partial^n : n \geq 0\}$ is an orthonormal basis for $D$. Since $j$ sends $\partial_x^n$ to $\partial^n$ for each $n \geq 0$, we conclude that it is an isometric isomorphism.
\end{proof}
Recall from Notation \ref{CongSubs} the \emph{Iwahori subgroup} $I$ of $GL_2(\cO_F)$, given by
\begin{equation}\label{IwahoriEq} I := \left\{ \begin{pmatrix} a & b \\ c & d\end{pmatrix}\in GL_2(\cO_F) : c\equiv 0 \mod \pi_F\cO_F\right\}.\end{equation}
We fix a closed subgroup $J$ of $I$ and let $X$ be a $J$-stable affinoid subdomain of $\bD$. We now recall the construction of the algebra $\w\cD(X, J)$ from \cite[Definition 3.3.1]{EqDCap}. 

Recall from \cite[Definition 3.2.13]{EqDCap} that $(\cJ,H)$ is an \emph{$\cA$-trivialising pair} if $\cJ$ is a $J$-stable $\cA$-Lie lattice in $\Der_K(\cO(X))$ and $H$ is an open normal subgroup of $J$ such that $\rho_X(H) \leq \exp(p^\epsilon \cJ)$ inside $\Aut_K(\cO(X))$, where $\rho_X : J \to \Aut_K(\cO(X))$ is the action of $J$ on $\cO(X)$. For each $\cA$-trivialising pair $(\cJ,H)$, there is a crossed product $\hK{U(\cJ)} \rtimes_H J$, and $\w\cD(X,J)$ is the inverse limit of these crossed products taken over all possible $\cA$-trivialising pairs $(\cJ,H)$.

\begin{lem}\label{TrivPair} Let $X$ be a $J$-stable affinoid subdomain of $\bD$. Then $(\pi_F^m \cL, J_{m + \epsilon v_{\pi_F}(p)})$ is an $\cA$-trivialising pair whenever $m \geq m_X$.
\end{lem}
\begin{proof} The affine formal model $\cA = \cO(X)^\circ$ in $\cO(X)$ is $J$-stable since $X$ is $J$-stable, and it is $\pi_F^m \cL$-stable by Lemma \ref{PiMstable}. Hence by \cite[Corollary 4.3.7]{EqDCap}, it will be enough to verify that $(\pi_F^m \cO(\bD)^\circ \partial_x, J_{m + \epsilon v_{\pi_F}(p)})$ is an $\cO(\bD)^\circ$-trivialising pair. Therefore we may assume that $X = \bD$. 

Recall congruence subgroup $J_n$ of $J$ from Definition \ref{CongSubs} and note that $\rho_{\bD}(J_n)$ acts trivially on $\cO(\bD)^\circ$ modulo $\pi_F^n$. If $\cE$ denotes the ring of $K^\circ$-linear endomorphisms of $\cO(\bD)^\circ$, then $\rho_{\bD}(J_{m + \epsilon v_{\pi_F}(p)}) \leq 1 + p^\epsilon \pi_F^m \cE$. Since every element of $\rho_{\bD}(J)$ is an $K^\circ$-algebra automorphism of $\cO(\bD)^\circ$, it follows from the proof of  \cite[Lemma 3.2.5(a)]{EqDCap} that $\log \rho_{\bD}(J_{m + \epsilon v_{\pi_F}(p)}) \leq p^\epsilon \pi_F^m \Der_{K^\circ} \cO(\bD)^\circ = p^\epsilon \pi_F^m \cO(\bD)^\circ \partial_x$. \end{proof}

After Lemma \ref{TrivPair}, \cite[Theorem 3.2.12]{EqDCap} and Corollary \ref{U=D}(b), we have at our disposal the crossed product
\[ \hsp \cD_{|{\pi_F}|^{-m}}(X) \underset{J_{m+\epsilon v_{\pi_F}(p)}}{\rtimes}{}J\]
for every $J$-stable affinoid subdomain $X$ of $\bD$ and every $m \geq m_X$. Note that the trivialisation of the $J_{m+ \epsilon v_{\pi_F}(p)}$-action on $\cD_{|{\pi_F}|^{-m}}(X)$ is $j \circ \beta_{\pi_F^m \cL}$, where $\beta_{\pi_F^m \cL}$ is the trivialisation of the $J_{m + \epsilon v_{\pi_F}(p)}$-action on $\hK{U(\pi_F^m \cL)}$ obtained from \cite[Theorem 3.2.12]{EqDCap} and Lemma \ref{TrivPair}, and $j$ comes from Corollary \ref{U=D}(b). On the other hand, when $m \geq m_X$, by Proposition \ref{SheafDnJ} we also have the crossed product
\[ \sD_m(X) \underset{J_{m+1}}{\rtimes}{} J \hsp = \hsp \cD^\dag_{\varpi / |{\pi_F}|^m}(X) \underset{J_{m+1}}{\rtimes}{} J.\]

\begin{defn}\label{DefnOfSandT} Let $X$ be a $J$-stable affinoid subdomain of $\bD$ and let $m \geq m_X$. Define
\be \item $\cT_m(X) \hsp := \hsp \cD_{|{\pi_F}|^{-m}}(X) \underset{J_{m+\epsilon v_{\pi_F}(p)}}{\rtimes}{}J$, and
\item $\cS_m(X) := \sD_m(X) \underset{J_{m+1}}{\rtimes}{} J \hsp = \hsp \cD^\dag_{\varpi / |{\pi_F}|^m}(X) \underset{J_{m+1}}{\rtimes}{} J$.
\ee\end{defn}

Note that the algebra $\cT_m(X)$ is functorial in $X$ and therefore defines a presheaf $\cT_m$ on the $G$-topology $\bD_m/J$ --- see Definition \ref{CongSubs}(d). 

\begin{prop}\label{TmPres} For every $J$-stable affinoid subdomain $X$ of $\bD$, the tower
\[ \cT_{m_X}(X) \gets \cT_{m_X + 1}(X) \gets \cT_{m_X + 2}(X) \gets \cdots\]
gives a Fr\'echet-Stein presentation for the algebra $\w\cD(X, J)$.
\end{prop}
\begin{proof}  Use Lemma \ref{TrivPair}, \cite[Lemma 3.3.4]{EqDCap} and the proof of \cite[Theorem 3.4.8]{EqDCap}.
\end{proof}
We will now make a precise comparison of the two constructions.
\begin{thm}\label{TmSmComp} For every $m \geq 0$, there is a commutative diagram 
\[ \xymatrix{ && \cT_{m+\epsilon v_{\pi_F}(p)} \ar[dd]_{f_{m + \epsilon v_{\pi_F}(p)}}\ar[rr] & & \cT_m \ar[dd]^{f_m} \\ \cD \rtimes J \ar[urr] \ar[drr] && && \\ && \cS_{m+\epsilon v_{\pi_F}(p)} \ar[rr]\ar[uurr]_{g_m} && \cS_m.}
\]
of presheaves of $K$-algebras on $\bD_m /J$.\end{thm}
\begin{proof} Fix $X \in \bD_m /J$. We will first construct the $K$-algebra map
\[f_m(X) : \cT_m(X) \to \cS_m(X).\] 
Let $r$ be a real number satisfying $\varpi / |{\pi_F}|^m < r < |{\pi_F}|^{-m} \min\{\varpi / |{\pi_F}|, 1\}$. Then in particular $r < |{\pi_F}|^{-m}$, so by Lemma \ref{BanachUnivProp} there is a $K$-Banach algebra homomorphism $f_{m,r} : \cD_{|{\pi_F}|^{-m}}(X)\to \cD_r(X)$ which fixes $\cO(X)$ and sends $\partial \in \cD_{|{\pi_F}|^{-m}}(X)$ to $\partial \in \cD_r(X)$. This map is $J$-equivariant for the $J$-actions on $\cD_{|{\pi_F}|^{-m}}(X)$ and $\cD_r(X)$ that were constructed in Proposition \ref{NormOfRhor}. Because also $\varpi / |{\pi_F}|^m < r < \varpi / |{\pi_F}|^{m+1}$, we can apply Lemma \ref{JG} together with Theorem \ref{GrGxTriv} to construct the crossed product $\cD_r(X) \underset{J_{m+1}}{\rtimes} J$ as in the proof of Proposition \ref{SheafDnJ}. Then by \cite[Lemma 2.2.7]{EqDCap}, $f_{m,r}$ extends to a $K$-algebra homomorphism
\[ f_{m,r} \rtimes 1_J : \cD_{|{\pi_F}|^{-m}}(X) \underset{J_{m + \epsilon v_{\pi_F}(p)}}{\rtimes}{} J \quad \longrightarrow \quad \cD_r(X) \underset{J_{m+1}}{\rtimes} J \]
provided we can show that it respects the two trivialisations 
\[j \circ \beta_{\pi_F^m \cL} : J_{m + \epsilon v_{\pi_F}(p)} \longrightarrow \cD_{|{\pi_F}|^{-m}}(X)^\times \qmb{and}\beta : J_{m+1} \to \cD_r(X)^\times.\]
Recall from \cite[Lemma 3.2.10(a)]{EqDCap} the action map $\psi_{\pi_F^m\cL} : \hK{U(\pi_F^m \cL)} \to \cB(\cO(X))$ so that $\psi_{\pi_F^m \cL} \circ \beta_{\pi_F^m \cL}$ is the restriction of $\rho_X$ to $J_{m+\epsilon v_{\pi_F}(p)}$ by \cite[Definition 3.2.11]{EqDCap}. Recall the action map $\sigma_r : \cD_r(X) \to \cB(\cO(X))$ from Lemma \ref{ActionOnO}, and note that $\psi_{\pi_F^m \cL}$ is just the restriction of $\sigma_r$ to $\hK{U(\pi_F^m \cL)}$; more precisely, we have 
\begin{equation}\label{PsiFactors}\sigma_r \circ f_{m,r} \circ j = \psi_{\pi_F^m \cL}\end{equation}
 where $j$ is the isomorphism from Corollary \ref{U=D}(b). Then
\[\sigma_r \circ f \circ j \circ \beta_{\pi_F^m \cL} = \psi_{\pi_F^m \cL} \circ \beta_{\pi_F^m \cL} = \rho_{X | J_{m+\epsilon v_{\pi_F}(p)}} = \sigma_r \circ \beta_{|J_{m + \epsilon v_{\pi_F}(p)}}\]
by Proposition \ref{SigmaRho}, and therefore $\beta_{|J_{m + \epsilon v_{\pi_F}(p)}} = f \circ j \circ \beta_{\pi_F^m \cL}$ because $\sigma_r$ is injective by Lemma \ref{ActionOnO}(b). This is precisely what is required to apply \cite[Lemma 2.2.7]{EqDCap}, and completes the construction of $f_{m,r} \rtimes 1_J$. It is clear that if $\varpi/|{\pi_F}|^m < r' \leq r$ then $f_{m,r'} \rtimes 1$ is compatible with $f_{m,r} \rtimes 1$ in the sense that $f_{m,r'} \rtimes 1_J$ is the composition of $f_{m,r} \rtimes 1_J$ with $\cD_r(X) \rtimes_{J_{m + 1}} J \to \cD_{r'}(X) \rtimes_{J_{m+1}} J$. We can now define
\[ f_m(X) :\cT_m(X) \to \cS_m(X)\]
to be the colimit of the maps $f_{m,r} \rtimes 1_J$ as $r$ approaches $\varpi/|{\pi_F}|^m$ from above. 

We will next similarly construct the $K$-algebra homomorphism 
\[g_m(X) : \cS_{m + \epsilon v_{\pi_F}(p)}(X) \to \cT_m(X).\]
Let $r$ be a real number such that $\varpi < r |{\pi_F}|^{m + \epsilon v_{\pi_F}(p)} < \varpi/{\pi_F}$, so that the crossed product $\cD_r(X) \rtimes_{J_{m + \epsilon v_{\pi_F}(p) + 1}} J$ is defined as in the proof of Proposition \ref{SheafDnJ}. Note that then $r|{\pi_F}|^m > \varpi / |{\pi_F}|^{\epsilon v_{\pi_F}(p)} = |p|^{\frac{1}{p-1} - \epsilon} \geq 1$, so $r \geq |{\pi_F}|^{-m}$. Hence by Lemma \ref{BanachUnivProp}, there is a natural $K$-Banach algebra homomorphism $g_{m,r} : \cD_r(X) \to \cD_{|{\pi_F}|^{-m}}(X)$ which fixes $\cO(X)$ and sends $\partial\in \cD_r(X)$ to $\partial \in \cD_{|{\pi_F}|^{-m}}(X)$. Now, the diagram
\[ \xymatrix{ \cD_r(X) \ar[rr]^{g_{m,r}} \ar[d]_{\sigma_r} & & \cD_{|{\pi_F}|^{-m}}(X)\ar[d]^{j^{-1}}_{\cong} \\\cB(\cO(X)) & &  \hK{U(\pi_F^m \cL)}\ar[ll]^{\psi_{\pi_F^m\cL}}}\]
is commutative, so that $\sigma_r = \psi_{\pi_F^m\cL} \circ j^{-1} \circ g_{m,r}$. Therefore
\[\psi_{\pi_F^m \cL} \circ j^{-1} \circ g_{m,r} \circ \beta = \sigma_r \circ \beta = \rho_X = \psi_{\pi_F^m \cL} \circ \beta_{\pi_F^m \cL|J_{m + \epsilon v_{\pi_F}(p) + 1}}\]
again by Proposition \ref{SigmaRho} and \cite[Definition 3.2.11]{EqDCap}. The factorisation (\ref{PsiFactors}) together with Lemma \ref{ActionOnO}(b) shows that $\psi_{\pi_F^m \cL}$ is injective, and we deduce that
 \[g_{m,r} \circ \beta = j \circ \beta_{\pi_F^m \cL|J_{m + \epsilon v_{\pi_F}(p) + 1}}.\]
Because the map $g_{m,r}$ is $J$-equivariant, this equation allows us to apply \cite[Lemma 2.2.7]{EqDCap} and deduce that $g_{m,r}$ extends to a $K$-algebra homomorphism
\[ g_{m,r} \rtimes 1_J : \cD_r(X) \underset{J_{m+\epsilon v_{\pi_F}(p) + 1}}{\rtimes}{} J \quad \longrightarrow \quad \cD_{|{\pi_F}|^{-m}}(X) \underset{J_{m + \epsilon v_{\pi_F}(p)}}{\rtimes}{} J. \]
Passing to the limit as $r$ approaches $\varpi / |{\pi_F}|^{m + \epsilon v_{\pi_F}(p)}$ from above, we obtain the required $K$-algebra homomorphism
\[g_m(X) : \cS_{m + \epsilon v_{\pi_F}(p)}(X) \to \cT_m(X).\]
Finally, we leave to the reader the straightforward verification that the diagram in the statement of the Theorem is commutative, and that the morphisms $f_m(X)$ and $g_m(X)$ are functorial in $X$.
\end{proof}

\begin{cor}\label{wDJ} For each $n \geq 0$ there is an isomorphism 
\[ \w\cD(-,J) \quad \stackrel{\cong}{\longrightarrow} \quad \lim\limits_{\stackrel{\longleftarrow}{m \geq n}} \sD_m \underset{J_{m+1}}{\rtimes}{} J\]
of presheaves of $K$-algebras on $\bD_n / J$.
\end{cor}
\begin{proof} Use Proposition \ref{TmPres}, Theorem \ref{TmSmComp} and Proposition \ref{TStowers}(a).
\end{proof}

For future use, we establish here the following flatness result.

\begin{thm}\label{SflatOverT} $\cS_n(X)$ is a flat right $\cT_n(X)$-module whenever $X \in \bD_n/J$. 
\end{thm}
\begin{proof} Write $c := \epsilon v_{\pi_F}(p)$. Form the crossed product $C := \cD^\dag_{\varpi / |{\pi_F}|^n}(X) \underset{J_{n+c}}{\rtimes}{} J$, and consider the following commutative diagram:
\[\xymatrix{ \cT_n(X)  = \cD_{|{\pi_F}|^{-n}}(X) \underset{J_{n + c}}{\rtimes}{} J \ar[d]\ar[rr] && \cS_n(X) = \cD^\dag_{\varpi / |{\pi_F}|^n}(X) \underset{J_{n+1}}{\rtimes}{} J \\ 
C = \cD^\dag_{\varpi / |{\pi_F}|^n}(X) \underset{J_{n+c}}{\rtimes}{J} \ar[rru] && D := \cD^\dag_{\varpi/|{\pi_F}|^n}(X). \ar[ll]\ar[u]}\]
Note that $C$ is a flat right $D$-module being a crossed product of $D$ with the group $J / J_{n+c}$, which allows us to deduce from the right-module version of \cite[Lemma 2.2]{SchmidtBB} that $\cS_n(X)$ is a flat right $C$-module. On the other hand, we deduce from Theorem \ref{flatcotangent} and Lemma \ref{flatcolim1} that $C$ is a flat right $\cT_n(X)$-module. Hence $\cS_n(X)$ is a flat right $\cT_n(X)$-module as claimed. \end{proof}

\subsection{Coadmissibility of \ts{j_\ast \sL} on the local Drinfeld space} \label{LocalCoad}

We begin with the following elementary result.

\begin{lem}\label{DfpSfp} Let $D \hookrightarrow S$ be an extension of rings such that $S$ is finitely generated and projective as a left $D$-module. Then every $S$-module $M$ which is finitely presented as a $D$-module is also finitely presented as an $S$-module.
\end{lem}
\begin{proof} We can find a finite generating set  $\{v_1,\ldots,v_n\}$ for $M$ as a $D$-module such that if $\theta : D^n \twoheadrightarrow M$ denotes the corresponding surjection then $\ker \theta$ is also finitely generated. Clearly $\{v_1,\ldots,v_n\}$ also generates $M$ as an $S$-module; let $\psi : S^n \twoheadrightarrow M$ be the corresponding surjection. Since $S$ is a projective $D$-module, Schanuel's Lemma implies that $S^n \oplus \ker \theta \cong D^n \oplus \ker \psi$ as $D$-modules. Since $S$ is finitely generated as a $D$-module, $\ker \psi$ is finitely generated as a $D$-module being a homomorphic image of $S^n \oplus \ker \theta$. Hence $\ker \psi$ is also finitely generated as an $S$-module.
\end{proof}

\textbf{We now return to the setting of $\S \ref{CompSect}$}.  We assume $[\sL] \in \PicCon^I(\Upsilon)_{\tors}$ satisfies Definition \ref{UpperHPhyp2}, and fix a closed subgroup $J$ of the Iwahori group $I$.

\begin{cor}\label{LnSfp} Let $n \geq v_{\pi_F}(e)$ and let $X \in \bD_n / J$. Then $\sL_n(X)$ is a finitely presented $\cS_n(X)$-module.
\end{cor}
\begin{proof} We know that $\sL_n(X)$ is a finitely presented $\sD_n(X)$-module, by Theorem \ref{AlgGensExist} and Corollary \ref{MnPres}. The $\sD_n(X)$-action on $\sL_n(X)$ extends to an $\cS_n(X)$-action by Theorem \ref{DnJnJmodule}. By construction, $\cS_n(X)$ is a crossed product of $\sD_n(X)$ with the finite group $J / J_{n+1}$ and is therefore finitely generated and free as a left $\sD_n(X)$-module. Now apply Lemma \ref{DfpSfp}.
\end{proof}

The following consequence will be useful later on.

\begin{prop}\label{LnSnCoh} Let $n \geq v_{\pi_F}(e)$ and let $X \in \bD_n / J$. Then the map
\[ \cS_n(X) \underset{\cS_n(\bD)}{\otimes}{} \sL_n(\bD) \longrightarrow \sL_n(X)\]
induced by the $\cS_n(X)$-action on $\sL_n(X)$ is an isomorphism.
\end{prop}
\begin{proof} We abbreviate $D := \sD_n(\bD), D' := \sD_n(X), S := \cS_n(\bD)$ and $S' := \cS_n(X)$; these rings form a commutative square
\[ \xymatrix{ D \ar[d]\ar[r] & D' \ar[d] \\ S \ar[r] & S'.}\] 
There is a natural transformation $\eta : D' \otimes_D - \to S' \otimes_S -$ between two functors from $S$-modules to $D'$-modules; since $S$ (respectively, $S'$) is a crossed product of $D$ (respectively, $D'$) with the same finite group $J / J_{n+1}$, we see that the map $\eta_S$ is an isomorphism. Since both functors are right exact, we conclude using the Five Lemma that $\eta_M$ is an isomorphism for every finitely presented $S$-module $M$.

Now let $M := \sL_n(\bD)$ and $M' := \sL_n(X)$ and consider the commutative triangle
\[ \xymatrix{ D' \underset{D}{\otimes}{} M \ar[rr]^{\eta_M}\ar[dr] && S' \underset{S}{\otimes}{} M \ar[dl] \\ & M'. & }\]
Since $M$ is a finitely presented $S$-module by Corollary \ref{LnSfp}, the horizontal arrow $\eta_M$ is an isomorphism by the above. Because the diagonal arrow on the left is an isomorphism by Lemma \ref{GenByGS}(a), we conclude that the diagonal arrow on the right is also an isomorphism. \end{proof}
Recall from Definition \ref{DefOfN} that $N = \begin{pmatrix} 1 & \cO_F \\ 0 & 1 \end{pmatrix}$.
\begin{prop}\label{LisQcoh} Let $X$ be a $J$-stable affinoid subdomain of $\bD$. Suppose that $d \mid (q+1)$ and $N_{m_0+1} \leq J$ for some $m_0 \geq \max\{m_X, v_{\pi_F}(e)\}$. Then the canonical action map
\[\cS_n(X) \underset{\cS_{n+1}(X)}{\otimes}{} \sL_{n+1}(X) \longrightarrow \sL_n(X)\]
is an isomorphism for all $n \geq m_0$.
\end{prop}
\begin{proof} This map appears as the top horizontal arrow in the commutative diagram
\[\xymatrix{ \cS_n(X) \underset{\cS_{n+1}(X)}{\otimes}{} \sL_{n+1}(X) \ar[r] &\sL_n(X) \\
\cS_n(X) \underset{\cS_{n+1}(X)}{\otimes}{} \left(\cS_{n+1}(X) \underset{\cS_{n+1}(\bD)}{\otimes}{} \sL_{n+1}(\bD)\right) \ar[u]\ar[r]_(0.63){\cong} & \cS_n(X) \underset{\cS_{n+1}(\bD)}{\otimes} \sL_{n+1}(\bD). \ar[u] }\]
Now $X \in \bD_{n+1}/ J$ because $n+1 > n \geq m_0 \geq m_X$, so the vertical arrow on the left is an isomorphism by Proposition \ref{LnSnCoh}. Also, because $n \geq m_0$ and because $N_{m_0 + 1} \leq J$ by assumption, we have
\[N_{n+1} = N_{m_0 + 1} \cap I_{n+1} \leq J \cap I_{n+1} = J_{n+1}.\]
Therefore the vertical arrow on the right is an isomorphism by Corollary \ref{JnGenBySn}. The bottom horizontal arrow is an isomorphism by the associativity of tensor products and the result follows.
\end{proof}
We can now invoke the formalism of $\S \ref{QcohTower}$ and prove our first main result. 
\begin{thm}\label{CoadmSects} Let $X$ be a $J$-stable affinoid subdomain of $\bD$. Suppose that $d \mid (q+1)$ and $N_{m_0+1} \leq J$ for some $m_0 \geq \max\{m_X, v_{\pi_F}(e)\}$. Then
\be \item $(j_\ast \sL)(X)$ is a coadmissible $\w\cD(X,J)$-module, and
\item for each $m \geq m_0$, the canonical map 
\[\cS_m(X) \underset{\w\cD(X,J)}{\otimes}{} (j_\ast \sL)(X) \longrightarrow \sL_m(X)\]
is an isomorphism.
\ee \end{thm}
\begin{proof} (a) For each $m \geq 0$, let $S_m := \cS_{m_0 + m}(X)$, $T_m := \cT_{m_0 + m}(X)$ and $M_m := \sL_{m_0 + m}(X)$. Then $S_\bullet$ is a tower of rings, and $M_\bullet$ is a quasi-coherent $S_\bullet$-module by Proposition \ref{LisQcoh}. Theorem \ref{TmSmComp} gives us a diagram of towers of rings
\[\xymatrix{T_\bullet[c] \ar[d]_{f_\bullet[c]}\ar[rr]^{\iota^c_{T}} & & T_\bullet \ar[d]^{f_\bullet}\\ S_\bullet[c] \ar[rr]_{\iota^c_{S}}\ar[urr]_{g_\bullet} && S_\bullet} \]
where $c := \epsilon v_{\pi_F}(p)$.  Note that $\invlim T_m \cong \w\cD(X,J)$ by Proposition \ref{TmPres}. Now $L_\bullet := T_\bullet \otimes_{S_\bullet[c]} M_\bullet[c]$ is a quasi-coherent $T_\bullet$-module by Lemma \ref{QCohBC}, and each
\[ L_m = T_m \otimes_{S_{m+c}} M_{m+c} \]
is a finitely presented $T_m$-module because $M_{m+c}$ is a finitely presented $S_{m+c}$-module by Corollary \ref{LnSfp}. In other words, $\invlim L_m$ is a coadmissible $\invlim T_m \cong \w\cD(X,J)$-module.  Now, Proposition \ref{TStowers}(b) implies that there is a natural isomorphism 
\begin{equation}\label{LimOfNs} \invlim M_m \stackrel{\cong}{\longrightarrow} \invlim L_m. \end{equation}
Since the restriction map $(j_\ast \sL)(X) \to \invlim M_m$ is an isomorphism by Proposition \ref{jLlimit}, we conclude that $(j_\ast \sL)(X)$ is a coadmissible $\w\cD(X,J)$-module as claimed.

(b) Write $T_\infty := \w\cD(X,J)$ and $M_\infty := \invlim M_m$. By \cite[Corollary 3.1]{ST}, the canonical map $\alpha_m : T_m \otimes_{T_\infty} M_\infty \to L_m$ is an isomorphism for each $m \geq 0$. This map, together with the map $\psi_m : S_m \otimes_{T_\infty} M_\infty \to M_m$ in question, appears in the following commutative diagram:
\[\xymatrix{ S_m \underset{T_\infty}{\otimes}{} M_\infty\ar[rrr]^{\psi_m}\ar[d]_{\cong} &&& M_m \\ S_m \underset{T_m}{\otimes}{} (T_m \underset{T_\infty}{\otimes}{} M_\infty) \ar[rr]_{1 \otimes\alpha_m}^{\cong} && S_m \underset{T_m}{\otimes}{} (T_m \underset{S_{m+c}}{\otimes}{} M_{m+c}) \ar[r]_(0.58){\cong} & S_m \underset{S_{m+c}}{\otimes}{} M_{m+c}\ar[u] }\]
where the arrow on the right is induced by the action of $S_m$ on $M_m$. We see that this arrow is an isomorphism by a repeated application of Proposition \ref{LisQcoh}, and conclude that $\psi_m$ is also an isomorphism as required. \end{proof}

\begin{lem}\label{LocFrech} $j_\ast \sL$ is a locally Fr\'echet $I$-equivariant $\cD$-module on $\bD$.
\end{lem}
\begin{proof} The local Drinfeld space $\Upsilon = \bD \cap \Omega$ is an admissible $I$-stable open subspace of $\bD$. Since $\sL$ is an $I$-equivariant $\cD$-module on $\Upsilon$, its pushforward $j_\ast \sL$ is an $I$-equivariant $\cD$-module on $\bD$. According \cite[Definition 3.6.1(a)]{EqDCap} we must show that $(j_\ast \sL)(U)$ carries a Fr\'echet topology for each affinoid subdomain $U$ of $\bD$, and that the action maps $g^{j_\ast \sL}(U) : (j_\ast \sL)(U) \to (j_\ast \sL)(gU)$ are continuous for each $g \in I$. We observed in $\S \ref{UpperHP}$ that $\Upsilon$ admits a quasi-Stein covering $(\Upsilon_n)_{n=0}^\infty$. Therefore $U \cap \Omega$ admits a quasi-Stein covering $(U \cap V_n)_{n=0}^\infty$ and there is a natural isomorphism
\[ \sL(U \cap \Omega) = \invlim \sL(U \cap V_n).\]
Since $\sL$ is a coherent $\cO$-module on $\Upsilon$ and since $U \cap V_n$ is affinoid, each $\sL(U \cap V_n)$ is naturally a $K$-Banach space. In this way, we see that $(j_\ast \sL)(U) = \sL(U \cap \Omega)$ carries a natural $K$-Fr\'echet space topology. 

The continuity of the maps $g^{j_\ast \sL}(U) : (j_\ast \sL)(U) \to (j_\ast \sL)(gU)$ now follows from the continuity of the maps $g^{\sL}(U \cap V_n) : \sL(U \cap V_n) \to \sL(g(U \cap V_n))$, which holds by viewing $\sL(g(U\cap V_n))$ as an $\cO(U\cap V_n)$-modulle via $g^\cO(U\cap V_n)$, automatically by \cite[Corollary 1.2.4]{FvdPut}.
\end{proof}

We can finally prove our first main local result.
\begin{thm}\label{LolCoad} Let $[\sL] \in \PicCon^I( \Upsilon)_{\tors}$. Suppose that the image $\omega[\sL]$ of $[\sL]$ in $\PicCon( \Upsilon)$ is killed by $q+1$. Let $j \colon \Upsilon \hookrightarrow \bD$ be the open embedding, Then 
\[j_\ast \sL \in \cC_{\bD / J}\]
for any closed subgroup $J$ of $I$ which contains an open subgroup of $N = \begin{pmatrix} 1 & \cO_F \\ 0 & 1 \end{pmatrix}$. \end{thm}
\begin{proof} It is easy to see that the isomorphism in Lemma \ref{jLlimit} is $J$-equivariant, and so by that result and Theorem \ref{CoadmSects}(a) we know that
\[(j_\ast \sL)(\bD) = \sL(\Upsilon) \cong \lim\limits_{\longleftarrow} \sL_n(\bD)\] 
is a coadmissible $\w\cD(\bD,J)$-module. Because $j_\ast \sL$ is a locally Fr\'echet $I$-equivariant $\cD$-module on $\bD$ by Lemma \ref{LocFrech}, and because $J \leq I$, by \cite[Definition 3.6.7]{EqDCap} it remains to exhibit a continuous $\cD$-linear $J$-equivariant isomorphism 
\[ \Loc^{\w\cD(\bD,J)}_{\bD}\left((j_\ast \sL)(\bD)\right) \stackrel{\cong}{\longrightarrow} j_\ast \sL.\]
Recall that by \cite[Definition 3.5.12]{EqDCap}, $\Loc^{\w\cD(\bD,J)}_{\bD}((j_\ast \sL)(\bD))$ is the unique extension of a sheaf $\cP^{\w\cD(\bD,J)}_{\bD}((j_\ast \sL)(\bD))$ on the basis $\bD_w$ of $\bD$,  defined at \cite[Definition 3.5.3]{EqDCap}. Because $\Loc^{\w\cD(\bD,J)}_{\bD}((j_\ast \sL)(\bD))$ and $j_\ast \sL$ are sheaves, by \cite[Theorem 9.1]{DCapOne} it will suffice to exhibit an isomorphism of $J$-equivariant $\cD$-modules on $\bD_w$
\[ \varphi : \cP^{\w\cD(\bD,J)}_{\bD}((j_\ast \sL)(\bD))  \stackrel{\cong}{\longrightarrow}(j_\ast \sL)_{|\bD_w}.\]
Let $X$ be an affinoid subdomain of $\bD$ and let $H$ be an $X$-small open subgroup\footnote{Since $I$ preserves the free $\cO(\bD)^\circ$-Lie lattice $\cO(\bD)^\circ \partial_x$ in $\cT(\bD)$, we see that in fact, in this situation \emph{every} open subgroup $H$ of $J_X$ is $X$-small.} of $J_X = \Stab_J(X)$, in the sense of \cite[Definition 3.4.4]{EqDCap}. Because $J$ contains an open subgroup of $N$ by assumption, so does its open subgroup $H$. Hence we can find an integer $m_0 \geq \max\{m_X, v_{\pi_F}(e)\}$ such that $N_{m_0 + 1} \leq H$.  For each $n \geq m_0$, let $\cS_n(X) := \sD_n(X) \rtimes_{H_{n+1}}H$ and note that the natural action map
\[\cS_n(X) \underset{\cS_n(\bD)}{\otimes}{} \sL_n(\bD) \longrightarrow \sL_n(X)\]
is an isomorphism by Proposition \ref{LnSnCoh} applied with $J$ replaced by $H$. Note that 
\[\w\cD(X,H) \cong \lim\limits_{\stackrel{\longleftarrow}{n \geq m_0}} \cT_n(X)  \qmb{and} (j_\ast \sL)(\bD) \cong \invlim \cN_n(\bD)\]
by Corollary \ref{wDJ} and equation (\ref{LimOfNs}), where we replace $J$ with $H$ so that
\[\cT_n(X) := \sD_n(X) \underset{H_{n+1}}{\rtimes}{} H \qmb{and} \cN_n(\bD) := \cT_n(\bD) \underset{\cS_{n+c}(\bD)}{\otimes}{} \sL_{n+c}(\bD)\]
and $c := \epsilon v_{\pi_F}(p)$. Now, for each $n \geq m_0$ there is a commutative diagram
\[ \xymatrix{ \cT_n(X) \underset{\cT_n(\bD)}{\otimes}{} \cN_n(\bD) \ar[rr]  & & \cT_n(X) \underset{\cS_{n+c}(X)}{\otimes}{} \sL_{n+c}(X) \\
\cT_n(X) \underset{S_{n+c}(\bD)}{\otimes}{} \sL_{n+c}(\bD) \ar[rr]_(0.4){\cong}\ar[u]^{\cong} && \cT_n(X) \underset{\cS_{n+c}(X)}{\otimes}{}\left(\cS_{n+c}(X) \underset{S_{n+c}(\bD)}{\otimes}{} \sL_{n+c}(\bD)\right) \ar[u]
}\]
where the maps on the left and on the bottom are isomorphisms arising from the associativity of tensor products. The vertical map on the right is an isomorphism by Proposition \ref{LnSnCoh}, so the top horizontal map is an isomorphism. Taking the inverse limit of these maps and using equation (\ref{LimOfNs}) together with the definition of $\w\otimes$ from \cite[\S 7.3]{DCapOne} and \cite[Lemma 3.4.13]{EqDCap}, we obtain a $\w\cD(X,H)$-linear isomorphism
\[ \varphi_H(X) : \w\cD(X,H) \underset{\w\cD(\bD,H)}{\w\otimes}{} (j_\ast \sL)(\bD) = \lim\limits_{\stackrel{\longleftarrow}{n \geq m_0}} \cT_n(X) \underset{\cT_n(\bD)}{\otimes}{} \cN_n(\bD) \quad\stackrel{\cong}{\longrightarrow}\quad (j_\ast \sL)(X).\]
Write $M(X,H) := \w\cD(X,H) \uwset{\w\cD(\bD,H)} (j_\ast \sL)(\bD)$, let $X'$ be an $H$-stable affinoid subdomain of $X$, let $H'$ be another open subgroup of $J$ containing $H$ and let $g \in J$. We leave to the reader the verification that the following diagram commutes:
\[\xymatrix{ M(X',H) \ar[rrrr]^{\varphi_H(X')} &&&& (j_\ast\sL)(X')  \\ && M(X, H') \ar[rrd]^{\varphi_{H'}(X)}  && \\ M(X,H) \ar[rrrr]^{\varphi_H(X)}\ar[rru]\ar[uu]\ar[d]_{g^M_{X,H}}  &&&& (j_\ast \sL)(X) \ar[uu]\\M(gX, {}^gH) \ar[rrrr]_{\varphi_{{}^gH}(gX)} &&&& (j_\ast \sL)(gX)\ar[u]_{g^{j_\ast \sL}(X)}}\]
It follows that the isomorphisms $\varphi_H(X)$ are compatible as $H$ shrinks to $1$, and commute with the restriction maps on both sides. Passing to the limit over all $X$-small $H$ we obtain the required $J$-equivariant $\cD$-linear isomorphism
\[ \varphi : \cP^{\w\cD(\bD,J)}_{\bD}((j_\ast \sL)(\bD)) = \lim\limits_{\stackrel{\longleftarrow}{H}} M(-,H) \quad  \stackrel{\cong}{\longrightarrow} \quad j_\ast \sL. \qedhere\]
\end{proof}

\subsection{Coadmissibility and irreducibility}\label{CoadIrredSect}
Let $j : \Omega \hookrightarrow \bP^1$ and $j_0 : \Upsilon \hookrightarrow \bD$ be the open inclusions and let $G_0 := \mathbb{GL}_2(\cO_F)$.  

\begin{thm}\label{PushForwardIsCoad} Let $[\sL] \in \PicCon^{G_0}(\Omega)_{\tors}$. Then 
\[ j_\ast \sL \in \cC_{\bP^1/H}\]
for every closed subgroup $H$ of $G_0$ which contains an open subgroup of $\mathbb{SL}_2(\cO_F)$.
\end{thm}
\begin{proof} Using the argument given in the proof of Lemma \ref{LocFrech}, we see that $j_\ast \sL$ is a $G_0$-equivariant locally Fr\'echet $\cD$-module on $\bP^1$. Let $w := \begin{pmatrix} 0 & 1 \\ 1 & 0 \end{pmatrix} \in G_0$ so that $\{\bD, w\bD\}$ forms an admissible affinoid covering of $\bP^1$ by two copies of the closed unit disc $\bD$. Note that the Iwahori subgroup $I$ from $(\ref{IwahoriEq})$ stabilises $\bD$, so the opposite Iwahori subgroup ${}^w I := w I w^{-1}$ stabilises $w \bD$; we will show that $j_\ast \sL$ is $\{\bD, w\bD\}$-coadmissible in the sense of \cite[Definition 3.6.7(a)]{EqDCap}.

Next, consider the restriction $\sL_{|\Upsilon}$ so that $[\sL_{|\Upsilon}] \in \PicCon^I(\Upsilon)_{\tors}$.  Because  $\omega[\sL] \in \PicCon(\Omega)^{G_0}$ is killed by $q+1$ by \cite[Corollary 4.3.9]{ArdWad2023}, we see that $\omega[\sL_0] \in \PicCon(\Upsilon)^I$ is also killed by $q+1$. Since $H$ contains an open subgroup of $\mathbb{SL}_2(\cO_F)$, $J := H \cap I$ contains an open subgroup of $N$. Therefore $(j_\ast \sL)_{|\bD} = j_{0,\ast} (\sL_{|\Upsilon}) \in \cC_{\bD / J}$ by Theorem \ref{LolCoad}. Entirely similarly we deduce that $(j_\ast \sL)_{| w\bD} \in \cC_{w \bD / (H \cap {}^w I)}$. Hence $j_\ast \sL \in \cC_{\bP^1 / H}$ by \cite[Definition 3.6.7]{EqDCap}.
\end{proof}

We will now work towards the irreducibility of $j_\ast \sL$ viewed as an object in the abelian category $\cC_{\bP^1/G_0}$. We start in the following abstract setting. Let $X$ be a set equipped with a $G$-topology, let $\cS$ be a sheaf of rings on $X$, let $\cU$ be a covering of $X$ and let $\cM$ be a sheaf of $\cS$-modules on $X$. We say $\cM$ is \emph{$\cU$-quasi-coherent} if the canonical map 
\[\cS(V) \underset{\cS(U)}{\otimes}{} \cM(U) \longrightarrow \cM(V)\]
is an isomorphism whenever $U \in \cU$ and $V$ is an admissible open subset of $U$. For $U, V \in \cU$, write $U \sim_{\cM} V$ if $\cM(U \cap V) \neq 0$ and let $\cong_{\cM}$ be the equivalence relation on $\cU$ generated by $\sim_{\cM}$. We say that $\cU$ is \emph{$\cM$-connected} if there is only one $\cong_{\cM}$ equivalence class in $\cU$.

\begin{prop}\label{GeneralIrred} Let $\cM$ be an $\cS$-module on $X$ and let $\cU$ be an admissible covering of $X$. Suppose that
\be \item $\cM$ is $\cU$-quasi-coherent,
\item $\cU$ is $\cM$-connected, and
\item $\cM(U)$ is a simple $\cS(U)$-module for all $U \in \cU$. 
\ee Then $\cM$ contains no non-zero proper $\cU$-quasi-coherent submodules.
\end{prop}
\begin{proof} Let $\cN$ be a $\cU$-quasi-coherent $\cS$-submodule of $\cM$. We first show that for any $U, V \in \cU$ with $U \sim_{\cM} V$, $\cN(U) = \cM(U)$ if and only $\cN(V)=\cM(V)$. If not, then because $\cM(U)$ and $\cM(V)$ are simple, we may assume that without loss of generality that $\cN(U) = 0$ and $\cN(V) = \cM(V)$. But then because both $\cN$ and $\cM$ are $\cU$-quasi-coherent,
\[ \begin{array}{lllll} \cM(U \cap V) &=& \cS(U \cap V) \underset{\cS(V)}{\otimes}{} \cM(V) &=& \cS(U \cap V) \underset{\cS(V)}{\otimes}{} \cN(V)   \\
&=& \cN(U \cap V) &=& \cS(U \cap V) \underset{\cS(U)}{\otimes}{} \cN(U) = 0\end{array}\]
which contradicts the hypothesis that $U \sim_{\cM} V$. 

Suppose now that the $\cS$-submodule $\cN$ is non-zero. Then we can find at least one $U \in \cU$ such that $\cN(U) \neq 0$, and hence $\cN(U) = \cM(U)$ because $\cM(U)$ is simple. Since $\cU$ is $\cM$-connected, we conclude that in fact $\cN(U) = \cM(U)$ for all $U \in \cU$. Since $\cN$ and $\cM$ are both $\cU$-quasi-coherent, we see that
\[ \cN(W) = \cS(W) \underset{\cS(U)}{\otimes}{} \cN(U) = \cS(W) \underset{\cS(U)}{\otimes}{} \cM(U)  = \cM(W)\]
for any $U \in \cU$ and any admissible open $W \subseteq U$; thus $\cN_{|U} = \cM_{|U}$ for all $U \in \cU$. Since $\cU$ is an admissible covering of $X$ and since $\cN$ and $\cM$ are sheaves, we conclude that $\cN = \cM$.\end{proof}

Here is how we apply this general result. Recall the sheaf of rings $\sD_n = \cD^\dag_{\varpi/|{\pi_F}|^n}$ from Definition \ref{CurlyDn} on the $G$-topology $\bD_n$ on $\bD$, and recall that the truncated line bundle $\sL_n := j_{0,\ast}(\sL_{\overline{\Upsilon_n}})$ is a $\sD_n$-module by Corollary \ref{LnDn}. 

\begin{thm}\label{LnDsimple} Let $[\sL] \in \PicCon^I( \Upsilon)_{\tors}$. Suppose that the image $\omega[\sL]$ of $[\sL]$ in $\PicCon(\Upsilon)_{\tors}$ is non-zero and of order $d$ dividing $q+1$. Let $j_0 : \Upsilon \hookrightarrow \bD$ be the inclusion, let $n \geq 0$ and let $\sL_n := j_{0,\ast}(\sL_{\overline{\Upsilon_n}})$. Then $\sL_n(\bD)$ is a simple $\sD_n(\bD)$-module.
\end{thm}
\begin{proof} Let $\cU_n$ be the covering $\{W_{1,n},\ldots, W_{h_n,n}, \Upsilon_n\}$ from Definition \ref{Win}. We can use Lemma \ref{GenByGS}(a) to see that $\sL_n$ is a $\cU_n$-quasi-coherent $\sD_n$-module on $\bD_n$. Then by construction, $\sL_n(W_{i,n} \cap \Upsilon_n)$ is non-zero for any $i = 1,\ldots, h_n$, so the covering $\cU_n$ is $\sL_n$-connected. 

Using Lemma \ref{UpsMeasures}(d), we see that $\omega[\sL|_{\Upsilon_n}] \in \Con(\Upsilon_n)^I[d]$ is non-zero. Hence $k\nu_n = M_{n,d}(\omega[\sL]) = \mu_{\Upsilon_n,d}(\theta_d(\omega[\sL|_{\Upsilon_n}]))$ is non-zero as well by \cite[Corollary 4.3.4]{ArdWad2023}. Since $k \in \{1,\cdots, d\}$ by Definition \ref{TheK}, Lemma \ref{UpsMeasures}(a) now implies that $k \neq d$. Hence $\sL_n(W_{i,n})$ is a simple $\sD_n(W_{i,n})$-module for each $i$ by Corollary \ref{LnWSimple}. 

On the other hand, $\sL_n(\Upsilon_n)$ is a simple $\sD_n(\Upsilon_n)$-module by Corollary \ref{LnYsimple}, because $\Upsilon_n$ is a connected affinoid subdomain of $\bD$. Hence $\sL_n$ contains no non-zero proper $\cU_n$-quasi-coherent $\sD_n$-submodules by Proposition \ref{GeneralIrred}.

Let $L$ be a $\sD_n(\bD)$-submodule of $\sL_n(\bD)$, and consider the presheaf of $\sD_n$-modules $\cL := \sD_n \underset{\sD_n(\bD)}{\otimes}{} L$. For any $\bD_n$-admissible covering $\cV$ of $\bD$, the augmented \v{C}ech complex $C^\bullet_{\aug}(\cV, \sD_n)$ is exact by Theorem \ref{NCDagTate}. Each of its terms is flat as a right $\sD_n(\bD)$-module by Lemma \ref{Flat}(b). Because the covering $\cV$ is finite by Definition \ref{DagSite}(b), the complex is bounded above, and therefore, by induction on its length, each of its syzygies is flat as a right $\sD_n(\bD)$-module. Tensoring this complex on the right by $L$ over $\sD_n(\bD)$ and splicing, we conclude that the augmented \v{C}ech complex $C^\bullet_{\aug}(\cV, \cL)$ is also exact. In particular, this means that $\cL$ is a sheaf on $\bD_n$ and $\cL(\bD) = L$. Applying Lemma \ref{Flat}(b) again, we see that $\cL$ is a subsheaf of $\sD_n \underset{\sD_n(\bD)}{\otimes}{} \sL_n(\bD) \cong \sL_n$. Since $\cL$ is $\cU_n$-quasi-coherent by construction, we conclude from the first paragraph that either $\cL = 0$ or $\cL = \sL_n$. Because $L = \cL(\bD)$, we conclude that either $L = 0$ or $L = \sL_n(\bD)$ as required. \end{proof}

Recall the sheaves $\cT_n$ and $\cS_n$ on $\bD_n/J$ from Definition \ref{DefnOfSandT}.

\begin{cor}\label{LnSsimple} With the notation and assumptions of Theorem \ref{LnDsimple}, $\sL_n(\bD)$ is a simple $\cS_n(\bD)$-module, whenever $n \geq v_{\pi_F}(e)$.
\end{cor}
\begin{proof} We know that the $\sD_n(\bD)$-action on $\sL_n(\bD)$ extends to the crossed product $\cS_n(\bD)  = \sD_n(\bD) \rtimes_{J_{n+1}} J$ by Theorem \ref{DnJnJmodule}, and that $\sL_n(\bD)$ is a simple $\sD_n(\bD)$-module by Theorem \ref{LnDsimple}. So it is per force simple as an $\cS_n(\bD)$-module.
\end{proof}

\begin{thm}\label{jLisTopSimpleOnD} Let $[\sL] \in \PicCon^I( \Upsilon)_{\tors}$. Suppose that the image $\omega[\sL]$ of $[\sL]$ in $\PicCon( \Upsilon)$ is non-zero and its order divides $q+1$. Then $\sL(\Upsilon)$ is a topologically irreducible $\w\cD(\bD, J)$-module for any closed subgroup $J$ of $I$ which contains an open subgroup of $N$.
\end{thm}
\begin{proof} We know that $M_\infty := \sL(\Upsilon)$ is a coadmissible module over the Fr\'echet-Stein algebra $\w\cD(\bD,J)$ by Lemma \ref{jLlimit} and Theorem \ref{CoadmSects}. Let $S_n := \cS_n(\bD)$ and $T_n := \cT_n(\bD)$ for each $n \geq 0$; it follows from Lemma \ref{PiMstable} that the integer $m_{\bD}$ from Definition \ref{mX} is zero, so the tower of rings $T_\bullet$ gives a Fr\'echet-Stein presentation of $T_\infty := \w\cD(\bD, J)$ by Proposition \ref{TmPres}. Hence each $T_n$ is a flat right $T_\infty$-module by \cite[Remark 3.2]{ST}. On the other hand, $S_n$ is a flat right $T_n$-module by Theorem \ref{SflatOverT} and hence it is also a flat right $T_\infty$-module.

Now suppose that $L$ is a non-zero, closed, proper $T_\infty$-submodule of $M_\infty$. Then by \cite[Corollary 3.3]{ST}, $T_n \otimes_{T_\infty}L$ and $T_n \otimes_{T_\infty} (M_\infty / L)$ are both non-zero for all sufficiently large integers $n$. Let $c := \epsilon v_{\pi_F}(p)$; because the map $T_\infty \to T_n$ factors through $S_{n+c}$ by Theorem \ref{TmSmComp}, it follows that $S_{n+c} \otimes_{T_\infty} L$ and $S_{n+c} \otimes_{T_\infty} (M_\infty / L)$ are also both non-zero. Since $S_{n+c}$ is a flat $T_\infty$-module by the first paragraph, we conclude that $S_{n+c} \otimes_{T_\infty}M_\infty$ is not a simple $S_{n+c}$-module. However this module is isomorphic to $M_{n+c}$ by Theorem \ref{CoadmSects}(b), which contradicts Corollary \ref{LnSsimple}. \end{proof}

\begin{cor}\label{SimpleEqDmoduleLocally} Let $[\sL] \in \PicCon^I( \Upsilon)_{\tors}$. Suppose that the image $\omega[\sL]$ of $[\sL]$ in $\PicCon( \Upsilon)$ is non-zero and its order divides $q+1$. Let $J$ be a closed subgroup of $I$ which contains an open subgroup of $N$. Then $j_{0,\ast} \sL$ is a simple object in $\cC_{\bD/J}$.
\end{cor}
\begin{proof} This follows from Theorem \ref{jLisTopSimpleOnD} and \cite[Theorem 3.6.11]{EqDCap}.
\end{proof}

\begin{thm}\label{PushForwardIsSimple} Let $[\sL] \in \PicCon^{G_0}(\Omega)_{\tors}$ whose image $\omega [\sL]$ in $\PicCon(\Omega)_{\tors}$ is non-zero. Then $j_\ast \sL$ is a simple object in $\cC_{\bP^1 / H}$ for \emph{every} closed subgroup $H$ of $G_0$ which contains an open subgroup of $\mathbb{SL}_2(\cO_F)$.
\end{thm}
\begin{proof} Note first that $j_\ast \sL \in \cC_{\bP^1 / H}$ by Theorem \ref{PushForwardIsCoad}. Let $w := \begin{pmatrix} 0 & 1 \\ 1 & 0 \end{pmatrix} \in G_0$ and recall the covering $\{\bD, w\bD\}$ of $\bP^1$ from the proof of Theorem \ref{PushForwardIsCoad}. 

We will apply Corollary \ref{SimpleEqDmoduleLocally} to $[\sL|_{\Upsilon}] \in \PicCon^I(\Upsilon)_{\tors}$, but first we have to check the required conditions on $\omega[\sL|_{\Upsilon}]$ hold. Since $(q+1) \cdot \omega[\sL] = 0$ by \cite[Corollary 4.3.9]{ArdWad2023}, the order of $\omega[\sL|_{\Upsilon}]$ in $\PicCon(\Upsilon)$ divides $q+1$. Consider the following commutative diagram:
\[\xymatrix{  \PicCon(\Omega)^{G_0}[q+1] \ar[rr]\ar[d]&& \PicCon(\Upsilon)^I[q+1] \ar[d] \\ 
\Con(\Upsilon_0)^{G_0}[q+1] \ar[rr] && \Con(\Upsilon_0)^I[q+1].}\]
Note that $\Upsilon_0 = \Omega_0$ in view of Definition \ref{UnCovering}(b) and \cite[Definition 4.2.12(b)]{ArdWad2023}. Then it follows from \cite[Proposition 3.1.9(b) and Proposition 4.3.8(a)]{ArdWad2023} that the vertical restriction map on the left is an isomorphism. Since the bottom horizontal map is an inclusion, it follows from Lemma \ref{UpsMeasures}(d) that the top horizontal map is injective. Since $\omega[\sL] \neq 0$ in $\PicCon(\Omega)^{G_0}[q+1]$ by assumption, it follows that $\omega[\sL|_{\Upsilon}] \neq 0$ in $\PicCon(\Upsilon)^I[q+1]$ as required.

The group $J := H \cap I$ contains an open subgroup of $N$ because $H$ contains an open subgroup of $\mathbb{SL}_2(\cO_F)$ by assumption. We can now apply Corollary \ref{SimpleEqDmoduleLocally} to see that $(j_\ast \sL)_{|\bD} = j_{0, \ast}(\sL_{|\Upsilon})$ is a simple object in $\cC_{\bD / H \cap I}$.  


Similarly, ${}^wH \cap I$ contains an open subgroup of $N$ because $H$ contains an open subgroup of $\mathbb{SL}_2(\cO_F)$. A similar argument to the above shows that we can apply Corollary \ref{SimpleEqDmoduleLocally} to $w^\ast (\sL_{|w \Upsilon}) \in \PicCon^I(\Upsilon)[e]$ and $J := {}^w H \cap I$ to deduce that $j_{0, \ast} w^\ast \sL_{|w \Upsilon}$ is a simple object in $\cC_{\bD / {}^wH \cap I}$. Hence $(j_\ast \sL)_{|w \bD} = w_\ast (j_{0, \ast} w^\ast \sL_{|w \Upsilon})$ is a simple object in $\cC_{w \bD / H \cap {}^wI}$. 

Now, consider a short exact sequence $0 \to \cM \to j_\ast \sL \to \cN \to 0$ in $\cC_{\bP^1/H}$ with $\cM \neq 0$. By \cite[Proposition 3.6.10(b)]{EqDCap} it induces exact sequences
$0 \to \cM_{|\bD} \to (j_\ast \sL)_{|\bD} \to \cN_{|\bD} \to 0$ in $\cC_{\bD/H \cap I}$ and $0 \to \cM_{|w \bD} \to (j_\ast \sL)_{|w \bD} \to \cN_{|w \bD} \to 0$ in $\cC_{w\bD/H \cap {}^wI}$. 
Since $\cM \neq 0$, we may assume without loss of generality that $\cM_{|\bD} \neq 0$.  Then the map $\cM_{|\bD} \to (j_\ast \sL)_{|\bD}$ is an isomorphism, and $\cN_{|\bD} = 0$. Next, $\cM_{|w \bD}$ cannot be zero, as otherwise $(j_\ast \sL)_{|\bD \cap w \bD} \cong (\cM_{|\bD})_{|\bD \cap w \bD} = (\cM_{|w \bD})_{|\bD \cap w \bD} = 0$ which is not the case. So in fact $\cM_{|w \bD} \neq 0$, which forces $\cN_{|w \bD} = 0$. Hence $\cN = 0$ and $\cM \to j_\ast \sL$ is an isomorphism.
\end{proof}

Let $D(H,K)$ denote the locally $F$-analytic distribution algebra of a locally $F$-analytic group $H$ with coefficients in $K$. Let $\frg = \Lie(H) \otimes_FK$ denote the $K$-Lie algebra of $H$  and let $\fr{m}_0 = D(H,K) \cdot ( \frg U(\frg) \cap Z(U(\frg)) )$ be the closed ideal of $D(H,K)$ generated by the kernel of the trivial infinitesimal character $Z(U(\frg)) \to F$ that sends $\frg$ to zero. We have the following general `rigid-analytic equivariant Beilinson-Bernstein localisation' result.
\begin{thm} \label{RigidEqBB} Let $\bG$ be a connected, split semisimple algebraic group over $F$, let $\bX$ be its flag variety and let $\bfX$ be the $K$-rigid analytification of $\bX \times_F K$. For every open subroup $H$ of $\bG(F)$, the functor of global sections 
\[\Gamma(\bfX,-) : \cC_{\bfX/H} \longrightarrow \{M \in \cC_{D(H,K)} : \fr{m}_0 \cdot M = 0\}.\]
is a well-defined equivalence of categories.
\end{thm}
\begin{proof} First, note that by \cite[Theorem 6.5.1]{EqDCap}, $D(H,K)$ is isomorphic as a Fr\'echet-Stein algebra to a certain completed skew-group algebra denoted $\w{U}(\frg,H)$ that was introduced at \cite[Definition 6.2.7]{EqDCap}. We wish to apply \cite[Theorem 6.4.9]{EqDCap} to the $K$-algebraic group $\bG \times_F K$ and the continuous inclusion $H \hookrightarrow \bG(K)$. In \cite{EqDCap} it was assumed that the algebraic group $\bG$ is simply connected; however this assumption is only used in the proof of \cite[Theorem 5.3.5]{EqDCap}, where it can be easily avoided altogether (by choosing an $K^\circ$-module basis for $\frh$) or instead by working (as we are currently doing) with a discretely valued ground field $K$. With this in mind, Theorem \ref{RigidEqBB} then follows from \cite[Theorem 6.4.9]{EqDCap}, once we observe that the functor of global sections $\Gamma(\bfX,-)$ is quasi-inverse to the localisation functor $\Loc_{\bfX}^{\w{U}(\frg,G)}$, in view of \cite[Theorem 6.4.8]{EqDCap}.  \end{proof}



Recall  that $G^0= \{ g \in \mathbb{GL}_2(F) : v_{\pi_F}(\det g) = 0 \}$.

\begin{cor}\label{LiftToG^0} Suppose $[\sL]\in \PicCon^{G^0}(\Omega)_{\tors}$ whose image $\omega[\sL]$ in $\PicCon(\Omega)$ is non-zero. Then $\sL(\Omega)$ is a coadmissible and topologically irreducible $D(H,K)$-module for \emph{any} closed subgroup $H$ of $G^0$ containing an open subgroup of $\mathbb{SL}_2(F)$. \end{cor}
\begin{proof} Let $Z$ denote the centre of $\GL_2(F)$; note that $Z$ is contained in $G^0$ and that it acts trivially on the $K$-analytic projective line $\bP^1$ via M\"{o}bius transformations. Since $[\sL]$ is a torsion element of $\PicCon^{G^0}(\Omega)$, $\sL^{\otimes e}$ is $G^0$-equivariantly isomorphic to $\cO_{\Omega}$ for some $e \geq 1$. Since $Z$ acts trivially on $\cO_{\Omega}$,  some finite-index open subgroup $Z'$ of $Z$ acts trivially on $\sL$. We can then choose a finite-index open subgroup $H'$ of $H$ such that $H' \cap Z \leq Z'$. Then the $H'$-action on $j_\ast \sL$ factors through $\overline{H'} := H'Z/Z$ and Theorem \ref{PushForwardIsSimple} tells us that $j_\ast \sL \in \cC_{\bP^1/\overline{H'}}$ is a simple object. 

Our condition on $H$ guarantees that its image $\overline{H} := HZ/Z$ is open in $\mathbb{PGL}_2(F) = \GL_2(F) / Z$. Hence $\overline{H'}$ is open in $\mathbb{PGL}_2(F)$ as well.  We now apply Theorem \ref{RigidEqBB} to $\bG := \mathbb{PGL}_{2,F}$ and $\overline{H'} \hookrightarrow \mathbb{PGL}_2(F)$ to see that $\sL(\Omega)$ is a coadmissible and topologically irreducible $D(\overline{H'},K)$-module, killed by $\frm_0$. Since $H \cap Z'$ acts trivially on $\sL(\Omega)$, it follows by inflation that the same statements hold for $\sL(\Omega)$ viewed as a $D(H',K)$-module. Hence they also hold \emph{a fortiori} for $\sL(\Omega)$ viewed as a $D(H,K)$-module.\end{proof}

\subsection{\ts{j_\ast \cO_{\Omega}} and proof of Corollary B}\label{JstarO}
We begin by recording a very general and very important example of a coadmissible $G$-equivariant $\cD$-module.

\begin{prop}\label{OXCXG} Let $G$ be a $p$-adic Lie group acting continuously on a smooth rigid $K$-analytic space $X$. Then 
\begin{enumerate}
\item $\cO_X$ is a coadmissible $G$-equivariant $\cD$-module on $X$: $\cO_X \in \cC_{X/G}$. 
\item $\cO_X$ is a simple object in $\cC_{X/G}$ whenever $X$ is connected. 
\end{enumerate}
\end{prop}
\begin{proof}(1) Certainly $\cO_X$ is a $G$-equivariant locally Fr\'echet $\cD$-module on $X$. Then in view of \cite[Definition 3.6.7]{EqDCap}, it is enough to consider the case where $(X, G)$ is small in the sense of \cite[Definition 3.4.4]{EqDCap}. We will show that in this case
\be \item $\cO(X)$ is a coadmissible $\w\cD(X,G)$-module, and
\item there is a natural continuous $G-\cD$-linear isomorphism 
\[\Loc^{\w\cD(X,G)}_X(\cO(X)) \stackrel{\cong}{\longrightarrow} \cO_X.\]
\ee
(a) Because $(X,G)$ is small, we can choose a $G$-stable affine formal model $\cA \subset \cO(X)$ and a $G$-stable free $\cA$-Lie lattice $\cL \subseteq \cT(X)$. By replacing $\cL$ by $\pi \cL$ if necessary, we may assume that $[\cL,\cL] \subseteq \pi \cL$ and $\cL \cdot \cA \subseteq \pi \cA$.  Since $\pi^n \cL \subseteq \cT(X)$ is an $\cA$-Lie lattice, $\cA$ is naturally a $U(\pi^n \cL)$-module. Since $\cA$ is $\pi$-adically complete, the $U(\pi^n \cL)$-action extends to $\cU_n := \h{U(\pi^n \cL)}$ and clearly $1 \in \cA$ generates $\cA$ as a $\cU_n$-module. Since $\cT(X)$ is a free $\cA$-Lie lattice, we can choose an $\cA$-module basis $\{\partial_1,\cdots,\partial_d\}$ for $\cL$. Then the left ideal $\cI$ in $\cU_n$ generated by this finite set annihilates $1 \in \cA$: $\cI \subseteq \ann_{\cU_n}(1)$. 

For the reverse inclusion, let $Q \in \ann_{\cU_n}(1)$ and write $Q = \sum\limits_{\alpha \in \mathbb{N}^d} f_\alpha (\pi^n \partial)^\alpha$ for some family $(f_\alpha)_{\alpha \in \mathbb{N}^d} \subset \cA$ converging to zero as $|\alpha| \to \infty$. Then $Q \cdot 1 = 0$ implies that $f_0 = 0$, and since $f_\alpha \to 0$ as $|\alpha| \to \infty$ we see that for every $m \geq 0$, $Q$ lies in $\cI + \pi^m \cU_n$. Since $[\cL,\cL] \subseteq \pi \cL$, \cite[Proposition 4.1.6(b)]{EqDCap} tells us that the cyclic $\cU_n$-module $\cU_n/\cI$ is $\pi$-adically separated. Hence $Q \in \cI$ and $\ann_{\cU_n}(1) = \cI$. This implies that $\cA$ is a finitely presented $\cU_n$-module, and hence that $\cO(X) = \cA \otimes_{K^\circ} K$ is a finitely presented $\cU_n \otimes_{K^\circ} K$-module for every $n \geq 0$. 

Using \cite[Corollary 3.3.7]{EqDCap}, choose a good chain $G_\bullet$ for $\cL$. In view of \cite[Definitions 3.3.3 and 3.2.13]{EqDCap} this means that $G_0 \geq G_1 \geq \cdots $ is a chain of open normal subgroups of $G$ with trivial intersection, such that $G_n \leq G_{\pi^n \cL}$ for all $n \geq 0$. Define
\[\sT_n(X) := \hK{U(\pi^n \cL)} = \cU_n \otimes_{K^\circ} K \qmb{and} \sS_n(X) := \sT_n(X) \rtimes_{G_n} G\]
for each $n \geq 0$; then by \cite[Lemma 3.3.4]{EqDCap} we have the Fr\'echet-Stein presentation
\[ \w\cD(X,G) \cong \invlim \sS_n(X).\]
Since $G_n \leq G_{\pi^n \cL}$, it follows from \cite[Definition 3.2.11]{EqDCap} that for every $g \in G_n$, the action of $\beta_{\pi^n \cL}(g) \in \cU_n$ on $\cA$ agrees with the action of $g \in G_n$. Therefore the natural $\cU_n \rtimes G$-action on $\cA$ factors through the crossed product $\cU_n \rtimes_{G_n} G$, and hence the natural $\sT_n(X)$-action on $\cO(X)$ extends to $\sS_n(X) = \sT_n(X) \rtimes_{G_n} G$. Since $\cO(X)$ is a finitely presented $\sT_n(X)$-module, Lemma \ref{DfpSfp} implies that it is also a finitely presented $\sS_n(X)$-module. To see that $\cO(X)$ is a coadmissible $\w\cD(X,G)$-module, it remains to see that for every $n \geq 0$, the following natural maps are isomorphisms:
\begin{equation}\label{SnOx} \sS_n(X) \underset{\sS_{n+1}(X)}{\otimes}{} \cO(X) \longrightarrow \cO(X).\end{equation}
Fix $n \geq 0$. We saw above that $\ann_{\sT_n(X)}(1) = \sum\limits_{i=1}^n \sT_n(X) \partial_i$. Therefore 
\begin{equation}\label{TnPres}\sT_n(X)^d \longrightarrow \sT_n(X) \to \cO(X) \to 0\end{equation}
is a presentation of $\cO(X)$ as a $\sT_n(X)$-module, where the first map sends $(Q_i)_{i=1}^d \in \sT_n(X)^d$ to $\sum\limits_{i=1}^d Q_i \partial_i$ and the second map sends $Q \in \sT_n(X)$ to $Q \cdot 1$. This presentation implies that the canonical $\sT_n(X)$-linear action map
\[ \sT_n(X) \underset{\sT_{n+1}(X)}{\otimes}{} \cO(X) \longrightarrow \cO(X)\]
is an isomorphism. Since this map factors through $\sS_n(X) \underset{\sS_{n+1}(X)}{\otimes}{} \cO(X)$ by Theorem \ref{IndModCoinv}, we conclude that the map (\ref{SnOx}) is an isomorphism as well.

(b) After possibly replacing $G$ by an open subgroup and inspecting \cite[3.5.12, 3.5.1 and 3.5.6]{EqDCap}, it suffices to prove that the natural action map
\[ \w\cD(Y, G) \underset{\w\cD(X,G)}{\w\otimes}{} \cO(X) \longrightarrow \cO(Y)\]
is an isomorphism, whenever $Y$ is a $G$-stable affinoid subdomain of $X$. For this, in view of \cite[Lemma 3.3.4]{EqDCap} and \cite[Lemma 7.3]{DCapOne}, it is enough to prove that 
\[ \sS_n(Y) \underset{\sS_n(X)}{\otimes}{} \cO(X) \longrightarrow \cO(Y)\]
is an isomorphism for all sufficiently large $n$. We can use Theorem \ref{IndModCoinv} again to reduce to showing that $\sT_n(Y) \underset{\sT_n(X)}{\otimes}{} \cO(X) \to \cO(Y)$ is an isomorphism for all sufficiently large $n$. Since $\{\partial_1,\cdots,\partial_d\}$ is still a basis for $\cT(Y)$ as an $\cO(Y)$-module, this follows immediately from the explicit presentations $\cO(X) \cong \sT_n(X) / \sum\limits_{i=1}^d \sT_n(X) \partial_i$ and $\cO(Y) \cong \sT_n(Y) / \sum\limits_{i=1}^d \sT_n(Y) \partial_i$ coming from (\ref{TnPres}).

(2) Suppose first that $X$ is a small affinoid, in the sense that $\cT(X)$ is a free $\cO(X)$-module. Then because $X$ is connected, $\cO(X)$ is a simple $\cD(X)$-module and therefore a simple coadmissible $\w\cD(X,H)$-module for any open subgroup $H$ of $G$ such that $(X,H)$ is small. Hence $\cO_X \in \cC_{X/H}$ is a simple object by \cite[Theorem 3.6.11]{EqDCap} in this case, and it is therefore \emph{a fortiori} simple in $\cC_{X/G}$. 

Returning to the general case, we can choose an admissible open covering $\cU$ of $X$ consisting of small, connected, affinoids. Suppose that $\cM$ is a subobject of $\cO_X$ and consider $\cV_1 := \{U \in \cU : \cM|_U=\cO_U\}$ and $\cV_2=\{U\in \cU: \cM|_U=0 \}$. Then $\cU$ is the disjoint union of $ \cV_1$ and $\cV_2$ by the first paragraph. Now if $U \in \cV_1$ and $V \in \cV_2$ with $U \cap V \neq \emptyset$, then \[(\cM|_V)|_{U\cap V}= \cM|_{U\cap V} =(\cM|_U)|_{U\cap V}=\cO_{U\cap V}\] since $U\in \cV_1$. Thus $\cM|_V \neq 0$ a contradiction. Since $X$ is connected if follows that $\cU=\cV_1$ or $\cU=\cV_2$ i.e. $\cM=\cO_X$ or $\cM=0$ as required. 
\end{proof}

Recall that $j : \Omega \hookrightarrow \bP^1$ is the open embedding. \textbf{We write $G:=\mathbb{GL}_2(F)$} and observe that there is a natural injective morphism of $G$-equivariant locally Fr\'echet $\cD$-modules on $\bP^1$
\[ \alpha : \cO_{\bP^1} \to j_\ast \cO_\Omega.\]

\begin{prop}\label{CokerCoad} $\coker \alpha$ lies in $\cC_{\bP^1/G}$.
\end{prop}
\begin{proof} By Proposition \ref{OXCXG}(1), $\cO_{\bP^1}$ lies in $\cC_{\bP^1/G}$. The trivial line bundle with flat connection $\cO_{\Omega}$ is lies in $\Con^G(\Omega)$ by \cite[Lemma 3.2.4]{ArdWad2023}. Applying  Theorem \ref{PushForwardIsCoad} to $[\cO_{\bP^1}]$ with $H =G_0$, we see that $j_\ast \cO_\Omega$ lies in $\cC_{\bP^1/G_0}$ as well. Since it is naturally a $G$-equivariant locally Fr\'echet $\cD$-module on $\bP^1$, \cite[Definition 3.6.7]{EqDCap} implies that $j_\ast \cO_\Omega$ lies in $\cC_{\bP^1/G}$. Finally, the morphism $\alpha$ lies in $\cC_{\bP^1/G}$, and we can conclude by applying \cite[Lemma 3.7.5]{EqDCap}.
\end{proof}

\begin{lem}\label{CokerSupp} $\coker \alpha$ is supported on $\bP^1(F)$ in the sense of \cite[Definition 1.2.1(a)]{ArdEqKash}. 
\end{lem}
\begin{proof} The restriction of $j_\ast \cO_\Omega$ back to $\Omega$ is equal to $\cO_\Omega$. Therefore $(\coker \alpha)_{|\Omega} = 0$. Since $\bP^1 = \Omega \cup \bP^1(F)$, this means that $\coker \alpha$ is supported on $\bP^1(F)$.
\end{proof}

Recall from $\S \ref{AffTrans}$ that $\bB = \{g \in \mathbb{GL}_2 : g_{21}=0\}$ denotes the subgroup scheme of upper-triangular matrices in $\mathbb{GL}_2$, and let $B := \bB(F)$. Let $\iota : \{0\} \hookrightarrow \bP^1$ be the closed embedding. Recall the algebra of smooth distributions $D^\infty(B,K)$ from \cite{ST3}. 

In what follows, we will silently identify the category $\cC_{D^\infty(B,K)}$ with $\cC_{\{0\}/B}$, using \cite[Theorem B(c)]{EqDCap}. Recall also the functors $\cH^0_{\{0\}}, \ind_{B}^{G}, \iota_+$ and $\iota^\natural$ from \cite{ArdEqKash}. 
\begin{cor}\label{IndEqApp} Let $\cN := \cH^0_{\{0\}}(\coker \alpha)$. Then there is an isomorphism in $\cC_{\bP^1/G}$
\[ \coker \alpha \cong \ind_B^G \cN. \]
\end{cor}
\begin{proof} The $G$-orbit of the point $0$ is equal to $\bP^1(F)$, and the stabiliser of $0$ under the M\"{o}bius action of $G$ on $\bP^1(F)$ is equal to $B$. By Proposition \ref{CokerCoad} and Lemma \ref{CokerSupp}, we see that $\coker \alpha$ is an object of $\cC^{G\cdot 0}_{\bP^1/G}$. Now use \cite[Theorem A]{ArdEqKash}.
\end{proof}

For the following calculations with local cohomology, recall that by \cite[\S 5]{SchVdPut} for any rigid analytic space $X$, there is a natural equivalence of categories between the abelian sheaves on $X$ and the abelian sheaves on the Huber space $\tilde{X}$ associated with $X$. See also $\S \ref{HuberSpace}$ and \cite[\S 2.1]{ArdEqKash} for further details.

We now begin to study the object $\cN = \cH^0_{\{0\}}(\coker \alpha)$ in $\cC^{\{0\}}_{\bP^1/B}$. This sheaf is supported at $\{0\}$ only, and therefore is completely determined by its restriction $\cN_{|\bD}$ to $\bD$. On the other hand, $\cN_{|\bD}$ lies in $\cC^{\{0\}}_{\bD/B_0}$ where $B_0 := \bB(\cO_F)$, and it is therefore completely determined by its global sections $\cN(\bD)$ as a $\w\cD(\bD,B_0)$-module, in view of \cite[Theorem B(c)]{EqDCap}. We first compute the $\cD(\bD)$-action on $\cN(\bD)$ as follows.

\begin{lem}\label{VDcomp} $\cN(\bD)$ is isomorphic to $H^1_{\{0\}}(\bD, \cO_{\bD})$ as a $\cD(\bD)$-module.\end{lem}
\begin{proof} Recall the notation $\Upsilon = \bD \cap \Omega$. Let $j_0 : \Upsilon \hookrightarrow \bD$ be the open inclusion and let $\alpha_0$ be the restriction of $\alpha$ to $\bD$. Then we have the short exact sequence 
\[ 0 \to \cO_{\bD} \stackrel{\alpha_0}{\longrightarrow} j_{0,\ast} \cO_\Upsilon \to \coker (\alpha_0) \to 0\]
of locally Fr\'echet $\cD$-modules on $\bD$. By the definition of local cohomology that can be found at \cite[p.2]{HartLocCoh}, we have
\[\cN(\bD) = H^0_{\{0\}}(\bD, (\coker \alpha)_{|\bD}) = H^0_{\{0\}}(\bD, \coker \alpha_0).\]
Now, using \cite[Proposition 1.1(b)]{HartLocCoh}, we get the long exact sequence
\begin{equation}\label{LocCohLES} \xymatrix@R-2pc{ 0 \ar[r] & H^0_{\{0\}}(\bD, \cO_{\bD}) \ar[r] & H^0_{\{0\}}(\bD, j_{0,\ast} \cO_\Upsilon) \ar[r] & H^0_{\{0\}}(\bD, \coker \alpha_0) \ar[r] & \\
      \ar[r] & H^1_{\{0\}}(\bD, \cO_{\bD}) \ar[r] & H^1_{\{0\}}(\bD, j_{0,\ast} \cO_\Upsilon) \ar[r] & H^1_{\{0\}}(\bD, \coker \alpha_0).  & }\end{equation}
Note that this is a long exact sequence of $\cD(\bD)$-modules.  We claim that the middle terms in this long exact sequence vanish:
\begin{equation}\label{HiLocCohVanish} H^0_{\{0\}}(\bD, j_{0,\ast} \cO_\Upsilon) = H^1_{\{0\}}(\bD, j_{0,\ast} \cO_\Upsilon) = 0.\end{equation}
By \cite[Corollary 1.9]{HartLocCoh}, these terms appear in the long exact sequence
\[ \xymatrix@R-2pc{ 0 \ar[r] & H^0_{\{0\}}(\bD, j_{0,\ast} \cO_\Upsilon) \ar[r] & H^0(\bD, j_{0,\ast} \cO_\Upsilon) \ar[r] & H^0(\bD \backslash \{0\}, j_{0,\ast} \cO_\Upsilon) \ar[r] &\\
\ar[r] & H^1_{\{0\}}(\bD, j_{0,\ast} \cO_\Upsilon) \ar[r] & H^1(\bD, j_{0,\ast} \cO_\Upsilon) \ar[r] & H^1(\bD \backslash \{0\}, j_{0,\ast} \cO_\Upsilon). & } \]
The third arrow in the first row is $\cO(\Upsilon) \to \cO(\Upsilon \backslash \{0\})$, which is an isomorphism since $0 \notin \Upsilon$. This already implies that $H^0_{\{0\}}(\bD, j_{0,\ast} \cO_\Upsilon) = 0$, and that the second arrow in the second row is injective.  The five-term exact sequence of low degree terms associated the Leray spectral sequence applied to the map $j_0 : \Upsilon \to \bD$ and the sheaf $\cO_\Upsilon$ begins with $0 \to H^1(\bD,  j_{0,\ast} \cO_\Upsilon) \to H^1(\Upsilon, \cO_\Upsilon)$. But $\Upsilon$ is a quasi-Stein rigid analytic space, so $H^1(\Upsilon, \cO_\Upsilon) = 0$ by Kiehl's Theorem. This completes the proof of (\ref{HiLocCohVanish}). The connecting map in (\ref{LocCohLES}) now gives a $\cD(\bD)$-linear isomorphism 
\[ \cN(\bD) = H^0_{\{0\}}(\bD, \coker \alpha_0) \cong H^1_{\{0\}}(\bD, \cO_{\bD}). \qedhere\]
\end{proof}
\begin{lem}\label{inatV} We have $\dim_K \iota^\natural \cN = 1$. \end{lem}
\begin{proof} 
The definition of $\iota^\natural$ can be found in the statement of \cite[Theorem 3.4.17]{ArdEqKash}:
\[ \iota^\natural(\cM) = \mathpzc{Hom}_{\cO_{\{0\}}}\left(\Omega_{\{0\}} , \hspace{0.05cm}\iota^\natural\hspace{0.05cm}(\Omega_{\bP^1} \uset{\cO_\bP^1} \cM)\right)\]
for any object $\cM$ of $\cC^{\{0\}}_{\bP^1/B_0}$: first one side-switches $\cM$ to make a $B_0$-equivariant \emph{right} $\cD$-module on $\bP^1$, then one applies the pullback functor $\iota^\natural$ for equivariant right $\cD$-modules, and then one side-switches back again. The pullback functor for right $\cD$-modules $\iota^\natural$ is defined at \cite[Definition 3.4.13]{ArdEqKash}. We see that only the $\cO_{\bP^1}$-module structure on $\Omega_{\bP^1} \uset{\cO_\bP^1} \cM$ is used in the definition of this functor. Since $\cM$ is supported at $\{0\}$, and since $\Omega_{\bP^1}(\bD) = \cO(\bD) dx$ is a free $\cO(\bD)$-module of rank $1$, we see that for any $\cM \in \cC^{\{0\}}_{\bP^1/B_0}$ we have a $K$-vector space isomorphism
\[ \Gamma\left(\bD, \iota^\natural(\Omega_{\bP^1} \uset{\cO_\bP^1} \cM)\right) \cong \cM(\bD)[x] := \ker\left(x : \cM(\bD) \to \cM(\bD)\right).\]
Now, applying \cite[Corollary 1.9]{HartLocCoh} again to $\cO_{\bD}$, we have the long exact sequence
\[ \xymatrix@R-2pc{ 0 \ar[r] & H^0_{\{0\}}(\bD, \cO_{\bD}) \ar[r] & H^0(\bD, \cO_{\bD}) \ar[r] & H^0(\bD \backslash \{0\}, \cO_{\bD}) \ar[r] &\\
\ar[r] & H^1_{\{0\}}(\bD, \cO_{\bD}) \ar[r] & H^1(\bD, \cO_{\bD}) \ar[r] & H^1(\bD \backslash \{0\}, \cO_{\bD}). & } \]
Since $H^1(\bD, \cO_{\bD}) = 0$, Lemma \ref{VDcomp} now implies that as a $\cD(\bD)$-module, we have
\[\cN(\bD) \cong H^1_{\{0\}}(\bD, \cO_{\bD}) \cong \frac{\cO(\bD \backslash \{0\})}{\cO(\bD)}.\]
Now $\cO(\bD) = K \langle x \rangle$, whereas $\cO(\bD \backslash \{0\})$ is a subring of $K\langle x, x^{-1}\rangle$. Then we calculate
\[ K \cdot x^{-1} \subseteq \left(\frac{\cO(\bD \backslash \{0\})}{\cO(\bD)}\right)[x] \subseteq \left(\frac{K \langle x, x^{-1} \rangle}{K \langle x\rangle}\right)[x] = K \cdot x^{-1}.\]
So, $\iota^\natural \cN = \cN(\bD)[x] = \left(\frac{\cO(\bD \backslash \{0\})}{\cO(\bD)}\right)[x] = K \cdot \frac{1}{x}$ has $K$-dimension equal to $1$.\end{proof}

Let $H$ be an open subgroup of $G$; our goal will be now to better understand the restriction of $\coker \alpha$ to an object in $\cC_{\bP^1 / H}$. Using \cite[Lemma 2.2.1]{ArdEqKash}, choose a finite set of representatives $\{s_1,=1,\cdots,s_m\}$ for the $H,B$-double cosets in $G$. For each $i=1,\cdots, m$, recall from \cite[p.18]{ArdEqKash} the object
\[ \cN_i := \Res^{{}^{s_i}B}_{H \cap {}^{s_i}B} \left([s_i] s_{i,\ast} \cN\right) \in \cC_{\bP^1 / (H \cap {}^{s_i}B)}.\]
\begin{lem}\label{ResIndN} $\Res^G_H \left(\ind_B^G \cN\right) \cong \bigoplus\limits_{i=1}^m \ind_{H \cap {}^{s_i}B}^H \cN_i$ in $\cC_{\bP^1/H}$.
\end{lem}
\begin{proof} This follows from \cite[Lemma 2.3.7]{ArdEqKash}. \end{proof}

Note that $\cN$ is supported at $0$. For each $i=1,\cdots, m$, we then see that $\cN_i$ is supported at $x_i := s_i \cdot 0 \in \bP^1(F)$. Therefore
\[ \cN_i \in \cC^{x_i}_{\bP^1 / (H \cap {}^{s_i}B)} \qmb{and} \ind_{H \cap {}^{s_i}B}^H \cN_i \in \cC^{H x_i}_{\bP^1 / H}.\]

\begin{cor}\label{LengthResGHjO} Let $H$ be an open subgroup of $G$.
\be \item $\Res^G_H(\coker \alpha)$ has length equal to $m = |H \backslash G / B|$ in $\cC_{\bP^1/H}$. 
\item $\Res^G_H(j_\ast \cO_\Omega)$ has length $m + 1$ in $\cC_{\bP^1/H}$. 
\ee\end{cor}
\begin{proof} (a) Fix $i=1,\cdots, m$. Using Corollary \ref{IndEqApp} and Lemma \ref{ResIndN}, we see that it is enough to show that $\ind_{H \cap {}^{s_i}B}^H \cN_i$ is a simple object in $\cC^{H x_i}_{\bP^1 / H}$. Let $\iota_i : \{x_i\} \hookrightarrow \bP^1$ be the closed embedding; using \cite[Corollary 1.3.2]{ArdEqKash}, it is enough to show that $\iota_i^\natural \cN_i$ is simple in $\cC_{\{x_i\} / H \cap {}^{s_i}B}$. However, the action of $s_i$ induces an isomorphism $s_{i,\ast} (\iota^\natural \cN) \cong \iota_i^\natural \cN_i$, so Lemma \ref{inatV} implies that $\iota_i^\natural \cN_i$ is a $1$-dimensional $K$-vector space and is therefore simple.

(b) This follows from (a) together with Proposition \ref{OXCXG}(b).
\end{proof}

\begin{cor} \label{LengthO(Omega)} Let $H$ be an open subgroup of $G = \GL_2(F)$. Then $\cO(\Omega)$ is a coadmissible $D(H,K)$-module of length $|H \backslash \bP^1(F)| + 1$.
\end{cor}
\begin{proof}  The centre $Z$ of $G$ acts trivially on $j_\ast \cO_\Omega$. Hence we may view $\Res^G_Hj_\ast \cO_{\Omega}$ as a object in $\cC_{\bP^1/\overline{H}}$ where $\overline{H}$ is the open subgroup $HZ/Z$ of $GZ/Z = \mathbb{PGL}_2(F)$, and its length is still $|H \backslash \bP^1(F)| + 1$ by Corollary \ref{LengthResGHjO}(b). Now apply Theorem \ref{RigidEqBB} with $\bG = \mathbb{PGL}_2$ and the open subgroup $\overline{H}$ of $\bG(F)$ to deduce that $\Gamma(\bP^1, j_\ast \cO_{\Omega}) =\cO(\Omega)$ is a coadmissible $D(\overline{H},K)$-module of length $|H \backslash \bP^1(F)| + 1$. The result now follows by inflating back to $D(H,K)$.\end{proof}

Recall  that $G^0= \{ g \in \mathbb{GL}_2(F) : v_{\pi_F}(\det g) = 0 \}$. 

\begin{cor} \label{LengthL(Omega)} Let $H$ be an open subgroup of $G^0$ and let $\sL$ be a torsion $G^0$-equivariant line bundle with connection on $\Omega$ such that $\omega[\sL] = 0$ in $\PicCon(\Omega)$. Then $\sL(\Omega)$ is a coadmissible $D(H,K)$-module of length $|H \backslash \bP^1(F)| + 1$.
\end{cor}
\begin{proof}By \cite[Proposition 3.2.14]{ArdWad2023}, $\sL$ is isomorphic as a $G^0$-equivariant line bundle with connection to the twist $(\cO_{\Omega})_{\chi}$ of $\cO_{\Omega}$, for some $\chi \in \Hom(G^0,K^\times)$. Since $\sL$ is torsion, $\chi$ necessarily has finite order. Since twisting by a character is an equivalence of categories, we deduce from Corollary \ref{LengthO(Omega)} that in this case $\sL(\Omega)$ is a coadmissible $D(H,K)$-module of length $|H \backslash \bP^1(F)| + 1$. \end{proof}
We can finally give a proof of our two main results from the Introduction.
\begin{proof}[Proof of Theorem \ref{ThmA}] When $\omega[\sL] \neq 0$ in $\PicCon(\Omega)$, use Corollary \ref{LiftToG^0} together with Theorem \ref{RigidEqBB}. When $\omega[\sL] = 0$, use Corollary \ref{LengthL(Omega)}.
\end{proof}
\begin{proof}[Proof of Corollary \ref{ThmB}]   Applying \cite[Proposition 2.3]{Taylor2} we have a decomposition
\[ f_\ast \cO_X = \bigoplus_{\psi \in \widehat{\Gamma}} \sL_\psi\]
of the $G^0$-equivariant vector bundle with flat connection $f_\ast \cO_X$ into a direct sum of torsion $G^0$-equivariant line bundles with connection. Here $\widehat{\Gamma} = \Hom(\Gamma, K^\times)$ is the character group of $\Gamma$. Note that $[\sL_\psi] \in \PicCon^{G^0}(\Omega)_{\tors}$ for each $\psi \in \widehat{\Gamma}$. Hence
\[ \cO(X) = \bigoplus_{\psi \in \widehat{\Gamma}} \sL_\psi(\Omega).\]
By Corollary \ref{LiftToG^0}, the summand $\sL_\psi(\Omega)$ is a simple coadmissible $D(H,K)$-module whenever $\omega[\sL_\psi] \neq 0$, and by Corollary \ref{LengthL(Omega)}, it is a coadmissible $D(H,K)$-module of length $|H \backslash \bP^1(F)| + 1$ when $\omega[\sL_\psi]=0$. Writing $c_X := \left\vert\{\psi \in \widehat{\Gamma} : \omega[\sL_\psi] = 0\}\right\vert$, we deduce that
\[ \ell_{D(H,K)} \left(\cO(X)\right) = (|\widehat{\Gamma}| - c_X) + c_X \cdot (|H \backslash \bP^1(F)| + 1). \]
It follows from \cite[Corollary 3.1.7]{ArdWad2023} and the proof of \cite[Lemma 2.3.4]{Taylor3} that 
\[c_X = \dim_{\bfC} \cO(X_{\bfC})^{\cT(X_{\bfC}) = 0}.\] 
But because every connected component of $X_{\bfC}$ is geometrically connected and quasi-Stein, this dimension is just the number of connected components $|\pi_0(X_{\bfC})|$ of $X_{\bfC}$, by \cite[Proposition 3.1.6]{ArdWad2023}. The result follows because $|\widehat{\Gamma}| = |\Gamma|$.
\end{proof}

\bibliographystyle{plain}
\bibliography{references}
\end{document}